\documentclass[11pt,reqno]{amsart}
\usepackage{fullpage}
\usepackage[mathscr]{eucal}
\usepackage{mathrsfs}
\usepackage{amsfonts}
\usepackage{amsmath}
\usepackage{amsthm}
\usepackage{amssymb}
\usepackage{latexsym}
\usepackage[all]{xy}
\usepackage{array}

\usepackage{fullpage}

\usepackage{graphicx}

\usepackage[dvipsnames]{xcolor}
\newcommand{\redd}[1]{{\color{red}#1}}

\usepackage{mathtools}
\usepackage{palatino}

\usepackage[bookmarks=false]{hyperref}
\usepackage{hyperref}
\usepackage{xr-hyper}

\theoremstyle{plain}
\newtheorem{thm}[equation]{Theorem}
\newtheorem{cor}[equation]{Corollary}
\newtheorem{prop}[equation]{Proposition}
\newtheorem{lem}[equation]{Lemma}

\theoremstyle{definition}
\newtheorem{defn}[equation]{Definition}

\theoremstyle{remark}

\newtheorem{ex}[equation]{Example}
\newtheorem{rem}[equation]{Remark}
\newtheorem{rems}[equation]{Remarks}
\newtheorem{claim}[equation]{Claim}

\makeatletter
\renewcommand{\subsection}{\@startsection{subsection}{2}{0pt}{-3ex
plus -1ex minus -0.2ex}{-2mm plus -0pt minus
-2pt}{\normalfont\bfseries}} \makeatother

\numberwithin{equation}{subsection}
\allowdisplaybreaks

\newlength{\dhatheight}

\def\dmo{\DeclareMathOperator}
\dmo{\ddim}{{\mathbf{dim}}}

   \let\D\cD
   
\let\a\alpha  \let\b\beta  \let\la\lambda        \let\th\theta
\let\Im\undefined   \let\det\undefined
\dmo{\cHom}{\text{\textrm{\itshape{\cH}\kern-.2ex{}om}}}
\dmo{\cEnd}{\text{\textrm{\itshape{\cE}\kern-.2ex{}nd}}}    
\dmo{\cExt}{\text{\textrm{\itshape{\cE}\kern-.2ex{}xt}}}

\dmo{\REnd}{{\mathrm{REnd}}}\dmo{\Lie}{{\mathrm Lie}}
\dmo{\chark}{{\mathrm{char}}}
\dmo{\Ext}{{\mathrm{Ext}}}
\dmo{\Hom}{{\mathrm{Hom}}}
\dmo{\End}{{\mathrm{End}}}
\dmo{\Aut}{{\mathrm{Aut}}}
\dmo{\Sym}{{\mathrm{Sym}}}
\dmo{\det}{{\mathrm{det}}}
\dmo{\rk}{{\mathrm{rk}}}
\dmo{\Tr}{{\mathrm{tr}}}
\dmo{\Ker}{{\mathrm{Ker}}}
\dmo{\coker}{{\mathrm{Coker}}}
\dmo{\Im}{{\mathrm{Im}}}
\dmo{\gr}{{\mathrm{gr}}}
\dmo{\supp}{{\mathrm{supp}}}
\dmo{\ad}{{\mathrm{ad}}}
\dmo{\Ad}{{\mathrm{Ad}}}
\dmo{\Spec}{{\mathrm{Spec}}}
\dmo{\Mor}{{\mathrm{Mor}}}

\newcommand{\opn}{\operatorname}
\newcommand{\scr}{\mathscr }
\newcommand{\gen}{\reg}
\newcommand{\hdot}{{\:\raisebox{2pt}{\text{\circle*{1.5}}}}}
\newcommand{\idot}{{\:\raisebox{2pt}{\text{\circle*{1.5}}}}}

\newcommand{\GG}{\mathbb G }
\newcommand{\gS}{{\mathfrak S}}
\newcommand{\cF}{{\mathcal F}}
\newcommand{\cf}{{\mathcal F}}

\newcommand{\red}{_{\op{red}}}
\newcommand{\fF}{{\mathfrak F}}
\newcommand{\DD}{{\mathsf D}}

\newcommand{\cR}{{\mathcal R}}
\newcommand{\cC}{{\scr C}}
\newcommand{\cX}{{\mathcal X}}
\newcommand{\D}{{\mathcal  D}}

\newcommand{\cM}{{\mathcal  M}}
\renewcommand{\AA}{\mathbb A }
\newcommand{\bvt}{_\bv^\th}

\newcommand{\vth}{{\wt\theta }}
\newcommand{\ZZ}{{\mathbb Z}}
\newcommand{\iso}{{\;\stackrel{_\sim}{\to}\;}}
\newcommand{\cd}{\!\cdot\!}
\newcommand{\erem}{\hfill$\lozenge$\end{rem}}
\newcommand{\beq}{\begin{equation}\label}
\newcommand{\eeq}{\end{equation}}
\newcommand{\triv}{{\op{triv}}}
\newcommand{\FDN}{\mathrm{F}}
\newcommand{\sch}{{\mathsf{AffSch}}}
\newcommand{\bo}{\mbox{$\bigotimes$}}
\renewcommand{\o}{\otimes }

\newcommand{\leg}{_{\op{legs}}}

\newcommand{\bplus}{\mbox{$\bigoplus$}}
\newcommand{\ccong}{\ \cong \  }

\newcommand{\mto}{\mapsto }
\newcommand{\sset}{\subset }
\newcommand{\g}{{\fg}}

\newcommand{\stack}{{}^{\,{\textit{\tiny{stacky}}}\!}}
\newcommand{\fZ}{{\mathfrak Z}}

\renewcommand{\k}{\mathsf k} \newcommand{\fg}{\mathfrak{g}}
\newcommand{\Lmod}[1]{#1\text{-}{\mathsf{mod}}}
\newcommand{\ccc}{{\scr C}}

\newcommand{\CV}{{\mathcal V}}
\newcommand{\h}{{\mathfrak h}}
\newcommand{\cb}{{\mathcal B}}
\newcommand{\pt}{{\opn{pt}}}
\newcommand{\pr}{{\opn{pr}}}
\newcommand{\Vect}{{\opn{Vect}}}

\newcommand{\GL}{{\opn{GL}}}
\dmo{\Coh}{{Coh}}

\newcommand{\PGL}{{\opn{PGL}}}

\newcommand{\E}{{\mathsf E}}
\newcommand{\BR}{{\mathbb R}}
\newcommand{\pgl}{\mathfrak{pgl}}
\newcommand{\gl}{\mathfrak{gl}}  \newcommand\ot\leftarrow
\newcommand{\RGam}{\text{\upshape R}\Gamma}
\newcommand{\BQ}{{\mathbb Q}}
\newcommand{\lat}{{\mathbb K}}

\newcommand{\BI}{{\mathsf I}}
\newcommand{\fl}{{\mathcal{F}^{\!}\ell}}
\newcommand{\bbm}{{\Xi(\bm)}}

\newcommand{\stab}{{\text{-stab}}}
\dmo{\chr}{char}

\newcommand{\fx}{{\mathfrak X}}

\DeclareMathOperator{\Loc}{{\textit{Loc}}}
\newcommand{\f}{ factorization }
\newcommand{\wtc}{\wt\cC }

\newcommand{\into}{\hookrightarrow}
\newcommand{\onto}{\twoheadrightarrow}

\newcommand{\tv}{{\wt\bv}}
\newcommand{\vpi}{\varpi }
\let\x\times

\let\eps\varepsilon

\newcommand{\wcz}{{\wt{\mathcal Z}}}
\newcommand{\cz}{{\mathcal Z}}
\def\ccirc{{{}_{\,{}^{^\circ}}}}
\newcommand{\IC}{\op{IC}}
\newcommand{\inv}{^{-1}}
\newcommand{\en}{{\enspace}} 
\newcommand{\vi}{${\en\sf {(i)}}\;$}
\newcommand{\vii}{${\;\sf {(ii)}}\;$}
\newcommand{\viii}{${\sf {(iii)}}\;$}
\newcommand{\iv}{${\sf {(iv)}}\;$}
\newcommand{\vv}{${\sf {(v)}}\;$}
\DeclareMathOperator{\Id}{{\mathrm Id}}
\newcommand{\bpsi}{{\boldsymbol{\psi}}}
\newcommand{\bphi}{{\boldsymbol{\phi}}}
\newcommand{\extf}{{\Ext^1_{\cC,F}}}

\newcommand{\sminus}{\smallsetminus}

\newcommand{\bw}{{\mathbf w}}
\newcommand{\bv}{{\mathbf v}} 
   \newcommand{\ft}{{\mathfrak t}}
\newcommand{\tg}{{\wt\g}}
\newcommand{\fX}{{\mathfrak X}}

\newcommand{\Z}{{\mathbb Z}}

\newcommand{\X}{{\mathfrak X}}
\newcommand{\F}{{\mathbb F}}

\newcommand{\al}{{\alpha}}
\newcommand{\CC}{{\mathbb C}}
\newcommand{\C}{{\underline{\mathbf C}}}
\newcommand{\op}{\operatorname}
\newcommand{\too}{\,\longrightarrow\,}
\newcommand{\half}{\mbox{$\frac{1}{2}$}}

\newcommand{\fN}{{\mathfrak N}}
\newcommand{\TT}{{\op{T}}}
\newcommand{\FF}{{\mathfrak F}}

\newcommand{\ki}{{\k{I}}}
\DeclareMathOperator{\Exp}{{\mathrm{Exp}}}

\DeclareMathOperator{\Rep}{{\mathrm{Rep}}}
\DeclareMathOperator{\AI}{{\mathrm{AI}}}

\newcommand{\cur}{{C}}

\newcommand{\si}{\gS }

\newcommand{\be}{\beta }
\newcommand{\ltr}{{\mathsf{Tr}_{_\fr}}|\,}
 
\newcommand{\wt}{\widetilde }
\renewcommand{\aa}{{\scr A}}

\newcommand{\fr}{{\mathsf{Fr}}}
\newcommand{\sign}{{\op{sign}}}

\DeclareMathOperator{\Irr}{{\mathrm Irr}}

\DeclareMathOperator{\res}{{\mathfrak R}}
\newcommand\cv{\coprod_\bv}

\newcommand{\Ga}{\Gamma }
\newcommand{\Aff}{{\op{\textsl{Aff}}}}
\newcommand{\wh}{\widehat }
\newcommand{\FD}{{\mathsf{F}}}

\newcommand{\sq}{\sqcup }
\newcommand{\ii}{{i\in I}}
\renewcommand{\a}{{\mathfrak A}}
\newcommand{\dis}{\displaystyle}

\newcommand{\oo}{{\mathcal O}}

\newcommand{\ww}{{w}}

\newcommand{\cL}{{\mathscr L}}

\newcommand{\aad}{{\!}_{_{\Ad}}}

\newcommand{\ql}{{\mathcal Q}}
\newcommand{\abs}{_{\op{abs}}}
\newcommand{\kabs}{{{\mathsf K}^I\abs}}

\newcommand{\ce}{{\mathcal E}}

\newcommand{\ev}{\op{ev}}

\newcommand{\vphi}{\varphi }

\renewcommand{\L}{{\mathbb L}}
\newcommand{\Ups}{{\mathfrak{P}}}
\newcommand{\vp}{\varpi }
\renewcommand{\ss}{^{\mathrm{ss}}}

\newcommand{\ck}{{\cC_\k}}
\newcommand{\rf}{\cR^{F,\phi}}
\newcommand{\rr}{{\mathcal R}}

\newcommand{\fB}{{\mathfrak B}}

\newcommand{\VV}{{{\mathcal{V}\textit{ect}}}}
\newcommand{\fz}{{\mathfrak z}}
\newcommand{\mmod}{\op{\text{-}\rm{mod}}}

\newcommand{\nat}{_\flat }

\newcommand{\bmu}{{\boldsymbol{\mu}}}

\newcommand{\he}{^{^{_\textrm{\tiny disj}}}}
\newcommand{\hec}{^{\circ}}

\newcommand{\qlb}{{\bar{\mathbb Q}}_\ell }

\renewcommand{\b}{{\mathfrak b}}
\renewcommand{\u}{{\mathfrak u}}
\newcommand{\fu}{\u}
\newcommand{\fp}{{\mathfrak p}}
\newcommand{\br}{{\mathbf r}}

\newcommand{\slope}{\op{\textit{slope}}}
\newcommand{\hn}{\op{HN}}
\newcommand{\cat}{{\scr C}}
\newcommand{\bun}{\textit{Bun} }

\newcommand{\buna}{{\fX_r(\VV_\gam(C))}}

\newcommand{\hig}{\textit{Higgs}}
\newcommand{\higgs}{{\textit{Higgs}_{\br}}}
\newcommand{\ppar}{\mathcal{PB}\textit{un}}
\newcommand{\higgsd}{{\textit{Higgs}_{\br,d}}}

\newcommand{\hit}{{\textit{Hitch}_{r}(D)}}
\newcommand{\hitirr}{{\textit{Hitch}_{r}^{\op{red,irr}}}}

\newcommand{\ri}{\textit{Res}_i }
\newcommand{\reshit}{{\textit{res}}_{\textit{Hitch}}}
\newcommand{\reshiggs}{{\textit{res}}_{\textit{Higgs}}}
\newcommand{\V}{{\mathcal V}}
\newcommand\Nil{{\opn{Nil}}}

\newcommand{\cl}{{\mathcal L}}
\newcommand{\gam}{\gamma }
\newcommand{\kap}{\kappa }
\newcommand{\vk}{\varkappa }
\newcommand{\Si}{\Sigma }

\newcommand{\bd}{{\mathbf d}}
\newcommand{\bm}{{\mathbf m}}
\newcommand{\W}{{\mathcal W}}
\newcommand{\cc}{{\wt C}}

\renewcommand{\P}{{\mathsf P}}

\renewcommand{\ij}{_i^{(j)}}
\newcommand{\mm}{{\mathcal M}}
\newcommand{\reg}{^\diamond }

\newcommand{\bvtn}{_{\op{nil}}^\th}
\newcommand{\nil}{_{\op{nil}}}

\newcommand{\ssum}{\mbox{$\sum$}}
\newcommand{\pprod}{\mbox{$\prod$}}

\newcommand{\tpp}{^{\tiny\mbox{${}^{_{^{_+}}}$}}\!}

\newcommand{\VGpp}{{\VV^\Ga_{\bm}(\cc)}}
\newcommand{\vgampp}{{\VV^\Ga_{\bm,\gam}(\cc)}}

\newcommand{\dsl}{/\!\!/}
\newcommand{\cx}{\cX }

\newcommand{\pig}{{\pi_1^{et}}}
\newcommand{\pia}{{\pi_1^{\op{arith}}}}

\newcommand{\hhh}{{\mathscr H}}

\newcommand{\fw}{{\mathfrak w}}
\newcommand{\fa}{{\mathfrak a}}
\newcommand{\fb}{{\mathfrak b}}
\newcommand{\sig}{\sigma }
\newcommand{\ig}{ $I$-graded }
\newcommand{\cde}{{C_\Delta}}

\newcommand{\dct}{ decomposition  }
\newcommand{\N}{{\mathsf N}}

\title{Moduli spaces, indecomposable objects and potentials\\ over a finite field}
\author{Galyna Dobrovolska}
\author{Victor Ginzburg}
\author{Roman Travkin}
\begin{document}

\begin{abstract} Given an $\F_q$-linear  category such that
the moduli space of its objects is a 
smooth Artin stack  (and some additional conditions)
we give  formulas for an exponential sum over
the set of absolutely indecomposable objects and a stacky sum over the set of all objects of the category, respectively,
in terms of the geometry of the cotangent
bundle on the moduli stack.
The first formula was inspired by the work of 
 Hausel,  Letellier, and  Rodriguez-Villegas. It
provides a new approach for counting  absolutely
indecomposable
quiver representations, vector bundles with parabolic structure
on a projective curve, and
 irreducible \'etale local systems (via a result of Deligne). Our second formula 
resembles formulas appearing in the theory of Donaldson-Thomas invariants.
\end{abstract}
\maketitle
{\small
\tableofcontents
}

\section{Main results}

\subsection{Exponential sums}\label{exp-sec}
Fix a finite field $\k$ and an additive  character $\psi:\k\to
\qlb^\times$.
Let ${\scr C}$ be  a  $\k$-linear  
(not necessarily abelian)  Karoubian  category (so, $\cC$ has finite direct sums).
Let $[Ob\,\cC]$ denote the set of isomorphism classes of 
objects of $\cC$.
We  fix a function $\phi: [Ob\,\cC]\to\k$, usually referred to as a `potential'.
such that  $\phi(x\oplus y)=\phi(x)+
\phi(y)$ for all $x,y\in [Ob\,\cC]$.
We also fix a finite set $I$ and a $\k$-linear    functor  $F: x\mto F(x)=\oplus_\ii F_i(x)$, from $\cC$
 to the category of finite dimensional $I$-graded vector spaces 
such that $F$ does not kill nonzero objects.
 The $I$-tuple $\bv:=(\dim F_i(x))_\ii\in \Z^I_{\geq 0}$ is
called the  dimension (or dimension vector) of an object $x$ of $\cC$.
We assume that for any $\bv$, the set  $[Ob_\bv\,\cC]$, of isomorphism classes of 
objects of dimension $\bv$, is finite.
 In this paper, we are  interested in the following exponential sums:
\beq{E0}E_\bv({\scr C}, F,\phi,\psi)=\sum_{^{x\in[\op{Ob}_\bv({\scr C})]}}\
\psi(\phi(x)),\qquad
\stack E_\bv({\scr C}, F,\phi,\psi)=\sum_{^{x\in[\op{Ob}_\bv({\scr C})]}}\
\mbox{$\frac{1}{\#\Aut(x)}$}  \cdot \psi(\phi(x)),
\eeq
where  $\Aut(x)$ denotes the group of 
automorphisms of $x$.
Our main result provides, for certain interesting classes of triples $(\cC,F,\phi)$,
a formula  for each of these sums in geometric terms.

In order to use geometry, we assume that our category $\cC$
has a well-behaved
 moduli space  of objects in the sense 
that there exists a {\em smooth} Artin   stack $\X(\cC)$, over $\k$,
and an identification of  groupoid $\X(\cC)(\k)$, of closed $\k$-points of $\X(\cC)$,
with the groupoid of objects of $\cC$.
Furthermore, we assume that the data discussed in the previous paragraph is
compatible with the stack structure as follows. 
We assume   that one has a
direct sum  morphism   $\oplus: {\mathfrak X}({\scr C})\times {\mathfrak X}({\scr
  C})\to
{\mathfrak X}({\scr C})$.
Let $pr_j: \X(\cC)\times\X(\cC)\to\X(\cC),\ j=1,2$, denote the projection to the $j$-th factor.
We assume that the potential comes from a regular function
${\mathfrak X}({\scr C})\to\AA^1$, to be also denoted by $\phi$ again,
which  is additive in the sense that  one has $\oplus^*\phi=pr_1^*\phi+ pr_2^*\phi$.
Further, 
we assume given, for each $\ii$, a   locally  free
sheaf $\fF_i$ on $\fx(\cC)$ equipped with an additivity isomorphism:\
\ $\oplus^*\fF_i\cong pr_1^*\fF_i\oplus pr_2^*\fF_i$ and,
 for all $x\in \X(\cC)(\k)$, with an isomorphism of
the geometric fiber of $\fF_i$ at $x$ with $F_i(x)$.
Finally, we require that the direct sum morphism  and the
additivity isomorphism satisfy appropriate 
 associativity constraints and are compatible with 
the natural isomorphisms $F(x\oplus y)\cong F(x)\oplus F(y)$.

There are two  classes of examples of the above setting.

Examples of the first class, to be discussed in more detail in \S\ref{par-sec},
  are categories of parabolic bundles on a smooth geometrically connected projective curve
$C$ over $\k$ with  a fixed set  of $\k$-rational marked points.
The functor $F$ assigns to a parabolic bundle 
the direct sum
 of its fibers over the marked points.
In the most basic special case of the category of vector bundles (without parabolic structure), 
one chooses an arbitrary $\k$-rational point $c\in C$ and let the functor
$F$ send a vector bundle to its fiber over $c$.
There are  various  modifications of
the above setting.
For example,  
one can consider  categories of  framed  coherent  sheaves on a projective variety, 
cf. \cite{MR} and \cite[Example 3.7]{Me}, in which case the functor $F$ is  taken to be the functor
of global sections. In those cases
It is also interesting to consider the category of locally projective
modules over a fixed hereditary order on $C$, cf. \cite{CI}.
In all these cases, there are no nonconstant potentials ~$\phi$.


Examples of categories of
the second class   are the categories $\Lmod A$ of finite dimensional
modules over a smooth  associative $\k$-algebra $A$. Recall that $A$ is called {\em  smooth} if  the kernel of the map
$A\o A\to A$, given by multiplication,
 is  finitely generated and projective as an   $A$-{\em bi}module. It is known that this condition 
insures that the stack $\fX(\Lmod A)$ is smooth.
Path algebras of quivers and coordinate rings of smooth
affine curves are examples of smooth algebras. Any localization
of a smooth algebra, eg. a multiplicative version of a path algebra,
is smooth. We let $F$ be the forgetful functor that
assigns to an $A$-module $M$  the underlying $\k$-vector space.
Given an element  $a\in A/[A,A]$, there is an associated
potential $\phi_a: {\mathfrak X}(\Lmod A)\to\AA^1,\ M\mto \phi_a(M):=\Tr(a,M)$, the trace of the $a$-action in $M$.
In the case where $A$ is the path algebra of a quiver $Q$,
we get  the category of finite dimensional representations of $Q$; here,
the  finite set $I$ is the
vertex set of $Q$ and an element of $A/[A,A]$ is a linear combination
of oriented cycles, i.e. cyclic paths, in $Q$.

\begin{rem}\label{stab-rem} 
The category  $\Lmod A$ is abelian and the forgetful functor is faithful.
On the contrary, categories from the first class,
such as the category of vector or parabolic bundles, 
are only quasi-abelian (usually not abelian) and the functor that sends a vector bundle
to a direct some of finitely many fibers is, typically, not faithful.
\erem

An important additional source of categories
where  our results are applicable
comes from considering stability conditions on 
 quasi-abelian categories.
In more detail, let $\cC$ be  a  quasi-abelian category
equipped with a  stability condition,
cf. \cite{An}. Any object   $x\in\cC$ has a canonical Harder-Narasimhan
filtration $x=x^0\supset x^1\supset\ldots$, of finite length.
Given  a segment $S\sset\BR$ (or a sector in the upper half plane, 
in the framework of Bridgeland stability)
let $\cC_S$ be  a full subcategory of $\cC$ whose objects
are  the  objects $x\in \cC$
such the slope of $x^j/x^{j+1}$ is contained in
$S$ for all $j$. The moduli stack $\fx(\cC_S)$ is easily seen,
see e.g. \cite{KS2}, to be an open substack
of $\fx(\cC)$. Hence,  $\fx(\cC_S)$ is a smooth Artin stack whenever so is
 $\fx(\cC)$.
Thus, starting with a data $(\cC,F,\phi)$, where $\cC$ is
either the category of modules over a smooth algebra or 
the category of parabolic bundles, the results of this paper  apply
to the
 triple $(\cC_{S},F,\phi)$,  for any choice of 
 stability condition on $\cC$ and a segment ~$S$.
\smallskip

Fix a triple $(\cC,F,\phi)$ as above. For each dimension vector $\bv=(v_i)_\ii$, let  $\fx_\bv({\scr C})$ be a substack of $\fx({\scr C})$
 such that the restriction of 
$\fF_i,\ \ii$, to  $\fx_\bv({\scr C})$  is a vector bundle of a fixed rank $v_i$.
Thus, one has
 a  decomposition
$\fx({\scr C})=\sqcup_{\bv\in \Z^I_{\geq 0}}\ \fx_\bv({\scr C})$.
We consider a stack $\fx_\bv(\cC,F)$ of `framed objects' whose closed
$\k$-points are pairs
$(x,b)$, where $x$ is an object of $\cC$ of dimension $\bv$
and $b$, the `framing',  is a collection $(b_i)_\ii$ where $b_i$ is a $\k$-basis of the vector
space $F_i(x)$, cf. \S\ref{frame-sec} for a complete definition. 
The group $\GL_\bv=\prod_i\ \GL_{v_i}$ acts on  $\fx_\bv(\cC,F)$
by changing the framing. 
One can show that the 1-dimensional torus $\GG_m\sset \GL_\bv$,
of scalar matrices, acts trivially  on  $\fx_\bv(\cC,F)$, so the  $\GL_\bv$-action factors through the
group 
${\PGL}_\bv=\GL_\bv/\GG_m$.

\begin{rem}  Forgetting the framing gives a $\GL_\bv$-torsor
 $\fx_\bv(\cC,F)\to \fx_\bv(\cC)$.
There are canonical isomorphisms $\fx_\bv(\cC,F)/\GL_\bv=\fx_\bv(\cC)$
and $\fx_\bv(\cC,F)/\PGL_\bv=\fx_\bv(\cC)/{B}\GG_m$.
Here,  ${B}\GG_m=\pt/\GG_m$,
the classifying stack, is viewed as the groupoid of 1-dimensional vector spaces equipped with an operation
of tensor product. Tensoring with a  1-dimensional vector space is a well defined functor on  a $\k$-linear category,
and this gives a natural action of $B\GG_m$ on $\X(\cC)$.
The stack $\fx_\bv(\cC,F)/\GL_\bv$ is a $\GG_m$-gerbe over $\fx_\bv(\cC,F)/\PGL_\bv$. 
One has  $\dim \fx_\bv(\cC,F)/\PGL_\bv=\dim \fx_\bv(\cC)/{B}\GG_m=\dim \fx_\bv(\cC) +1$.
\erem


We need some notation.
Write $\si_v$ for the Symmetric group on $v$ letters,
so  $\si_\bv=\prod_i \si_{v_i}$ is the Weyl group of ${\PGL}_\bv$.  Let $\Irr \si_\bv$ be the set
of isomorphism classes of irreducible $\si_\bv$-representations and
let $\op{sign}$, resp.   $\op{triv}$, 
denote the sign, resp. trivial, character.
Let $\ft_\bv$ be the Cartan subalgebra of $\pgl_\bv$ formed by diagonal matrices.
By the  Chevalley isomorphism, we have $\pgl_\bv^*\dsl\PGL_\bv=\ft_\bv\reg\dsl \si_\bv$, where $\pgl_\bv^*\dsl\PGL_\bv$ denotes the categorical quotient.
We will define a $\si_\bv$-stable Zariski open subset $\ft_\bv\reg$ of $\ft_\bv^{\op{reg}}$,
see Definition \ref{verygen},   where $\ft_\bv^{\op{reg}}\sset\ft_\bv^*$ is the complement of root hyperplanes.

Let $\TT^*\fx_\bv(\cC,F)$ be
the cotangent stack of $\fx_\bv(\cC,F)$ and  $\pgl_\bv=\Lie {\PGL}_\bv$.
Associated with the $\PGL_\bv$-action on $\TT^*\fx_\bv(\cC,F)$, there is a moment map
$ \TT^*\fx_\bv(\cC,F)\to \pgl_\bv^*$.
This map is $\PGL_\bv$-equivariant, hence descends to a map $\mu:\, \TT^*\fx_\bv(\cC,F)/\PGL_\bv\to \pgl_\bv^*/\PGL_\bv$,
of stacky quotients  by $\PGL_\bv$. Further, the  pull-back of the potential $\phi$ via the projection $\fx_\bv(\cC,F)\to \fx_\bv(\cC)$ is clearly a $\PGL_\bv$-invariant function.
Hence, this function descends to a function $\fx_\bv(\cC,F)/\PGL_\bv\to\AA^1$, to be denoted by $\phi$ again.
We obtain the following diagram
\beq{f*}
\ft_\bv\reg\times_{\pgl_\bv\dsl \PGL_\bv} \pgl_\bv/\PGL_\bv\
\xrightarrow{pr}\
\pgl_\bv^*/\PGL_\bv\  
\xleftarrow{\ \mu\ }\ 
\TT^*\fx_\bv(\cC,F)/\PGL_\bv\ \xrightarrow{q} \  \fx_\bv(\cC,F)/\PGL_\bv\
\xrightarrow{\phi}\  
\AA^1.
\eeq
Here,  $pr$ is  the second projection and   $q$ is induced by the vector bundle 
projection $\TT^*\fx_\bv(\cC,F)\to \fx_\bv(\cC,F)$.

Write   ${\mathcal{AS}}_\psi$ 
for the Artin-Schreier local system
on $\AA^1$ associated with  the additive character $\psi$ and $\gr_{\mathsf{W}}(-)$ for an associated graded with respect to the weight filtration.

\begin{prop} Let $(\cC,F,\phi)$ be a triple satisfying all the conditions above.
Then,  for any dimension vector $\bv$, the constructible complex $\gr_W(pr^*\mu_!\phi^*{\mathcal{AS}}_\psi)$
is a constant sheaf
on $\ft_\bv\reg\times_{\pgl_\bv\dsl \PGL_\bv} \pgl_\bv/\PGL_\bv$.
\end{prop}

The image of  the map $pr$ in \eqref{f*} has the form  $\pgl_\bv\reg/ \PGL_\bv$, where $\pgl_\bv\reg$ is a  $\PGL_\bv$-stable   Zariski  open 
and dense subset of the set  of 
regular semisimple
elements of $\pgl_\bv^*\cong\pgl_\bv$. Therefore,  $pr$ is a Galois covering with Galois group $\si_\bv$ and  the sheaf 
$\mu_!\phi^*{\mathcal{AS}}_\psi\big|_{\pgl_\bv\reg\PGL_\bv}$ is a locally constant sheaf on  $\pgl_\bv\reg/ \PGL_\bv$,
by the above proposition.
A $\k$-rational coadjoint orbit $O\sset \pgl_\bv\reg$ gives a point of
$(\pgl_\bv\reg/\PGL_\bv)(\k)$ and one has  a
short exact sequence 
\[1\to \pia(\ft_\bv\reg\times_{\pgl_\bv/\PGL_\bv} \pgl_\bv/\PGL_\bv)
\too \pia(\pgl\reg_\bv/\PGL_\bv)\xrightarrow{u}\si_\bv\to1\]
where the group in the middle is the arithmetic fundamental group of $\pgl\reg_\bv/\PGL_\bv$
at  $O\in \pgl_\bv\reg/\PGL_\bv$.
 The stack $\mm_O=\mu\inv(O)$ may be viewed as
 a stacky Hamiltonian reduction with respect to the moment map $\TT^*\fx_\bv(\cC,F)\to\pgl^*_\bv$.
Thanks to the proposition above, 
 the monodromy action of the subgroup $\Ker(u)\sset \pia(\pgl\reg_\bv/\PGL_\bv)$ on the cohomology
$H_c^\hdot(\mm_O,\, \phi^*{\mathcal{AS}}_\psi)=
H_c^\hdot(\mu_!\phi^*{\mathcal{AS}}_\psi|_O\big)$ is unipotent.
It follows that
for any $\rho\in\Irr \si_\bv$ one has a well-defined
space $H^\hdot_c(\mm_O,\, \phi^*{\mathcal{AS}}_\psi)^{\langle \rho\rangle}$, the generalized
isotypic component of  $\rho$ under the monodromy action.

To state our   main result  about exponential sums
it is  convenient to package the sums  in \eqref{E0} into generating series.
To this end, let $(z_i)_\ii$ be an $I$-tuple of indeterminates and write
 $z^\bv:=\prod_{i\in I} z_i^{v_i}$. Put $\bv\cdot\bv=\sum_\ii v_i^2$ and $|\bv|=\sum_\ii v_i$. 
We form the following generating series
\[E({\scr C}, F,\phi)=\sum_\bv\ z^\bv E_\bv({\scr C}, F,\phi),\qquad
\stack E({\scr C}, F,\phi)=\sum_\bv\ 
(-1)^{|\bv|}\cdot z^\bv\cdot q^{\frac{\bv\cdot\bv}{2}}\cdot \stack E_\bv({\scr C}, F,\phi).
\]

We will assume that  for each $\bv$ the stack
$\fx_\bv(\cC)$ has finite type.
It follows that the set $[\op{Ob}_\bv{\scr C}]:=[\fx_\bv({\scr C})(\k)]$  is finite.
We put $d_\bv=\dim\X_\bv(\cC)+1+\frac{\bv\cdot\bv-|\bv|}{2}+1$.
Let $\ltr H^\hdot_c(\mm_O,\phi^*{\mathcal{AS}}_\psi)^{\langle\rho\rangle}$ denote
the alternating sum of traces of the Frobenius. Finally, write
$\Exp$ for the plethystic exponential. 

Then, we

\begin{thm}\label{A-ind} 
Let $\k=\F_q$, let $O\sset \pgl_\bv\reg$ be a 
coadjoint orbit defined over $\k$,
and  $\mm_O:=\mu_\bv\inv(O)/{\PGL}_\bv$. Then, 
one has  the following equations:
\begin{align}
E({\scr C},
F,\phi)\ &=\ \Exp\left(\sum_{^{\bv>0}}\ 
z^\bv\cdot q^{-d_\bv}\cdot 
\ltr H^\hdot_c(\mm_O, \phi^*{\mathcal{AS}}_\psi)^{\langle\op{sign}\rangle}\right)\label{thm11}\\
\stack E({\scr C}, F,\phi)\ &=\
\Exp\left(\frac{1}{q^{\frac{1}{2}}-q^{-\frac{1}{2}}}\cdot\sum_{^{\bv>0}}\ 
z^\bv\cdot q^{-d_\bv}\cdot
\ltr H^\hdot_c(\mm_O, \phi^*{\mathcal{AS}}_\psi)^{\langle\op{triv}\rangle}\right).\label{thm1stack}
\end{align}
\end{thm}

\begin{rems}\label{bd} \vi The quantity  $\frac{\bv\cdot\bv-|\bv|}{2}$ is the dimension of the flag variety for the group
$\PGL_\bv$. 

\vii We have 
$\ltr H^\hdot_c(\fx_\bv({\scr C})/B\GG_m,\qlb)=(q-1)\cdot \ltr H^\hdot_c(\fx_\bv({\scr C}),\qlb)$
and $\dim\X_\bv(\cC)+1=\dim \X_\bv(\cC)/B\GG_m$.
In many cases, including all the examples considered in this paper,
one has an equality  $d_\bv=\frac{1}{2}\dim\mm_O$.
Also,  the stack $\mm_O$ in Theorem \ref{A-ind}  is, typically, a smooth scheme,
cf.  ~\S\ref{quiv-sec}.

\viii
The equations of Theorem \ref{A-ind} have motivic nature. In particular, there are counterparts of these
equations in the case where $\k=\CC$, the field of complex numbers.
In that case, the natural analogue of  the function $\psi\ccirc \phi$
is  the exponential function $e^{\phi}$. 
The analogue of the exponential sum  $\stack E(\cC,F,\phi)$ is an element $\frac{[\![\X_\bv(\cC),\, e^\phi]\!]}{[\![\PGL_\bv]\!]}$,
where $[\![\PGL_\bv]\!]$ is the motive of $\PGL_\bv$, \ $[\![\X_\bv(\cC), e^\phi]\!]$ is a class in
the Grothendieck 
group of the category of
{\em exponential monodromic mixed Hodge structures} defined by Kontsevich and Soibelman \cite{KS2},\
 and the fraction is an element of an appropriate localization of the
Grothendieck group. It will become clear from section \ref{iner} that the correct  analogue of the exponential sum $E(\cC,F,\phi)$ 
is the element $\frac{[\![I(\X_\bv(\cC)),\, e^\phi]\!]}{[\![\PGL_\bv]\!]}$, where $I(\X_\bv)$ is the inertia stack of $\X_\bv(\cC)$.
Finally, one can define a Hodge theoretic
counterpart,
$[\![\mm_O,\, e^\phi]\!]^{\langle\rho\rangle}$,  of the  quantity $\ltr H_c^\hdot(\mm_O, \wp^\phi)^{\langle\rho\rangle}$.

\vv
The motivic counterpart of formula \eqref{thm1stack} with $\phi=0$ implies, in a special case  where $\cC$ is the category of representations of a 
symmetric quiver and zero potential, a special case of
the `positivity and integrality conjecture' for DT-invariants due to  Kontsevich and Soibelman \cite{KS2}.
This special case of the conjecture has been  proved in \cite{HLV1}
(cf. also  \cite{DM} in the general case).
We believe that  formula \eqref{thm1stack} might be relevant for understanding the  conjecture in full generality. 
The essential point is that, in the case of quiver representations, the stack $\mm_O$ is
a scheme.
In such a case, formula \eqref{thm1stack}  indicates
that every  coefficient in the motivic counterpart  of the formal power series $\op{Log}(\!\stack E(\cC,F,\phi))$
 is an exponential motive 
of a {\em scheme} rather than a general stack.
\end{rems}

\subsection{Counting absolutely indecomposable objects}
For any finitely generated commutative $\k$-algebra   $K$   one defines
${K}\o_\k\cC$, the base change category,
 as a $K$-linear category whose objects are pairs $(x,\al)$, where $x$ is an object of
$\cC$ and
$\al: K\to \End(x)$ is an algebra homomorphism,
to be thought
of as a `$K$-action on $x$'. 
A morphism $(x,\al)\to (x',\al')$, in $K\o_\k\cC$ is, by definition, a morphism $u: x\to x'$ in $\cC$
such that $\al'(a)\ccirc u=u\ccirc \al(a)$ for all $a\in K$.
By construction, the algebra
$K$ maps to the center of the category ${K}\o_\k\cC$.
It is easy to show that if  the category $\cC$ is  Karoubian, resp. exact,
then so is ${K}\o_\k\cC$.

In the case where $K$ is a finite field extension of $\k$,
the natural forgetful functor ${K}\o_\k\cC\to\cC$ is known to have
a left adjoint functor $\cC\to{K}\o_\k\cC,\ x \to K\o_\k x$, see \cite[\S 2]{So}.


\begin{defn} An object $x$ of $\cC$ is  called {\em absolutely indecomposable}
if $K\o_\k x$ is an indecomposable object of ${K}\o_\k\cC$,
for any  finite field extension $K\supset \k$.
\end{defn}

 It turns out that the generating series $E({\scr C}, F,\phi)$
can be expressed in terms of
 absolutely indecomposable objects.
Specifically, write 

Let $\op{AI}_\bv({\scr C})$, or $\op{AI}_\bv({\scr C}, \k)$, be the subset  of
$[\op{Ob}_\bv({\scr C})]$ formed by the isomorphism classes of   absolutely indecomposable
objects.
 Let
 $E_\bv^{\op{AI}}({\scr C}, F,\phi)$
be an  exponential sum similar to $E_\bv({\scr C}, F,\phi)$, see \eqref{E0}, where summation
over   $[\op{Ob}_\bv({\scr C})]$ is replaced by summation
over the set ~$\op{AI}_\bv({\scr C})$.
It follows from  a generalization of Hua's lemma \cite{H} that one has
$E({\scr C}, F,\phi)=\op{Exp}\big(E_\bv^{\op{AI}}({\scr C}, F,\phi)\big)$
Therefore,  formula  \eqref{thm11} yields the following result.

\begin{thm}\label{thm1} Let $\k=\F_q$. Then, we have
\[E^{\AI}({\scr C}, F,\phi)=q^{-d_\bv}\cdot
\ltr H^\hdot_c(\mm_O, \wp^{\phi_\bv})^{\langle\op{sign}\rangle}.
\]
\end{thm}

In the special case $\phi=0$, the local system $\wp^\phi$ reduces
to a constant sheaf and Theorem  \ref{thm1}
yields  a formula for the number of isomorphism
classes of absolutely indecomposable objects:

\beq{ai-intro}\# \op{AI}_\bv(\cC,\F_q)\ =\
q^{-d_\bv}\cdot \ltr
H^\hdot_c(\mm_O,\qlb)^{\langle\op{sign}\rangle}.
\eeq

Fix a finite set $I$ and write $\ki $ for the
algebra of  functions $I\to\k$, with pointwise
operations. Let $A$ be a $\ki$-algebra,
i.e.  an associative $\k$-algebra
equipped with  a $\k$-algebra homomorphism
 $\ki\to A$ (with a not necessarily central image).
Let $\cC=\Lmod{A}$ and let $F$ be the forgetful
functor. Then,
the  stack  $\fx_\bv(\cC,F)$ is nothing but the  representation  scheme
 $\Rep_\bv(A)$, so we have $\fx_\bv(\cC)= \Rep_\bv(A)/\GL_\bv$, the moduli stack of
$\bv$-dimensional $A$-modules. The assumption that $A$ be smooth
implies that  $\Rep_\bv(A)$ is a smooth scheme. Hence,
 $\fx_\bv(\cC)$ is a smooth Artin stack of dimension 
$\dim\big(\Rep_\bv(A)/{\GL}_\bv\big)=\dim\Ext^1(x,x)-\dim\End(x)$, for
any $x\in \fx_\bv(\cC)$.

To a large extent, our paper grew out of our
 attempts to better understand a remarkable
proof of Kac's positivity conjecture by T. Hausel, E. Letellier, and F. Rodriguez-Villegas, \cite{HLV1}.
The present paper provides, in particular, 
a new proof of  the main result of \cite{HLV1}.
Our approach  is quite different
from the one used in \cite{HLV1}. 

In more detail, let $\cC$ be  the category of
finite dimensional representations a quiver $Q$, equivalently,  the category
$\Lmod{A}$ where $A$ is the path algebra of $Q$.
In that case the variety $\mm_O$ becomes in this case a certain quiver variety.
Our  formula \eqref{ai-intro}  is then equivalent to the main result of \cite{HLV1},
modulo the assertion,  proved   in  \cite{HLV2}, that the 
cohomology of quiver  varieties is pure and Tate, i.e.
 $H_c^{\text{odd}}(\mm_O,\qlb)=0$ and the Frobenius acts
on $H^{2j}_c(\mm_O,\qlb)$ as a multiplication by $q^j$. In particular,
we have that $\ltr H_c^\hdot(\mm_O,\qlb)=\sum_{j\geq 0}\ q^j\cdot \dim H^{2j}_c(\mm_O,\qlb)$
is a polynomial in $q$ with non-negative integer coefficients, proving the Kac's positivity conjecture.
 \begin{rem} The fact that the cohomology of the quiver  variety is pure and Tate
also follows, using a deformation argument, cf. \S\ref{pure-sec},  from a much more general  result 
of Kaledin \cite{Ka} that says that rational cohomology groups of an arbitrary symplectic resolution
are generated by algebraic cycles.
\end{rem}
\subsection{Deformation to the nilpotent cone}\label{nilp-sec}
The RHS of formulas \eqref{thm11}-\eqref{thm1stack} involve an arbitrarily chosen coadjoint orbit $O\sset\pgl\reg_\bv$
while  the LHS  of the formulas are independent of that choice. Therefore, it is tempting to find a replacement of the cohomology of the
stack
$\mm_O$ by a more canonical object.    Below, we explain how to do this 
in the case   of quiver representations (and zero potential)
or the case of parabolic bundles. In those cases, 
one can  choose a stability condition $\th$ and consider  the substack of  $\th$-semistable 
objects of the stack $\TT^*\fX_\bv(\cC,F)/\PGL_\bv$ and the corresponding
 coarse moduli space   $\mm^\th$,  a quasi-projective variety 
whose closed $\bar\k$-points are polystable objects.
In the  case where $\cC$ is the category  of  representations of a quiver $Q$,
the coarse moduli space  is a GIT quotient $\mm^\th=\Rep_\bv \bar Q\dsl_\th \PGL_\bv$,
where $\Rep_\bv \bar Q$ is the representation scheme of the double of the quiver $Q$.
In the  case of parabolic bundles the corresponding  coarse moduli space is
a well studied variety   $\mm^\th$  of semistable parabolic Higgs bundles,  cf.  \cite{Y}.

In the above setting, the GIT counterpart of the moment $\TT^*X_\bv(\cC,F)\to\pgl_\bv^*$
is a morphism
 $f^\th:\ \mm^\th\to\pgl_\bv^*\dsl \PGL_\bv$, where
$\pgl_\bv^*\dsl \PGL_\bv$  is a categorical quotient.
An orbit $O\sset \pgl_\bv^*$ maps to a point
$\eta_O\in\g^*\dsl \PGL_\bv$.
The result below  expresses  the quantity in the RHS of
\eqref{ai-intro}  in terms of  a `limit'  of the cohomology of the generic fiber $(f^\th)\inv(\eta)$
as $\eta\to0$, that is, the regular semisimple orbit $O$ degenerates to the nilpotent cone $\Nil\sset\pgl_\bv^*$.



The dimension vector $\bv$ will be fixed throughout this subsection,
so we drop it from the notation and write $G=\PGL_\bv$, resp. $\g=\pgl_\bv$.  
Thus, $\ft$ is the Cartan subalgebra of
$\g$ and $W=\si_\bv$ is the Weyl group.
Let  $\mm\gen:=(f^\th)\inv(\g\gen\dsl G)$.
This is a Zariski  open and dense subset
contained in the stable locus of  $\mm^\th$.
Further, let $\mm\bvtn=(f^\th)\inv(0)$,  a closed subvariety  of $\mm^\th$.
We also consider another  variety, $\mm^\th_{0}$,  a closed subset of  $\mm\bvtn$.
In the case of quiver representations, $\mm^\th_{0}$ is defined as the image in
$\Rep_\bv \bar Q\dsl_\th \GL_\bv$ of the $\th$-semistable locus of the zero fiber of the moment map
$\Rep_\bv \bar Q\to \g^*$.
In the case of parabolic bundles, the variety $\mm^\th_{0}$ is formed by the Higgs bundles
such that the Higgs field has no poles, equivalently, its residue at every marked point vanishes.
Thus,  one has
a commutative diagram:
\beq{bs1}
\xymatrix{
\mm^\th_{0}\ar@{^{(}->}[r]&\mm\bvtn\
\ar[d]^<>(0.5){f^\th}\ar@{^{(}->}[r]^<>(0.5){\imath}
\ar@{}[dr]|{\Box}&
\ \mm^\th\  \ar[d]^<>(0.5){f^\th}\ar@{}[dr]|{\Box}&
\ \mm\gen\
\ar@{_{(}->}[l]_<>(0.5){\jmath}\ar[d]^<>(0.5){f^\th}\ar@{}[dr]|{\Box}& \
\mm\gen\times_{\ft\gen\dsl  W}\ft\gen\ar@{->>}[l]_<>(0.5){pr}\ar[d] \\
&\{0\}\ \ar@{^{(}->}[r]&\ \g^*\dsl G \ &\ \ft\gen\dsl  W\ \ar@{_{(}->}[l]
\ &\ \ft\gen\ \ar@{->>}[l]}
\eeq
Here, $\imath$, resp. $\jmath$, is a closed, resp.  open, imbedding, and all squares in the diagram are cartesian.

The map $pr$, the first projection, in the above diagram is an unramified covering, so one has a direct sum
decomposition
$\dis
pr_!\qlb\
=\
\oplus_{\rho\in \op{Irr}( W)}\ \
\rho\o \cl_\rho,
$
where $\op{Irr}( W)$ is the set of isomorphism classes of
irreducible representations of the group $ W$,\
 $\cl_\rho$ is an irreducible local system on $\mm\gen$
and $\qlb$ stands for a constant sheaf on $\mm\gen\times_{\ft\gen\dsl  W}\ft\gen$.
Let $\IC(\cl_\rho)$ be the  IC-extension of  $\cl_\rho$ (here, we use an unconventional normalization
such that $\IC(\cl_\rho)|_{\mm^\th}=\cl_\rho$ is placed in homological degree zero
rather than $-\dim\mm^\th$).  
The following result will be proved in section \S\ref{quiv-sec}.

\begin{thm}\label{la-thm} Let $\cC$ be either the category  $\Lmod{A}$
for a smooth algebra $A$
or the category of parabolic bundles. Then, for any stability condition $\th$  we have

\vi For any $\rho\in\op{Irr}( W)$, the sheaf $\imath^*\IC(\cl_\rho)$  is a semisimple perverse
sheaf on $\mm\bvtn$, moreover, 
 the cohomology $H^\hdot_c(\mm\bvtn, \IC(\cl_\rho))$ is pure. 

\vii The support of the sheaf $\imath^*\IC(\cl_{\op{sign}})$ is 
contained in $\mm^\th_{0}$. Furthermore, if $\th$ is   sufficiently general then,
 in the `coprime case',  the map  $f^\th$ is a
 smooth morphism  and 
$\ \imath^*\IC(\cl_{\op{sign}})\ $ is a constant sheaf on
$\mm^\th_{0}$, up to shift.

\viii For any $O\sset \g\gen$, the scheme  $(f^\th)\inv(\eta_O)$ is smooth
and one has an isomorphism $\mm_O\cong (f^\th)\inv(\eta_O)$, of stacks.

\iv If $\cC$ is either the category of quiver representations  or the category of parabolic bundles
then the
monodromy action on $H^\hdot_c(\cM_O, \qlb)$ factors through a $ W$-action
and one has a canonical isomorphism
\[H^\hdot_c(\mm_O, \qlb)^{\rho}\
=\
H^{\hdot} _c\big(\mm\bvtn,\ \imath^*\IC(\cl_\rho)\big),
\qquad\forall\ \rho\in\Irr W.
\]

In particular, for $\th$  sufficiently general   in the `{\em coprime}' case we have
\beq{nopole}
\ltr
H^\hdot_c(\mm_O, \qlb)^{\op{sign}}\
=\ q^{-\frac{1}{2}\dim\mm_O}\cdot
\# \mm^\th_{0}(\k).
\eeq
\end{thm}

\begin{rem} We expect that parts (i)-(iii) of the theorem hold for more general quasi-abelian categories whenever
one can construct coarse moduli spaces $\mm^\th$ with reasonable properties.
Quite differently, the statement of part (iv) crucially relies on the fact that $\mm^\th$ is a {\em semi-projective} variety
in the sense of \cite{HV2}. We do not know examples beyond quiver representations and parabolic bundles
where  $\mm^\th$ is  semi-projective. This is the reason we had to require in (iv) that the algebra $A$ is 
a path algebra.
\erem

The proof of Theorem \ref{la-thm}  exploits an idea borrowed from
\cite{HLV2}.  We use   a   deformation argument 
and  the following general result that seems to be new and may be of
independent interest.
\begin{prop}\label{virt} Let $\wt X\xrightarrow{\pi} X\xrightarrow{p}S$
be a diagram of morphisms of schemes such that the scheme $S$ is smooth,
$\pi$ is a proper, resp.   $p\ccirc \pi$ is a smooth, morphism and, moreover, there is a Zariski open and
dense subset $U\sset S$ such that the morphism $(p\ccirc \pi)\inv(U)\to p\inv(U)$, induced by  from $\pi$  by restriction, is smooth.

Then, we have $\pi_*\qlb=\IC(\pi_*\qlb|_{p\inv(U)})$, i.e., the map $\pi$ is {\em virtually small}
in the sense of  Meinhardt and Reineke \cite{MeR}.

 Assume in addition that   $\pi$ is generically  finite. Then $\pi$ is small. Furthermore, for any closed point $s\in S$
the map  $(p\ccirc \pi)\inv(s)\to p\inv(s)$, induced by $\pi$, is  semi-small.
\end{prop}

A basic example of the setting of the proposition is provided by the Grothendieck-Springer
resolution $\mu:\wt\g\to \g$ for a semisimple Lie algebra $\g$.
In that case,  we  let $S=\ft$ be the Cartan subalgebra, resp. 
 $\wt X=\wt \g$ and
 $X=\g\times_{\ft/W}\ft$, where $W$ denotes the Weyl group.
There is  a natural smooth morphism $\wt\g\to \ft$  that factors as a composition
of a proper birational morphism $\pi:\wt \g\to \g\times_{\ft/W}\ft$ and the second projection
$p: \g\times_{\ft/W}\ft\to \ft$.    

For applications to the proof of Theorem \ref{la-thm} one takes
$S=\ft^*$ and $X=\mm^\th\times_{\ft^*\dsl  W}\ft^*$.
The corresponding variety   $\wt X$ is defined as a certain GIT
quotient for a stability condition $\th'$ which is
  a sufficiently general deformation of  the given 
 stability condition $\th$, see \S\ref{quiv-sec}. Proposition \ref{virt} insures that the resulting variety is a small resolution 
of $\mm^\th\times_{\ft^*\dsl  W}\ft^*$.
There is  a strong similarity with the construction of
M. Reineke \cite{Rei}, although the setting in {\em loc cit} involves quiver moduli spaces while our setting
corresponds to the cotangent bundle of those moduli spaces.

\subsection{Parabolic bundles}\label{par-sec}
Let $C$ be  a smooth projective geometrically connected curve over $\F_q$,
and  $D=\{c_i\in C(\F_q),\, i\in I\}$ a fixed 
collection  of  pairwise distinct   marked points sometimes thought of as a divisor $D=\sum_\ii\,c_i$.
Associated with a dimension vector
$\bm=(m_i)_\ii$ with positive coordinates, there is a category
$\cC=\VV(C,D,\bm)$ of parabolic bundles.
An object of $\VV(C,D,\bm)$ is a vector bundle, i.e. a locally free
sheaf
  $\V$ on $C$ equipped,  for each $\ii$, with an $m_i$-step  partial flag
$ V_i^\bullet=( \V|_{c_i}= V^{(0)}_i\supseteq  V^{(1)}_i\supseteq \ldots\supseteq
 V^{(m_i)}_i=0)$ in
the fiber of $\V$ at $c_i$. Put ${r}^{(j)}_i =\dim(V\ij/V^{(j+1)}_i)$ and call the
 array $\br=\{{r}^{(j)}_i\in\Z_{\geq 0},\ (i,j)\mid \ii,\ j=1,\ldots,m_i\}$
 the {\em type} of the parabolic bundle.
The integer  $|\br|=\sum_{j=1}^{m_i} \ {r}^{(j)}_i$ is the same for all $\ii$
and it is the rank of the vector bundle $\V$. 
Let  $\bbm=\{(i,j)\mid \ii,\ j=1,\ldots,m_i\}$.
It will be  convenient to view various types $\br$
as dimension vectors $\br=(r\ij)\in \Z_{\geq 0}^\bbm$.
A parabolic type such that $m_i=1\ \forall i$, will be called
`trivial'. For any such type the corresponding category $\VV(C,D,\bm)$ is the category  of vector bundles
without  any flag data. 

A morphism in  $\VV(C,S,\bm)$ is, by definition, a 
morphism $f:\V\to\W$ of  locally free sheaves  that respects the partial flags,
i.e. such that $f( V^{(j)}_i)\sset W^{(j)}_i$, for all $i,j$.
A direct sum of parabolic bundles $\V$ and $\W$ is
the vector bundle $\V\oplus\W$ equipped
with the partial flags defined by
$(V\oplus W)^{(j)}_i:=V^{(j)}_i\oplus W^{(j)}_i$.
This makes  $\VV(C,D,\bm)$
 a $\k$-linear, not necessarily abelian, Karoubian category. 
The assignment
\beq{bun-functor}\V\ \mto\ F(\V)=\bigoplus_{(i,j)\in \bbm}\ 
V^{(j)}_i/V^{(j+1)}_i.
\eeq
gives a $\k$-linear  functor from  category $\VV(C,D,\bm)$ to the 
category of finite dimensional $\bbm$-graded vector spaces. One has
a decomposition
$\X(\VV(C,D,\bm))=\sqcup_{\br\in \bbm}\ \X_\br(\VV(C,D,\bm))$,
where $\X_\br(\VV(C,D,\bm))$ is the stack parametrizing
parabolic bundles of type $\br$.
The corresponding framed stack  $\fX_\br(\VV(C,D,\bm), F)$ 
 parametrizes the data
$\big(\V, (\V^\bullet_i)_\ii, b=(b^{(j)}_i)_{(i,j)\in \bbm}\big)$, where  $(\V, (\V^\bullet_i)_\ii)$  is  a 
parabolic
bundle   of type $\br$ and 
 $b^{(j)}_i$ is a $\k$-basis  of the vector space
 $V^{(j)}_i/V^{(j+1)}_i$.
The group 
$\GL_\br:= \prod_{(i,j)\in \bbm} \
\GL_{r\ij}$ acts on  $\fX(\VV(C,D,\bm), F)$ by changing the bases.
One can show that this action factors through an action of
$G_\br=G_\br/\GG_m$.

\begin{rem} 
Note that in order to have a nonzero functor $F$ the set $I$ of marked points must always be nonempty,
even if one is only interested in  vector  bundles without flag data, i.e., parabolic bundles of trivial type.
\erem

In this paper, a Higgs field on a  parabolic bundle $\CV$
 is, by definition,  a morphism $\V \to \V\o\Omega^1_C(D)$
such that 
\beq{phi-flag}
\textit{res}_i(u)(V\ij)\ \sset\ V\ij\,\o\, \Omega^1_C(D)|_{c_i},
\quad\forall\ (i,j)\in \bbm.
\eeq
Here, $u$ is viewed as a section of  ${\scr H}\!om(\V,\V)\o \Omega^1_C$ with simple poles at the
marked points and $\textit{res}_i(u) \in \End(\V|_{c_i})$ denotes the residue of $u$ at
the point $c_i$. 
\begin{rem} Our terminology is slightly different from the 
standard one. The objects that we call `parabolic bundles' are  usually referred to as `quasi-parabolic bundlles',
while the name `parabolic bundle' is reserved to objects equipped with additional stability parameters.
Since we work with stacks, these parameters will play no role in our considerations except for the section
\ref{quiv-sec}.
Also, Higgs fields are often defined as morphisms $u: \V\to \V\o\Omega^1_C(D)$
satisfying a stronger
requirement:\  $\textit{res}_i(u)(V\ij)\ \sset\ V_i^{(j+1)}\,\o\, \Omega^1_C(D)|_{c_i},\ \forall\ (i,j)\in \bbm$.
\erem

The  stack $\TT^*\fX_{\br}(\VV(C,D,\bm), F)$ 
may be identified with the stack of framed parabolic Higgs bundles.
A $\k$-point of this stack is a triple  $(\V,b,u)$
where $\CV$ is a  parabolic bundle, $b$ is a framing, and $u$ is a Higgs field on $\V$.
 We let $\g_\br=\Lie G_\br$ and use the natural identification
$\g_\br^*=\g_\br$. The moment map 
$\TT^*\fX_{\br,d}(\VV(C,D,\bm), F)\to \g_\br^*$ associated with the $G_\br$-action
on the cotangent stack
 sends $(\CV,u,b)$ to the collection
$(\gr\textit{res}_i(u))_\ii$. Here,  $\gr\textit{res}_i(u))\in \oplus_j\ \End(V^{(j)}_i/V^{(j+1)}_i)$ 
is the
map induced by $\textit{res}_i(u)$ and the
framing $b$ provides an identification $\End(V^{(j)}_i/V^{(j+1)}_i)=
\gl_{r^j_i}$; thus,  the collection $(\gr\textit{res}_i(u))_\ii$ may be identified with an element of $\g_\br$.
Note that the $G_\br$-conjugacy class of that element is independent of the choice of the
framing $b$.

Let $\#\op{AI}(\ppar_{{\mathbf r},d},\F_q)$ be 
the number   of isomorphism classes
of  $\F_q$-rational absolutely indecomposable
 parabolic  bundles of type  ${\mathbf r}$ and
  degree $d$.
Given a conjugacy class $O\sset\g_\br$, let $\hig_{{\mathbf r},d}(O)$
be (the Beilinson-Drinfeld modification, cf. Remark \ref{bd}(i),  of)  the stack of parabolic  Higgs bundles 
whose objects are pairs $(\CV,u)$
such that $\gr\textit{res}(u)\in O$. 
Using formula  \eqref{ai-intro} 
we prove
\begin{thm}\label{ind-bundles}  For  a very general  $\F_q$-rational semisimple conjugacy class $O$ in $\g_\br$,
we have an equation
\beq{ind-for}
\#\op{AI}(\ppar_{{\mathbf r},d},\F_q)= q^{-\frac{1}{2}\dim 
\hig_{{\mathbf r},d}(O)}\cdot
\ltr H^\hdot_c(\hig_{{\mathbf r},d}(O),\qlb)^{\langle\op{sign}\rangle}.
\eeq
\end{thm}

We remark that  $\ppar_{{\mathbf r},d}$ is an Artin stack of  {\em infinite type}
and the  corresponding set $[\ppar_{{\mathbf r},d}(\F_q)]$
is not finite. Therefore, Theorem \ref{ai-intro} 
can not be applied
to category $\VV(C,D,\bm)$ directly.
We will prove formula \eqref{ind-for}
by approximating category  $\VV(C,D,\bm)$
 by an increasing family  of subcategories
such that the corresponding moduli stacks have finite type.
We show that the sets of absolutely indecomposable objects of these categories stabilize. 
This allows to find
$\#\op{AI}(\ppar_{{\mathbf r},d},\F_q)$, in particular, this number is finite.
\vskip 3pt

 In the case of a trivial parabolic type, the LHS of the above formula
reduces to  $\#\op{AI}_{r,d}(\bun_{r,d},\F_q)$, the number of isomorphism classes
of  $\F_q$-rational absolutely indecomposable
vector  bundles of  rank $r$ and degree $d$.
If, in addition,  $r$ and $d$ are 
{\em coprime}  the RHS of  formula \eqref{ind-for} can be expressed in terms of
the coarse moduli space,
$\hig^{\op{ss}}_{r,d}$,
 of semistable
Higgs bundles  of rank $r$, degree $d$, and such that the Higgs field is regular
on the whole of $C$. Specifically, formula \eqref{nopole} yields
\beq{sch}
\#\op{AI}(\bun_{r,d},\F_q)= q^{-\frac{1}{2}\dim \hig^{\op{ss}}_{r,d}}
\cdot \#\hig^{\op{ss}}_{r,d}(\F_q) \quad\text{if}\en(r,d)=1.
\eeq

Another, much more explicit, formula for $\#\op{AI}(\bun_{r,d},\F_q)$
has been obtained earlier by O. Schiffmann \cite{Sch}, cf. also \cite{MS}.
The approach in {\em loc cit} is totally different
from ours and it is close, in spirit, to the strategy
used in \cite{CBvdB} in the case of quiver representations with
an indivisible dimension vector. 
It seems unlikely that
such an approach can be generalized  to the case of vector bundles with
$(r,d)>1$,
resp.  the 
case of quiver representations with divisible dimension vector.
Note that the more general formula
 \eqref{ind-for} applies regardless of whether  $(r,d)=1$
or not.
The relation between \eqref{ind-for} and the approach 
 used in \cite{Sch} and \cite{CBvdB} is discussed in more detail
in \S\ref{quiv-sec}.


\subsection{Higgs bundles and local systems}\label{deligne-sec}
Remarkably,   indecomposable parabolic bundles
 provide a bridge between
the geometry of Higgs bundles, on one hand,
 and counting {\em irreducible} $\ell$-adic  local systems on 
the punctured curve with  prescribed monodromies
at the punctures, on the other hand.
In more detail,
let  $C_i^*=C_i\sminus \{c_i\}$ where $C_i$ is the
henselization of  $C$ at  $c_i$.
Fix  an $I$-tuple $({\scr L}_i)_\ii$,
where  ${\scr L}_i$ is  a  rank $r$ tame
 semisimple 
$\qlb$-local system 
 on $C_i^*$ such that  ${\textrm{Fr}}^*{\scr
  L}_i\cong{\scr L}_i$. 

In  \cite{De}, Deligne addressed the problem of counting
 elements of the set $\Loc_{\textrm{et}}(({\scr L}_i)_\ii,\, q)$, of isomorphism classes of
irreducible  $\qlb$-local
 systems ${\scr L}$ on $C^*:=C\sminus (\bigcup_i c_i)$,
 such that 
${\textrm{Fr}}^*{\scr L}\cong {\scr L}$ and
${\scr L}|_{C_i^*}\cong {\scr L}_i$.
Motivated by an old result of Drinfeld
\cite{Dr} in the case $r=2$, one expects that the function
$n\mto \# \Loc_{\textrm{et}}(({\scr L}_i)_\ii,\, q^n)$ behaves
as if  it counted $\F_{q^n}$-rational points of an  algebraic variety over $\F_q$.
The main result of Deligne confirms this expectation under
the following additional
assumptions on the collection  $({\scr L}_i)_\ii$:
\begin{enumerate}

\item Each of the local systems ${\scr L}_i$ is semisimple and tame.

\item One has $\prod_\ii\ \det{\scr L}_i=1$,  see 
\cite{De}, equation  (2.7.2).

\item {\em Genericity condition:}
for any  $0< r' <r$ and any $I$-tuple
$({\scr L}'_i)_\ii$ where ${\scr L}'_i$ is a direct summand of ${\scr L}_i$
of rank $r'$, one has  $\prod_\ii\ \det{\scr L}'_i\neq 1$,
see \cite{De},  equations   (2.7.3)
  and (3.2.1).
\end{enumerate}
Recall that the tame quotient of
the geometric fundamental group of $C^*_i$ is isomorphic
to $\wh{\mathbb Z}(1)$, cf. \cite[\S2.6]{De}. Therefore, the local system
${\scr L}_i$ being tame by (1), it gives a homomorphism $\rho_i: \wh{\mathbb Z}(1)\to
\GL_r(\qlb)$. Condition (2)   says that, writing
$\prod_i \det\rho_i(z)=1$ for all $z\in \wh{\mathbb Z}(1)$. This equation is neccesary 
for the existence of a local system $\cL$ on $C\sminus D$
such that $\cL|_{C_i^*}\cong {\scr L}_i,\ \forall \ii$, see
\cite[\S 2.7]{De}. 
Furthermore,  condition (3) guarantees
that any such  local system $\cL$ is automatically
irreducible.
Condition (1) implies also that each local system
${\scr L}_i$ is isomorphic to  a direct sum of  rank one   local systems.
Thus, one can write 
$${\scr L}_i=({\scr L}_{i,1})^{\oplus
  r^{(1)}_i}\ \bplus\ldots\bplus\
({\scr L}_{i,m_i})^{\oplus
  r^{(m_i)}_i},\qquad\forall\ \ii,
$$
where  the tuple $(r^{(1)}_i,\ldots,r^{(m_i)}_i)$ is a partition of $r$
with $m_i$ parts and ${\scr L}_{i,j},\ j\in [1,m_i]$,  is an unordered collection
of pairwise nonisomorphic rank one  tame local systems.
We say that the $I$-tuple
$({\scr L}_i)_\ii$ has type  $(r\ij)$. 

According to Deligne \cite[Th\'eor\`eme 3.5]{De}, for an  $I$-tuple $({\scr L}_i)_\ii$  of type $\br=(r\ij)$  one has
\beq{de}
\# \Loc_{\textrm{et}}(({\scr L}_i)_\ii,\, q)\ =\ \#\op{AI}(\ppar_{{\mathbf r},d},\F_q),
\eeq
for any fixed degree $d$.
The proof of this equation, which is the main result of  \cite{De},  rests on the validity of the  Langlands conjecture for $\GL_n$ proved by Lafforgue.
Comparing \eqref{de} and  \eqref{ind-for}, we obtain 

\begin{cor}\label{deligne} Let $({\scr L}_i)_\ii$ be an $I$-tuple such that
  conditions (1)-(3) hold and let $\br$ be the corresponding type.
Then,
for any degree $d$ and a very general conjugacy class
$O$ in $\g_\br$ 
one has
\beq{DN}\# \Loc_{\textrm{et}}(({\scr L}_i)_\ii,\, q)\ =\ q^{-\frac{1}{2}\dim 
\hig_{{\mathbf r},d}(O)}\cdot
\ltr H^\hdot_c(\hig_{{\mathbf r},d}(O),\qlb)^{\langle\op{sign}\rangle}.
\eeq
\end{cor}

\begin{rem} It would be interesting to find an analogue of 
\eqref{DN} for a  general
reductive group $G$ that relates
the number of  $\ell$-adic $G$-local systems to the number 
of $\F_q$-rational points of
a variety of Higgs bundles for the Langlands dual group of $G$.
Note that  the set of 
indecomposable parabolic bundles
has no obvious
counterpart for an arbitrary $G$.
\erem

The RHS of \eqref{DN} simplifies in the generic case, i.e. the case where one has
  ${\scr L}_i\not\cong {\scr L}_j$ for all $i\neq j$.
This forces all $r^{(j)}_1$ be equal to 1, so ${\mathbf r}$ corresponds to
the parabolic structure such that
all partial flags 
are  {\em complete} flags.  The group $G_\br$ is in this case a torus.
Therefore,  the group $\si_\br$ reduces to the identity element and
  the  cohomology group  in the
RHS of  \eqref{DN}
is equal to its
sign-isotypic component. Hence, by the Lefschetz trace formula
one finds that the RHS of \eqref{DN}
equals $q^{-\frac{1}{2}\dim \hig_{r,d,{\text{\rm complete flags}}}(O)}
\cdot
\#\hig_{r,d,{\text{\rm complete flags}}}(O)(\F_q)$.

Some very interesting phenomena also occur in the 
case where the set $I$ consists of one element,
so there is only one marked point $c\in C$. 
Let  $\zeta$ be  a primitive  root of unity of order $r$
and use simplified notation $\Loc_{\textrm{et}}^\zeta(q)=\Loc_{\textrm{et}}(({\scr
  L}_i)_\ii,\, q)$
for the set   of isomorphism classes
of tame  rank $r$ local systems ${\scr L}$ on $C\sminus \{c\}$ 
such that ${\textrm{Fr}}^*{\scr L}\cong{\scr L}$ and
the monodromy of ${\scr L}$ at $c$  equals
$\zeta\cdot \Id$. Such a local system is automatically simple
and the corresponding type $\br$ is  {\em trivial}.
Therefore, formula
\eqref{sch} is applicable and using Corollary \ref{deligne},  for any $d$ coprime to $r$,
we obtain
\beq{et}
\# \Loc_{\textrm{et}}^\zeta(q)\ = 
\ q^{-\frac{1}{2}\dim \hig^{\op{ss}}_{r,d}}
\cdot \ltr H^\hdot_c(\hig^{\op{ss}}_{r,d},\qlb).
\eeq

On the other hand, let  $C_{top}$ be a  compact Riemann surface 
of the same genus as the algebraic curve $C$ and $\pt\in C_{top}$
a marked point.
For any field $\k$
that contains $r$-th roots of unity, one has a smooth 
affine algebraic variety  $\Loc_{Betti}^\zeta$,
called {\em character variety}, that parametrizes isomorphism
classes of $r$-dimansional representations of the fundamental group of
$C_{top}\sminus \{\pt\}$
such that a small loop around the puncture $\pt$ goes to $\zeta\cdot \Id$.
It is known that for $\k=\CC$  there is a natural
diffeomorphism   $\Loc_{Betti}^\zeta(\CC)\cong 
\hig^{\op{ss}}_{r,d}(\CC)$,
{\em of  $C^\infty$-manifolds}.
In particular,
these two  $C^\infty$-manifolds have isomorphic 
singular cohomology groups.
However, these two manifolds
 are {\em not} isomorphic as complex algebraic
varieties and the isomorphism of the cohomology induced by the diffeomorphism
does not respect the  Hodge structures. 
The number of  $\F_q$-rational points of $\Loc_{\textrm{Betti}}^\zeta$ is controlled by  the $E$-polynomial
associated with the Hodge structure.
Specifically, it was shown in  \cite{HV}  using a result of N. Katz that one has
\beq{ltop}
\# \Loc_{\textrm{Betti}}^\zeta(\F_q)\ =\ \sum_j\ q^{j}\cdot 
\left(\sum_k\ (-1)^k\,\dim\, \gr^W_{2j} H^k_c(\Loc_{Betti}^\zeta(\CC),\ \CC)\right),
\eeq
where $\gr^W_\idot(-)$ denotes an associated graded space with respect to the weight filtration
(the  Hodge structure on $H^\hdot(\Loc_{Betti}^\zeta(\CC),\, \CC)$ is not pure).

It would be interesting to compare  formulas \eqref{et} and \eqref{ltop} in the context of 
`$P=W$'-conjecture, see \cite{CHM}.

\begin{rem}\label{azumaya}  The above formulas may also be reformulated in terms of de Rham 
local systems. To do so, let  $R$ be a domain over $\Z$ that
admits homomorphisms  $R\to\CC$ and $R\to\k$, where $\k$ is a finite field of characteristic $p$. 
Let $C_R$ be a smooth projective curve 
over  $R$ with a good reduction to
a smooth curve $C_\k$ over $\k$, and let $C_\CC$ be  a  complex curve  obtained from
$C_R$ by specialization.
Let  $\Loc_{\textrm{DR}}^{\frac{1}{r}}(C_{\CC})$, resp. $\Loc_{\textrm{DR}}^{\frac{1}{r}}(C_{\k})$,
be the the  moduli space of  `de Rham local systems', that is, of pairs $(\V,\nabla)$ where $\V$ is
a rank $r$ algebraic vector bundle on $C_{\CC}$, resp. $C_\F$, and $\nabla$ is a 
connection on $\V$ with regular singularity at $\pt$ such that $\textsl{res}_{\pt}\nabla=\frac{1}{r}\cdot\Id$.
The Riemann-Hilbert correspondence provides an isomorphism 
$\Loc_{\textrm{Betti}}^\zeta(\CC)\cong \Loc_{\textrm{DR}}^\zeta(\CC)$, of complex analytic manifolds
 (but
not algebraic varieties) that respects the Hodge structures.
Therefore, the RHS of \eqref{ltop} is not affected by replacing $\Loc_{Betti}^\zeta(\CC)$
by $\Loc_{DR}^\zeta(\CC)$.

On the other hand, it is well known that crystalline differential operators on $C_\k$ give rise to 
an  Azumaya algebra on  $(\TT^*C)^{(1)}$, the Frobenius twist of  $\TT^*C$.
Furthermore, the results proved in  \cite{Tr2}, \cite{Gro} imply that the restriction
of this Azumaya algebra to any $\k$-rational `spectral curve' in $(\TT^*C)^{(1)}$
admits a splitting over $\bar\k$. 
The Brauer group of the (finite) field $\k$ being trivial, one can actually find a splitting defined over $\k$.
Now,   given a connection  $\nabla$ on a vector
bundle $\V$ on $C_\k$,  there is an associated $p$-curvature, $\op{curv}_p(\nabla)$, defined as the map $\xi\mto \nabla^p(\xi)-\nabla(\xi^p)$.
It follows from the above mention splitting resultIt  that the assignment $(\V,\nabla)\mto (\fr^*\V, \op{curv}_p(\nabla))$,
provides a bijection between the set  $\Loc_{\textrm{DR}}^{\frac{1}{r}}(C_{\k})(\k)$, of isomorphism classes of 
$\k$-rational points of $\Loc_{\textrm{DR}}^{\frac{1}{r}}(C_{\k})$, and the set of isomorphism classes of 
 rank $r$   Higgs bundles on $C_\k^{(1)}$, the Frobenius twist of
$C_\k$, respectively. Assume that  $p>r$ and fix an integer $m$ prime to $r$.
One can show that if $\nabla$  has a simple pole at $\pt$ with
residue $\frac{m}{r}\Id$ then  $\op{curv}_p(\nabla)$ has
no singularity at $\pt$. Moreover, this way one obtains an isomorphism of the stack
$\Loc_{\textrm{DR}}^{\frac{1}{r}}(C_\k)$  and the Frobenius twist  of the stack $\hig_{r,1}$.
Thus, for the corresponding numbers of  isomorphism classes of $\k$-rational points one gets 
$\#\Loc_{\textrm{Betti}}^\zeta(C_\k)(\k)\ =\ q^{-\frac{1}{2}\dim \hig^{\op{ss}}_{r,d}}
\cdot \# \hig^{\op{ss}}_{r,d}(\k)$.
\end{rem}
\subsection{Factorization sheaves}\label{Phi-sec}
Factorization sheaves have appeared in the
literature
in various contexts, see eg. \cite{FG}.
We will use the following settings.

Fix a smooth scheme $C$ over $\k$ and write $\Sym^v C$ for the $v$-th symmetric pover of $C$.
For any pair $v_1,v_2\in\Z_{\geq0}$,  let $C_{v_1,v_2}\he$ be an open subset of $C^{v_1+v_2}$
formed by the tuples $(c_1,\ldots,c_{v_1}, c'_1,\ldots,c_{v_2})\in C^{v_1+v_2}$
such that $c_j\neq c'_k$ for all $j,k$. Let $\Sym^{v_1+v_2}_{\op{disj}}C$ be the
image of $C_{v_1,v_2}\he$ in $\Sym^{v_1+v_2}C$.
A \f sheaf on $\Sym C$ is, by definition, a collection
$\cf=(\cf_v)_{v\geq 0}$, where  $\cf_v\in D\abs(\Sym^vC)$,
equipped, for each pair $v_1,v_2$,  with an isomorphism $\cf_{v_1+v_2}|_{\Sym^{v_1+v_2}_{\op{disj}}C}\iso
(\cf_{v_1}\boxtimes \cf_{v_2})|_{\Sym^{v_1+v_2}_{\op{disj}}C}$,
satisfying certain associativity 
and commutativity constraints, see \S \ref{sec3}.
Here, 
the notation  $D\abs(X)$ stands for the the triangulated category
of absolutely convergent Weil complexes on a stack $X$, as defined by Behrend
\cite{B1},\cite{B2} and reminded in \S\ref{nnn} below. 

There are two other settings sheaves 
where the sheaf $\cf_v$ on $\Sym^v C$
is  replaced, in the above definition,
by  either a $\si_v$-equivariant sheaf on $C^v$ (equivalently, the categorical quotient
$\Sym^v C=C^v\dsl \si_v$ is replaced by
$C^v/\si_v$, a stacky quotient of $C^v$ by $\si_v$) or by a sheaf on $\Coh_v(C)$, the stack of
length $v$ coherent sheaves on $C$.
In the latter case, the subset  $\Sym^{v_1+v_2}_{\op{disj}}C$ should be replaced by
$\Coh_{v_1,v_2}\he$,  an open  substack of
$\Coh_{v_1+v_2}(C)$
whose objects are  coherent sheaves of the form $M_1\oplus M_2$
such that $\supp(M_1)\cap\supp(M_2)=\emptyset$ and
$\textrm{length}(M_j)=v_j,\ j=1,2$. 
In these two settings, the collections $\cf=(\cf_v)_{v\geq 0}$ where $\cf_v\in D\abs(C^v/\si_v)$, 
resp.  $\cf_v\in D\abs(\Coh_v(C))$,
will be  referred to as  `\f sheaves on $C$', resp.  `\f sheaves on $\Coh_v(C)$'. 

Factorization sheaves on $\Coh C$ and $\Sym C$ are related via a pair of adjoint functors.
Specifically,
for each $v$   one has an adjoint pair $\dis D\abs(\Coh_v C)\xymatrix{
\ar@<0.3ex>[r]^<>(0.5){_{\mathfrak{res}}}&D\abs(C^v/\si_v)\ar@<0.3ex>[l]^<>(0.5){_{\mathfrak{ind}}},
}
$
of functors of parabolic restriction and induction. One shows that these functors take
factorization sheaves  to factorization sheaves, cf. Remark  \ref{res-rem}.

Finally, there is also a version of the notion of
factorization sheaf in the coherent world  where
one considers  collections $M=(M_v)$ with
  $M_v\in D_{Coh}(C^v/\si_v)$, an object of the derived category of coherent sheaves.

All the above can be extended to an $I$-graded setting by replacing
a given scheme $C$ by a disjoint union $\sqcup_\ii C_i$, of several copies of $C$ each `colored'
by an element of the set $I$. Given a dimension vector $\bv=(v_i)$, let $C^\bv=\prod_\ii\ C_i^{v_i}$.
Further, elements of $\Sym^\bv C$ are colored 
effective divisors on $C$ with degree $v_i$ for the $i$-th color. Similarly,
the stack $\Coh_\bv(C)$ parametrizes  finite length  coherent sheaves whose
support is  a colored divisor of multi-degree $\bv$. Thus, one may consider an object $\cf_\bv$ of
$D\abs(\Sym^\bv C)$, resp. $D_{Coh}(\Sym^\bv C),\ D\abs(C^\bv/\si_\bv)$, and $D\abs(\Coh_\bv C)$.

We will apply factorization sheaves to moduli problems
using a construction that associates to a triple   $(\cC,F,\phi)$,    as  in \S\ref{exp-sec},
and an affine curve $C$ a \f sheaf on $\Coh C$.
To explain the construction, write $K=\k[C]$ for the  coordinate ring of $C$ and let $ K\o\cC$ be the corresponding
 base change category.
For any object $(x,\al)$ of $ K\o\cC$,
the composite  homomorphism $K\xrightarrow{\al} \End(x)\xrightarrow{F} \End F(x)$
makes $F(x)$ a $ K$-module. So  the assignment $(x,\al)\mto F(x)$ yields a functor
$K\o F:\ K\o\cC\to \Lmod{K}$.
Thus,  one obtains a diagram
of natural maps
\beq{Cdiag}
\Coh_\bv(C)=\fx_\bv(\Lmod{ K})
\ \xleftarrow{\nu}\ 
\fx_\bv(K\o\cC)\ \xrightarrow{p }\ \fx_\bv(\cC)\
\xrightarrow{\phi_\bv}\
\AA^1
\eeq
where the map $\nu$ is induced by the functor $K\o F$ and we have used an equivalence
between the category of 
finite dimensional $\k[C]$-modules and of finite length coherent sheaves on $C$, respectively.

\begin{rem}
For
$C=\GG_m$, one has  isomorphisms $\Coh_\bv(\GG_m)\cong \GL_\bv/\aad \GL_\bv\cong
I(\pt/\GL_\bv)$, resp. $\fx(\k[\GG_m]\o\cC)\cong I(\fX(\cC))$, where $I(\mathfrak Y)$ denotes 
the 
 {\em inertia stack} of a stack $\mathfrak Y$. 
Thus, in the special case $C=\GG_m$, the map $\nu$ in \eqref{Cdiag} 
may be viewed as a map $I(\fX(\cC))\to  I(\pt/\GL_\bv)$. This map of inertia stacks
is induced by the morphism
$\fx_\bv(\cC)\to\pt/\GL_\bv$ that comes from $\mathfrak F$, the rank $\bv$ vector bundle
 on $\fx_\bv(\cC)$, cf. \S\ref{exp-sec}.
\erem

The following result plays  a crucial role in the proof of Theorem \ref{A-ind}.

\begin{thm}\label{cruc}
For any  smooth affine curve $C$ and a data $(\cC,F,\phi)$,  as in \S\ref{exp-sec}, the collection
$\rr^{F,\phi,C}=(\rr_\bv^{F,\phi,C})$ defined by the formula $\rr^{F,\phi,C}_\bv=\nu_!\wp^{\phi_\bv\ccirc p}$
is a factorization sheaf on $\Coh C$.
\end{thm}

\subsection{Cohomology of \f sheaves}\label{h-intro}
Fix a smooth scheme $C$ with tangent, resp. cotangent, sheaf $\TT_C$, resp. $\TT^*_C$.
Let $\bv$ be a dimension vector, $C\sset C^\bv$ the principal diagonal,
and $\N_\bv\to C$, resp. $\N^*_\bv\to C$, the normal, resp.
conormal bundle.
The natural $\si_\bv$-action on $C^\bv$ by permutation of factors induces 
a $\si_\bv$-action on $\N_\bv$, resp. $\N^*_\bv$.
We will define a  Zariski open  $\si_\bv$-stable subset $\N_\bv\reg\sset \N^*_\bv$.
To simplify the exposition we assume that there exists  $\k$-rational section $\eta_\bv: C\to \N_\bv\reg/\si_\bv$.
We consider
the following  chain of functors
 \beq{Phi-fun}
\xymatrix{
D\abs(C^\bv/\si_\bv)\ar[r]^<>(0.5){{\mathsf{sp}}_\N}&
D\abs(\N_\bv/\si_\bv)
\ar[r]^<>(0.5){\FDN_\N}&
D\abs(\N^*_\bv/\si_\bv)
\ar[r]^<>(0.5){\eta_\bv^*}&
D\abs(C),
}
\eeq
where $\FDN_\N$ is the Fourier-Deligne transform along the fibers of the vector bundle $\N$
and ${\mathsf{sp}}_\N$
is the Verdier specialization to the normal bundle.
Let $\Phi:\ D\abs(C^\bv/\si_\bv)\to D\abs(C)$ be the composite functor.
Also, write $\{n\}$ for the functor $[n](\frac{n}{2})$, of homological  shift and a Tate twist.
Then, our first important result concerning factorization sheaves  reads, cf. Theorem \ref{expC}.

\begin{thm}
Let    $\cf=(\cf_ \bv)_{ \bv\in\Z^I_{\geq0}\sminus\{0\}}$ be a factorization sheaf    on $\Sym C$
such that  for every $\bv$ the sheaf $\Delta^*_ \bv\cf_ \bv$ is a 
locally constant sheaf on $C$. Then,
 one has 
\beq{expC-intro} 
  \sum_{ \bv\geq0}\  z^\bv\cdot [\RGam_c(C^\bv/\si_\bv, \cf_ \bv)]\ =\
\Sym\Big(\sum_{ \bv\neq 0}\ z^\bv\cdot[\RGam_c\big(C,\
\Phi(\cf_ \bv)^{\langle\op{triv}\rangle}\big]\{\dim C(1-|\bv|)\}\Big).
\eeq
\end{thm}
\noindent The above equation holds  in  $K\abs(\pt)[\![z]\!]$,
where we use the notation $K\abs(X)$ for a suitable
Grothen-dieck group of the category $D\abs(X)$.
The symbol  $H_c^\hdot(-)^{\langle\op{triv}\rangle}$ in the RHS of the equation refers to the 
{\em generalized} isotypic component 
of the trivial representation of a  certain monodromy action of the group $\si_v$,
cf. \S\ref{unipc}.

Observe  that the LHS, resp. RHS, of the equation  \eqref{expC-intro}
depens only
on the class of $\cf_\bv$ in $K\abs(C^\bv))$, resp. of $\Phi(\cf_ v)^{\langle\op{triv}\rangle}$ in 
$K\abs(C)$. It will be explained in the main body of the paper that
the map  $K\abs(\Coh_\bv(C))\to K\abs(C)$ given the
assignment
$[\cf_\bv]\to [\Phi(\cf_ v)^{\langle\op{triv}\rangle}]$ 
can be defined in a canonical way that does not involve a choice of  $\eta_\bv$.
Furthermore, 
we will actually establish a version of   \eqref{expC-intro}
that holds without any additional assumptions and choices, see Theorem \ref{expC}.

Applications to exponential sums involve
factorization sheaves  on $\Coh(C)$, where either $C=\AA^1$ or $C=\GG_m$. There is
a well-known isomorphism of stacks $\Coh_\bv(\AA^1)\cong\gl_\bv/\GL_\bv$, 
resp. $\Coh_\bv(\GG_m)\cong\GL_\bv/_{_{\Ad}}\GL_\bv$.
 The group  $\op{Aff}=\GG_m\ltimes \GG_a$ of
affine-linear transformations acts on $\gl_\bv$ by $g\mto a\cdot g+b\cdot \Id$.
For an $\Aff$-equivariant    factorization sheaf $\cf=(\cf_\bv)$, on $\gl=\Coh \AA^1$,
formula \eqref{expC-intro} simplifies. 
Specifically, the sheaf $\cf_\bv$  descends to a $\GG_m$-equivariant
sheaf $\bar\cf_\bv$ on $\pgl_\bv/\GL_\bv$. In such a case,
taking Verdier's specialization  becomes unnecessary.
Furrther,
we have a diagram
$\GL_\bv\stackrel{\eps}\into\gl_\bv \stackrel{i}\hookleftarrow \{0\}$,
where
$\eps$ is an open imbedding of invertible matrices into all matrices.
Restricting $\cf$ via $\eps$ gives
a $\GG_m$-equivariant  factorization sheaf $\eps^*\cf=(\eps^*\cf_\bv)$ on $\GL=\Coh\GG_m$. 
Using that 
 the operation of taking $\Sym$ corresponds,
under the trace of Frobenius, to taking plethystic exponential, 
from  formula \eqref{expC-intro}  we deduce 
 the following equations, cf. Theorem \ref{GL-H},
\begin{align}
 &\sum_\bv\ z^\bv\cdot \ltr H^\hdot_{c,\GL_\bv}(\GL_\bv,\ \eps_\bv^*\cf_\bv)
=\
\Exp\Big((q^{-\frac{1}{2}}-q^{\frac{1}{2}})\cdot\sum_{^{\bv\neq 0}}\
(-1)^{|\bv|}\cdot z^\bv \cdot q^{\frac{|\bv|}{2}}\cdot \ltr (\fw_!\FDN_{\pgl_\bv}\bar\cf_\bv|_{\eta_\bv})^{\langle\op{sign}\rangle}
\Big),\nonumber\\
 &\sum_\bv\ (-1)^{|\bv|}\cdot z^\bv\cdot q^{\frac{\bv\cdot\bv}{2}}\cdot \ltr  H^\hdot_{c,\GL_\bv}(i_\bv^*\cf_\bv)
=\
\Exp\Big(-\sum_{^{\bv\neq 0}}\ 
z^\bv\cdot  q^{\frac{|\bv|}{2}}\cdot \ltr (\fw_!\FDN_{\pgl_\bv}\bar\cf_\bv|_{\eta_\bv})^{\langle\op{triv}\rangle}
\Big).\label{intro2}
\end{align}
where   $\eta_\bv\in\pgl_\bv\reg$ is a $\k$-rational  element
and $\fw:\pgl_\bv/\GL_\bv\to\pgl_\bv\dsl\GL_\bv$ is the natural projection from a stacky quotient to the corresponding categorical quotient
by $G$.

\begin{ex}
In the special case where  
the set $I$ consists of one element
and the factorization sheaf $\cf$, on $\gl$, is formed by the constant sheaves $\cf_v=\qlb,\ \forall v$, 
one finds that $\FDN_{\pgl_v}\bar\cf_\bv|_{\eta_v}=0$ for all $v\neq 1$.
terms in the sum in the RHS of formula
\eqref{intro2} vanish except for  term for $v=1$. 
In that case, equation  \eqref{intro2} is the
standard identity for the
{\em quantum dilogarithm}:
$$
\sum_{v\geq0}\ z^v\cdot \frac{(-1)^v q^{\frac{v^2}{2}}}{\#\GL_v(\F_q)}=
\Exp\Big(z\cdot\frac{q^{\frac{1}{2}}}{1-q}\Big).\eqno\lozenge
$$
\end{ex}

The \f sheaf $\rr^{F,\phi,\AA^1}$  is
 automatically $\op{Aff}$-equivariant, so the above equations
are applicable for $\cf=\rr^{F,\phi,\AA^1}$.
Theorem \ref{A-ind} is an easy consequence of these equations.

\begin{rem} There is a natural isomorphism of stacks
$\Coh_\bv \pt\cong \pt/\GL_\bv={B}\GL_\bv$. Hence,  $H^\hdot(\Coh_\bv \pt)=H^\hdot({B}\GL_\bv)$ is an algebra of
symmetric polynomials, so one has
$\Spec H^\hdot(\Coh_\bv \pt)=\Sym^\bv \AA^1$.
For any scheme $C$,  
one has a natural  map $\Coh_\bv C\to \Coh_\bv \pt={B}\GL_\bv$,
induced  by a constant map $C\to \pt$.
This map takes a finite lentgh sheaf to its space of sections.
The pull-back  of the cohomology makes $H^\hdot(\Coh_\bv C)$ a super-commutative $H^\hdot({B}\GL_\bv)$-algebra,
equivalently, a coherent sheaf, ${\mathcal H}^\hdot(\Coh_\bv C)$, of super-commutative algebras on  $\Sym^\bv \AA^1$.
By abstract nonsense, the collection 
$({\mathcal H}^\hdot(\Coh_\bv C))_{\bv\in\Z^I_{\geq 0}}$ gives a coherent
\f sheaf on $\Sym \AA^1$ such that the factorization maps respect the algebra structures.

Now, given a sheaf $\cf_\bv\in D\abs(\Coh_\bv C)$,
the cohomology $H_c^\hdot(\Coh_\bv(C),\,\cf_\bv)$ has the canonical 
structure of a graded  module over the graded algebra $H^\hdot(\Coh_\bv C)$.
One can show that if a collection $\cf=(\cf_\bv)$ is a \f sheaf  on $\Coh C$
then the collection $H^\hdot_c(\cf)=(H^\hdot_c(\Coh_\bv C,\ \cf_\bv))$, gives a
coherent  factorization sheaf  on $\Sym \AA^1$ and, moreover,
 the factorization maps respect the action of  $(H^\hdot(\Coh_\bv C))_{\bv\in\Z^I_{\geq 0}}$,
a coherent factorization sheaf of super-commutative algebras.

All the above applies, in particular, to the case where $C$ is a smooth affine curve and $\cf=\cR^{F,\phi,C}$ is
the \f sheaf from Theorem \ref{cruc}. In that case 
we have
\beq{aamod}
H^\hdot_c(\Coh_\bv(C),\ \rr^{F,\phi}_\bv)\cong
H^\hdot_c(\fx_\bv(K\o\cC),\ \wp^{\phi_\bv\ccirc p }).
\eeq
We conclude from that the  collection
of the cohomology spaces \eqref{aamod} has the natural structure of a coherent \f sheaf on $\Sym\AA^1$
equipped with a compatible action of the  commutative algebras $({\mathcal H}^\hdot_c(\Coh_\bv C,\ \cf_\bv))$.
This is a generalization of the construction from \cite[\S6.5]{KS2}.
\erem

\begin{rem} One can show that   there is a natural
graded algebra isomorphism
$$
H^\hdot(\Coh_\bv \GG_m)\cong H^\hdot(\GL_\bv/\aad \GL_\bv)\cong
\big(\op{S} {\mathfrak h}\o \wedge {\mathfrak h}\big)^{\si_\bv},
$$
where ${\mathfrak h}=\AA^\bv$ is the Cartan subalgebra of $\gl_\bv$.
There is a generalization of this isomorphism to the case where $\GG_m$ is replaced by an arbitrary smooth affine 
connected curve $C$.
In the general case, the role of the super vector space ${\mathfrak h}\o (\AA^1\oplus \AA^1[1])$ is played by 
the super vector space  ${\mathfrak h}\o (H^0(C)\oplus H^1(C))$.
\erem

\subsection{Relation to earlier works}
Formula \eqref{intro2} and Theorem \ref{cruc} are  related, in a way, to works of Davison, Meinhardt and Reineke, \cite{MeR}, \cite{Me}, \cite{DM}.
In  these papers, the authors study Donaldson-Thomas invariants of the moduli stack $\X(\cC)^\th$ of $\th$-semistable objects
of a finitary  abelian category $\cC$. 
Let $\mm^\th$ be the corresponding coarse moduli space and $\mm^{\th-st}\sset \mm^\th$ the stable locus of $\mm^\th$.
There is  a well-defined morphism $p: \X(\cC)^\th\to \mm^\th$
that sends  an object of $\cC$ to an associated graded object with respect to 
a Jordan-H\"older filtration.
Further, any cosed point of $\mm^\th$ is represented by  a polystable object,  i.e.  a finite direct sum
of stable objects. Hence,  the direct sum operation yields an isomorphism
$\Sym \mm^{\th-st}\iso \mm^\th$, so the map $p$  gives a morphism $\X(\cC)^\th\to \Sym \mm^{\th-st}$.
Let  $\C$ be a constant sheaf on $\X(\cC)^\th$, with appropriated shifts and Tate twists.
In this language, it was shown in  \cite{Me}, \cite{DM} that $p_!\C$ is a factorization sheaf on $\Sym \mm^{\th-st}$.
Furthermore, the main result of {\em loc cit} says, roughly speaking, that in $K\abs(\mm^{\th-st})$ one has
$[p_!\C]=\Sym [\op{IC}(\mm^{\th-st})]$. This result may be compared with equation
\eqref{intro2}.

We observe that in  Donaldson-Thomas theory
 the category $\cC$ in question is assumed to be   {\em abelian}, furthermore, the framing functor $F$
is assumed to be exact and faithful. Our setting is more general, e.g. it covers the case 
of  parabolic bundles  where the category  is not
abelian and the functor $F$ is  neither exact nor faithful.
In such a case, there is no meaningful notion of  Jordan-H\"older series, so there is no
natural analogue of the morphism $\X(\cC)^\th\to \mm^\th$.
On  the other hand, one has the notion of an absolutely indecomposable object.
The structure of the set of (isomorphism classes of) absolutely indecomposable objects is,
typically,  rather chaotic  from an algebra-geometric
viewpoint. There is no natural way, in general, to upgrade this set to a scheme or a stack.
However, 
one can consider an  {\em ind-constrictible scheme},  $X(\cC)$, that parametrizes  isomorphism classes of 
objects of  $\cC$, and the set of  isomorphism classes of absolutely indecomposable objects  is the
set of $\k$-rational points of
an ind-constrictible subscheme $X^{\op{ind}}(\cC)$ of  $X(\cC)$.
Further, thanks to the Krull-Schmidt theorem  there is a natural isomorphism
$X(\cC)\cong \Sym X^{\op{ind}}(\cC)$, of constructible schemes.
This  may be viewed  as a substitute
of the isomorphism $\Sym \mm^{\th-st}\iso \mm^\th$, resp.  the Krull-Schmidt theorem as a substitute
of  the Jordan-H\"older theorem.  
It seems that the formalism of constrictible schemes 
is not rich enough  to do interesting geometry. This was one reason
for us to  consider the factorization sheaf
$\cR^{F,\phi,C}$, on $\Coh C$, rather than factorization sheaves on $\Sym X^{\op{ind}}(\cC)$.
More importantly, our approach is `microlocal' in the sense that it involves the
geometry of the cotangent bundle $\TT\fx(\cC,F)$ and the monent map.

It is interesting to observe that the effect of replacing the zero fiber of the monent map by
the {\em generic} fiber  (the stack $\mm_O$ in Theorem \ref{A-ind},
cf. also  the RHS of \eqref{expC-intro})  
is similar, in a way, to replacing $[H^\hdot_c(\fx(\cC))]$ by  $[\op{IH}(\mm^{\th-st})]$
(where $[\ldots]$ stands for the class  in an appropriate Grothendieck group),
in the setting  of  \cite{MeR}, \cite{Me}, \cite{DM}. Understanding this phenomenon
requires a better understanding of the geometry of the `forgetting the legs' map from
the {\em indecomposable} locus  of $\X(\wt\cC,\wt F)^{\wt \th}$ to $\X(\cC,F)^\th$, cf. \S\ref{resolution} 
and also Proposition \ref{nochar}.
We plan to discuss this in the future.

\subsection{Layout of the paper}
In section 2 we recall basic notions concerning   absolutely convergent constructible complexes
and the  Grothendieck-Lefschetz trace formula on stacks, 
as developped by Behrend \cite{B1}, \cite{B2}.
This section also contains the proof of a generalization of Hua's lemma and of Lemma \ref{jordan} that allows  to define the map $\Phi$, see \S\ref{Phi-sec}, in full generality.
In section 3, we introduce \f sheaves on $\Sym C$ and prove our main technical results
involving cohomology of factorization sheaves. In section 4 we upgrade the results of section 3 to the
setting of  factorization sheaves on the stack $\Coh C$. Specifically we  relate  factorization sheaves
on the stack $\GL=\op{Coh} \GG_m$ to  factorization sheaves on symmetric powers of $\GG_m$ via
a radial part (aka `parabolic restriction') functor. This  is necessary for applications to moduli spaces.
In section 5 we study equivariant inertia stacks.
Section 6 begins with some generalities cocerning sheaves 
of categories and associated moduli stacks.  This subject was studied by several authors in various contexts but we  have been unable
to find the setting that fits our purposes in the literature.
We then study the equivariant inertia stack  of a moduli stack and prove a 
 factorization property 
which is equivalent to Theorem \ref{cruc}.
In section \ref{if} we explain a relation, via the Fourier transform,
between  the equivariant inertia stack of  a $G$-stack and the moment map for the cotangent stack
of that $G$-stack. Such a relation is well known  in the case  of $G$-schemes but its generalization to stacks is not straightforward.
It requires a duality formalism for vector bundles on stacks which we develop in \S\ref{Fstacks} and that may be of 
an independent interest. The proof of  Theorem \ref{A-ind} is completed in \S\ref{fff}.
 In section 8 we study indecomposable parabolic bundles.
We use the Harder-Narasimhan filtration to
define a family of certain truncated categories such that the corresponding moduli stacks have finite type
and Theorem \ref{A-ind} is applicable.
The counting of absolutely indecomposable 
 parabolic bundles is then reduced to  counting absolutely indecomposable objects of the truncated categories.

Sections 9, 10,  and 11 do not depend on the rest of the paper and each of these sections
may be of independent interest.
Section  9 is devoted to the proof of Theorem \ref{la-thm}.
In \S \ref{acyc} we prove Proposition \ref{virt}.
Section \ref{positive-sec} contains a new  proof of   deformation invariance of
the cohomology of a semi-projective variety.
Although the result itself is not new, our proof is algebraic while
the original proof in \cite{HLV2} is based on a reduction to the
case of the ground field $\CC$ and a transcendental argument involving
Riemannian metrics.
In section 11 we show that the geometry of cells in the Calogero-Moser variety studied
by G. Wilson \cite{Wi} has a natural interpretation in terms of indecomposable
representations of the Calogero-Moser quiver. The results of section 11 have also appeared in the recent paper by
Bellamy and Boos \cite{BeBo}.

\subsection{Acknowledgments}
We are grateful  to Asilata Bapat for pointing out a missing condition in the definition of the set $\Sigma$
in section \ref{CM} and also to Daniel Halpern-Leistner and  Jochen Heinloth for explanations regarding   GIT formalism
for stacks and for providing references.
The work of the first author was supported by an NSF Postdoctoral Fellowship and the Max-Planck Institute in Bonn.
The second author  was supported in part by the NSF grant DMS-1303462.
This work was  conducted during the period the third author 
was employed by the Clay Mathematics Institute as a Research Scholar.
This material is based upon work supported by The Wolfensohn Fund and the National Science Foundation under agreement No. DMS-1128155.

\section{Absolutely convergent sheaves and Hua's formula}\label{nnn} 
\subsection{}\label{nnn1}
We write $\bar\k$ for an algebraic closure of a field $\k$.
We choose a prime $\ell\neq\chark\k$ and 
let $\C_X$ denote the constant sheaf on a scheme (or stack) $X$ with
fiber $\qlb$.

Let $\k$ be a finite field and $X$ a scheme over $\k$.
We write
$\fr=\fr_{\bar\k/\k}$ for the  Frobenius automorphism of $\k$
and $\pig(X)$ for the {\em geometric} \'etale fundamental
group of $X$, that is, for the kernel of
the canonical homomorphism of the full \'etale fundamental
group of $X$ to $\op{Gal}(\bar\k/\k)$.

Let $G$ be
 a connected algebraic  group  over a  finite field $\F$.
Associated with
a homomorphism
$\psi: G(\F)\to\qlb^\times$ there is a  rank one
$\qlb$-local system $\wp_\psi$,
 on $G$, such that $\fr^*\wp_\psi\cong\wp_\psi$ and one has
$\psi(g)=\ltr (\wp_\psi)_g$, for any $g\in G(\F)$.
To define this local system one considers
the Lang map ${{\textit{Lang}}}: G\to G,\ g\mto g\inv\fr(g)$.
Thanks to Lang's theorem, the map  ${{\textit{Lang}}}$ is a Galois covering
with Galois group $G(\F)$.
One defines $\wp_\psi$ to be the $\psi$-isotypic component
of the local system ${{\textit{Lang}}}_*\C_G$, on $G$.
This construction is functorial in the sense that, for
any morphism $\chi: G'\to G$ of algebraic  groups  over $\F$, there is a canonical
isomorphism $\wp_{\psi\ccirc\chi}\cong\chi^*\wp_\psi$.

We write $\AA:=\AA^1$ for an affine line and
$\GG:=\GG_m$ for the multiplicative group.
Thus, $\AA^n$  is an $n$-dimensional affine space,
resp. $\GG^n$  an $n$-dimensional torus.
From now on, we reserve the notation $\psi$ for an 
additive  character, i.e. for the case  $G=G_a$. The corresponding
local system  $\wp_\psi$ is 
the Artin-Schreier local system.
Given a scheme (or a stack) $X$ and a morphism
$f: X\to \AA$, we write $\wp^f:=f^*\wp_\psi$,
 a rank 1 local system on $X$ obtained from $\wp_\psi$ by pull-back
via $f$.

\subsection{}
We will systematically use the formalism of stacks.
Throughtout the paper, {\em a stack} means {\em an Artin
stack  over $\k$}.
Given a stack $X$ over $\k$ and a field extension $K/\k$
we write  $X(K)=X(\Spec K)$ for the groupoid  (= category)
of all 1-morphisms $x: \Spec K\to X$
in the 2-category of stacks.
If the morphism $x$ is defined over $\k$,
we  let $[x]$ denote
 the isomorphism class of
objects of the corresponding category over $\k$.
In this case we write  $[x]\in [X(\k)]$, where $[X(\k)]$ stands for the
set of  isomorphism classes of all $\k$-objects of  $X$.
 Given a pair of objects $x,x'\in X(\k)$ let $\Mor_\k(x,x')$ denote
the set of morphisms $f: x\to x'$.
We write $\TT X$, resp. $\TT^* X$, for the tangent, resp.   cotangent,
stack on $X$. 

Schemes will be identified with stacks via the functor of points,
where sets are viewed as
groupoids with trivial isomorphisms.
For a  scheme  $X$ over $\k$,
we will often identify 
the set $X(\k)$
of $\k$-rational points of $X$
with   the
set $[X(\k)]$ of isomorphism classes of objects 
of the corresponding stack, defined over $\k$.

Given a stack $X$
let $D(X)$ be the {\em bounded above} derived category
of constructible complexes on  $X$.
Abusing the terminology,
we will often refer to  an arbitrary object of  $D(X)$
as a  ``sheaf''.

We refer the reader to \cite{Ro}
for a detailed exposition of the formalism of group actions on stacks.
We will only consider $G$-actions 
where $G$ is a linear algebraic group over $\k$.
A stack $X$ equipped with a {\em strict} $G$-action,
cf. \cite[Definition 1.3]{Ro}, is called a
$G$-stack. 
There is a quotient stack  $X/G$ that comes equipped
with a canonical projection  $X\to X/G$.
The category $D(X/G)$ is equivalent
to  the $G$-equaivariant constructible derived category
on $X$. 
Given  a $G$-action on a scheme $X$ we
write $X\dsl G$ for the corresponding categorical quotient.
Then, there is a canonical morphism
${\mathfrak{w}}:\ X/G\to X\dsl G$.

\subsection{}\label{frobabs}
We will  need to take the trace of Frobenius for sheaves on
a  stack $X$ over a finite field $\k$. The corresponding theory is developed in
\cite{B2},  and \cite{LO}. 
In particular, Behrend introduces  \cite[\S6]{B2} a triangulated category
of {\em  absolutely
convergent} $\ell$-adic Weil complexes on $X$.
We will not reproduce a complete definition; an essential part being
 that  for any point $x:\pt \to X$ defined over $\k$, the
trace of the Frobenius inverse   on the (co)stalk  $x^!\cF$ of an absolutely
convergent constructible complex $\cF$ gives,
for every field imbedding ${\mathfrak j}: \bar{\mathbb Q}_\ell\into \CC$,
an  absolutely convergent series
\[\sum_{n\in\Z}\ \Bigl(\ \underset{\la\in \Spec
  (\fr\inv; {H^n(x^!\cF)})}\ssum\
\textrm{mult}(\la)\cdot
|{\mathfrak j}(\la)|\Bigr)\ <\infty,\]
where  $\|-\|$  stands for the 
absolute value of a complex number,
$\Spec(-)$  is the set  of eigenvalues
of $\fr\inv$ in $H^n(x^!\cF)$, and $\textrm{mult}(\la)$ 
is the corresponding multiplicity.

It will be more convenient for us to use a dual notion
of an absolutely
convergent constructible complex that involves
 the stalks $\cF_x:=x^*\cF$ rather than the costalks.
 We let $D\abs(X)$ denote
the corresponding bounded above triangulated category, 
for which one additionally requires that the weights of Frobenius 
on all the stalks be uniformly bounded. The cohomology sheaves
of an object of  $D\abs(X)$ vanish in degrees $\gg0$
(whereas, in  Behrend's definition the cohomology sheaves
 vanish in degrees $\ll 0$). 
 The results of \cite{B2} insure
that for any morphism $f: X\to Y$, of stacks over $\k$,
there are well defined functors
$f_!:\ D\abs(X)\to D\abs(Y)$ and
$f^*:\ D\abs(Y)\to D\abs(X)$.
Objects of the category $D\abs(\pt)$ will be referred to as
absolutely convergent 
complexes. Thus, for any $\cf\in D\abs(X)$ there
is an absolutely convergent complex $\RGam_c(X,\cf):=f_!\cf$,
where $f: X\to\pt$ is a constant map.

\begin{rem} Introducing a category of   absolutely convergent 
 constructible complexes is necessary already in the case
where $X=\pt/\GG$. Since $\pt/\GG$ is homotopy
equivalent to ${\mathbb P}^\infty$.
Thus, $H^\hdot(\pt/\GG,\CC)=\CC[u]$, where $\deg u=2$, so the complex $\RGam(\pt/\GG,\CC)$
has nonzero cohomology in infinitely many {\em nonnegative} degrees
and it is an absolutely convergent
constructible complex on $\pt$ in the sense of \cite{B2}.
On the other hand,  the complex $\RGam_c(\pt/\GG,\CC)$
 is an object of $D\abs(\pt)$, in our definition.
We have $H_c^\hdot(\pt/\GG,\CC)=\CC[u,u\inv]/\CC[u]$, so 
the complex $\RGam_c(\pt/\GG,\CC)$
has nonzero cohomology in infinitely many {\em negative} degrees:
$-2=2\dim(\pt/\GG),\,-4,-6,\ldots$.
\erem

Given an absolutely convergent 
complex $\cf\in D\abs(X)$, let $\hhh(\cf)$ denote the complex with zero differential
and terms $\hhh^n(\cf),\ n\in\Z$, the cohomology sheaves of $\cf$.
Let $K_{\textrm{abs}}(X)$ be a quotient of the Grothendieck group
of the category $D\abs(X)$ by the relations
$[\cf]=[\hhh(\cf)]$ for all $[\cf]\in D\abs(X)$ (this relation
is not, in general, a formal consequence of the relations
in the  Grothendieck group itself since $\cf$ may have infinitely
many nonzero cohomology sheaves).
In the special case $X=\pt$ we use simplified notation $K\abs=K\abs(\pt)$.
For any field imbedding ${\mathfrak j}: \bar{\mathbb Q}_\ell\into \CC$,
the map $V \mto \sum_{n\in\Z}\ (-1)^n\cdot {\mathfrak j}(\Tr(\fr; H^n(V)))$
induces a well-defined group homomorphism $K_{\textrm{abs}}\to\CC$.
From now on, we fix an imbedding ${\mathfrak j}$
and denote the corresponding homomorphism by $\mathsf{Tr}_{_\fr}$.
In particular, from the Remark above we find $\ltr \RGam_c(\pt/\GG,\qlb)=\frac{1}{q}+\frac{1}{q^2}+\ldots= \frac{1}{q-1}$,
where $q=\#\k$.

By definition, objects of $D\abs(X)$ come equipped
with a canonical weight filtration. 
 The graded pieces of the weight   filtration on $\cf$
are pure,
hence semisimple. We write $\cf\ss $ for  the associated graded 
 semisimple object and we call it 
the {\it semisimplification} of $\cf$. If the stack $X$ is an algebraic
variety or, more generally, an orbifold then in $K_{\textrm{abs}}(X)$
we have $[\cf]=[\cf\ss]$. This can be proved using the equations
$[\ce]=[\hhh(\ce)]$ and the fact that the cohomology sheaves
${\scr H}^i(\ce)$ of any perverse sheaf $\ce$ on $X$ vanish in degrees
$i>|\dim X|$.

We put
$\kabs:=\prod_{\bv\in\Z^I_{\geq0}} \ K_{\textrm{abs}}\cong K\abs [\![z]\!]$, where $z=(z_i)_\ii$.
We write an element of $\kabs$ either as
$V=\sum_\bv z^\bv V_\bv$ or just $V=\sum_\bv V_\bv$, if no confusion is possible.
We call $V_\bv$ the {\em component of  $V$ of dimension} $\bv$.
The
group $\kabs$ has the natural structure
of  a  pre $\lambda$-ring, in particular,
one has a product
 $[M]\cdot[M']:=[M\o M']$ and an operation $\Sym:\ \prod_{\bv>0}\,K\abs \to \kabs,\ [M]\mto \Sym[M]=\prod_\bv\Sym^\bv[M]$,
where $\Sym^\bv[M]=
\oplus_{m\geq 0}\ [((M^{\o m})_\bv)^{\si_m}]$ and $(M^{\o m})_\bv$ is the
component of $M^{\o m}$ with respect to the decomposition 
induced by the decomposition $M$ into components of various dimensions.
We identify $K\abs$ with the component of $\kabs$ corresponding to $\bv=0$.
The class $1=[\qlb]\in K\abs$  is a unit of $\kabs$.
Also, let  $\L$ be the class  $[\RGam_c(\AA, \qlb)]=\qlb[-2](-1)\in K\abs(\pt)$.
It will be convenient to enlarge $K\abs$ by adjoining a formal square root $\L^{1/2}:=[\qlb(\frac{1}{2})]$, of $\L$.
In the enlarged group, one has\
$\Sym(z\cdot\L^{1/2})=1+z\cdot\L^{1/2}+z^2\cdot\L+\ldots=\frac{1}{1-z\cdot\L^{1/2}}$. For $\cf\in D\abs(X)$,
we often use the notation $\cf\{n\}=\cf[n](\frac{n}{2})$. Thus, in $K\abs(X)$ one has $[\cf\{n\}]=\L^{-n/2}\cdot[\cf]$.


\subsection{Generalized Hua's formula}\label{eai-pf}

Let $f=\sum_{\bv\in \Z^I_{\geq 0}}\ z^\bv\cdot f_\bv$ be a formal power series
where each coefficient $f_\bv$ is a function $\Z^I_{\geq 1}\to \qlb,\ m\mto f_\bv(m)$.
For any such $f$ with $f_0=0$ one defines,  following \cite[\S 4]{Mo}, the plethystic exponential
of $f$ as a  formal power series $\Exp f=\sum_{\bv\in \Z^I_{\geq 0}}  z^\bv\cdot \Exp_\bv f$, where
the functions $\Exp_\bv f$ are determined by the equation
$$
\sum_{^{\bv>0}}
z^\bv\cdot (\Exp_\bv f)(m)\ :=\
\exp\left(\sum_{n\geq1}\sum_{^{\bv>0}}\  \frac{z^{n\cdot\bv}}{n}\cdot
f_\bv(n\cdot m)\right),\qquad m\in \Z^I_{\geq 1}.
$$

In a special case where the functions $f_\bv$ have the form
$f_\bv(n)={\mathbf f}_\bv(q^n)$,  for some $q\in\qlb$ and some Laurent polynomials
${\mathbf f}_\bv\in \qlb[{\mathbf q}]$ 
in one variable, the above formula
takes a more familiar form
\beq{exp-pol}\Exp\left(
\sum_{^{\bv>0}}
z^\bv\cdot {\mathbf f}_\bv({\mathbf q})\right)\ :=\
\exp\left(\sum_{n\geq1}\sum_{^{\bv>0}}\ \frac{z^{n\cdot\bv}}{n}
\cdot {\mathbf f}_\bv({\mathbf q}^{n})\right).
\eeq

Let $\k$ be a finite field and
$K/\k$ a finite field extention.
Let $(\cC,F)$ be a data as in \S\ref{exp-sec}, so $\cC$ is a $\k$-linear category and $F$
is an additive $\k$-linear functor from $\cC$ to the category of finite dimensional $I$-graded
$\k$-vector spaces that does not kill nonzero objects.
Let $\fX(\cC)=\sqcup_\bv\ \fX_\bv(\cC)$ be the corresponding
decomposition of the moduli stack of $\cC$.
 In \S\ref{obj}, we will define, by `extension of scalars',
a $K$-linear category $\cC_K=\cC_K$ equipped with a $K$-linear functor $F_K=K\o F$ such that  
the pair $(\cC_K, F_K)$ satisfies the assumptions of \S\ref{exp-sec} over the
the ground field $K$ and such that
for the corresponding moduli stacks 
there are canonical isomorphisms
$\fX_\bv(\cC_K)\cong K\o_\k\fX_\bv(\cC)$.
In particular, for the sets of isomorphism classes of objects
one has 
$[\op{Ob}_\bv(\cC_K)]=[\fX_\bv(\cC)(K)]$.
Further, for each $\bv\neq 0$, let $\AI_\bv(K)$ let 
be a subset of $[\op{Ob}_\bv(\cC_K)]$ formed by the absolutely indecomposable
objects of the category  $\cC_K$. 

Now,
for each $n\geq 1$ let $K_n$ be a degree $n$ extension of $\k$ and $\Tr_{K_n/\k}: \  K_n\to\k$
the trace map. Thus,  $\psi_n:=\psi\ccirc \Tr_{K_n/\k}$ is an  additive character of $K_n$  associated 
with the  additive character $\psi: \k\to\qlb^\times$.
Given a potential $\phi: \fX(\cC)\to \AA$, we define a 
function $\E_\bv(\cC, F,\phi): \Z_{\geq0}\to \qlb$, resp. $\E_\bv^{\AI}(\cC, F,\phi): \Z_{>0}\to \qlb$,
by the formula
$$\E_\bv(\cC, F,\phi)(n)=\sum_{^{x\in [\op{Ob}_\bv(\F_{q^n}\o\cC)}]}\
\psi_n(\phi(x)),\quad\text{resp.}\quad
\E_\bv^{\AI}(\cC, F,\phi)(n)=\sum_{^{x\in [\AI_\bv(\F_{q^n})]}}\
\psi_n(\phi(x)).
$$

Then, our generalization of Hua's formula reads

\begin{prop}\label{eai} For any data $(\cC,F,\phi)$ as in section \ref{exp-sec}, we have
$$
\sum_{\bv\in \Z_{\geq0}^I}\ z^\bv\cdot \E_\bv(\cC, F,\phi)\ = \
\Exp\left(\sum_{^{\bv>0}}\ z^\bv\cdot \E_\bv^{\AI}(\cC, F,\phi)\right).
$$
\end{prop}

In the special case where  $\phi=0$ and $\cC$ is the category of finite dimemsional quiver representations
over $\k=\F_q$, it was shown by V. Kac that the functions $\#[\fX_\bv(\cC)(\F_{q^n})$ and
$\#[\AI_\bv(\F_{q^n})]$ are Laurent polynomials in $q$.
In this case one can use \eqref{exp-pol} and  the formula of the proposition  reduces to the 
original formula due to Hua \cite{H}.

The proof the proposition given below amounts, essentially, to  a
more conceptual  reinterpretation of Hua's argument.
In fact, the proof  in {\em loc cit} can be carried over
in our, more  general, setting as well.
Indeed, the only property of
the function $\psi\ccirc\phi:\ X(\cC)(\bar\k)\to\bar\k$ which is necessary
for  Hua's argument to go through is the
equation $(\psi\ccirc\phi)(x\oplus x')=(\psi\ccirc\phi)(x)\cdot
(\psi\ccirc\phi)(x'),\ x,x'\in X(\cC)(\bar\k)$. The latter equation holds
thanks to the additivity of the potential $\phi$.

The proof of the Hua formula begins with  the following

\begin{lem}\label{krull-lem} For any finite field extension $K/\k$,
all $\Hom$-spaces in the the category $\cC_K$ have 
finite dimension over $K$ and, moreover, $\cC_K$ is a Krull-Schmidt category, i.e. any object is
isomorphic to a finite direct sum of indecomposable objects,
defined uniquely up to isomorphism and permutation of direct summands.
\end{lem}
\begin{proof}
Recall that automorphism groups
of  objects of an Artin stack are algebraic groups.
For $x\in\cC_K$, the group $\Aut(x)$ is the group of
invertible elements of the algebra $\End(x)$, hence it is
Zariski open in  $\End(x)$. Thus, we have
$\dim\End(x)=\dim\Aut(x)<\infty$.
Applying this to a direct sum $x\oplus y$ we deduce
that $\Hom(x,y)<\infty$.
The second statement holds by Theorem A.1 in \cite{CYZ}, 
which asserts that an additive Karoubian category has the 
Krull-Schmidt property if and only if $\End(x)$ is semi-perfect
for every object $x$ of this category. In particular this is true
in our case since $\End(x)$ is finite-dimensional, hence semi-perfect.
\end{proof}

Given a reduced scheme $Y$,  let $\Sym^n Y= Y^n\dsl\si_n$ and write $p_n: Y^n\to \Sym^n Y$ for the quotient map.
We put $\Sym Y=\sqcup_{n\geq 0}\ \Sym^n Y$ where $\Sym^0 Y=\{\pt\}$.
Given a sheaf $\cf\in D\abs(Y)$, for each $n\geq 1$ put  $\Sym^n\cf=\big((p_n)_!\cf^{\boxtimes n}\big)^{\si_n}$.
Also let $\Sym^0\cf=\C_\pt$.
Let $\Sym\cf$ be a sheaf on $\Sym Y$
such that  the restriction of $\Sym\cf$ to $\Sym^n Y$ equals $\Sym^n\cf$. 
Then, 
one has the following equations  for generating functions in one
variable $z$ (these equations are  well known from calculations involving the
$L$-function associated with the sheaf $\cf$):
\begin{multline}
\sum_{n\geq 0} z^n\cdot \left(\sum_{^{w\in (\Sym^n Y)^\fr}}\ \ltr H^\hdot(\Sym^n\cf|_w)\right)\ =\
\sum_{n\geq 0} z^n\cdot \ltr H^\hdot_c(\Sym^n Y, \,\Sym^n\cf)\label{expn}\\
=\ \exp\left(\sum_{n\geq 1}\ \frac{z^n}{n}\cdot \Tr_{\fr^n}| H^\hdot_c(Y,\cf)\right)\ 
=\ \exp\left(\sum_{n\geq 1}\ \frac{z^n}{n}\cdot 
\left(\sum_{^{y \in Y^{\fr^n}}}\
\Tr_{\fr^n}| H^\hdot(\cf|_y)\right)\right).
\end{multline}
Here, the first and third equality hold by the
Grothendieck-Lefschetz trace formula, and the second equality is a
consequence of the K\"unneth formula one

Observe that for any finite decomposition $Y=\sqcup_\al\ Y_\al$,
where   $i_\al: Y_\al\into Y$ is a locally closed subvariety,
in ${K\abs}$ one has
$[H^\hdot_c(Y,\cf)]=\sum_\al\ [H^\hdot_c(Y_\al, i_\al^*\cf)]$.
It follows that one can assign a well defined
element $[H^\hdot_c(Y,\cf)] \in{K\abs}$ and extend equations \eqref{expn} to 
sheaves on constructible schemes.
 Recall that a constructible scheme is, roughly speaking, a reduced scheme $Y$ defined up to an equivalence given
by
partitioning $Y$ into smaller,  locally closed subsets, cf. \cite{KS1}.

We will use a graded version of the above.
Specifically, let
$Y_\bv,\ \bv\in\Z^I_{\geq0}$, be a collection of reduced schemes of finite type over $\k$
and $\cf_\bv\in D\abs(Y_\bv)$, a collection of sheaves.
We put $Y=\sqcup_\bv\ Y_\bv$, a disjoint union of the schemes $Y_\bv$.
For each $n\geq 1$ and $\bv\in \Z^I_{\geq0}$ let $\Sym^\bv Y_n$ be  the image 
of $\cup_{\{\bv_1,\ldots\bv_n\in \Z^I_{\geq0}\mid
\bv_1+\ldots+\bv_n=\bv\}}\ \Sym^{\bv_1}Y\times\ldots \Sym^{\bv_n}Y$
under the quotient map  $p_n: Y^n\to \Sym^n Y$.
Let $\Sym^\bv \cf_n$ be  the restriction of the sheaf $\Sym^n\cf$ to  $\Sym^\bv Y_n$.
It is clear that $\Sym^\bv Y=\sqcup_n\ \Sym^\bv Y_n$. This is a disjoint union of finitely many schemes of finite type
and we have $\Sym^\bv \cf\in D\abs(\Sym^\bv Y)$, where 
$\cf$ is a sheaf such that $\cf|_{\Sym^\bv Y_n}=\Sym^\bv \cf_n$.
Further, it is clear that one has a decomposition
$\Sym Y=\sqcup_\bv\  \Sym^\bv Y$. 

For each $\bv$, write $z^\bv=\prod_\ii\ z_i^{v_i}$ and let $\E(Y_\bv,\cf_\bv):\ \Z_{>0}\to \qlb$ be a function defined by the
assignment $n\mto \sum_{^{y \in (\Sym^\bv Y)^{\fr^n}}}\
\Tr_{\fr^n}| H^\hdot(\cf_y)$.
With this notation, a generalization of \eqref{expn} to the
$\Z^I_{\geq0}$-graded setting reads
\beq{graded-exp}
\sum_{^{\bv\in\Z^I_{\geq0}}}\
z^\bv\cdot \left(\sum_{^{w\in (\Sym^\bv Y)^\fr}}\ \ltr H^\hdot(\Sym^\bv\cf|_w)\right)\ =\
\Exp\left(\sum_{^{\bv>0}}\  z^\bv\cdot \E(Y_\bv,\cf_\bv)\right).
\eeq

\begin{proof}[Sketch of proof of Proposition \ref{eai}]  
The argument below is based on an observation,
due to Kontsevich and Soibelman  \cite[\S2]{KS1}, 
that for each $\bv$ there is a canonically defined {\em  constructible} scheme  $X_\bv$ of finite type over $\k$
such that one has a bijection of sets
$[\fX_\bv(\bar\k)]\cong X_\bv(\bar\k)$. Here,  $[\fX_\bv(\bar\k)]$ stands for the set
of isomorphism classes of 
objects $\fX_\bv(\cC)(\bar\k)$, resp. $X_\bv(\bar\k)$ for the set of closed $\bar\k$-points of $X_\bv$.
Similarly, it was argued in  \cite{KS1} that for each $\bv\neq 0$ there is a canonically defined
constructible
$\k$-subscheme $X^{\textrm{ind}}_\bv$, of $X_\bv$, such that
 one has a bijection between the subset $[\fX^{\textrm{ind}}_\bv(\cC)(\bar\k)]\sset [\fX_\bv(\cC)(\bar\k)]$,
of isomorphism classes of 
 indecomposable objects, and the set $X^{\textrm{ind}}_\bv(\bar\k)$ of closed $\bar\k$- points of  
the constructible scheme  $X^{\textrm{ind}}_\bv$. 
It follows that for any finite field extension $K/\k$ there is a bijection
$X_\bv^{\textrm{ind}}(K)\cong [\AI_\bv(\cC_K)]$. Further, one has an ind-constructible scheme
$X=\sqcup_\bv\ X_\bv$, resp. $X^{\textrm{ind}}=\sqcup_{\bv>0}\ X^{\textrm{ind}}_\bv$. 
Then,
the Krull-Schmidt property implies that the direct sum map,
$\oplus: \Sym^n X^{\textrm{ind}}\to X,\ (x_1,\ldots,x_n)\mto x_1\oplus\ldots\oplus x_n$,
induces a bijection of sets
$\dis\oplus:
(\Sym X^{\textrm{ind}})(\bar\k)\iso
X(\bar\k)$.
This bijection sends $(\Sym^\bv X^{\textrm{ind}})(\bar\k)$ to  $X_\bv(\bar\k)$.
Further, a potential $\phi: \fx(\cC)\to\AA$ induces  morphisms $X_\bv\to\AA$, of  constructible schemes.
The additivity  of the potential implies an isomorphism $\wp^{\phi\ccirc\oplus}\cong
\wp^\phi\boxtimes\ldots\boxtimes\wp^{\phi}$, of sheaves on $\Sym X^{\textrm{ind}}$.
We conclude that 
equation \eqref{graded-exp} becomes, in the case where $Y_\bv=X^{\textrm{ind}}_\bv$
and $\cf_\bv=\wp^{\phi}|_{X^{\textrm{ind}_\bv}}$,
the required equation of Proposition \ref{eai}.
\end{proof}

\subsection{Generalized isotypic decomposition}\label{unipc} The construction of this subsection will (only) be
used in \S\ref{fac-res}.

Consider a diagram 
$\wt X\xrightarrow a  X\xrightarrow{b} Y$
of morphisms of schemes and put  $c:=b\ccirc  a $.
We make the following assumptions:
\begin{enumerate}
\item The morphism $ a $
is  an unramified finite Galois covering;
\item Each of the morphisms  $b$ and $c$ is
a Zariski locally
trivial fibration;
\item All fibers of the  morphisms $b$ and $c$
are connected.
\end{enumerate}

Let  $\Gamma$ be 
the Galois group  of the Galois covering $a$ and
the group   $\Gamma$ acts naturally on $\wt X$.
A   local system  on $\wt X$ is said to be constant along $c$ if
it is isomorphic to a pull-back via $c$ of a local system  on $Y$.
Let $\wt\cf$ be a  $\Gamma$-equivariant local
system on $\wt X$ which is constant along $c$.
Therefore, on any local system $\cf$ on $X$  one obtains,thanks to
assumption (3), 
an action 
of the local system of groups which is obtained by taking  
the associated bundle for the action of $\Gamma$ on itself by conjugation.
The group $\Gamma$ being finite, this yields  a canonical direct sum
decomposition into isotypic components:
\beq{isotyp1}
\cf\ =\ \oplus_{\rho\in \Irr(\Ga)}\ \cf^\rho,
\eeq
where   $\Irr(\Ga)$ denotes
 the set
of isomorphism classes of irreducible $\Ga$-representations in
finite dimensional $\qlb$-vector spaces.

Next, we say that a  $\qlb$-local system $\tilde\cf$ on $\wt X$    is {\em unipotent along $c$}
if  it  admits
 a finite filtration  by local sub-systems such that 
$\gr\tilde\cf$,
an associated graded  local system, is constant along $c$.
The restriction of such an $\tilde\cf$ to any fiber of $c$ is a local
system with  unipotent monodromy.

\begin{lem}\label{jordan} 
Let $\cf$ be a  $\qlb$-local system 
on $X$ such that the local system $a^*\cf$, on $\wt X$, is
  unipotent  along $c$.
Then, there is a canonical  `{\em generalized} isotypic decomposition':
\beq{isotypic} \cf\ =\ \oplus_{\rho\in \Irr(\Ga)}\ \cf^{\langle\rho\rangle}.
\eeq
The  direct summand $\cf^{\langle\rho\rangle}$ is uniquely determined by the requirement that
it admits a filtration such that
 $a^*\gr(\cf^{\langle\rho\rangle})$, a pull-back of an associated graded local system,
is constant along
$c$ and using the notation of \eqref{isotyp1}, we have
$\gr(\cf^{\langle\rho\rangle})=(\gr(\cf^{\langle\rho\rangle}))^\rho$,
equivalently,  $(\gr(\cf^{\langle\rho\rangle}))^\sig=0,\ \forall \sig\in
\Irr(\Ga),\
\sig\neq\rho$.
\end{lem}

\begin{proof} It will be convenient to identify  $\qlb$-local systems
on a scheme
with representations of the geometric (\'etale) fundamental group of that
 scheme.

We have the following
 diagram 
$$
{\small
\xymatrix{
1\ar[r] &\ \pig(\wt X)\  \ar@{->>}[d]^<>(0.5){c_*}\ar[r] & \ \pig(X)\  
\ar@{->>}[d]^<>(0.5){b_*}\ar[r] &\ \Gamma\ar[r]
& 1\\
& \ \pig(Y)\ \ar@{=}[r]^<>(0.5){\Id}& \ \pig(Y)\  &&
}}
$$
where  the maps $b_*$ and $c_*$
are induced by the morphisms $b$ and $c$, respectively.
It follows from assumption (1)
that the horizontal row in the diagram is an exact sequence,
and it follows  from assumptions (2)-(3) that the vertical maps are
surjective.

Let $K=\Ker(b_*)$, resp. $\wt K=\Ker(c_*)$.
We may view $\pig(\wt X)/\wt
K$ and $K/\wt K$ as normal subgroups of $\pig(X)/\wt K$.
Then, diagram chase yields $\pig(X)/\wt K=(\pig(\wt X)/\wt K)\cdot (K/\wt K)$ and
$(\pig(\wt X)/\wt K)\cap (K/\wt K)=\{1\}$.
It follows that the group $\pig(X)/\wt K$ is a direct product of the
subgroups $\pig(\wt X)/\wt
K$ and $K/\wt K$. Further, we
 have that $K/\wt K\cong\Gamma$ and $\pig(\wt X)/\wt
K=\pig(Y)$.
We deduce
an exact sequence
\[
1\to \wt K \to \pig(X)\to \Gamma\times \pig(Y)\to 1,
\]
such that the group $\pig(\wt X)\sset \pig(X)$ gets identified with the
preimage of $\{1\}\times \pig(Y)\sset \Gamma\times \pig(Y)$.

Now, let $f: \pig(X)\to \GL(V)$ be the representation  that corresponds to a
local system $\cf$ on $X$. The  assumption that  $a^*\cf$
be unipotent along $c$ translates into the condition
that $f|_{\wt K}$ is a unipotent representation of the group $\wt K$.
In such a case, by elementary group theory, there is a canonical
generalized isotypic  decomposition $V|_{K}= \oplus_{\rho\in \Irr(\Ga)}\
V^{\langle\rho\rangle}$ where each $V^{\langle\rho\rangle}$ is $K$-stable. The direct summand  $V^{\langle\rho\rangle}$ is
characterized by the property that it is the
maximal  $K$-subrepresentation of $V|_{K}$ 
such that any irreducible $K$-subquotient of  $V^{\langle\rho\rangle}$ is isomorphic
to a pull-back of $\rho$ via the projection $K\onto \Gamma$.

We claim that each space  $V^{\langle\rho\rangle}$ is in fact
$\pig(X)$-stable. To see this, choose a $\pig(X)$-stable increasing
filtration $0=V_0\sset V_1\sset\ldots \sset V_m=V$ such that
the action of $\wt K$ on $V_i/V_{i-1}$ is trivial.
To prove the claim it suffices to show that for any
 $i\geq 1$ 
the space $V_i\cap V^{\langle\rho\rangle}$ is
$\pig(X)$-stable and, moreover, one has
\[V_i=\oplus_{\rho\in \Irr(\Ga)}\
(V_i\cap V^{\langle\rho\rangle}).
\]
This is easily proved by  induction on $i$
using that the $\pig(X)$-action on  $V_i/V_{i-1}$
factors through  an action of
$\pig(X)/\wt K=\Gamma\times \pig(Y)$
and  the characterizing property of the subspaces
$ V^{\langle\rho\rangle}$.

The claim insures that, for each $\rho\in\Irr(\Gamma)$,
there is a local subsystem $\cf^{\langle\rho\rangle}$, of $\cf$,
that corresponds to the $\pig(X)$-subrepresentation
$ V^{\langle\rho\rangle}\sset V$.
We leave to the reader to check that the resulting decomposition 
$\cf=\oplus_{\rho\in \Irr(\Ga)}\ \cf^{\langle\rho\rangle}$ satisfies all
the required
properties.
\end{proof}

Similarly, we say that a  local system $\cf$ on $X$    is {\em
  unipotent relative
to $c$}
if it  admits
 a finite filtration  by local sub-systems
such that $a^*\gr\cf$, a pull-back of an associated graded  local system, is
  isomorphic to a  local system of the form $c^*\cf_Y$ for some local system $\cf_Y$ on $Y$.
Let $D^{\op{unip}}\abs(X,c)$,
denote the full triangulated subcategory of $D\abs(X)$ whose objects are
complexes $\cf$ such that each cohomology sheaf of $\cf$
is  a  local system unipotent relative to $c$.
Let  $K^{\op{unip}}\abs(X,c)$,
denote the corresponding 
Grothendieck group with the modification similar to
the one explained in \S\ref{frobabs}.
Also, let $K^{\op{loc}}\abs(Y)$ be the Grothendieck group
of the category of absolutely convergent local systems on $C$.
The functor $\RGam_c(Y,-)$ induces a homomorphism
$K^{\op{loc}}\abs(Y)\to K\abs(\pt),\
[\cf_Y]\mto [\RGam_c(Y,\cf_Y)]$.

Let $\Irr(\Gamma)$ be the set
of isomorphism classes of irreducible $\Gamma$-representations in
finite dimensional $\qlb$-vector spaces
For each  $\rho\in\Irr(\Gamma)$, we define
a group  homomorphism
\beq{hor}
K^{\op{unip}}\abs(X,c)\to 
K^{\op{loc}}\abs(Y),\quad [\cf]\mto [\cf]^{\langle\rho\rangle},
\eeq
 as follows.

Let $\cf\in D^{\op{unip}}\abs(X,c)$.
One can assume,  without changing the class of
$\cf$ in the Grothendieck group, that
$\cf$ is  a  local system and, morover, that $ a^* \cf\cong c^*\cf_Y$ for
some local system  $\cf_Y$ on $Y$
The  local system  $ a^* \cf$ comes equipped with a canonical
$\Gamma$-equivariant
structucture which induces, via the isomorphism  $ a^* \cf\cong
c^*\cf_Y$, a $\Gamma$-equivariant
structucture on $c^*\cf_Y$.
The  group $\Gamma$ acts along the fibers of $ c$
and these fibers are connected.
Therefore, we have $\cf_Y=R^0 c_*  a^* \cf$ (an underived
direct image) 
and giving the
$\Gamma$-equivariant
structucture on $ c^*\cf_Y$ is equivalent to giving an action 
 $\Gamma \to \Aut(\cf_Y)$, of $\Gamma$ on $\cf_Y$.
Hence,  one has   a canonical direct sum
decomposition into $\Gamma$-isotypic components:
\beq{isotyp}
\cf_Y\ =\ \oplus_{\rho\in \Irr(\Ga)}\ \cf_Y^\rho
\eeq
By definition, the map \eqref{hor} sends the class
$[\cf]\in K^{\op{unip}}\abs(X,c)$
to the class $[\cf]^{\langle\rho\rangle}:=[(R^0 c_*  a^*
\cf)^\rho]\in K^{\op{loc}}\abs(Y)$.
It is easy to check that this map is well-defined and the class thus defined
does not depend on the choices of various filtrations
involved in the construction. In particular, note that the uniform boundedness of weights in the definition of $K_{\rm abs}(X)$ implies that the K-group above  is the same as the K-group of pure complexes, hence the K-group is the same as for sheaves with monodromy along the fibers of $c$. Therefore the above K-group decomposes as a direct sum over representations of $\Gamma$, and the $\Gamma$-isotypic components along the fibers and descend to the base.

\begin{rem} Under these conditions we also have an isomorphism obtained by taking the direct sum of the generalized isotypic component morphisms over all irreducible representations of the finite group $\Gamma$
$$K_{\rm abs}^{\rm unip}(X,c) \to \oplus_{\rm Irr(\Gamma)} K_{\rm abs}^{\rm loc}(Y)$$
\end{rem}


\section{Factorization sheaves}\label{sec3}

In this section we introduce factorization sheaves and prove
a result (Theorem \ref{expC}) involving  cohomology of such sheaves and the Fourier transform.
This theorem, which has an independent interest, plays a crucial
role in relating exponential sums over indecomposable objects 
with the geometry of moment maps.

\subsection{}\label{cond-gen}
Given  a nonempty  finite set  $J$ let
 $\gS_J$ denote the group of bijections $J\to J$
and $\sign: \si_J\to\{\pm1\}$  the sign-character.
Given a positive integer $v$ we write
 $[v]:=\{1,2,\ldots,v\}$  and
$\si_v:=\si_{[v]}$, the Symmetric group on 
$v$ letters. 
For $v=0$, we define $[v]=\emptyset$, resp.
$\si_v=\{1\}$. Given a dimension vector $\bv= \{v_i\}_{i\in I}$,
we put $[\bv]=\sq_i\ [v_i]$, resp. $|\bv|=\sum_i\ v_i$,
and $\si_\bv=\prod_i \si_{[v_i]}=\prod_i \si_{v_i}$.

Let $\cur$ be a connected scheme over a field $\k$
and $\cur^J$ the scheme of maps
$\gamma: J\to \cur$.
 This scheme  is isomorphic
to a cartesian power of $\cur$  (for $J=\emptyset$ we declare that
$\cur^J=\{\pt\}$)
and it comes equipped with
a natural $\gS_J$-action.
For an integer $v\geq 0$, resp.  dimension vector $\bv= \{v_i\}_{i\in I}$,
we  use simplified notation
$\cur^v:=\cur^{[v]}$, resp.    $\cur^\bv:=\cur^{[\bv]}$.

Let  $\bv^1,\ldots,\bv^m$ be an $m$-tuple   of dimension vectors.
One has the following
 canonical  maps
\beq{imath}
\pprod_\al\ \cur^{\bv_\al}/\gS_{\bv_\al}
\ \cong\ \cur^{\sqcup_\al\, [\bv_\al]}\big/\big(\pprod_\al \si_{[\bv_\al]}\big)\
\to\
\cur^{\sqcup_\al\, [\bv_\al]}/\si_{\sqcup_\al\, [\bv_\al]}\
\cong\
\cur^{\bv^1+\ldots+\bv^m}\!/\gS_{\bv^1+\ldots+\bv^m}.
\eeq
Here, the last isomorphism uses the fact that for any set $J$ with $v$ elements there is
a canonical isomorphism $C^J/\si_J\cong C^v/\si_v$,
of stacks. 
Let $\imath_{\bv^1,\ldots,\bv^m}$ denote the
composite map in \eqref{imath}.

Further,  we put
\beq{hec}
\cur\he_{\bv^1,\ldots,\bv^m}:=\{\gamma:
\sq_\al [\bv^\al]\to C\en \big|\en 
\gamma([\bv^\al])\cap \gamma([\bv^\beta])=\emptyset,
\quad\forall\ \al\neq\be\}.
\eeq
Thus, $\cur\he_{\bv^1,\ldots,\bv^m}$ is a Zariski open and dense,
$(\prod_\al \si_{\bv^\al})$-stable
subscheme of $\prod_\al\ \cur^{\bv^\al}$.
We obtain a diagram
\beq{ij}
\xymatrix{
\cur\he_{\bv^1,\ldots,\bv^m}\!/(\prod_\al\gS_{\bv^\al})\
\ar@{^{(}->}[rr]^<>(0.5){\jmath_{\bv^1,\ldots,\bv^m}} &&
\  \prod_\al (\cur^{\bv^\al}/\gS_{\bv^\al})\
\ar@{->>}[rr]^<>(0.5){\imath_{\bv^1,\ldots,\bv^m}} &&
\   \cur^{\bv^1+\ldots+\bv^m}/\si_{\bv^1+\ldots+\bv^m},
}
\eeq
where $\jmath_{\bv^1,\ldots,\bv^m}$ is the
natural open imbedding.

Let $\tau: C^\bv/\si_\bv\times C^\bw/\si_\bw\iso
C^\bw/\si_\bw\times C^\bv/\si_\bv$ be
the flip morphism.
 It is clear that $\tau$ maps
 the set $\cur\he_{\bv,\bw}/(\si_\bv\times\si_\bw)$
to $\cur\he_{\bw,\bv}/(\si_\bw\times\si_\bv)$.
 Furthermore, the following diagram commutes
\[
\xymatrix{
\cur\he_{\bv,\bw}/(\si_\bv\times\si_\bw)\ar[rrrr]^<>(0.5){\tau}
\ar[drr]_<>(0.5){\imath_{\bv,\bw}^{^{}}\hphantom{x}\,}&&&&
\cur\he_{\bw,\bv}/(\si_\bw\times\si_\bv)
\ar[dll]^(0.5){\,\hphantom{x}{\,}^{^{}}\imath_{\bw,\bv}}\\
&&\cur^{\bv+\bw}/\si_{\bv+\bw}&&
}
\]
It follows that there is a isomorphism of functors
$\psi_{\bv,\bw}: 
\imath_{\bv,\bw}^*\iso
\tau^*\ccirc\imath_{\bw,\bv}^*$.

\begin{defn}\label{fac-def} 
A {\em  weak factorization sheaf } $\cF$ on 
 $\Sym\cur$
is the data of a collection
  $\ (\cF_\bv)_{\bv\in\Z^I_{\geq0}}$,\ $\cF_\bv\in D(\cur^\bv/\gS_\bv)$, equipped,
for each pair $\bv_1,\bv_2$ of dimension vectors, with an isomorphism
\[\vphi_{\bv_1,\bv_2}:\ \jmath_{\bv_1,\bv_2}^*\imath_{\bv_1,\bv_2}^*\cF_{\bv_1+\bv_2} \to 
\jmath_{\bv_1,\bv_2}^*(\cF_{\bv_1} \boxtimes \cF_{\bv_2}),\]
of sheaves on $\cur\he_{\bv_1,\bv_2}/(\si_{\bv_1} \times\si_{\bv_2})$,
such that:
\begin{description}
\item[Associativity constraint]
 For any triple $\bv_1,\bv_2,\bv_3$, of dimension vectors
(using slightly imprecise notation) one has
\beq{fact2}
\jmath_{\bv_1,\bv_2,\bv_3}^*\big((\vphi_{\bv_1,\bv_2}\boxtimes \Id_{\bv_3})\ccirc\vphi_{\bv_1+\bv_2,\bv_3}\big)\
=\ \jmath_{\bv_1,\bv_2,\bv_3}^*\big(\vphi_{\bv_1,\bv_2+\bv_3}\ccirc(\Id_{\bv_1}\boxtimes
\vphi_{\bv_2,\bv_3})\big).
\eeq
\item[Commutativity constraint] 
 For any $\bv,\bw$, the following
diagram commutes
\beq{fac-diag}
\xymatrix{
\jmath_{\bv,\bw}^*\imath_{\bv,\bw}^*\cF_{\bv+\bw}\
\ar[d]^<>(0.5){\jmath_{\bv,\bw}^*(\psi_{\bv,\bw})}
\ar[rr]^<>(0.5){\vphi_{\bv,\bw}}
&&
\ \jmath_{\bv,\bw}^*(\cF_{\bv} \boxtimes
\cF_{\bw})\ar[d]^<>(0.5){\cong}\ 
\\
\tau^*\jmath_{\bw,\bv}^*\imath_{\bw,\bv}^*
\cF_{\bv+\bw}\
\ar[rr]^<>(0.5){\tau(\vphi_{\bw,\bv})}&&
\ \tau^*\jmath_{\bw,\bv}^*(\cF_{\bw}\boxtimes \cF_{\bv})
}
\eeq
\end{description}
\noindent
Here the vertical map on the right comes from the
natural isomorphism $\cF_{\bv} \boxtimes
\cF_{\bw}\iso{\tau_{\bv,\bw}^*(\cF_{\bw} \boxtimes
\cF_{\bv})}$.
\end{defn}

Given an action on $\cur$ of a group $H$,
we get the diagonal $H$-action on $\cur^\bv$ for any \ig set $\bv$.
An  $H$-{\em equivariant}   weak factorization sheaf 
 is a  weak factorization sheaf  $\cF=(\cF_\bv)$
 equipped, for each $\bv$,
 with an $H$-equivariant structure on  $\cF_\bv$
and such that the morphisms
$\vphi_{\bv^1\ldots,\bv^m}$ respect the  $H$-equivariant
structures (here  $\cF_{\bv^1} \boxtimes \ldots\boxtimes\cF_{\bv^m}$, an $H\times\ldots\times
H$-equivariant sheaf, 
is viewed as an $H$-equivariant sheaf
via  the diagonal imbedding $H\into H\times\ldots\times H$).

Let $\cf=(\cf_\bv)$ be a factorization sheaf.
Using the associativity constraint  one constructs inductively,
for  any $m\geq 2$ and $\bv_1,\ldots,\bv_m$,
 an  isomorphism
\beq{phi-m}
\vphi_{\bv^1,\ldots,\bv^m}:\
\jmath^*_{\bv^1,\ldots,\bv^m}\imath^*_{\bv^1,\ldots,\bv^m}
\cF_{\bv_1+\ldots+\bv_m}\to
\jmath^*_{\bv^1,\ldots,\bv^m}(\cF_{\bv^1}\boxtimes\ldots\boxtimes\cF_{\bv^m})
\eeq
MacLane's  coherence
argument implies that this isomorphism is independent of
the induction process.

Let  
$N(\si_{\bv_1,\ldots,\bv_m})$ be the normalizer of
the Young subgroup $\si_{\bv_1}\times\ldots\times \si_{\bv_m}$ in
$\si_{\bv_1+\ldots+\bv_m}$. An element $\sig\in N(\si_{\bv_1,\ldots,\bv_m})$ induces a 
an automorphism $\tau_\sig$ of  $C^{\bv_1}\times\ldots\times C^{\bv_m}$.
One has a commutative diagram
\[
\xymatrix{
\cur\he_{\bv_1,\ldots,\bv_m}/\prod_\al\,\si_{\bv_\al}\ar[rrrr]^<>(0.5){\tau_\sig}
\ar[drr]_<>(0.5){\imath_{\bv_1,\ldots,\bv_m}}&&&&
\cur\he_{\bv_1,\ldots,\bv_m}/\prod_\al\,\si_{\bv_\al}
\ar[dll]^(0.5){\,\hphantom{x}{\,}^{^{}}\imath_{\bv_1,\ldots,\bv_m}}\\
&&\cur^{\bv_1+\ldots+\bv_m}/\si_{\bv_1+\ldots+\bv_m}&&
}
\]
Thus, there is a isomorphism of functors
$\psi_\sig: 
\imath_{\bv_1,\ldots,\bv_m}^*\iso
\tau^*_\sig\ccirc\imath_{\bv_1,\ldots,\bv_m}^*$.
 Using that symmetric groups are generated
by transpositions, from  the commutativity constraint  \eqref{fac-diag}
 one deduces that
for any $m\geq 2$ and $\sig\in N(\si_{\bv_1,\ldots,\bv_m})$ the
following diagram commutes
\beq{nor}
\xymatrix{
\jmath_{\bv_1,\ldots,\bv_m}^*\imath_{\bv_1,\ldots,\bv_m}^*\cF_{\bv_1+\ldots+\bv_m}\
\ar[d]^<>(0.5){\jmath_{\bv_1,\ldots,\bv_m}^*(\psi_{\bv_1,\ldots,\bv_m})}
\ar[rr]^<>(0.5){\vphi_{\bv_1,\ldots,\bv_m}}
&&
\ \jmath_{\bv_1,\ldots,\bv_m}^*(\cF_{\bv_1} \boxtimes\ldots\boxtimes
\cF_{\bv_m})\ar[d]^<>(0.5){\cong}\ 
\\
\tau_\sig^*\jmath_{\bv_1,\ldots,\bv_m}^*\imath_{\bv_1,\ldots,\bv_m}^*
\cF_{\bv_1+\ldots+\bv_m}\
\ar[rr]^<>(0.5){\tau_\sig^*(\vphi_{\bv_1,\ldots,\bv_m})}&&
\ \tau_\sig^*\jmath_{\bv_1,\ldots,\bv_m}^*(\cF_{\bv_1}\boxtimes\ldots\boxtimes \cF_{\bv_m})
}
\eeq

\begin{rems}\label{fact-rems} The following is clear
\begin{enumerate}
\item Given a Weil (as opposed to a general absolutely convergent) sheaf $\ce$ on
a scheme (or a Deligne-Mumford stack),
 its Verdier dual is again a Weil sheaf, to be denoted $\ce^\vee$. Let
$\cF=(\cF_\bv)$ be a weak factorization sheaf
where each $\cf_\bv$ is a Weil sheaf. 
Then, the collection  $\cF^\vee=(\cF_\bv^\vee)$ is a weak factorization sheaf.

\item For any pair   $\cF=(\cF_\bv),\
\cF'=(\cF_\bv')$, of weak factorization sheaves on $\Sym \cur$,
the collection
  ${\cF\o\cf'} =(\cF_\bv\o\cf_\bv')$ has the natural
structure  of a weak factorization sheaf on $\Sym \cur$.

\item Let $\CC^\sign_{_{C^\bv}}$ be a rank 1 constant sheaf on $C^\bv$ equipped with
a (unique) $\si_\bv$-equivariant structure such that the group $\si_\bv$ acts in
the stalks over points of the  principal diagonal of $C^\bv$  via the sign character
of the group $\si_\bv$.
The collection $\{\CC^\sign_{_{C^\bv}}[d_\bv]\}$, where $d_\bv=(\sum v_i)\cdot\dim C$
has the natural structure of  a weak factorization sheaf. We remark that without
homological shift by $d_\bv$, the commutativity of \eqref{fac-diag} fails already in the
case where the set $I$ consists of 1 element, $C=\AA$, and $v=w=1$.

\item For any  \f sheaf $\cf$ on $\Sym C$ and an open subset  $C_0\sset C$,
the collection
  $\cF|_{C_0}=(\cf_\bv|_{C_0^\bv})$
is a weak factorization sheaf on $\Sym C_0$.
\end{enumerate}
\end{rems}
 We  often consider  weak factorization sheaves
$\cF=(\cF_\bv)$  such that each  $\cf_\bv$ is an absolutely convergent
sheaf and all the data involved in
the above defininition is
compatible with the Frobenius. 
Similarly, in the case where the ground field is  $\k=\CC$
we may (and will)  use the category
of mixed Hodge $\D$-modules as a replacement of the
category of mixed $\ell$-adic sheaves. 
We will not explicitly
mention  `absolutely convergent', resp. `mixed $\ell$-adic' or `mixed Hodge',
when dealing with 
weak factorization sheaves  in those settings.

There are several other, slightly different  variants of the 
notion of   weak factorization sheaf  which are stronger than the
one given in Definition \ref{fac-def}. In particular, one has

\begin{defn}\label{fac-str}
A {\em factorization sheaf} on 
 $\Sym \cur$ is a collection $(\cf_\bv)_{\bv\in\Z_{\geq0}^I}$ equipped, for each $\bv,\bw$, with morphisms
$\bar\vphi_{\bv,\bw}:\
\imath_{\bv,\bw}^*\cF_{\bv,\bw} \to 
\cF_{\bv} \boxtimes\cF_{\bw}$ (of sheaves on
$\cur^{\bv}\times\cur^{\bw}$ rather than on the open subset
$\cur\he_{\bv,\bw}$)
such that the associativity constrain \eqref{fact2} is replaced by
$(\bar\vphi_{\bv_1,\bv_2}\boxtimes \Id_{\bv_3})\ccirc\bar\vphi_{\bv_1+\bv_2,\bv_3}\
=\ \bar\vphi_{\bv_1,\bv_2+\bv_3}\ccirc(\Id_{\bv_1}\boxtimes
\bar\vphi_{\bv_2,\bv_3})$
 and also the following analogue of diagram \eqref{fac-diag} commutes
$$
\xymatrix{
\imath_{\bv,\bw}^*\cF_{\bv+\bw}\
\ar[d]^<>(0.5){\psi_{\bv,\bw}}
\ar[rr]^<>(0.5){\bar\vphi_{\bv,\bw}}
&&
\  \cF_{\bv} \boxtimes
\cF_{\bw}\ar[d]^<>(0.5){\cong}\ 
\\
\tau^*\imath_{\bw,\bv}^*
\cF_{\bv+\bw}\
\ar[rr]^<>(0.5){\tau(\bar\vphi_{\bw,\bv})}&&
\ \tau^*(\cF_{\bw}\boxtimes \cF_{\bv})
}
$$

A factorization sheaf $\cf$ is called a {\em strong factorization
  sheaf}
if each of the  morphisms
$\bar\vphi_{\bv,\bw}$ is itself an  isomorphism.
Clearly, we have
\[\text{strong factorization
  sheaves}\en\Rightarrow\en
\text{factorization
  sheaves}\en\Rightarrow\en\text{weak factorization
  sheaves.}
\]
\end{defn}

\begin{rem} There are analogues of the  notions
of factorization
  sheaves, resp. weak and strong factorization
  sheaves,
involving  higher categories where condition \eqref{fact2}
 has an infinite
sequence of  higher analogues.
In that  context,
analogues of  weak  factorization sheaves are are usually called
`factorization algebras', cf. e.g. \cite[\S 2.5]{FG}.
These are sheaves  on 
$\cv\cur^\bv/\si_\bv$ where the latter is viewed as a factorization $\infty$-stack.
Higher analogues of  factorization sheaves considered in Definition \ref{fac-str}(i)
are called   \emph{factorization $E_1$-algebras}, cf. \cite{Lur}.
These are associative algebras in the category of factorization
algebras. Finally,  objects analogous to
 strong  factorization sheaves turn out to form a
rather small subclass of objects which are determined by their
components with  dimension vectors $\bv=1_i,\ \ii$.
\erem

\subsection{A cohomology result on factorization sheaves}
\label{norm-bun} Let $\bv$ be a dimension vector and
$\AA\sset\AA^\bv$   the principal diagonal.
We put $\ft=\AA^\bv/\AA$. 
Using  the identification  $(\AA^\bv)^*=\AA^\bv$, we have
$\dis
\ft^*\ =\ \big\{(z_i^\al)_{\ii,\al\in[1,v_i]}\in\AA^\bv
\en\big|\en  \ssum_{i,\,\al}\ z_i^\al\ =\
  0\big\}$,
a codimension one hyperplane of
$\AA^\bv$. 
We define a subset $\ft_\bv\reg\sset\ft_\bv^*$  as follows
\beq{de2}
\ft_\bv\reg=\biggl\{(z_i^\al)\in\AA^\bv\ \bigg|\ 
{\small\begin{array}{l}
\text{(1)\en $\sum_{\ii}\sum_{\al\in[v_i]}\ z_i^\al=0$;
\quad (2)\en 
$z_i^\al\neq z_i^\be$\en for any $\al\neq \be$;
 and}\\
\text{(3)\en for  any $I$-tuple
 of subsets $J_i\subseteq [v_i],\ \ii$, such that
$J_i\neq\emptyset$,}\\ \text{resp. $J_i\neq [v_i]$, for at least one $i$,
one has  $\sum_\ii \sum_{\al\in J_i}\  z_i^\al  \neq
0$.}
\end{array}
}
\biggr\}
\eeq
This is  a $\GG\times\si_\bv$-stable, Zariski  open and dense subset of $\ft_\bv^*$.

\begin{rem} 
Condition (2) in \eqref{de2} says that  the $\si_\bv$-action on  $\ft_\bv\reg$
is free.
The meaning of  condition (3) will be explained in  \S\ref{diag} below.
We note that conditions (1) and (3) in the RHS of \eqref{de2}
are additive analogues of
conditions (2) and  (3), respectively, from section \ref{deligne-sec}).
\erem

Now, fix a  a smooth scheme $C$
with tangent, resp. cotangent, sheaf $T_C$, resp. $T^*_C$.
The  normal, resp.
conormal, sheaf to $C\sset C^\bv$, the principal diagonal,  is canonically isomorphic to
$\ft\o T_C$, resp. $\ft^*\o T^*_C$.
Therefore, $\N_\bv$, resp. $\N_\bv^*$, the total space of the normal, resp. conormal, bundle on $C$
is isomorphic to the  total space of the vector bundle associated with the sheaf $\ft\o T_C$, resp.
$\ft^*\o T^*_C$. We have $\rk \N=\dim \ft\cdot \dim C=(|\bv|-1)\dim C$.
Let $T^\times C$ be the complement of the zero section in the 
cotangent bundle $T^*C\to C$. We define $\N_\bv\reg:=\ft_\bv\reg\times_\GG T^\times C$.
This is  a  $\si_\bv$-stable, Zariski open and dense  subset of $\N_\bv^*$ that has been mentioned
in
\S\ref{h-intro}.
A geometric interpretation of the set $\N_\bv\reg$ will become clear from Lemma  \ref{veryreg}.

%



Let $\Delta(\bv): C\to C^\bv$ be the diagonal imbedding, $u: C\into \N_\bv$ the zero section, and
$v: \N_\bv^*\to C$ the projection. 
Let $\FDN_\N: D\abs(\N_\bv)\to D\abs(\N^*_\bv)$ be the (relative) Fourier-Deligne transform 
functor along the fibers of the vector bundle $\N\to C$.
We choose a self-dual normalization of $\FDN_\N$
such that $\FDN_\N$ takes perverse sheaves to perverse sheaves.
Then, for any sheaf $\ce$ on $C$, one has $\FDN_\N(u_!\ce)=v^*\ce[\rk\N_\bv](\frac{\rk\N_\bv}{2})$.
There is also  a similar  functor $\FDN_\N: D\abs(\N_\bv/\si_\bv)\to D\abs(\N^*_\bv/\si_\bv)$, between $\si_\bv$-equivariant categories.
Let $\Phi_\N: D\abs(C^\bv/\si_\bv)\to D\abs(\N_\bv\reg/\si_\bv)$ denote the composition 
functor $\cf\mto (\FDN_{\N}\ccirc\op{sp}_\N\cf)|_{\N\reg/\si_\bv}$.


Note that $\si_\bv$-action on $\N_\bv\reg$ is free and the setup of Section \ref{unipc} applies to the natural
maps
 $\N_\bv\reg\xrightarrow{a}\N_\bv\reg/\si_\bv\xrightarrow{b} C$ and  $c=b\ccirc a$. 
Write $\gr_W(-)$ for the associated graded with respect to the weight filtration.
The main result of section \ref{sec3}, which is the precise version of equation 
\eqref{expC-intro},  reads

\begin{thm}\label{expC}
Let $\cF=(\cF_\bv)_{\bv\in\Z^I_{\geq0}}$
 be a  weak factorization sheaf  on $\Sym C$
such that  for every $\bv$ the sheaf $\Delta(\bv)^*\wt\cf_\bv$ is a
local system  on $C$. Then, we have

\vi For each $\bv$, the sheaf $a^*\gr_W\Phi_\N(\cf_\bv)$ is a pull-back of a sheaf on $C$ via $b$, therefore,
$\Phi_\N(\cf_\bv)$ is an object of
$D^{\op{unip}}\abs(\N\reg/\si_\bv,c)$.

\vii In $\kabs$, one has
 the following equation
$$
  \bigoplus_\bv\  \big[\RGam_c(C^\bv/\si_\bv,\ \cF_\bv)\big]\ =\
\Sym\left(\bigoplus_{\bv\neq 0}\ \RGam_c\big(C,\
[\Phi_\N(\cf_\bv)]^{\langle\triv\rangle}\big)\{\dim C(|\bv|-1)\}\right).
$$
\end{thm}

\noindent
Here,  $\{n\}=[n](\frac{n}{2})$ and the symbol $[\ldots]^{\langle\triv\rangle}$ in the RHS of the formula of part (ii) refers to
the map \eqref{hor} in the case $\rho=\triv$. This makes sense thanks to  the statement of part (i).

The above theorem will be proved in \S\ref{aa-sec}.


\subsection{The case $C=\AA$}
\label{fac-res}
In this paper, we will only be  interested in the case where the scheme
$\cur$ is either $C=\AA$ or $C=\GG$.
Thus,
$\AA^\bv:=\cur^\bv$ is
an affine space, resp.  $\GG^\bv:=\cur^\bv$ is
an algebraic torus. 
Let $\Aff=\GG\ltimes \GG_a$ be the group of
affine-linear transformations $x \mto ax+b$, of $\AA$.
The natural $\Aff$-action  on $\AA$ induces, for any dimension vector
$\bv$, the diagonal $\Aff$-action 
on $\AA^\bv$ that descends to  $\AA^\bv/\si_\bv$.
Let  $\Aff$ act on $\ft$ through its quotient $\GG$ by dilations
and act on $\ft\times\AA$, resp. $\ft^*\times \AA$, diagonally.
Then, there is a natural $\Aff$-equivariant
trivialization 
$\N=\ft\times\AA$, resp. $\N^*=\ft^*\times \AA$.
Thus, we have $\N\reg:=\AA\times \ft_\bv\reg$.
The action of $\si_\bv$ on $\ft_\bv\reg$ being free,  the canonical
map $\ft_\bv\reg/\si_\bv\to\ft_\bv\reg\dsl\si_\bv$ is an isomorphism, so we will
make no distinction between $\ft_\bv\reg/\si_\bv$ and $\ft_\bv\reg\dsl\si_\bv$.
Let $\delta: \N^*/\si_\bv=\ft^*/\si_\bv\times \AA\to\ft_\bv^*/\si_\bv$
be the first projection.

Let  $\cf_\bv$ 
be an $\Aff$-equivariant sheaf 
on $\AA^\bv/\si_\bv$. Then, $\cf_\bv$ descends
to a sheaf $\bar\cf_\bv$ on $(\AA^\bv/\AA)/\si_\bv$. Moreover, $\bar\cf_\bv$ 
is $\GG$-equivariant  with respect to the dilation action on  $\AA^\bv/\AA$.
This implies a natural isomorphism
 $\op{sp}_\N(\cf_\bv)\cong\cf_\bv$.

Let $\FDN_{\ft}$ denote the Fourier-Deligne transform
$D\abs(\ft_\bv/\si_\bv)\to D\abs(\ft_\bv^*/\si_\bv)$.
The $\Aff$-equivariant structure on $\cf_\bv$ provides
the sheaf $\FDN_{\N}(\cf_\bv)$,
on $\N_\bv^*/\si_\bv=\AA\times \ft_\bv^*/\si_\bv$,
with a natural  equivariant structure with respect to the
action of the additive group on $\ft_\bv^*/\si_\bv\times \AA$
by translation along the second factor and one has
a canonical isomorphism $\FDN_{\N}(\cf_\bv)\cong\delta^*\FDN_{\ft}(\bar\cf_\bv)$.
We put $\Phi_\ft(\cf_\bv):=\FDN(\bar\cf_\bv)|_{\N_\bv\reg/\si_\bv}$.
By construction, we have $\Phi_\N(\cf_\bv)=\delta^*\Phi_\ft(\cf_\bv)$.

Let $\cF:=(\cF_\bv)_{\bv\in\Z^I_{\geq0}}$ be
be an $\Aff$-equivariant weak factorization sheaf on $\Sym \AA$.
Then,
 each of the sheaves $\Delta(\bv)^*\wt\cf_\bv$
is an $\Aff$-equivariant, hence a geometrically constant,
sheaf on $\AA$. Thus, applying part (i) of Theorem
\ref{expC} we deduce that the local system $\Phi_\ft(\cf_\bv)$
is unipotent relative to the constant map
$\ft_\bv\reg/\si_\bv\to\pt$.
It follows that there is a canonical  `{\em generalized} isotypic decomposition':
$$ \Phi_\ft(\cf_\bv)\ =\ \oplus_{\rho\in \Irr(\si_\bv)}\ 
\Phi^{\langle\rho\rangle}_\ft(\cf_\bv).
$$

Let $\eta$ be a $\k$-rational closed point of $\ft_\bv\reg/\si_\bv$ and
 $\Phi_\ft(\cf_\bv)^{\langle\rho\rangle}_\eta$ be the restriction of $\Phi^{\langle\rho\rangle}_\ft(\cf_\bv)$ to $\eta$.

For  every $\bv$, one has a natural diagram
$\ \GG^\bv/\gS_\bv\ \xrightarrow{\eps}\ \AA^\bv/\si_\bv\  \xleftarrow{i} \ 0/\si_\bv$,
where the  map $\eps$ is  an open imbedding induced by the natural  imbedding
$\GG\into\AA$ and $i$ is a closed imbedding. Let
\[\xymatrix{
D\abs(\GG^\bv/\gS_\bv)\ &
\ D\abs(\AA^\bv/\si_\bv)\ \ar[l]_<>(0.5){\eps^*}\ar[r]^<>(0.5){i^*}&
 \ D\abs(0/\si_\bv)
}
\]
be the corresponding restriction functors.
For any $\Aff$-equivariant   weak factorization sheaf $\cf$  on $\Sym \AA$,
the collection  $\eps^*\cf=(\eps^*\cf_\bv)$ gives
a $\GG$-equivariant weak factorization sheaf  on  $\Sym \GG$, cf. Remark \ref{fact-rems}(4).


The following result will be deduced from Theorem \ref{expC}.

\begin{thm}\label{Exp}
Let $\cF=(\cF_\bv)_{\bv\in\Z^I_{\geq0}}$
 be an $\Aff$-equivariant  weak factorization sheaf  on $\Sym \AA$.

\vi The monodromy action of the subgroup  $\pia(\ft_\bv\reg)\sset \pia(\ft_\bv\reg/\si_\bv)$ on
$\Phi_\ft(\cF_\bv)_{\eta}$ is unipotent. 

\vii In $\kabs$, one has
 the following equations
\begin{align}
 \sum_\bv\ z^\bv\cdot [\RGam_c(\AA^\bv/\gS_\bv,\cF_\bv)]   
&=\Sym\Big(-\L^{\frac{1}{2}}\cdot\sum_{^{\bv>0}}\ (-1)^{|\bv|}\cdot z^\bv\cdot 
\L^{\frac{|\bv|}{2}}\cdot [\Phi_\ft(\cf_\bv)_\eta]^{\langle\op{triv}\rangle}\Big)\label{2eqs}\\
\sum_\bv\
[\RGam_c(\GG^\bv/\gS_\bv,\,\eps^*\cF_\bv)]
& =
\Sym\Big((\L^{-\frac{|\bv|}{2}}-\L^{\frac{|\bv|}{2}})\cdot\sum_{^{\bv>0}}\ (-1)^{|\bv|}\cdot z^\bv\cdot 
\L^{\frac{|\bv|}{2}}\cdot [\Phi_\ft(\cf_\bv)_\eta^{\langle\op{triv}\rangle}]\Big)\label{3eqs}\\
\sum_\bv
\ [\RGam_c(0/\gS_\bv,\,i^*\cF_\bv)]
& =
\Sym\Big(-\L^{\frac{1}{2}}\cdot \sum_{^{\bv>0}}\ (-1)^{|\bv|}\cdot z^\bv\cdot \L^{-\frac{|\bv|}{2}}\cdot [\Phi_\ft(\cf_\bv)_\eta]^{\langle\op{triv}\rangle}\Big).
\label{4eqs}
\end{align}
\end{thm}

%

%
\subsection{Diagonal stratification}\label{diag}
Given a dimension vector $\bv\neq 0$,
let  $\wt\Ups(\bv)$ be the set of
decompositions
$[\bv]=\sq_\al\ J^\al$  into a disjoint union of  {\em nonempty} subsets
$J^\al$.
Let  $\Ups(\bv)$ be the set of  unordered collections $(\bv^\al)_{\al}$,
of   {\em nonzero}
dimension
vectors, such that 
 $\bv=\sum_\al\ \bv^\al$.
For $(\bv^\al)_{\al}\in \Ups(\bv)$,
let  $d_\bw$  be the number of occurences of a dimension vector
$\bw$ as an element of the collection $(\bv^\al)_{\al}$.
We see that giving  an element of  $\Ups(\bv)$
is equivalent to giving a  collection $\{d_\bw\in \Z_{\geq 0},\ \bw\in Z^I_{\geq 0}\sminus
\{0\}\}$ such that $\sum_\bw\ d_\bw\cdot \bw=\bv$.
Thus, the set   $\Ups(\bv)$ may be identified with the set of multi-partitions of ~$\bv$.

For $\a,\fB\in\wt\Ups(\bv)$, resp.
 $\fa,\fb\in\Ups(\bv)$,
we write  $\a\leq {\mathfrak B}$,  resp.
 $\fa\leq\fb$, whenever 
$\fB$ is a refinement of $\a$, resp.
 $\fb$ is a refinement of 
 $\fa$.
This gives a partial order on  $\wt\Ups(\bv)$,
resp. ~$\Ups(\bv)$.
By definition, $\fB\geq\a=(J^\al)$ holds 
iff the parts of $\fB$ are obtained by partitioning
further each of the sets $J^\al$.
For any  tuple
$\{\fb^\al=(\bw^{\al,\beta})_\beta\in \Ups(\bv^\al),
\ \al\in A\}$, where $A$ is a finite set $A$,
 we have
$\sum_{\al,\beta}\ \bw^{\al,\beta}=\sum_\al\ \bv^\al$,
so the collection $(\bw^{\al,\beta})_{\al,\beta}$
gives an element of $\Ups(\sum_\al\ \bv^\al)$.
This way one obtains  a bijection
\[\prod\nolimits_{\al}\ \Ups(\bv^\al)\ \iso\ 
\{\fb\in\Ups(\mbox{$\sum_\al\ \bv^\al$})\mid \fb\geq\fa\},
\quad \{(\bw^{\al,\beta})_\beta\in \Ups(\bv^\al),\
\ \al\in A\}\ \mto\ (\bw^{\al,\beta})_{\al,\beta}.\]


Any \dct $\a=(J^\al)\in \wt\Ups(\bv)$
gives, for each $\ii$, a \dct $[v_i]=\sqcup_\al\ ([v_i]\cap J^\al)$.
We put $\si_\a=\prod_{i,\al}\ \si_{[v_i]\cap J^\al}$,
a subgroup of $\si_\bv$.
Further, let $\bv^\al:=(v_i^\al)_\ii$ be a  dimension vector
with components  $v_i^\al:=\#([v_i]\cap J^\al)$.
Since $\bv=\sum_\al\ \bv^\al$
we obtain a multi-partition $\fa=(\bv^\al)\in \Ups(\bv)$.
The assignment $\a\mto\fa$ provides   a natural map
$\wt\Ups(\bv)\to\Ups(\bv)$.

The group $\si_\bv$ acts naturally on the set  $\wt\Ups(\bv)$.
The  stabilizer in  $\si_\bv$
of a \dct $\a=(J^\al)_{\al\in A}\in \wt\Ups(\bv)$ equals
$N(\si_\a)$, the normalizer of the subgroup
$\si_\a\sset\si_\bv$. We have an isomorphism
\beq{N(si)}
N(\si_\a)\cong \prod_{\bw\in\Z^I}\ \big((\si_\bv)^{d_\bw}\rtimes\si_{d_\bw}\big)
\quad\text{where}\quad
d_\bw:=\#   \{\al\in A\mid \bw=\bv^\al\}.
\eeq
Observe further that
the fibers of the map $\wt\Ups(\bv)\to\Ups(\bv)$
are precisely the $\si_\bv$-orbits. 
We will identify $\Ups(\bv)$
with the set of $\si_\bv$-orbits in   $\wt\Ups(\bv)$.


Fix a smooth curve $\cur$. For $r>0$   let $\cur^{r,\circ}\sset \cur^r$
be a  subset formed by the
$r$-tuples of pairwise distinct points of $\cur$.
Given a dimension vector $\bv$
and
$\a=\{J^\al,\ \al\in[r]\}\in \wt\Ups(\bv)$,
we consider  a  closed imbedding $C^r\into C^\bv$  defined, using the
\dct $[\bv]=\sqcup_\al\ J^\al$, by the assignment
\[
z=\big([r]\to C,\ \al\mto z(\al)\big)\en\longmapsto\en 
\bar\Delta(z)=\big(\sqcup_\al J^\al \to C,\ J^\al\mto z(\al)\big).
\]
We will usually identify $C^r$, resp.  $\cur^{r,\circ}$, with its
image, to be denoted $\cur^\bv_\a$, resp.  ${\cur^{\bv,\circ}_\a}$.
Thus,  $\cur^\bv_\a$ is
a smooth closed subscheme of $C^\bv$ and  ${\cur^{\bv,\circ}_\a}$
is a Zariski open and dense subset of
$\cur^\bv_\a$.
Write $\Delta(\a): C^{r,\circ}=C^{\bv,\circ}_\a\into C^\bv$, resp.
$\bar\Delta(\a): C^r=\cur^\bv_\a\into C^\bv$,
for  the 
corresponding  locally closed, resp. closed, imbedding.
 It is clear that, for $\a,{\mathfrak B}\in\wt\Ups(\bv)$,
we have  $\a\leq {\mathfrak B}$  iff
one has an inclusion $\cur^\bv_\a\subseteq \cur^\bv_{\mathfrak B}$.
In particular,  $ \cur^\bv_\a=\cur^\bv_{\mathfrak B}$ holds iff
$\a={\mathfrak B}$. Note also that
the sets
 ${\cur^{\bv,\circ}_\a}^{}$ and $\cur^{\bv,\circ}_{\mathfrak
  B}$ are equal if $\a={\mathfrak B}$, 
and 
are disjoint otherwise.
It follows that $\cur^\bv_{\mathfrak B}=
\bigsqcup_{\a\leq {\mathfrak B}}\ \cur^{\bv,\circ}_\a$.
In particular, one has 
 a stratification
$\cur^\bv=\bigsqcup_{\a\in\wt\Ups(\bv)}\ \cur^{\bv,\circ}_\a$.
The action of the
group $\si_\bv$ on $\cur^\bv$ respects the stratification
and induces 
a  stratification  $\cur^\bv/\si_\bv$ with
 locally-closed smooth strata which are labeled
by the set ~$\Ups(\bv)$.
These stratifications of either $\cur^\bv$ or $\cur^\bv/\si_\bv$  will be
called   `{\em diagonal stratification}'.

In the special case where
$r=1$ and $\a=\{J^1\}$ is the \dct with a single part $J^1=[\bv]$,
we have that
  $\cur^\bv_\a=\cur^\bv_\Delta$ is  the principal diagonal and
$\Delta(\a)=\bar\Delta(\a)=\Delta(\bv)$ is the diagonal imbedding.
In this case, abusing the notation, we write $\a=\Delta_\bv$
for the corresponding \dct. 
Observe that the principal diagonal
 is contained in the closure of any other
stratum, and it
is the  unique closed stratum of the diagonal stratification.

For any  $\a\in\wt\Ups(\bv)$, we have closed imbeddings
$C\xrightarrow{\Delta(\bv)} C^\bv_\a\to C^\bv$.
Let $\TT_C (C^\bv_\a)$ denote the normal bundle to $\Delta(\bv)(C)\sset C^\bv_\a$,
resp.  $\TT_C ^\perp(C^\bv_\a)$ a vector sub-bundle of
$\N_\bv^*$ whose fibers are 
the annihilators of the corresponding fibers of $\TT_C (C^\bv_\a)$.
Thus, one has  natural
diagrams
\beq{TT}
\xymatrix{
C\ &
\ \TT_C(C^\bv_\a)\ \ar@{->>}[l]_<>(0.5){u_\a}\ar@{^{(}->}[r]^<>(0.5){v_\a}&
\ \TT_C(C^\bv)=\N_\bv,\ &\
C\ &
\ \TT_C ^\perp(C^\bv_\a)\ \ar@{->>}[l]_<>(0.5){u_\a^\perp}
\ar@{^{(}->}[r]^<>(0.5){v_\a^\perp}&
\ \TT^*_C(C^\bv)=\N_\bv^*.
}
\eeq

The following result is  straighforward; it explains the geometric meaning of the set $\N_\bv\reg$.

\begin{lem}\label{veryreg} The set  $\N_\bv\reg$ is equal to the set of elements  $z\in\N_\bv^*$ such that
$w(z)\neq z$ for any $w\in\si_\bv,\ w\neq ~1$, 
and, moreover,
$z\notin\TT_C ^\perp(C^\bv_\a)$
for any $\a\in \Ups(\bv),\ \a\neq\Delta$. \qed
\end{lem}

\subsection{}\label{AB}
Fix $\a\in\wt\Ups(\bv)$ and
let  $\fa=(\bv=\bv^1+\ldots+\bv^r)\in\Ups(\bv)$ be an associated multi-partition.
We have a product of diagonal imbeddings  $\prod_{\al\in[r]}\ \Delta(\bv^\al):\
C^r\into \prod_{\al\in[r]}\ C^{\bv^\al}$.
Note that the image of the open subset $C^{r,\circ}\sset C^r$
is contained in $C\he_{\bv^1,\ldots,\bv^r}$, cf. \eqref{hec}.
We obtain  a commutative diagram
\beq{circ}
\xymatrix{
  C\he_{\bv^1,\ldots,\bv^r}\ \ar@{->>}[d]^<>(0.5){\vp_{\bv^1,\ldots,\bv^r}}&&
\   C^{r,\circ}=   C^{\bv,\circ}_\a\
\ar[rr]^<>(0.5){\Delta(\a)}
\ar@{_{(}->}[ll]_<>(0.5){\prod_{\al\in[r]}\ \Delta(\bv^\al)}
&&
\   C^\bv\  \ar@{->>}[d]^<>(0.5){\varpi_\bv}\\
 C\he_{\bv^1,\ldots,\bv^r}\!/(\prod_{\al=1}^r\gS_{\bv^\al})\
\ \ar@{^{(}->}[rr]^<>(0.5){\jmath_{\bv^1,\ldots,\bv^r}} &&
\ \prod_{\al=1}^r\ ( C^{\bv^\al}/\gS_{\bv^\al})\
\ar@{->>}[rr]^<>(0.5){\imath_{\bv^1,\ldots,\bv^r}} &&
\   C^\bv/\si_\bv.
}
\eeq

It is often convenient to choose and fix an  identification
$[\bv^1]\sqcup\ldots\sqcup[\bv^r]\cong [\bv]$. Such an   identification
 determines a lift of the  bijection $\Ups(\bv^1)\times\ldots\times
\Ups(\bv^r)
\ \iso\ \{\fb\in\Ups(\bv)\mid \fb\geq\fa\}$
to a bijection 
$\gamma: \wt\Ups(\bv^1)\times\ldots\times\wt\Ups(\bv^r)
\ \iso\ \{\fB\in\wt\Ups(\bv)\mid \fB\geq\a\}$.
This provides
an identification $
\prod_\al  C^{\bv^\al}
\iso  C^\bv$ and, therefore,
 an identification
of $C\he_{\bv^1,\ldots,\bv^r}$ with an open subset of
$C^\bv$ that lifts the map
$\imath_{\bv^1,\ldots,\bv^r}\ccirc\jmath_{\bv^1,\ldots,\bv^r}$.
It is immediate to see that this way, for any
$\fB^1\in  \wt\Ups(\bv^1),\ldots,\fB^r\in \wt\Ups(\bv^r)$, we get
\beq{strat1}
C\he_{\bv^1,\ldots,\bv^r}\ \cap\
( C^{\bv^1,\circ}_{\fB^1}\times\ldots\times C^{\bv^r,\circ}_{\fB^r})\ =\
 C^{\bv,\circ}_{\gamma(\fB^1,\ldots,\fB^r)}.
\eeq
It will sometimes be convenient to use a separate notation
$ C^\circ_{\fB^1,\ldots,\fB^r}$ for the set on the left of this
equation whenever we want to think of it as
  a Zariski open  subset of
$ C^{\bv^1,\circ}_{\fB^1}\times\ldots\times C^{\bv^r,\circ}_{\fB^r}$.

Thus,  the diagonal stratification of $C^\bv$
induces the following  stratification
\beq{strat2} C\he_{\bv^1,\ldots,\bv^r}=
\bigsqcup_{\fB^1\in \wt\Ups(\bv^1),\ldots,\fB^r\in\wt\Ups(\bv^r)}\
C^\circ_{\fB^1,\ldots,\fB^r}\ \ =\ \
\bigsqcup_{\fB\geq \a}\  C^{\bv,\circ}_\fB.
\eeq

Below, we will freely use the above identifications without further
mention.

\begin{rem}\label{hec-rem} Note that we have $
  C\hec_{\Delta(\bv^1),\ldots,\Delta(\bv^r)}=C^\bv_\a\cong  C^{r,\circ}$
is the unique closed stratum of the
stratification.
\end{rem}


Let  $\cf=(\cf_\bv)$ be  a weak  factorization sheaf  on $ C$
and recall the notation $\wt\cf=(\wt\cf_\bv),\
\wt\cf_\bv=\varpi_\bv^*\cf_\bv$.

We say that  $\cf$ is  diagonally constructible, resp.
diagonally  constant,  if for all $\bv$ and  $\a\in \wt\Ups(\bv)$
the sheaf
$\Delta(\a)^*\wt\cf_\bv$
is a locally constant,  resp.  {\em geometrically} constant,
sheaf on $C^{\bv,\circ}_\a$.

\begin{lem}\label{constr} For $\cf$ to be  diagonally constructible, resp.
diagonally  constant, it is sufficient that
$\Delta(\bv)^*\wt\cf_\bv$, the restriction
of $\wt\cf_\bv$ to the principal diagonal,
be a locally constant,  resp. geometrically constant, sheaf
 for all $\bv$.
\end{lem}

\begin{proof} 
Fix $\bv$ and $\a=\{J^\al, \al\in[r]\}\in\wt\Ups(\bv)$. 
Let   $(\bv=\bv^1+\ldots+\bv^r)\in\Ups(\bv)$ be the  decomposition
associated with $\a$ and for each $\al=1,\ldots,r$ let 
 $\fB^\al\in \Ups(\bv^\al)$.

Observe that the stratum $C^\bv_\a$ is stable under the action
of the group  $N(\si_\a)\sset \si_\bv$. Using \eqref{phi-m} and
 isomorphisms \eqref{strat1}-\eqref{strat2}
we deduce an isomorphism
\beq{BBB}
(\Delta(\gamma(\fB^1,\ldots,\fB^r))^*\wt\cF_\bv)|_{C\he_{\bv^1,\ldots,\bv^r}}
\ \iso\
(\Delta(\fB^1)^*\wt\cF_{\bv^1}
\boxtimes\ldots\boxtimes\Delta(\fB^r)^*\wt\cF_{\bv^r})|_{C\he_{\bv^1,\ldots,\bv^r}}.
\eeq
Moreover, diagram \eqref{nor} says that this isomorphism
is compatible with the natural action of the group
$N(\si_\a)\cong N(\si_{\bv_1,\ldots,\bv^r})$
on each side of \eqref{BBB}.

If  $\Delta(\bv)^*\wt\cf_\bv$ is  locally constant for all
  $\bv$ then so is the sheaf $(\Delta(\bv^1)^*\wt\cf_{\bv^1}\boxtimes\ldots
\boxtimes\Delta(\bv^r)^*\wt\cf_{\bv^r})|_{C\hec_{\Delta,\ldots,\Delta}}$,
for all decompositions $\bv^1+\ldots+\bv^r=\bv,\ r\geq1$.
Combining  Remark \ref{hec-rem} and isomorphism \eqref{BBB} in the special case 
$(\fB^1,\ldots,\fB^r)=(\Delta(\bv^1),\ldots,\Delta(\bv^r))$, we deduce
that $\Delta(\a)^*\wt\cf_\bv$
is a locally constant
sheaf on $C^{\bv,\circ}_\a$ for all $\a\in \wt\Ups(\bv)$.
An identical argument applies in the geometrically constant case as well.
\end{proof}

\subsection{Semisimple factorization sheaves}
\label{ss-sec}

Let  $Z$ be  a  geometrically connected variety.
An object of $D\abs(Z)$  is called {\em a semisimple local system}
 if  it is
isomorphic to a finite direct sum  of  objects
of the form  $V\o \cl$,
where  $V\in D\abs(\pt)$ is an absolutely convergent complex with zero
differential
and $\cl$ is a  finite dimensional  {\em geometrically} irreducible
 local system on $Z$.


\begin{defn}\label{ss-def} 
A sheaf  on $C^\bv$ is said to be {\em  nice} if it is
isomorphic to a  direct sum  of the form
$\bplus_{\a\in\wt\Ups(\bv)}\ 
\bar\Delta(\a)_*\ce^{\a}$, 
where $\ce^{\a}$ is a  semisimple local system on
$C^\bv_\a$.
A weak  factorization sheaf $\cf$ is called   nice, resp. semisimple,
if so is  the sheaf
$\wt\cf_\bv$ for every $\bv$.
\end{defn}

\newcommand{\tcf}{\wt\cf}

\begin{lem}\label{ss}

\vi Let  $\cf=(\cf_\bv)$ be a 
nice weak factorization sheaf so,  for each $\bv$, 
 we have  $\cf_\bv=\bplus_{\a\in\wt\Ups(\bv)}\ \cf^{\a}$,
where  $\cf^{\a}=\bar\Delta(\a)_*\ce^{\a}$  for some   semisimple local system $\ce^{\a}$ on
$C^\bv_\a$. 
Then,  for  every $\bv$ and $\a\in\wt\Ups(\bv)$, there is an isomorphism
\beq{ssfact}
\ce^\a\ccong \boxtimes_{\al\in [r]}\ 
\ce^{\Delta(\bv^\al)},
\eeq
of $N(\si_\a)$-equivariant  local systems on
$C^\bv_\a\cong  \prod_{\al\in[r]}\  C^{\Delta(\bv^\al)}$,
where $\bv=\bv^1+\ldots+\bv^r$ is the decomposition
associated with $\a$ and $r$ is the number of parts.

\vii Let  $\cf=(\cf_\bv)$ be a semisimple weak
factorization sheaf such
that $\Delta(\bv)^*\wt\cf_\bv$   is a  semisimple local system
for any $\bv$.
 Then $\cf$ is nice.
\end{lem}



 \begin{proof} 
\vii Fix $\a\in\wt\Ups(\bv)$ and use the notation as above.
Given an $r$-tuple $\fB^\al\in\wt\Ups(\bv^\al),\ \al\in[r]$,
let 
$$C_{\fB^1,\ldots,\fB^r}:=
C\he_{\bv^1,\ldots,\bv^r}\ \cap\
( C^{\bv^1}_{\fB^1}\times\ldots\times C^{\bv^r}_{\fB^r})
\ =\
C\he_{\bv^1,\ldots,\bv^r}\ \cap\
C_{\gamma(\fB^1,\ldots,\fB^r)}.
$$ 
Thus, the set $C_{\fB^1,\ldots,\fB^r}$ equals  the closure
of the set $C^\circ_{\fB^1,\ldots,\fB^r}=C\he_{\bv^1,\ldots,\bv^r}\ \cap\
( C^{\bv^1,\circ}_{\fB^1}\times\ldots\times C^{\bv^r,\circ}_{\fB^r})$ in
$C\he_{\bv^1,\ldots,\bv^r}$, cf.  \eqref{strat2}.
Let $\Delta\he(\fB^1,\ldots,\fB^r):\
C_{\fB^1,\ldots,\fB^r}\into 
C\he_{\bv^1,\ldots,\bv^r}$ denote the closed imbedding.

We have the following diagram

$$
\xymatrix{
C^\circ_{\fB^1,\ldots,\fB^r}
\ar@{=}[d]_<>(0.5){\eqref{strat1}}
\ar@{^{(}->}[rr]&&C_{\fB^1,\ldots,\fB^r}\ar@{^{(}->}[d]
\ar@{^{(}->}[rrr]^<>(0.5){\Delta\he(\fB^1,\ldots,\fB^r)}&&&
C\he_{\bv^1,\ldots,\bv^r}\ar@{^{(}->}[d]^<>(0.5){\wt\jmath_{\bv^1,\ldots,\bv^r}}\\
C^{\bv,\circ}_\fB\ar@{^{(}->}[rr]&&C^\bv_\fB\ar@{^{(}->}[rrr]^<>(0.5){\bar\Delta(\fB)}
&&& C^\bv
}
$$

Hence, we find
\begin{align*}
\wt\cf_\bv|_{C\he_{\bv^1,\ldots,\bv^r}}
&=\bigoplus_{\fB \in\Ups(\bv)}\ \cf^\fB|_{C\he_{\bv^1,\ldots,\bv^r}}
\ =\  \bigoplus_{\fB \geq\a}\ 
(\bar\Delta(\fB)_*\ce^\fB)|_{C\he_{\bv^1,\ldots,\bv^r}} \\
&=\ \bigoplus_{(\fB^1,\ldots,\fB^r)\in
  \wt\Ups(\bv^1)\times\ldots\times\wt\Ups(\bv^r)}\
(\Delta\he(\fB^1,\ldots,\fB^r)_*
\bigl(\ce^{\gamma(\fB^1,\ldots,\fB^r)}|_{C_{\fB^1,\ldots,\fB^r}}\bigr).
\end{align*}

Similarly, we have
\begin{align*}
\bigl(\boxtimes_\al\ \wt\cf_{\bv^\al}\bigr)|_{C\he_{\bv^1,\ldots,\bv^r}}
&=\Bigl(\boxtimes_\al\
\bigl(\bplus_{\fB^\al \in\Ups(\bv^\al)}\ 
\cf^{\fB^\al}\bigr)\Bigr)|_{C\he_{\bv^1,\ldots,\bv^r}}\\
&=\ \bigoplus_{(\fB^1,\ldots,\fB^r)\in
    \wt\Ups(\bv^1)\times\ldots\times\wt\Ups(\bv^r)}\
\bigl(\boxtimes_\al\
\bar\Delta(\fB^\al)_*\ce^{\fB^\al}\bigr)|_{C\he_{\bv^1,\ldots,\bv^r}}
\\
& =\
\bigoplus_{(\fB^1,\ldots,\fB^r)\in
  \wt\Ups(\bv^1)\times\ldots\times\wt\Ups(\bv^r)}\
(\Delta\he(\fB^1,\ldots,\fB^r)_*
\bigl((\boxtimes_\al\ \ce^{\fB^\al})|_{C_{\fB^1,\ldots,\fB^r}}\bigr).
\end{align*}

Comparing the above isomorphisms,
we see that the factorization isomorphism $\vphi_{\bv^1,\ldots,\bv^r}$
yields, for each tuple $(\fB^1,\ldots,\fB^r)\in
  \Ups(\bv^1,\ldots,\bv^r)$,
an isomorphism 
\[\ce^{\gamma(\fB^1,\ldots,\fB^r)}|_{C_{\fB^1,\ldots,\fB^r}}
\cong (\boxtimes_\al\ \ce^{\fB^\al})|_{C_{\fB^1,\ldots,\fB^r}},
\]
of  local systems on
$C_{\fB^1,\ldots,\fB^r}$.
The local system on each side of the isomorphism
is a restriction of a  local system on $C^\bv_{\gamma(\fB^1,\ldots,\fB^r)}$.
The homomorphism
$\pig\Bigl(C_{\fB^1,\ldots,\fB^r})
\to \pig(C^\bv_{\gamma(\fB^1,\ldots,\fB^r)})$ being surjective,
we conclude that the  isomorphism of  local systems
on  $C_{\fB^1,\ldots,\fB^r}$ extends to an isomorphism 
$\ce^{\gamma(\fB^1,\ldots,\fB^r)}
\cong \boxtimes_\al\ \ce^{\fB^\al}$ of the corresponding 
 local systems
on $C^\bv_{\gamma(\fB^1,\ldots,\fB^r)}$.

In the special case where $(\fB^1,\ldots,\fB^r)=
(\Delta(\bv^1),\ldots,\Delta(\bv^r))$ one has
 $\gamma(\Delta(\bv^1),\ldots,\Delta(\bv^r))=\a$.
Thus, $\ce^{\gamma(\fB^1,\ldots,\fB^r)}
\cong \boxtimes_\al\ \ce^{\fB^\al}$ reduces to
an isomorphism as in \eqref{ssfact}. 
Furthermore, one checks that this isomorphism  respects
the $N(\si_\a)$-equivariant structures.

\vii From the factorization isomorphism
$(\Delta(\bv^1)^*\wt\cf_{\bv^1}\boxtimes\ldots
\boxtimes\Delta(\bv^r)^*\wt\cf_{\bv^r})|_{C\hec_{\Delta,\ldots,\Delta}}
\cong \Delta(\a)^*\wt\cf_\bv$ we deduce that $\cf$ is diagonally
constructible, cf.  proof of Lemma \ref{constr}.
Therefore, the semisimple sheaf
$\wt\cf_\bv$ is isomorphic to
$\oplus_\a\ \Delta(\a)_{!*}\ce^{\a,\circ}$, a direct sum of 
intersection cohomology complexes associated to some
local systems $\ce^{\a,\circ}$ on $C^{\a,\circ}_\bv$.
Furthermore, since the local system $(\Delta(\bv^1)^*\wt\cf_{\bv^1}\boxtimes\ldots
\boxtimes\Delta(\bv^r)^*\wt\cf_{\bv^r})$ is semisimple
and $C^\bv_\a\cong C_{\Delta(\bv^1)}\times\ldots\times C_{\Delta(\bv^r)}$, we deduce
that the local system $ \Delta(\a)^*\wt\cf_\bv$
extends to a semisimple local system on $C^\bv_\a$.
The local system  $\ce^{\a,\circ}$ is a direct summand of $
\Delta(\a)^*\wt\cf_\bv$.
It follows that $\ce^{\a,\circ}$ extends  to a semisimple local system 
$\ce^\a$ on $C^\bv_\a$. Hence, we have
$\Delta(\a)_{!*}\ce^{\a,\circ}=\bar\Delta(\a)_*\ce^\a$, proving that
$\cf_\bv$ is nice.
\end{proof}

In the nice case we have the following
stronger version of Theorem \ref{expC}.
Recall the projection  $c: \N_\bv\reg=\TT^*_C (C^\bv)\reg\to C$.
\begin{prop}\label{semisimple}
Let $\cF=(\cF_\bv)_{\bv\in\Z^I_{\geq0}}$
 be a nice weak factorization sheaf  on $\Sym C$.
Then, writing
 $\wt\cf_\bv=\bplus_{\a\in\wt\Ups(\bv)}\ (\bar\Delta(\a))_*\ce^{\a}$,
one has:

\vi For each $\bv$ we have
  $\dis(\FDN_\N\ccirc \op{sp}_{C^\bv/C}(\wt\cf_\bv))|_{\N_\bv\reg}\cong c^*\ce^{\Delta(\bv)}[\rk\N_\bv](\mbox{$\frac{\rk\N_\bv}{2}$})$.

\vii In  $D^I\abs(\pt)\ $\ (rather than  in $K^I\abs$), one has
a natural isomorphism
$$
  \bplus_\bv\  \RGam_c(C^\bv/\si_\bv,\ \cF_\bv)\ =\
\Sym\left(\bplus_{\bw\neq 0}\  \RGam_c(C_{\Delta(\bw)},\,
\ce^{\Delta(\bw)})^{\si_\bw}\right).
$$
\end{prop}

\begin{proof} 
The construction of Verdier specialization involves
a deformation
to the normal bundle. From that deformation and the 
factorization isomorphism \eqref{ssfact}, it is easy to see that
one has
$\op{sp}_{C_\a/C_\Delta}(\ce^\a)=u_\a^*(\ce^\a|_\cde)$,
where we use the notation of diagram \eqref{TT}.
It follows that $\op{sp}_{C^\bv/C_\Delta}(\wt\cf^\a)=
(v_\a)_*u_\a^*(\ce^\a|_\cde)$. By the properties of  Fourier-Deligne
transform,
this implies that 
\beq{fdsp}
\FDN_\N\ccirc \op{sp}_{C^\bv/C_\Delta}(\wt\cf^\a)=
(v_\a^\perp)_*(u_\a^\perp)^*(\ce^\a|_\cde)[\rk\N_\bv](\half\rk\N_\bv).
\eeq

The sheaf above is supported on $\TT^\perp_C(C^\bv_\a)$.
We see that the restriction of
this sheaf to the open set  $\N_\bv\reg=\TT^*_C (C^\bv)\reg$ vanishes unless
$\a=\Delta(\bv)$ and $C^\bv_\a=C$ is the principal diagonal.
In the latter case, we have
$\op{sp}_{C_\Delta/C_\Delta}(\ce^{\Delta(\bv)})=(u_\Delta^\perp)^*\ce^{\Delta(\bv)}$
and $\TT_C^\perp(C^\bv_{\Delta(\bv)})\reg=\N_\bv\reg$.
Thus,  we obtain
\begin{align*}
(\FDN_\N\ccirc\op{sp}_{C^\bv/C}\cf_\bv)|_{_{\N_\bv\reg}} &=
\bplus_{\a\in\wt\Ups(\bv)}\ (\FDN_\N\ccirc\op{sp}_{C^\bv/C}(\wt\cf^\a))|_{_{\N_\bv\reg}}\\
&=
{(u_\Delta^\perp)^*\ce^{\Delta(\bv)}}^{}[\rk\N_\bv](\mbox{$\frac{\rk\N_\bv}{2}$})|_{_{\N_\bv\reg}}=
c^*\ce^{\Delta(\bv)}[\rk\N_\bv](\mbox{$\frac{\rk\N_\bv}{2}$}).
\end{align*}
This proves (i).

To prove (ii), let
$\wt\Ups(\bv,\fa)$ denote the fiber of the canonical projection
$\wt\Ups(\bv)\to\Ups(\bv)$ over   a decomposition
$\fa\in \Ups(\bv)$.  Further, we put
$\tcf^\fa:=\bplus_{\fB\in\wt\Ups(\bv,\fa)}\ \tcf_\bv^\fB$.
This is a $\si_\bv$-invariant direct summand of $\tcf_\bv$.
 Therefore it descends to  a direct summand $\cf_\bv^\fa$ of $\cf_\bv$.

Choose a \dct $\a\in\wt\Ups(\bv,\fa)$.
 The group $\si_\bv$ acts transitively on
$\Ups(\bv,\fa)$ and the stabilizer of the element
$\a$ equals $N(\si_\a)$.
Therefore,
 $\RGam_c(C^\bv, \tcf^\fa)=\bplus_{\fB\in\wt\Ups(\bv,\fa)}\
 \RGam_c(C^\bv, \tcf_\bv^\fB)$, viewed as a
  representation of $\si_\bv$,
 is induced from the $N(\si_\a)$-representation
  $\RGam(C^\bv, \tcf^\a)$.
 Thus, using Lemma \ref{ssfact} and the notation therein we get
\begin{align*}\RGam_c(C^\bv/\si_\bv, \cf^\fa)&=
\RGam_c(C^\bv, \tcf^\fa)^{\si_\bw}=
\RGam_c(C^\bv, \tcf^\a)^{N(\si_\a)}\\
&=
\RGam_c(\mbox{$\prod$}_\al\ C^{\bv^\al},\ \boxtimes_\al\ \cf^{\Delta(\bv^\al)})^{N(\si_\a)}=
\bigl(\bo_\al\ \RGam_c(C_{\Delta(\bv^\al)},
\ce^{\Delta(\bv^\al)}\bigr)^{N(\si_\a)}.
\end{align*}

Observe next that the set
 $\Ups(\bv)$, of multipartitions, has a natural 
identification with the set of collections
$\{d_\bw\}_{\bv\in\Z_{\geq0}^I}$ such that one has $\sum_\bw\
d_\bw\cdot\bw= \bv$. Let   $\{d_\bw\}$ be 
the collection  
that corresponds to a decomposition $\fa\in
\Ups(\bv)$. Then, by \eqref{N(si)} we have
$N(\si_\a)\cong
\prod_\bw\,\big((\si_\bw)^{d_\bw}\rtimes\si_{d_\bw}\big)$.
Hence,
we get
\begin{multline*}
\Bigl(\underset{\al\in[r]}\bo\ \RGam_c(C_{\Delta(\bv^\al)},
\ce^{\Delta(\bv^\al)}\Bigr)^{N(\si_{\bv^1,\ldots,\bv^r})}
\ =\
\Bigl(\bigotimes_\bw\ \RGam_c(C_{\Delta(\bw)},
\ce^{\Delta(\bw)})^{\o
  d_\bw}\Bigr)^{\prod_\bw\,\big((\si_\bw)^{d_\bw}\rtimes\si_{d_\bw}\big)}\\
=\ \bigotimes_\bw\ \Sym^{d_\bw}\bigl(\RGam_c(C_{\Delta(\bw)},
\ce^{\Delta(\bw)})^{\si_\bw}\bigr).
\end{multline*}

Thus, using the direct sum decomposition
$\cf_\bv  =\bplus _{\fa\in \Ups(\bv)} \  \cf_\bv^\fa$, we compute

\begin{align*}
\bigoplus_\bv\ \RGam_c(C^\bv/\si_\bv,&\, \cf_\bv) =\ 
\bigoplus_\bv\ \bigoplus_{\fa\in \Ups(\bv)} \  \RGam_c(C^\bv/\si_\bv,\, \cf_\bv^\fa)\\
&=\ \bigoplus_\bv\ \ \bigoplus_{\{(d_\bw)_{\bw\in \ZZ_{\ge0}^I \sminus
  \{0\}}\ |\ \sum_\bw\, d_\bw\cdot\bw=\bv\}}\
\left(\bigotimes_\bw\ \Sym^{d_\bw}\bigl(\RGam_c(C_{\Delta(\bw)},
\ce^{\Delta(\bw)})^{\si_\bw}\bigr)\right)\\
&=\quad \bigoplus_{(d_\bw)_{\bw\in \ZZ_{\ge0}^I \sminus
  \{0\}}}\ \left(\bigotimes_\bw\ \Sym^{d_\bw}\bigl(\RGam_c(C_{\Delta(\bw)},
\ce^{\Delta(\bw)})^{\si_\bw}\bigr)\right)\\
&=\Sym\big(\bplus_{\bw\neq 0}\ \ \RGam_c(C_{\Delta(\bw)},
\ce^{\Delta(\bw)})^{\si_\bw}\big).
\end{align*}
Part (ii) follows.
\end{proof}


\subsection{Proofs of Theorem \ref{expC} and Theorem \ref{Exp}}
\label{aa-sec}

Let $D_\Ups(C^\bv)$ be a full  subcategory
of $D\abs(C^\bv)$ whose objects are
constructible with respect to the diagonal stratification
and let $K\abs(C^\bv)$ be the corresponding Grothendieck group.
The semisimplification,  $\cf\ss$, of any object  $\cf\in D_\Ups(C^\bv)$
is again an object of  $D_\Ups(C^\bv)$. The semisimplification operation
extends to the $\si_\bv$-equivariant setting.

\begin{lem}\label{ss3} 
For any weak   factorization sheaf   $\cf=(\cf_\bv)$ the collection
$\ \cf\ss:=(\cf\ss_\bv)_{\bv\in \Z^I_{\geq 0}}\ $ has the natural structure
of  a weak   factorization sheaf.
\end{lem}

\begin{proof}[Sketch of Proof]
Since the operation of open restriction preserves the weight filtration,  
we can consider semisimplification more generally
for sheaves on any  open substack $U\sset  C^\bv/\si_\bv$.
Furthermore, it is clear that for any such $U$, and
$\ce\in D_{_\Ups}( C^\bv/\si_\bv)$,  one has $(\ce\ss)|_U=
(\ce|_U)\ss$. In particular, for any decomposition
$\bv=\bv^1+\ldots+\bv^r$ we have an  isomorphism 
$(\jmath_{\bv^1,\ldots,\bv^r}^*\imath_{\bv^1,\ldots,\bv^r}^*\cf_\bv)\ss=\jmath_{\bv^1,\ldots,\bv^r}^*\imath_{\bv^1,\ldots,\bv^r}^*(\cf\ss)$,
of sheaves on  $ C\he_{\bv^1,\ldots,\bv^r}$.
Also, it is clear that
$(\cf_{\bv^1}\boxtimes\ldots\boxtimes\cf_{\bv^r})\ss=
\cf_{\bv^1}\ss\boxtimes\ldots\boxtimes\cf_{\bv^r}\ss$.
Since the only operations with sheaves involved in
the definition of    factorization sheaf  are the operations
$\jmath_{\bv^1,\ldots,\bv^r}^*\imath_{\bv^1,\ldots,\bv^r}^*$
and $\boxtimes$, and the operation of taking semisimplification is functorial, 
it follows that for any collection $(\cf_\bv)_{\bv\in\Z^I_{\geq0}}$
that satisfies that definition the collection
$(\cf_\bv\ss)_{\bv\in\Z^I_{\geq0}}$ satisfies it as well.
\end{proof}

\begin{proof}[Proof of Theorem \ref{expC}]
Each of the functors $\op{sp}_\N$ and $\FDN_\N$
 is t-exact with respect to the perverse
t-structure. Hence, for any weak factorization sheaf $\cf$, we
have $(\Phi_\N(\cf_\bv))\ss=\Phi_\N(\cf_\bv\ss)$.
Further, part (i) of Proposition \ref{semisimple}
implies that $(\Phi_\N(\wt\cf_\bv))\ss$ is a pull-back of a semisimple local system
on $C$ via the projection $\N_\bv\reg\to C$. We deduce that 
the local system $\Phi_\N(\wt\cf_\bv)$ is unipotent along 
the map $c:\N_\bv\reg\to C$. Part (i) of Theorem \ref{expC} follows.

Observe next that 
in  $K\abs(C^\bv)$ we have  $[\ce]=[\ce\ss]$.
The equation of  part (ii) of Theorem \ref{expC} is an equation
in $\kabs$. Such an equation is unaffected by replacing
each of the sheaves $\cf_\bv$ by $\cf_\bv\ss$.
The sheaf  $\cf_\bv\ss$ is nice. For such a sheaf, the equation of  part (ii) of Theorem \ref{expC} 
is an immediate consequence of the formula of Proposition \ref{semisimple}(ii)
and the isomorphism 
of part (i) of that proposition, using that $\rk\N_\bv=\dim C(|\bv|-1)$.
\end{proof}

\begin{proof}[Proof of Theorem \ref{Exp}]
Part (i) of the theorem follows directly from Theorem \ref{expC}(i).
Also, it is clear from the equivariance of the sheaf $\cf_\bv$ under the diagonal translations that
one has $[\Phi_\N(\cf_\bv)]^{\langle\rho\rangle}=
[\RGam_c(\AA,\C_\AA)]\o [\Phi_\ft(\cf_\bv)^{\langle\rho\rangle}_\eta]=\L\cdot [\Phi_\ft(\cf_\bv)^{\langle\rho\rangle}_\eta]$,
resp.
$[\Phi_\N(\eps^*\cf_\bv)]^{\langle\rho\rangle}=
\RGam_c(\GG,\C_\AA)\o [\Phi_\ft(\cf_\bv)^{\langle\rho\rangle}_\eta]=(\L-1)\cdot [\Phi_\ft(\cf_\bv)^{\langle\rho\rangle}_\eta]$, for any $\rho\in\Irr(\si_\bv)$.
Thus,
equations \eqref{2eqs} and \eqref{3eqs} follow from Theorem \ref{expC}(ii).

We now prove \eqref{4eqs}. Since this is an equation
in $\kabs$ it suffices to prove it for semisimple weak factorization sheaves $\cf$.
Furthermore,  without loss of generality one may assume in addition 
that each of  the corresponding local systems
$\ce^\a$ that appear in Definition \ref{ss-def} has {\em finite} rank, that is,  the semisimple sheaves
$\cf_\bv$ are  Weil  sheaves. In that case, the collection $\cf^\vee=(\cf_\bv^\vee)$ 
is a semisimple weak factorization sheaf again, cf. Remark \ref{fact-rems}(1).
Thus, we are in a position to apply equation \eqref{2eqs} in the case of the  factorization sheaf $\cf^\vee$.
Note that $\overline{\cf_\bv^\vee}=(\bar\cf_\bv)^\vee[2](1)$.
Using that the Fourier-Deligne functor $\FDN_\ft$ commutes
with the Verdier duality we deduce
$[\Phi_\ft(\cf_\bv^\vee)_\eta]=[\Phi_\ft(\cf_\bv)_\eta^*]\{2\}=\L\inv\cdot [\Phi_\ft(\cf_\bv)_\eta^*]$.
Also, we 
one has $\RGam_c(\AA^\bv/\si_\bv, \cf_\bv^\vee)=\RGam(\AA^\bv/\si_\bv, \cf_\bv)^*$.
Thus, using \eqref{2eqs} for $\cf^\vee$,
we compute
\begin{multline*}
\sum_\bv\ z^\bv\cdot [\RGam(\AA^\bv/\gS_\bv,\cF_\bv)^*]=
\sum_\bv\ z^\bv\cdot [\RGam(\AA^\bv/\gS_\bv,\cF_\bv^\vee)]\\
=\Sym\Big(-\L^{\frac{1}{2}}\cdot\sum_{^{\bv>0}}\ (-1)^\bv\cdot z^\bv\cdot \L^{\frac{|\bv|}{2}}
\cdot[\Phi_\ft(\cf_\bv^\vee)_\eta]^{\langle\op{triv}\rangle}\Big)\\
=\Sym\Big(-\L^{\frac{1}{2}}\cdot\L\inv\cdot\sum_{^{\bv>0}}\ (-1)^\bv\cdot z^\bv\cdot \L^{\frac{|\bv|}{2}}
\cdot\big[(\Phi_\ft(\cf_\bv)_\eta^*\big]^{\langle\op{triv}\rangle}\Big).
\end{multline*}
Now, the sheaf $\cf_\bv$ being $\GG$-equivariant, the well known result of Springer
yields a canonical isomorphism $\RGam(\AA^\bv/\gS_\bv,\cF_\bv)\cong
\RGam(0/\si_\bv, i_\bv^*\cf_\bv)$. Applying $(-)^*$ to the displayed formula above yields
the required equation \eqref{4eqs}.
\end{proof}

\section{Factorization sheaves on $\GL$}
\subsection{From $G$ to $T$} \label{GtoT}
We will use the notation $\pt_X$ for a constant map $X\to \pt$.

Let $G$ be a split connected reductive group. For all
Borel subgroups $B$ of $G$, the tori $T=B/[B,B]$
are canonically isomorphic to each other and  
we write ${W}$ for the corresponding Weyl group.
Let $\g,\,\b$, and $\ft$, be the Lie algebras of
$G,\,B$, and $T$, respectively.
The imbedding $\b\into\g$, resp. the projection
$\b\onto \b/[\b,\b]$, gives a morphism
of quotient stacks $\kappa_\b: \b/B\to\g/G$, resp. 
$\nu_\b: \b/B\to\ft/T$.
Observe
that the action of $T$ on $\ft$ being trivial, there
is a canonical isomorphism $\ft/T\cong\ft\times \pt/T$.
This isomorphism respects the natural actions of the Weyl
group $W$ on each side, where $W$ acts on $\ft\times \pt/T$
diagonally.  By the Chevalley isomorphism
$\g\dsl G\cong\ft\dsl W$, we have
a commutative  diagram:
\beq{ww}
\xymatrix{
0/G\ \ar@{^{(}->}[r]^<>(0.5){i_\g}\ar[d]^<>(0.5){\pt_{0/G}}&\g/G\ \ar[d]_<>(0.5){{\mathfrak w}_{\g}}
 & \  \b/B\ \ar@{_{(}->}[l]_<>(0.5){\kappa_\b}\ar@{->>}[r]^<>(0.5){\nu_\b}&
\ \ft/T\ar@{=}[r]&\ft\times {B} T\
\ar[d]^<>(0.5){r} &\  \{0\}\times {B} T=0/T\  \ar@{_{(}->}[l]_<>(0.5){i_\ft}\ar[d]^<>(0.5){\pt_{0/W}}\\
\pt & \g\dsl G\ar@{=}[r]&\ft\dsl W\ &\ \ft/W \ \ar[l]_<>(0.5){{\mathfrak w}_{\ft}}&
\ \ft\ \ar[l]_<>(0.5){q}&\ \pt
}
\eeq

We claim that for any  sheaf $\cf$ on $\g/G$,
the sheaf 
$r_!(\nu_\b)_!\kappa_\b^*\cf$ on $\ft$  has a canonical
${W}$-equivariant structure. This seems to be well known 
(its version for parabolic restriction of character sheaves  is
essentially due
to Lusztig \cite{L2}), but we have been
unable to find the statement in the generality that we need in the literature.

To state our result, 
put  $n_G=\dim G/B=\dim\fb-\dim\ft$, resp. $\rk=\dim T$. Write
$\sign\o(-)$  for a twist by the sign character of the $W$-action on  a $W$-representation or, more generally,
a twist of the equivariant structure on a $W$-equivariant sheaf.
In particular, one has a functor $\sign\o(-)$ on $D\abs(\ft/W)$ induced by tensoring
with a rank 1 local system   $\textit{\small{Sign}}$.

In the proofs below, we will freely use the following known results
about {\em compactly supported} cohomology of  $G/T,\,{B} T$, and ${B} G$.
The top degree nonzero compactly supported cohomology groups are 
$H^{4n_G}_c(G/T)$, resp. $H_c^{-2\rk}({B} T)$ and $H_c^{-2\dim G}({B} G)=H_c^{-4n_G-2\rk}({B} G)$.
These groups are 1-dimensional, furthermore, the Weyl group acts on 
$H^{4n_G}_c(G/T)$ and $H_c^{-2\rk}({B} T)$ via the sign character, where we have used
the isomorphisms
$H^\hdot(G/B)=H_c^\hdot(G/B)=H_c^{\hdot+2n_G}(G/T)$.
Also,  each of the groups $H_c^{\op{odd}}(G/B),\ H_c^{\op{odd}}({B} T)$, and $HH_c^{\op{odd}}({B} G)$,
vanishes.
Therefore, the spectral sequence for the fibration ${B} T\xrightarrow{G/T} {B} G$ collapses,
yielding a $W$-equivariant isomorphism $H_c^{\op{ev}}({B} T)=H_c^{\op{ev}}(G/T)\o H_c^{\op{ev}}({B} G)$.
In particular, one has  $W$-equivariant isomorphisms 
\[H_c^\hdot({B} T)^W=H_c^{\hdot+2n_G}({B} G)\quad\text{resp.}\quad
H_c^\hdot({B} T)^{\sign}=\sign\o H_c^{\hdot+4n_G}({B} G).\]

The functors constructed in the following lemma are  called parabolic restriction functors.

\begin{lem} \label{xi}
\vi There exists a functor $\res_\ft\colon D\abs(\g/ G)\to D\abs(\ft/{W})$
such that one has an isomorphism of functors $r_!\,(\nu_\b)_!\,\kappa_\b^*\,\cong q^*\,\res_\ft$, resp.
\beq{xieq} 
r_!\,(\nu_\b)_!\,\kappa_\b^*\,\cong q^*\,\res_\ft,\quad
({\mathfrak w}_\g)_!\cong (({\mathfrak w}_\ft)_!\res_\ft)^W,\quad
(\pt_{0/G})_!\, i^*_\fg\,\cong\, ((\pt_{0/W})_!\, (i^*_\ft\res_\ft\o\sign))\{-2n_G\}.
\eeq

\vii There is   also a group analogue, a functor  $\res_T\colon D\abs(G/\aad G)\to D\abs(T/{W})$
with similar properties. 
\vskip 2pt
\viii For any smooth affine curve $C$ there is a  functor
$\res_C: D\abs(\Coh_\bv C)\to D\abs(C^\bv/\si_\bv)$ such that analogues of isomorphisms \eqref{xieq} hold.
Given an open imbedding $\eps: C_1\into C$, one
has 
an isomorphism of functors 
\beq{epep}\eps_{\Sym}^*\ccirc\res_C\cong\res_{C_1}\ccirc\eps^*_{\Coh},
\eeq
where $\eps_{\Coh}^*:\ D\abs(\Coh_\bv C)\to D\abs(\Coh_\bv C_1)$,
resp. $\eps_{\Sym}^*:\ D\abs(C^\bv/\si_\bv)\to D\abs(C_1^\bv/\si_\bv)$,
is the natural restriction.

\iv The functor $\res_C$ reduces to the functor $\res_\ft$, resp. $\res_T$,  in the case $G=\GL_\bv$ and $C=\AA$, resp. $C=\GG$.
\end{lem}

Let  $\wt \g$ be the variety of pairs $(\b,g)$ such that $\b$ is a Borel
subalgebra
and $g$ is an element of $\b$.
It is convenient to reformulate the lemma in terms of the   Grothendieck-Springer
resolution $\wt \g\to \g,\ (\b,g)\mto g$.
This map is  $G$-equivariant, hence descends to a morphism
$\wt\kappa: \wt \g/G\to \g/G$.
There is a natural isomorphism $\wt \g/G\cong \b/B$,
of quotient stacks. Therefore, one also has  a morphism
$\wt \nu: \wt \g/G= \b/B\to \ft/T=\ft\times\pt/T,\ (\b,g)\mto \nu_\b(g)$.
Thus we have
the following commutative  diagram:
\beq{GGTT}
\xymatrix{
&&\wt \g/G\
\ar[dll]_<>(0.5){\wt\kappa}\ar[d]^<>(0.5){\pi}
\ar[r]^<>(0.5){\wt\nu}&
\ \ft\times\pt/T\
\ar[d]^<>(0.5){r}
& 
\\
\g/G\ \ar[d]_<>(0.5){{\mathfrak w}_\g}
&&\  (\g/G)\times_{\ft\dsl W} \ft\
\ar[ll]_<>(0.3){\kappa}\ar@{}[dr]|{\Box}
\ar[d]^<>(0.5){\pi'}
\ar[r]^<>(0.5){pr_2}
&\  \ft\ar[d]^<>(0.5){q}
&\\
\g\dsl G=\ft\dsl W\ &&(\g/G) {\times_{\ft\dsl W}} (\ft/W)\
\ar[ull]_<>(0.3){\kappa'}\ar[r]^<>(0.5){pr_2'}\ar[ll]_<>(0.5){pr}
&\ \ft/W\ \ar[r]^<>(0.5){{\mathfrak w}_\ft}&
\ \ft\dsl W
}
\eeq





\begin{proof}[Proof of Lemma \ref{xi}] 
We have a chain of canonical isomorphisms of functors:
\beq{rnu}r_!\,(\nu_\b)_!\,\kappa^*_\b=r_!\,\wt\nu_!\,\wt\kappa^*\,=\,(\pr_2)_!\,\pi_*\,\wt\kappa^*\,=\,
(\pr_2)_!\,\pi_*\,\pi^*\,\kappa^*\,=\,
(pr_2)_!\,\big(\nu_!\,\kappa^*\,\o\ \pi_!\C_{\wt\g/G} \big)
\quad\eeq
where  the last  isomorphism follows from
the projection formula and we have used that  $\pi_!=\pi_*$.


The sheaf $\pi_!\C_{\wt\g/G}$ can be equipped with a $W$-equivariant structure
as follows. Let $\g^{rs}$ be the set of  regular semisimple
elements of $\g$ and  $U:=(\g^{rs}/G )\,\times_{\ft\dsl W}\,\ft$.
The morphism $\bar\pi: \wt \g\to \g\times_{\ft\dsl W}\,\ft$,
a non-stacky counterpart of $\pi$,
is a proper morphism which  is an isomorphism over
$U$, moreover,
this morphism is known to be small \cite[Section 3]{L1}.
It follows that $\bar\pi_!\C_{\wt\g}=\op{IC}(U)[-\dim U]$,
where $\op{IC}(U)$ is the IC-sheaf corresponding to the constant sheaf
$\C_U$. The natural $G\times W$-equivariant structure
on $\C_U$ induces one on $\op{IC}(U)$ and, hence, on $\bar\pi_!\C_{\wt\g}=\op{IC}(U)[-\dim U]$.
It is clear that the sheaf $\pi_!\C_{\wt\g/G}$ on $(\g/G)\times_{\ft\dsl W}\,\ft$ is, up to a shift, obtained 
from $\bar\pi_!\C_{\wt\g}$ by $G$-equivariant descent. Since the actions of $G$ and $W$ on $\g\times_{\ft\dsl W}\,\ft$
commute, the $W$-equivariant structure
on $\bar\pi_!\C_{\wt\g}$ induces a  $W$-equivariant structure on $\pi_!\C_{\wt\g/G}$.
Thus, there is a canonically defined sheaf $\IC$ on
$(\g/G)\times_{\ft\dsl W}\,(\ft/W)$ such that $\pi_!\C_{\wt\g/G}=(\pi')^*\op{IC}$.

We now define the required funtor  $\res_\ft$ as follows
\[\res_\ft(\cF):=(pr'_2)_!((\kappa')^*\cF\,\o\,\op{IC}).\]
Using base change with respect to the cartesian square
at the bottom right corner of diagram \eqref{GGTT}  we  deduce
\begin{align*}
q^*\res_\ft=q^*(pr'_2)_!\big((\kappa')^*\o\,\op{IC})
&=(\pr_2)_!(\pi')^*((\kappa')^*\,\o\,\op{IC}(U/{W})\big)\\
&=(\pr_2)_!\big((\pi')^*(\kappa')^*\,\o\,(\pi')^*\op{IC}\big)\\
&=(\pr_2)_!\big(\kappa^*\,\o\,\op{IC}(U)\big)=(\pr_2)_!\big(\kappa^*\,\o\,
\pi_!\C_{\wt\g/G})\, \stackrel{\eqref{rnu}}{=\!=}r_!\,(\nu_\b)_!\,\kappa_\b^*,
\end{align*}
proving the first isomorphism in \eqref{xieq}.

To prove  the second  isomorphism in \eqref{xieq},
 we observe that the natural $W$-action on
$\kap_!\pi_!\C_{\wt\g/G}$ gives, for any sheaf $\cf\in D\abs(\g/G)$,
a  $W$-action on 
\[\wt\kap_!\wt\kap^*\cf=
\cf\o \wt\kap_!\wt\kap^*\C_{\wt\g/G}=
\cf\o \kap_!\pi_!\pi^*\C_{\g/G}.\]
Furthermore, the
 well  known isomorphism
$(\kap_!\pi_!\C_{\wt\g/G})^W=\C_{\wt\g/G}$ yields
\[(\wt\kap_!\wt\kap^*\cf)^W
=
\cf\o (\kap_!\pi_!\C_{\wt\g/G})^W=
\cf\o \C_{\g/G}
=\cf.\]
We deduce that $(\fw_\g)_!\cf=(\fw_\g)_!(\wt\kap_!\,\wt\kap^*\cf)^W=
((\fw_\g\ccirc\wt\kap)_!\,\wt\kap^*\cf)^W$.
On the other hand, diagram \eqref{GGTT}
shows that $\fw_\g\ccirc\wt\kap=\fw_\ft\ccirc q\ccirc r\ccirc\wt\nu$.
Therefore, we have an isomorphism of functors

\[(\fw_\g\ccirc\wt\kap)_!\,\wt\kap^*=
(\fw_\ft)_!\, q_!\, r_!\,\wt\nu_!\wt\kap^*=
(\fw_\ft)_!\, q_!\,q^*\,\res_\ft,\]
by the first isomorphism in \eqref{xieq}.
The proof of the second  isomorphism in \eqref{xieq} is now
completed by the following
chain of isomorphisms of functors
\begin{align*}
(\fw_\g)_!&=((\fw_\g\ccirc\wt\kap)_!\,\wt\kap^*(-))^W
=((\fw_\ft)_!\, q_!\,q^*\,\res_\ft(-))^W\\
&=(\fw_\ft)_!\, (\res_\ft(-)\o (q_!q^*\C_{\ft/W})^W)
=((\fw_\ft)_!\, (\res_\ft(-)\o \C_{\ft/W})
=(\fw_\ft)_!\,\res_\ft.
\end{align*}

Next, let ${\mathcal N}\sset \g$ be the nilpotent variety.
Let $i_0:  0/G\into  {\mathcal N}/G$ and $i_{\mathcal N}: {\mathcal N}/G\into \g/G$
be the natural imbeddings, and $\pt_{\mathcal N}: {\mathcal N}/G\to \pt$.
According to the theory of Springer representations, one has
$$
i_{\mathcal N}^*\kap'_!\big(\op{IC}\o (pr'_2)^*\textit{\small{Sign}}\big)\ =\  (i_0)_*\C_{\{0/G\}}[-2n](-n),
$$
where $n=n_G$ and the shift is due to the fact that $\dim{\mathcal N}=2n$.
We deduce
\[i_{\mathcal N}^*\kap'_!\big((\kap')^*\cf\o \op{IC}\o (pr'_2)^*\textit{\small{Sign}}\big)\ =\  
(i_0)_*i_\g^*\cf[-2n](-n).
\]

Note that $(\kap')\inv({\mathcal N}/G)= {\mathcal N}/G\times 0/W=
(pr'_2)\inv(0/W)$.  Let $\iota:\  {\mathcal N}/G\times 0/W\into (\g/G)\,\times_{\ft\dsl W}\,(\ft/W)$
be the imbedding. Since
$(\pt_{\g/G})_!i_\g^*=(\pt_{\mathcal N})_!\, (i_0)_!\,i_\g^*$,
using  the isomorphism above and base change, we find
\begin{align*}
(\pt_{\g/G})_!i_\g^*\cf&=(\pt_{\mathcal N})_!\, (i_0)_*\,i_\g^*\cf
=(\pt_{\mathcal N})_!\, i_{\mathcal N}^*\kap'_!\big((\kap')^*\cf\o \op{IC}\o (pr'_2)^*\textit{\small{Sign}}\big)[2n](n)\\
&=(\pt_{\mathcal N}\boxtimes\pt_{\ft/W})_!\iota^*\big[(\kap')^*\cf\o \op{IC}\o (pr'_2)^*\textit{\small{Sign}}\big)[2n](n)\\
&=(\pt_{\ft/W})_!\,i_\ft^*\,(pr'_2)_!\big((\kap')^*\cf\o \op{IC}\o (pr'_2)^*\textit{\small{Sign}}\big)[2n](n)\\
&=(\pt_{\ft/W})_!\,i_\ft^*\,(\res_\ft\cf\o\textit{\small{Sign}})[2n](n).
\end{align*}
The third isomorphism in \eqref{xieq} follows.
The proof of part (ii) is left for the reader.

To prove (iii) one replaces $\wt\g/G$ by the stack $\wt \Coh_\bv C$ be the stack whose objects are
complete flags $0=M_0\sset M_1\sset\ldots \sset M_d=M$, where $M_j$ is a length $j$ coherent sheaf on $C$.
One has a natural diagram $\Coh_\bv C\xleftarrow{\wt\kap}\wt\Coh_\bv C\xrightarrow{\wt\nu} C^\bv$ that reduces to the diagram
$\g/G\xleftarrow{\wt\kap}\wt\g/G\xrightarrow{\wt\nu} \ft$ in the case where $C=\AA$ and $\g=\gl_\bv$. The diagram for a general curve $C$ is isomorphic,
\'etale locally, to the diagram for $C=\AA$. In particular, the map $\wt\Coh_\bv C\to \Coh_\bv C\times_{C^\bv\dsl \si_\bv}C^\bv$ is small.
The isomorphism of part (iii) is an immediate consequence of base change
in the following diagram with cartesian squares:
\[\xymatrix{
\ {\Coh_\bv C_1}_{}\ \ar@{^{(}->}[d]&\ {\wt\Coh_\bv C_1}_{}\ \ar@{^{(}->}[d]\ar[l]\ar[r]&
\ {C_1^\bv}_{}\ \ar[r]\ar@{^{(}->}[d]&\
\ {C^\bv_1/\si_\bv}_{}\
\ar@{^{(}->}[d]\\
{\Coh_\bv C}^{}\ &\ {\wt\Coh_\bv C}^{}\ \ar[l]\ar[r]&
\ {C^\bv}^{}\
\ar[r]&\ {C^\bv/\si_\bv}^{}\
}
\]

Generalization to the the $I$-graded case is straightforward, as well as the proof of (iv), is straightforward.
\end{proof}

\begin{rem}\label{ind-rem} One can also define an induction functor
${\mathfrak{ind}}: \ D\abs(\ft/W)\to D\abs(\g/G)$ by the formula
${\mathfrak{ind}}(\cf)=(\kap'_!((pr_2')^*\cf\o \op{IC}))^W$.
\erem

Next, we dualize the maps in the top row of diagram
\eqref{ww}. Thus, writing $(-)^\top$ for the transposed map,
we get a diagram

Let $\FDN_{\g/G}:\ D\abs(\g/G)\to D\abs(\g^*/G)$ be the Fourier-Deligne transform,
that is, a relative Fourier trasform on the vector bundle
$\g/G\to\pt/G$.
With obvious modifications of the notation of Lemma \ref{xi} we have the following result.

\begin{lem}   \label{qfourierr} 
There exists  a functor  $\res_{\ft^*}\colon D(\g^*/G)\to D(\ft^*/W)$,
such that 
analogues of the three  isomorphisms in \eqref{xieq} hold and, in addition,
there is an isomorphism of functors
\beq{fo2}
\FDN_{\ft/W}\ccirc\res_\ft\cong\sign\o(\res_{\ft^*}\ccirc\FDN_{\g/G}).
\eeq
\end{lem}

\begin{proof} The construction of the functor
$\res_{\ft^*}$   mimics the proof of Lemma
\ref{xi}. Further, let ${\mathfrak n}=[\b,\b]$ be the nilradical of $\b$.
The proof of \eqref{fo2} is well known, 
it is based on the fact that the Killing form provides an identification of the 
following natural diagram
$$
\xymatrix{
\g^*/G\ \ar[d]_<>(0.5){{\mathfrak w}_{\g^*}}& {\mathfrak n}^\perp/B
\ar@{_{(}->}[l]
 \ar@{->>}[r]&
\ \ft^*/T\ar@{=}[r]&\ft^*\times{B} T\
\ar[d]^<>(0.5){r_{\ft^*}} 
\\
 \g^*\dsl G\ar@{=}[r]&\ft^*\dsl W\ &\ \ft^*/W \ \ar[l]_<>(0.5){{\mathfrak w}_{\ft^*}}&
\ \ft^*\ \ar[l]_<>(0.5){q_{\ft^*}}
}
$$
with the rectangle in the middle of diagram \eqref{ww}.
We refer to Brylinski's paper  \cite{Br} for details.
\end{proof}



Later on, we will also need the following result.
Let $\chi: G\to \GG$ be a morphism of algebraic groups
and $\wp$ a local system on $\GG$. The local system
$\chi^*\wp$, on $G$, is clearly $\Ad G$-equivariant,
hence it descends to a local system
$\wp^G$ on $G/\aad G$. The morphism $\chi$ gives, by restriction,
a $W$-invariant morphism $\chi_T: T\to\GG$.
The corresponding local system $\chi_T^*\wp$ is therefore
$W$-equivariant, hence it descends to a local
system $\wp^T$ on $T/W$.

\begin{cor} With the above notation, for any $\cf\in D\abs(G/\aad G)$,
  there is a canonical isomorphism
\beq{align2}
\RGam_c(G_\bv/\aad G_\bv,\ \wp^G\o\eps^*_G\cf)\ccong
\RGam_c(T_\bv/\si_\bv,\ \wp^T\o \eps^*_T\res_\ft\cf).
\eeq
\end{cor}
\begin{proof} Let $\chi_B: B\to \GG$ be the restriction
of $\chi$ to $B$. As above, the local system $\chi_B^*\wp$ descends
to a local system $\wp^B$  on $B/\aad B$.  Clearly, we have
$\mu_B^*\wp^G\cong\wp^B\cong \nu^*_B\wp^T$.
Using the composite isomorphism
$\mu_B^*\wp^G\cong\nu^*_B\wp^T$, for any $\ce\in D\abs(G/\aad G)$,
we compute
\[(\nu_B)_!\mu_B^*(\wp^G\o\ce)=
(\nu_B)_!(\mu_B^*\wp^G\o\mu_B^*\ce)=
(\nu_B)_!(\nu^*_B\wp^T\o\mu_B^*\ce)=
\wp^T\o(\nu_B)_!\mu_B^*\ce,\]
were the last isomorphism is  the projection formula.
We deduce that $\res_T(\wp^G\o\ce)\ccong
\wp^T\o\res_T(\ce)$ canonically.

Now, given $\cf\in D\abs(\g/G)$, we apply the above isomorphism
for  $\ce:=\eps_G^*\cf$. We compute
\begin{align*}
\RGam_c(G/\aad G,\ \wp^G\o\eps^*_G\cf)
 &=\ \RGam_c(T/W,\ \res_T(\wp^G\o\eps^*_G\cf))
\qquad\text{(Lemma
  \ref{xi}) }\\
&=\ \RGam_c(T/W,\ \wp^T\o(\res_T\eps^*_G\cf))\\
& =\ \RGam_c(T/W,\ \wp^T\o(\eps^*_T\res_\ft\cf)).\qquad\text{(Lemma
  \ref{qfourierr})}
\qedhere
\end{align*}
\end{proof}

\begin{rem}\label{GL-pgl}
Below, we will use analogues of the above results in the case of $\GL_\bv$-equivariant
(rather than $\PGL_\bv$-equivariant) sheaves on $\pgl_\bv$, resp. $\PGL_\bv$. In such a case, the stack $\wt\g/G$ must be replaced
by $(\GL_\bv\times_{B_\bv} {\mathfrak{pb}}_\bv)/\GL_\bv\cong {\mathfrak{pb}}_\bv/B_\bv$,
where $B_\bv\sset \GL_\bv$ is the Borel subgroup of $\GL_\bv$ and ${\mathfrak{pb}}_\bv\sset \pgl_\bv$ is the
image of $\Lie B_\bv$ in $\pgl_\bv$, a Borel subalgebra of
$\pgl_\bv$. Similarly, the stack $\ft_\bv/\GG^\bv$, where $\ft_\bv=\AA^\bv/\AA$ is the Cartan subalgebra of
$\pgl_\bv$, plays the role of $\ft/T$.
\end{rem}
\subsection{Factorization sheaves on $\GL$ and $\gl$}
\label{gfact} In this subsection we write $G_\bv$ for $\GL_\bv$, resp. $\g_\bv$ for $\gl_\bv$.

We will consider factorization sheaves on the collection
of  Lie algebras $\g_\bv$, resp. groups ${G}_\bv$.
To this end, for any $\bv_1,\bv_2$, we define 
\beq{g12}
\g\he_{1,2}=
\{(a,b)\in \g_{\bv_1}\times \g_{\bv_2} \mid
a=(a_i)_{i\in I},\ b=(b_i)_{i\in I},\en \op{spec}(a_i)\cap\op{spec}(b_j)=\emptyset
\en\ \forall\ i,j\in I\},
\eeq
where $\op{spec}(g)$ stands for the set of eigenvalues of a matrix $g$.
It is clear that $\g\he_{1,2}$
is an  affine, ${G}_{\bv_1}\times {G}_{\bv_2}$-stable
Zariski open and dense
subset of  $\g_{\bv_1}\times \g_{\bv_2}$. 
Thus, we have   natural  maps
\[
\g\he_{\bv_1,\bv_2}\ \xrightarrow{\jmath_{\bv_1,\bv_2}}\ 
 \g_{\bv_1}/{G}_{\bv_1}\times \g_{\bv_2}/{G}_{\bv_2}\ 
\xrightarrow{\imath_{\bv_1,\bv_2}}\  \g_{\bv_1+\bv_2}/ {G}_{\bv_1+\bv_2},\]
where the first, resp. second, map is an open, resp. closed, imbedding.
For any triple $\bv',\bv'',\bv'''$,
of dimension vectors, we let
 $\g\he_{\bv',\bv'',\bv'''}$ be a subset of $\g_{\bv'}\times \g_{\bv''}\times \g_{\bv'''}$
formed by the
triples $\big((a_i)_{i\in I},\ (b_i)_{i\in I}$, $(c_i)_{i\in I}\big)$, 
of $I$-tuples,
where $a_i\in \g_{v_i'},b_i\in \g_{v_i''}, c_i\in \g_{v_i'''}$
such that any two of the three  sets $\cup_i\op{spec}(a_i)$,
$\cup_i\op{spec}(b_j)$, and  $\cup_i\op{spec}(c_i)$, are disjoint.

We have an open imbedding $\GL_\bv\into \gl_\bv$  of invertible matrices into the space of all
matrices, and induced imbeddings
$G_{\bv_1}\times G_{\bv_2}\into \g_{\bv_1}\times \g_{\bv_2}$, etc. Let
${G}\he_{\bv_1,\bv_2}:=\g_{\bv_1,\bv_2}\he \cap (G_{\bv_1}\times G_{\bv_2})$,
resp. ${G}\he_{\bv',\bv'',\bv'''}:=\g\he_{\bv',\bv'',\bv'''}\cap ({G}_{\bv'}\times
{G}_{\bv''}\times {G}_{\bv'''})$.
One can equivalently define $\g\he_{\bv',\bv'',\bv'''}$, resp. ${G}\he_{\bv_1,\bv_2}$,
as the preimage of $\AA\he_{\bv_1,\bv_2}\dsl(\si\times\si_{\bv_2}$, resp $\GG\he_{\bv_1,\bv_2}$, under the
Chevalley isomorphism $\g_{\bv_1}\dsl G_{\bv_1}\times \g_{\bv_2}\dsl G_{\bv_2} \cong \AA\dsl\si_{\bv_1}\times\AA^{\bv_2}\dsl\si_{\bv_2}$,
resp. 
$G_{\bv_1}\dsl\aad G_{\bv_1}\times G_{\bv_2}\dsl\aad G_{\bv_2}\cong \GG^{\bv_1}\dsl\si_{\bv_1}\times\GG^{\bv_2}\dsl
\si_{\bv_2}$.

The general notion of factorization sheaves on $\Coh C$ in the case $C=\AA$, resp. $C=\GG$, translates into the following
\begin{defn}\label{gfact-def}
\vi A {\em factorization sheaf} $\cF$ on $\gl$
is the data of a collection
  $(\cF_\bv)_{\bv\in\Z^I_{\geq0}}$, where
$\cF_\bv\in D(\g_\bv/{G}_{\bv})$, 
equipped,
for each pair $\bv',\bv''$, with a morphism
$\bar\vphi_{\bv',\bv''}:\ \imath_{\bv',\bv''}^*\cF_{\bv'+\bv''} \to 
\cF_{\bv'} \boxtimes \cF_{\bv''}$ such that analogues of 
\eqref{fact2} and \eqref{fac-diag}, where $\varphi$ is replaced by
$\bar\varphi$ and $C\he_{\bv_1,\bv_2}/(\si_{\bv_1}\times\si_{\bv_2})$
is replaced by $\g\he_{\bv_1,\bv_2}/({G}_{\bv_1} \times {G}_{\bv_2})$, hold.

%

\vii Equivariant factorization sheaves and factorization sheaves on
$\GL$ are defined  similarly.
\end{defn}

\begin{rem} We will not use the notion of `weak factorization sheaf' on
  either
${G}$ or $\g$. \erem

We let the group $\GG$ act on ${G}_\bv$, resp. $\g_\bv$, by scalar multiplication, resp. 
$\Aff$ act on $\g_\bv$
by $"az+b": g\mto ag+ b\cdot\Id$.
These actions commute
with the ${G}_\bv$-action, hence  descend to well-defined actions on
the corresponding quotient stacks $\g_\bv/{G}_\bv$ and ${G}_\bv/\aad {G}_\bv$.

\begin{lem} \label{res}
 The functor $\res_\ft$, resp. $\res_T$, sends factorization sheaves
 on $\g$, resp. ${G}$, to  factorization sheaves on $\Sym \AA$, resp. $\Sym \GG$.
A similar statement holds for
 $\Aff$-equivariant, resp. $\GG$-equivariant, factorization sheaves in
 the Lie algebra
resp. group, case.
\end{lem}

\begin{proof} 
Let $\g_1=\g_{\bv_1}, \
\g_2=\g_{\bv_2}$, let  $\g_1\oplus\g_2\into\g:=\g_{v_1+v_2}$ be a block
diagonal imbedding., and
put 
$\g_{1,2}\he=\g_{\bv_1,\bv_2}\he$. Let $\b_i=\ft_i+\fu_i,\ i=1,2$, be a Borel
subalgebra
in $\g_i$,  let $\fu$ be the upper right
$\bv_1\times\bv_2$-block.
Thus $\ft=\ft_1\times\ft_2$ is a Cartan subalgebra of $\g$,
$\fp=\g_1\oplus\g_2\oplus\fu$ is a parabolic subalgebra
of $\g$ with nilradical $\fu$,
and $\b=\b_1\times\b_2\times\fu$ is a Borel subalgebra
of $\g$ contained in $\fp$.
Let $G_i,B_i,U,$ etc. be the 
subgroups of $G$ corresponding to the Lie algebras $\g_i,\b_i,\fu$, etc.
We put 
\[\g\he=\g_{\bv_1,\bv_2}\he,\quad
\ft\he=\ft_{\bv_1,\bv_2}\he,\quad
(\b_1\oplus\b_2)\he=\ft_{\bv_1,\bv_2}\he\times (\fu_1\oplus\fu_2),
\quad\b\he=\ft\he\times\fu.
\]
Then, the variety $\b\he$ is stable under
the adjoint action of $B$, furthermore, for any
\begin{claim}\label{u-act}
The variety $\b\he$ is $B$-stable and
the action the map yields  an isomorphism
\[
B\times_{B_1\times B_2} (\b_1\oplus\b_2)\he\to
\b\he.
\]
\end{claim}
\begin{proof} For any $b\in \b\he$,
the variety $b+\fu$ is $U$-stable.
The differential of the adjoint action  $U\to b+\fu, u\mto u(b)$ is injective.
This is clear if $b\in \ft\he$ since in that case
all the weights of the $\ad b$-action in
$\fu$ are nonzero by the definition of the set $\ft\he$.
The case of a general  $b\in \b\he$ reduces to the one
where  $b\in \ft\he$ by
considering Jordan decomposition. Thus, the 
map $u\mto u(b)$ gives an  imbedding $U\into b+\fu$.
The image of this map is open since
$\dim U=\dim(b+\fu)$, and it is also closed in $b+\fu$
since orbits of a unipotent acting on an affine variety are
known to be closed.
We deduce that
the  action map $U\times_{U_1\times U_2} (\b_1\oplus\b_2)\he\to
\b\he$ is  an isomorphism. This implies the claim.
\end{proof}

Let $Z_i:=\g_i\times_{\g_i\dsl G_i}\ft_i$,
resp. $Z:=\g\times_{\g\dsl G}\ft$, and let $\wt\g_i$, resp. $\wt\g$, denote the Grothendieck-Springer resolution.
For any scheme (or stack) ${\mathcal G}$ over $\ft$
we write ${\mathcal G}\he:={\mathcal G}\times_\ft\ft\he$.
We have a commutative diagram:
$$
\xymatrix{
G\times_{_{G_{1}\times G_{2}}}(\wt\g_1\times\wt\g_2)\he\ar[d]_<>(0.5){\wt f}^<>(0.5){\cong}
\ar[rrr]^<>(0.5){\Id_G\times\bmu_1\times\bmu_2}&&&G\times_{_{G_{1}\times G_{2}}}(Z_1\times
Z_2)\he
\ar[d]_<>(0.5){f}
\\
\wt\g\he\ar[rrr]^<>(0.5){\bmu}&&&Z\he
}
$$
Note that
$$
G\times_{G_{1}\times G_{2}}(\wt\g_1\times\wt\g_2)
=G\times_{G_{1}\times G_{2}}((G_1\times G_2)\times_{B_1\times B_2}
(\b_1\oplus\b_2))=G\times_B(B\times_{B_1\times B_2}
(\b_1\oplus\b_2)).
$$
Therefore, Claim \ref{u-act} implies that the map $\wt f$ in the
commutative
diagram is an isomorphism. Further, it is known that
one has $(\bmu_1\times\bmu_2)_*\oo_{\wt\g_1\times\wt\g_2}=
\oo_{Z_1\times Z_2}$, resp. $\bmu_*\oo_{\wt\g}=\oo_Z$.
Hence, we deduce an isomorphism
\[f_*\oo_{G\times_{G_{1}\times G_{2}}(Z_1\times
Z_2)\he} \iso \oo_{Z\he}.\]
The morphism $f$ being affine, it follows
that this morphism is an isomorphism. 
Thus, taking quotients by the $G$-action
in the diagram above yields the following commutative diagram
$$
\xymatrix{
\ft_1\times\ft_2
\ar[d]^<>(0.5){\cong}_<>(0.5){\jmath_{\ft}}
&&(Z_1\times
Z_2)\he/(G_{1}\times G_{2})\ar[ll]_<>(0.5){pr_{\ft_1}\times pr_{\ft_2}}
\ar[d]_<>(0.5){f}^<>(0.5){\cong}
\ar[rr]^<>(0.5){pr_{\g_1}\times pr_{\g_2}}
&&(\g_1\times\g_2)\he/(G_{1}\times
G_{2})\ar[d]^<>(0.5){\jmath_{\g}}
\\
\ft&&
Z\he/G\ar[rr]^<>(0.5){pr_\g}\ar[ll]_<>(0.5){pr_{\ft}}&&\g/G
}
$$
Recall the open sets $U_i:=\g^{rs}\times_{\g_i\dsl G_i}\ft_i$,
resp. $U:=\g^{rs}\times_{\g\dsl G}\ft$.
The morphism $f$ being an isomorphism,
we deduce  an isomorphism $\op{IC}(U_1/G_1\times U_2/G_2)
\iso f^*\op{IC}(U/G)$.
Now, for any $\cf\in D\abs(\g/G)$, we compute

\begin{align*}(\res_{\ft_1}\times\res_{\ft_2})(\jmath_{\g}^*\cf)&=
(pr_{\ft_1}\times pr_{\ft_2})_!((pr_{\g_1}\times pr_{\g_2})^*\jmath_\g^*\cf\o
\op{IC}(U_1/G_1\times U_2/G_2))\\
&=(pr_{\ft_1}\times pr_{\ft_2})_!(f^*\op{IC}(U/G)\o f^*pr_\g^*\cf)\\
&=
\jmath_\ft^*(pr_\ft)_!(\op{IC}(U/G)\o pr_\g^*\cf)=\jmath_\ft^*\res_\ft(\cf).
\end{align*}
The sheaf $\op{IC}(U/G)$ being $W$-equivariant,
the functors
$\jmath_\ft^*$ and $f^*\op{IC}(U/G)\o f^*(-)$ clearly
respect $W_1\times W_2$-equivariant structures,
proving the lemma.
\end{proof}

\begin{rem}\label{res-rem} 
An argument similar to the proof of Lemma \ref{res} shows
that each of the two  functors   $\dis D\abs(\Coh_\bv)\xymatrix{
\ar@<0.3ex>[r]^<>(0.5){_{\mathfrak{res}}}&D\abs(C^\bv/\si_\bv)\ar@<0.3ex>[l]^<>(0.5){_{\mathfrak{ind}}}
}
$
sends factorization sheaves to factorization sheaves.
We will neither use nor prove this result.
\end{rem}

\subsection{$\Aff$-equivariant factorization sheaves on $\gl$}
We are going to formulate a $\gl$-counterpart of the result from
\S\ref{sec3} about the cohomology of $\Aff$-equivariant factorization sheaves on $\AA$.

We identify $\ft_\bv=\AA^\bv/\AA$ with  the Cartan subalgebra of $\pgl_\bv$
and also  identify 
$\pgl_\bv^*\dsl \PGL_\bv$ with $\ft_\bv^*\dsl \si_\bv$ via  the Chevalley isomorphism.

\begin{defn}\label{verygen} 
Let $\pgl_\bv\reg$ be the preimage of the set $\ft_\bv\reg\dsl \si_\bv$, see \eqref{de2}, under the quotient map
$\pgl_\bv^*\to \pgl_\bv^*\dsl \PGL_\bv=\ft_\bv^*\dsl \si_\bv$.
\end{defn}

The group $\GL_\bv$ acts on $\pgl_\bv$, resp. $\pgl_\bv^*$, through its quotient $\PGL_\bv$ and the imbedding
$\ft_\bv\reg\into \pgl_\bv\reg$ induces an isomorphism $\ft_\bv\reg/\GG^\bv\iso \pgl_\bv\reg/\GL_\bv$,
where $\GG^\bv$, the maximal torus of $\GL_\bv$, acts trivially on $\ft_\bv\reg$.
Thus, the canonical map $\fw: \pgl_\bv\reg/\GL_\bv\to \pgl_\bv\reg\dsl\GL_\bv=\pgl_\bv\reg\dsl\PGL_\bv=\ft_\bv\reg\dsl\si_\bv$
may be identified with the first projection $\ft_\bv\reg\dsl\si_\bv\times\pt/\GG^\bv\to \ft_\bv\reg\dsl\si_\bv$.

The group $\Aff$ acts on $\gl_\bv$ by $ax+b:\ g\mto a\cdot g+b\cdot\Id$. This action commutes with the $\GL_\bv$-action,
hence, it descends to $\gl_\bv/\GL_\bv $. 
An $\Aff$-equivariant sheaf $\cf_\bv$ on $\gl_\bv/\GL_\bv$ descends to a $\GG$-equivariant sheaf $\bar\cf_\bv$
on $\pgl_\bv/\GL_\bv$, where $\GG$ acts on $\pgl_\bv$ by dilations.
We put $\Phi_{\pgl_\bv/\GL_\bv}(\cf_\bv)=\fw_!\big(\FDN_{\pgl_\bv/\GL_\bv}(\bar\cf_\bv)|_{\pgl_\bv\reg/\GL_\bv}\big)$, where $\FDN_{\pgl_\bv/\GL_\bv}$
 denotes the Fourier-Deligne transform
on $\pgl_\bv/\GL_\bv$. Thus, $\Phi_{\pgl_\bv/\GL_\bv}(\cf_\bv)$ is a sheaf on $\ft_\bv\reg\dsl\si_\bv$ and we
let $\Phi_{\pgl_\bv/\GL_\bv}(\cf_\bv)_\eta$ be the restriction of $\Phi_{\pgl_\bv/\GL_\bv}(\cf_\bv)$ to $\eta=\Spec \overline{\k(\ft_\bv\reg\dsl\si_\bv)}$.

We have natural imbeddings $\GL_\bv/\aad\GL_\bv\xrightarrow{\eps} \gl_\bv/\GL_\bv \xleftarrow{i} 0/\GL_\bv$.
The image of each imbedding 
is  stable under the action of the subgroup $\GG\sset\Aff$.

\begin{thm}\label{GL-H} 
For any
 $\Aff$-equivariant   factorization sheaf  $\cF=(\cF_\bv)_{\bv\in\Z^I_{\geq0}}$
 on $\gl$ we have:

\vi The monodromy action of the subgroup  $\pig(\ft_\bv\reg)\sset \pig(\ft_\bv\reg/\si_\bv)$ on
$\Phi_{\pgl_\bv/\GL_\bv}(\cF_\bv)$ is unipotent and there is a Frobenius stable generalized isotypic decomposition
\[\Phi_{\pgl_\bv/\GL_\bv}(\cf_\bv)_{\eta}\ =\ \oplus_{\rho\in \Irr(\si_\bv)}\ 
\Phi_{\pgl_\bv/\GL_\bv}(\cf_\bv)^{\langle\rho\rangle}_{\eta}.
\]

\vii In $\kabs$, one has
 the following equations
\begin{align}
 &\sum_{^\bv}\ z^\bv\cdot [\RGam_c(\gl_\bv/\GL_\bv,\cF_\bv)]   
=\Sym\Big(-\L^{\frac{1}{2}}\cdot\sum_{^{\bv>0}}\  (-1)^{|\bv|}\cdot 
z^\bv\cdot \L^{\frac{|\bv|}{2}}\cdot [\Phi_{\pgl_\bv/\GL_\bv}(\cf_\bv)_\eta]^{\langle\op{sign}\rangle}\Big);
\label{2eqsG}\\
&\sum_{^\bv}\ z^\bv\cdot 
[\RGam_c(\GL_\bv/\aad \GL_\bv,\,\eps^*\cF_\bv)] 
=\Sym\Big((\L^{-\frac{1}{2}}-\L^{\frac{1}{2}})\cdot\sum_{^{\bv>0}}\ (-1)^{|\bv|}\cdot 
z^\bv\cdot \L^{\frac{|\bv|}{2}}\cdot 
[\Phi_{\pgl_\bv/\GL_\bv}(\cf_\bv)_\eta]^{\langle\op{sign}\rangle}\Big);
\label{3eqsG}\\
&\sum_{^\bv}\  (-1)^{|\bv|}\cdot z^\bv\cdot  \L^{\frac{\bv\cdot\bv}{2}}\cdot \big[\RGam_c(0/\GL_\bv,\,  i^*\cf_\bv)\big]
=
\Sym\Big(-\sum_{^{\bv>0}}\  z^\bv\cdot \L^{\frac{|\bv|}{2}}\cdot[\Phi_{\pgl_\bv/\GL_\bv}(\cf_\bv)_\eta]^{\langle\op{triv}\rangle}\Big).
\label{4eqsG}
\end{align}
\end{thm}

\begin{proof}
First, applying the functor $\RGam_c(\AA^\bv/\si_\bv,-)$ to the second isomorphism in
\eqref{xieq} in the case of the group $G=\GL_\bv$, we obtain
\begin{multline*}
\RGam_c(\GL_\bv/\aad \GL_\bv,\,\eps^*\cF_\bv)=
\RGam_c(\GL_\bv\dsl\aad \GL_\bv,\,({\mathfrak w}_{\GL_\bv})_!\eps^*\cF_\bv)\\=
\RGam_c(\GG^\bv\dsl\si_\bv,\,(({\mathfrak w}_{\GG^\bv})_!\res_\GG\eps^*\cF_\bv)^{\si_\bv})
=\RGam_c(\GG^\bv/\si_\bv,\,\res_\GG\eps^*\cF_\bv)=\RGam_c(\GG^\bv/\si_\bv,\,\eps^*\res_\AA\cF_\bv).
\end{multline*}

Next, let $T_\bv=\GG^\bv/\GG$ be the Cartan torus, resp. $\ft_\bv=\AA^\bv/\AA$ the Cartan subalgebra,
for $\PGL_\bv$. Abusing the notation, we write  $\res_\ft$ for the functor $D\abs(\pgl_\bv/\GL_\bv)\to
D\abs(\ft_\bv/\si_\bv)$ described in Remark \ref{GL-pgl}.
For each $\bv$, the sheaf $\cf_\bv$, resp. $\res_\AA\cf_\bv$,  descends, thanks to
the $\Aff$-equivariance,  to a sheaf
$\bar\cf_\bv$ on $\pgl_\bv/\GL_\bv$, resp. $\overline{\res_\AA\cf_\bv}$ on $\ft_\bv$,
furthermore, one has $\res_\ft\bar\cf_\bv=\overline{\res_\AA\cf_\bv}$.
Applying the analogue of the second isomorphism in
\eqref{xieq} in the case of $\pgl_\bv^*$ and using  Lemma \ref{qfourierr}, we find
\begin{multline*}
(\fw_{\pgl^*_\bv})_!\FDN_{\pgl_\bv/\GL_\bv}(\bar\cf_\bv)=
(\fw_{\ft^*_\bv})_!\res_{\ft^*}\FDN_{\pgl_\bv/\GL_\bv}(\bar\cf_\bv)\\=
\sign\o (\fw_{\ft^*_\bv})_!\FDN_{\ft_\bv}(\res_{\ft}\bar\cf_\bv)=
\sign\o (\fw_{\ft^*_\bv})_!\FDN_{\ft_\bv}\overline{\res_\AA\cf_\bv}.
\end{multline*}
We deduce that for any $\rho\in\Irr(\si_\bv)$ one has
\begin{align}
\Phi_{\pgl_\bv/\GL_\bv}(\cf_\bv)_\eta^{\langle\rho\rangle}&=
((\fw_{\pgl^*_\bv})_!\FDN_{\pgl_\bv/\GL_\bv}\bar\cf_\bv)_\eta^{\langle\rho\rangle}\label{fw-rho}\\
&=
\big(\sign\o (\fw_{\ft^*_\bv})_!\FDN_{\ft_\bv}\overline{\res_\AA\cf_\bv}\big)_\eta^{\langle\rho\rangle}=
\Phi_{\ft_\bv}(\res_\AA\cf_\bv)_\eta^{\langle\op{sign}\o\rho\rangle}.\nonumber
\end{align}

By  Lemma \ref{res}, the collection
  $\res_\AA\cf=(\res_\AA\cf_\bv)$
is an $\Aff$-equivariant factorization sheaf on $\AA$.
Thus, using \eqref{3eqs}, we compute
\begin{multline*}
\sum_{\bv}\ z^\bv\cdot [\RGam_c(\GL_\bv/\aad \GL_\bv,\,\eps^*\cF_\bv)]
=\sum_{\bv}\ z^\bv\cdot [\RGam_c(\GG^\bv/\si_\bv,\,\eps^*\res_\AA\cF_\bv)]\\
=\Sym\Big((\L^{-\frac{1}{2}}-\L^{\frac{1}{2}})\cdot\sum_{^{\bv>0}}\ (-1)^{|\bv|}\cdot 
z^\bv\cdot \L^{\frac{|\bv|}{2}}\cdot [\Phi_\ft(\res_\AA\cF_\bv)_\eta]^{\langle\op{triv}\rangle}\Big)\\
=\Sym\Big((\L^{-\frac{1}{2}}-\L^{\frac{1}{2}})\cdot\sum_{^{\bv>0}}\ (-1)^{|\bv|}\cdot 
z^\bv\cdot \L^{\frac{|\bv|}{2}}\cdot [\Phi_{\pgl_\bv/\GL_\bv}(\cf_\bv)_\eta^{\langle\op{sign}\rangle}]\Big).
\end{multline*}

This proves \eqref{3eqsG}. The proofs of  \eqref{2eqsG} and \eqref{4eqsG} are similar
using
the corresponding equation of Theorem  \ref{expC}. 
For example,  to prove formula \eqref{4eqsG} recall first that $n_{_{\GL_\bv}}=\frac{\bv\cdot\bv-|\bv|}{2}$.
We compute
\begin{multline*}[\RGam_c(0/\GL_\bv,\ i^*\cf_\bv)]\cdot\L^{n_{_{\GL_\bv}}}
=
[(\pt_{0/\GL_\bv})_!i^*\cf_\bv]\{-2n_{_{\GL_\bv}}\}\\=
[((\pt_{0/\si_\bv})_!i^*\res_\AA\cf_\bv)]^{\op{sign}}\{2n_{_{\GL_\bv}}-2n_{_{\GL_\bv}}\}=
[\RGam_c(0/\si_\bv,\, i^*(\sign\o \res_\AA\cf_\bv))].
\end{multline*}

By Remarks \ref{fact-rems}(2)-(3), the collection $(\sign\o \res_\AA\cf_\bv\big[|\bv|\big])$ is a
factorization sheaf on $\AA$. 
Thus, using \eqref{fw-rho}, we compute 
\begin{multline*}
\sum_\bv\ (-1)^{|\bv|}\cdot z^\bv\cdot\L^{\frac{\bv\cdot\bv-|\bv|}{2}}\cdot\big[\RGam_c(0/\GL_\bv,\ 
i^*\cf_\bv)]\\=\sum_\bv\ z^\bv\cdot [(\pt_{0/\GL_\bv})_!i^*\cf_\bv]\{-n_{\GL_\bv}\}\big[|\bv|\big]
=\sum_\bv\  z^\bv\cdot \big[\RGam_c(0/\si_\bv,\, i^*(\sign\o \res_\AA\cf_\bv\big[|\bv|\big])\big]\\
=\Sym\Big(-\L^{-\frac{1}{2}}\cdot\sum_{^{\bv>0}}\  
  (-1)^{|\bv|}\cdot z^\bv\cdot\L^{\frac{|\bv|}{2}}\cdot 
\big[\Phi_\ft(\sign\o \res_\AA\cf_\bv\big[|\bv|\big])_\eta\big]^{\langle\op{triv}\rangle}\Big)\\
=
\Sym\Big(-\L^{-\frac{1}{2}}\cdot\sum_{^{\bv>0}}\  (-1)^{|\bv|}\cdot 
  z^\bv\cdot\L^{\frac{|\bv|}{2}}\cdot 
\big[\Phi_\ft(\res_\AA\cf_\bv)_\eta\big]^{\langle\op{sign}\rangle}\big[|\bv|\big]\Big)\\
=
\Sym\Big(-\L^{-\frac{1}{2}}\cdot\sum_{^{\bv>0}}\  
 z^\bv\cdot\L^{\frac{|\bv|}{2}}\cdot 
\big[\Phi_{\pgl_\bv/\GL_\bv}(\cf_\bv)_\eta\big]^{\langle\op{triv}\rangle}\Big).
\end{multline*}

Finally, it is known that in $\kabs$ the following two equations are equivalent
\beq{a-exp}
\Sym\Big(\sum_{^{\bv>0}}  z^\bv\cdot a_\bv\Big)=\sum_{^{\bv}}  z^\bv\cdot f_\bv\quad\Leftrightarrow\quad
\Sym\Big(\L^{-\frac{1}{2}}\cdot\sum_{^{\bv>0}}  z^\bv\cdot a_\bv\Big)=
\sum_{^{\bv}} z^\bv\cdot \L^{-\frac{|\bv|}{2}}\cdot f_\bv.
\eeq
We apply this  in the case where
$a_\bv=-\L^{\frac{|\bv|}{2}}\cdot \big[\Phi_{\pgl_\bv/\GL_\bv}(\cf_\bv)_\eta\big]^{\langle\op{triv}\rangle}$
and formula \eqref{4eqsG} follows.
\end{proof}

\section{Inertia stacks}\label{iner}

\subsection{$G$-stacks}\label{gstack}
Given  an algebraic group $G$ and a $G$-stack $Y$, over $\k$, we write $q_Y: Y\to Y/G$ for
the quotient morphism. 
The categories $Y(\bar\k)$ and $(Y/G)(\bar\k)$ have the same objects
and the quotient morphism $q_Y: Y\to Y/G$ is the identity on
objects. The groups of morphisms between the
corresponding objects of $Y$ and $Y/G$ may  differ,
in general. In particular, for
 $y\in Y(\bar\k)$, the  automorphism group of the
object $q_Y(y)\in (Y/G)(\bar\k)$ equals
\[
\Aut(q_Y(y))=\{(g,f)\mid g\in G, f\in\Mor(y,gy)\},
\]
where the group operation is given by
$(g,f)\cdot (g',f')=(gg', g'(f)\ccirc f')$.

There is a natural
exact sequence
\beq{algaut}
\xymatrix{
1\ar[r]&\Aut(y)\ar[rr]^<>(0.5){f\mto
  (1,f)}&&\Aut(q_Y(y))\ar[rr]^<>(0.5){p_y:\ (g,f)\mto g}&&
G,
}
\eeq 
 of algebraic groups.
The  group $G_y:=\Im({p}_y)$ may be viewed as the stabilizer
of $y$ in $G$. 

Recall that a  morphism
$Y\to X$, where $Y$ is a $G$-stack and $X$ is a stack,
 is called a $G$-torsor (on $X$) if, for any
test scheme $S$ and an $S$-point $S\to X$, the first projection
 $S\times_X Y\to S$ is a $G$-torsor on $S$. 
Any morphism of  $G$-torsors on $X$ is an isomorphism.
For any  $G$-stack $Y$, the
quotient   morphism $q_Y: Y\to Y/G$ is a $G$-torsor.
Conversely,
for any $G$-torsor $Y\to X$  the canonical
morphism $Y/G\to X$ is an isomorphism of stacks.

One has the universal $G$-torsor
 $\pt\to \pt/G$, where  $\pt/G={B} G$ is the classifying stack 
of the group $G$. For any   $G$-torsor $Y\to X$
the constant map $Y\to\pt$
induces a map $Y/G\to\pt/G$
It follows, since $Y/G=X$ and any morphism of  $G$-torsors on $X$ is an
isomorphism,
 that the  $G$-torsor $Y\to X$ 
may be obtained from the universal $G$-torsor
 by base change, equivalently, we have $Y\cong (Y/G)\times_{_{\pt/G}}\pt$.

Let $X\times \k^v\to X$
be a trivial   rank $v$ vector bundle on a stack $X$.
For any   rank $v$ vector bundle $\V$ on $X$,
let $Y(\V)$ be the sheaf (on $X$)
of vector bundle isomorphisms
$X\times \k^v\iso \V$. There is a $\GL_v$-action on $Y(\V)$
induced by the natural $\GL_v$-action on $\k^v$.
The projection $Y(\V)\to X$ is
a $\GL_v$-torsor on $X$, to be 
 called the `frame bundle'
of the vector bundle $\V$.

\subsection{$G$-inertia stacks}\label{ZZsec} 
The {\em inertia stack}, 
$I(X)$, of a stack $X$ is defined as a fiber product 
$I(X)=X\times_{X\times X}
X$,
where each of the two maps $X\to X\times X$ is the diagonal morphism.
An object of the inertia stack is a pair $(x,f)$
where $x\in X$  and $f\in \Aut(x)$. Morphisms
$(x,f)\to (x',f')$ are defined as morphisms
$\vphi: x\to x'$ such that $f'\ccirc \vphi=\vphi\ccirc f$.

Let $G$ be an algebraic group
with Lie algebra $\g$, and fix  a $G$-stack $Y$.
One defines
equivariant analogues of the inertia stack as follows.
The morphism
$$a_G:\ G\times Y\to Y\times Y,\quad\text{resp.}\quad
a_\g:\ \g\times Y\to  \TT{Y},
$$
given by $(g,y)\mto (gy,y)$,
makes $G\times Y$ a stack over $Y\times Y$,
resp. $\g\times Y$ a stack over $\TT Y$.
We define the {\em $G$-inertia stack}, resp.
{\em infinitesimal  $G$-inertia stack}, of $Y$, by
\[ I(Y,G)=  Y\times_{Y\times Y}(G\times Y),\quad\text{resp.}\quad
I(Y,\g)=  Y\times_{\TT{Y}} (\g\times Y),\]
where $Y\to Y\times Y$ is the diagonal
morphism, resp. $Y\to \TT{Y}$ is the zero section.
In the special case where 
$Y$ is a $G$-scheme, viewed as a stack,
the $G$-inertia stack becomes a scheme and the corresponding
set of closed points
can also be defined in the following  equivalent,  more
familiar way
\beq{setZ}
I(Y,G)= \{(y,g)\in Y\times G\mid gy=y\}\ =\
\{(y,g)\in Y\times G\mid y\in Y^g\}.
\eeq

We return to the general case. 
It is easy to see  that  objects of the  stack $I(Y,G)$
are triples of the form
$(y,g, f)$, where  $y$ is an object of $Y$,
$g\in G$, and $f: y\to g(y)$ is a morphism.
 Similarly,
an object of  $I(Y,\g)$ 
is a triple $(y,g,f)$ where
$y\in Y,\ g\in \g$, and
$f: gy \to 0$ is an isomorphism in  the  groupoid $\TT_yY$.

We let the group $G$
act on itself, resp. its Lie algebra, via the adjoint action.
Then, it is clear that $I(Y,G)$, resp. $I(Y,\g)$,
is a $G$-stack. The first projection
$\pr_Y: I(Y,G)\to Y$, resp. the composition,
$\pr_G$, of the  second projection $I(Y,G)\to G\times Y$
and the first projection $G\times Y\to G$,
are morphisms of $G$-stacks.
Further, it follows from definitions that the composition
$I(Y,G)\to Y\times Y\to  Y/G\times Y/G$
factors through a morphism $q_I: I(Y,G)\to  I(Y/G)$,
furthermore, 
the latter morphism factors through $I(Y,G)/G\to I(Y/G)$.

\begin{lem}\label{ZZ} \vi The morphism $I(Y,G)/G\to I(Y/G)$
is an isomorphism.

\vii The  morphism $q_I$ is
a $G$-torsor that fits into a diagram of cartesian squares
\beq{zz}
\xymatrix{
Y\ar[d]^<>(0.5){q_Y}&&  I(Y,G)
\ar@{}[dll]|{\Box}\ar@{}[drr]|{\Box}\ar[d]^<>(0.5){q_I}\ar[ll]_<>(0.5){\pr_Y}
\ar[rr]^<>(0.5){\pr_G}
&& G\ar[d]^<>(0.5){q_G}\\
Y/G&&I(Y/G)\ar[ll]^<>(0.5){p_{_{Y/G}}}\ar[rr]_<>(0.5){p_{_{G}}}&&
G/\aad G
}
\eeq 
\end{lem}

\begin{proof} One has natural isomorphisms
$\ Y=(Y/G) \times_{ _{\pt/G}} \pt$,\ resp.
$\ G = \pt \times_{  _{\pt/G}} \pt$, \ and also $\
G \times  Y = Y \times _{Y/G} Y$.
Using these isomorphisms
we find
\begin{align*}
I(Y,G)=
Y &\times_{ _{Y \times  Y}} (G \times  Y) 
=  Y \times_{  _{Y \times  Y}} (Y \times_{  _{Y/G}} Y)  
 =  Y\times _{ _{Y \times  Y}}(Y \times  Y)
\times _{ _{Y/G \times  Y/G}} (Y/G)\\
&=  Y\times _{ _{Y/G \times  Y/G}} (Y/G) =
Y\times_{ _{Y/G}}
(Y/G) \times _{ _{Y/G \times  Y/G}}(Y/G)=
 Y\times_{ _{Y/G}} I(Y/G).
\end{align*}
The resulting  isomorphism $I(Y,G)\cong Y\times_{ _{Y/G}} I(Y/G)$
yields the left cartesian square in \eqref{zz}.
Further, since $Y\to Y/G$ is a $G$-torsor,
it follows that $I(Y,G)\to I(Y/G)$  is a $G$-torsor.
This implies ~(i).

By the universal property
of quotient stacks, the map
$\pr_G\times \pr_Y: I(Y,G)\to G\times Y$ followed by the projection to $(G\times Y)/ G$
factors through $I(Y,G)/G$, which is isomorphic to $I(Y/G)$ by (1).
Thus, we obtain a diagram
\beq{pgy}
\xymatrix{
I(Y/G)\ \ar@{=}[rr]^<>(0.5){\text{part (1)}}&&\  I(Y,G)/G\
\ar[rr]^<>(0.5){\pr_G\times \pr_Y}
& &\
(G\times Y)/ G\ 
 \ar[r]^<>(0.5){\pr_1} &\  G/\aad G.
}
\eeq
Let  $p_{_{G}}: I(Y/G)\to G/\aad G$ be the composite map.
Note that the  map $q_I\times p_{_{G}}: I(Y,G)
\to I(Y/G)\times_{_{G/\aad G}} G$ is a morphism 
of $G$-torsors on $I(Y/G)$, hence it is an isomorphism.
It follows that the right square in \eqref{zz} is cartesian,
completing the proof.
\end{proof}


In the Lie algebra case,
we put  $I\tpp(Y/G):=Y/G\,\times_{(\TT Y)/G}\,(\g\times Y)/G$,
where $Y/G\to (\TT Y)/G$ is given by the zero section.
The Lie algebra counterpart of diagram \eqref{zz} reads
\beq{llie}
\xymatrix{
Y\ar[d]^<>(0.5){q_Y}&&  I(Y,\g)
\ar@{}[dll]|{\Box}\ar@{}[drr]|{\Box}\ar[d]^<>(0.5){q_I^+}\ar[ll]_<>(0.5){\pr_Y^+}
\ar[rr]^<>(0.5){\pr_\g}
&& \g\ar[d]^<>(0.5){q_G^+}\\
Y/G&&I\tpp(Y/G)\ar[ll]^<>(0.5){p^+_{_{Y/G}}}\ar[rr]_<>(0.5){p_{_{\g}}}&&
\g/G.
}
\eeq 
In particular, the map $q_I^+$ is a $G$-torsor, so one has a canonical
isomorphism $I\tpp(Y/G)=I(Y,\g)/G$.

\subsection{A trace formula} Below, we will freely use the notation of diagram
\eqref{zz} and put $X=Y/G$. Thus, we have the projections $p_X: I(X)\to X$
and $p_{G}: I(X)\to G/\aad G$, respectively.
Recall that an object of $I(X)$ is a pair $z=(x,f)$ where
$x\in X$ and $f\in\Aut(x)$.

\begin{prop}\label{tr_prop} Let $\k$ be a finite field.
Then, for any  $\ce\in D\abs(X)$ and a locally constant sheaf
$\cF\in D\abs(G/\aad G)$,
we have

\beq{tr2}
\ltr \RGam_c(G/\aad G,\ (p_{G})_!p^*_{X}\ce\o\cf)\
 = \sum_{[x]\in [X(\k)]}\ \sum_{f\in \Aut(x)(\k)} \mbox{\large{$\frac{1}{\#\Aut(x)(\k)}$}}
\cdot(\ltr\ce_x)\cdot(\ltr\cf_{p_{_G}(x,f)}).
 \eeq
\end{prop} 

\begin{proof} Let $\scr J:=p_{X}^*\ce\o p_G^*\cf$. This is
a sheaf on $I(X)$ and we compute the trace of  Frobenius 
on $\RGam_c(I(X), \scr J)$ in two ways as follows.

On the one hand, using the projection formula,
we find
\begin{align*}\ltr \RGam_c(I(X), \scr J)&=
\ltr \RGam_c(G/\aad G,\ (p_{G})_!(p_X^*\ce\o p_{G}^*\cf))\\
&=\ltr \RGam_c(G/\aad G,\ (p_{G})_!p_X^*\ce\o\cf)\ =\ 
\op{LHS}\eqref{tr2}.
\end{align*}

On the other hand, by the Grothendieck-Lefschetz  formula 
for stacks due to Behrend \cite[Theorem 6.4.9]{B1}, one has
\[\ltr \RGam_c(I(X), \scr J)\ =
\ \sum\nolimits_{[z]\in [I(X)(\k)]}\ \mbox{\large{$\frac{1}{\#\Aut(z)}$}}
\ltr {\scr J}_z.
\]

We write $z=(x,f)$ and put $A(x,f):=\Aut(z)(\k)$. An object $z'=(x',f')\in I(X)$
is isomorphic to $(x,f)$ iff there is an isomorphism
$\vphi: x\to x'$ such that $f'=\vphi\ccirc f\ccirc\vphi\inv$.
Hence, the set $[I(X)(\k)]$ is formed by the
pairs $([x], C)$ where  $[x]\in [X(\k)]$ and $C$ is a conjugacy
class in the group $A(x):=\Aut(x)(\k)$. Furthermore,
for an object $z\in I(X)(\k)$ in the isomorphism
class  $([x], C)$
the group $\Aut(z)(\k)$ is isomorphic to $A(x, C)$, the centralizer
in $A(x)$
of an element of the conjugacy class $C$.
Thus, writing $\op{Cl} A(x)$ for the set of conjugacy
classes of the group $A(x)$, we compute
\begin{align}
\ltr \RGam_c(I(X),\ \scr J) &=
\sum_{([x],f)\in [I(X)(\k)]}\mbox{\large{$\frac{1}{\# \Aut(x,f)(\k)}$}}\cdot
\ltr (\ce_{x}\o \cf_{p_{_G}(x,f)})\label{I-tr}\\
&=\ \sum_{[x]\in [X(\k)]}\ \sum_{^{C\in \op{Cl}  A(x)}}\
\sum_{f \in C}\
\mbox{\large{$\frac{1}{\# A(x,C)}$}}\cdot
(\ltr \ce_x)\cdot(\ltr\cf_C)\nonumber\\
&=\ \sum_{[x]\in [X(\k)]}\ 
\mbox{\large{$\frac{1}{\# A(x)}$}}\cdot(\ltr \ce_{x})\cdot
\biggl(\sum_{^{C\in \op{Cl}  A(x)}}\mbox{\large{$\frac{\# A(x)}{\#
      A(x,C)}$}}\cdot\ltr\cf_C\biggr)\nonumber\\
&=\ \sum_{[x]\in [X(\k)]}\ 
\mbox{\large{$\frac{1}{\# A(x)}$}}\cdot
(\ltr \ce_{x})\cdot
\Big(\sum_{^{f\in A(x)}}\ \ltr\cf_{p_{_G}(x,f)}\Big)=
\op{RHS}\eqref{tr2}.\qedhere
 \end{align}

\end{proof}
\section{Factorization property of inertia stacks}
\subsection{Moduli stacks}\label{obj}
The goal of this subsection is to formalize the concept of moduli stack of objects of
a given $\k$-linear category.  The approach suggested below is, perhaps,  neither  the most optimal nor the
most general  one. In any case, it will be sufficient for the limited purposes of the present paper.

It is well known that
the  adequate setting for studying moduli problems
is that of {\em stacks} since, usually, moduli spaces do not
exist as schemes. Let $\k$ be an arbitrary field and  $\ck$ a $\k$-linear category.
We would like to consider a moduli stack of 
 objects of  $\ck$.
For such  a stack to be  an Artin stack
requires, in particular,
 being able to define the notion of a  flat  family of objects
of $\ck$ over a scheme. 
For the purposes of this paper, 
we formalize this as follows.

Let ${\mathsf{AffSch}}$ be the category of  affine schemes  over $\k$.
Given a  fibered category
$\pi_\cC: \cC \to {\mathsf{AffSch}}$,
let  $\fX(\cC)$ be a  not necessarily full subcategory of $\cC$
that has the same objects as $\cC$ and such that the
morphisms in $\fX(\cC)$  are defined to be the cartesian arrows, cf. \cite{Ols}.
The  category  $\fX(\cC)$ is a prestack, i.e. it is  fibered in groupoids over ${\mathsf{AffSch}}$.

\begin{defn}\label{as1} 
A {\em  sheaf-category} over $\k$ is a  fibered category
$\pi_\cC: \cC \to {\mathsf{AffSch}}$ equipped, for each $S\in {\mathsf{AffSch}}$, 
with  the structure  of  a $\k[S]$-linear category on the fiber $\cC_S := \pi_\cC\inv(S)$, the fiber of $\cC$ over $S$,
such that 
\begin{enumerate}
\item  For every morphism 
$f: S'\to S$, in ${\mathsf{AffSch}}$,  the pull-back functor $f^*: \cC_S \to \cC_{S'}$ (defined up to a canonical isomorphism) is $\k[S]$-linear.


\item The  prestack  $\fX(\cC)$ is an Artin stack  (to be called {\em the moduli stack} of objects
of ~$\cC$). 
\end{enumerate}
\noindent

We let $\ck:=\cC(\Spec \k)$, a $\k$-linear category, and  refer to
$\cC$ as a {\em sheaf-category enhancement} of $\ck$.
\end{defn}

The descent property of  $\X(\cC)$, which is part of the definition of a stack,
guarantees that the  presheaf $\cC: S\mto \cC_S$ is, in fact, a sheaf in fppf topology.
Thus,  the name `sheaf-category' is legitimate.

\begin{rem}\label{toen} \vi In the case of an abelian category,
the notion of  flat  family of objects has been studied by Lowen and Van den Bergh
\cite{LvdB}, and also by Joyce \cite{J}. Unfortunately, it is not sufficient 
for applications we are interested in to restrict to
the case of abelian categories (there is a more general notion of 
{\em quasi-abelian category}, cf. \cite{An}, \cite[Appendix B]{BvdB}, 
that  is  sufficient 
for most  applications). In any case, all examples  of $\k$-linear
categories considered in this paper come equipped with a natural
 enhancement to a sheaf-category. So, in the rest of the paper
we simply work with sheaf-categories and
will not discuss whether or not a given $\k$-linear category 
comes from some sheaf-category.

\vii A much better, and a more correct approach would be
to use formalism of derived algebraic geometry
as follows. 
Let ${\mathcal D}$ be  a finite-type dg category, i.e. a dg-category such that its category of modules is equivalent 
to the category modules over some homotopically finitely presented dg-algebra.
 For any  t-structure on  ${\mathcal D}$,   To\"en   and Vaquie \cite{TV} constructed 
 a moduli space of objects of the heart of the   t-structure. This moduli space  is 
an $\infty$-stack, which is a direct limit of derived $n$-stacks for all $n$. 
Using  \cite[Corollary 3.21]{TV}, one can find a substack which is in fact a derived Artin $1$-stack.
 Finally, one can consider
 an underlying classical stack of that Artin $1$-stack.

\viii A finite type category is saturated in the sense of Kontsevich.
In such case, an analogue of  (2) in Definition \ref{as1} would says that the natural functor
$\k[S']\o_{\k[S]} \cC_S\to\cC_{S'}$ is an equivalence.
\erem

The sheaf-category $\cC$ is said to be {\em Karoubian} if,
for any affine scheme $U$ of finite type over $\k$,
the category $\cC(U)$ is Karoubian, i.e. for any $M\in\cC(U)$
every idempotent morphism
$p\in \Hom_{\cC(U)}(M,M),\ p^2=p$,  is the projection to a direct
summand of $M$.

\begin{lem}\label{clear} Let  $\cC$ be a Karoubian sheaf-category.
 Then, for any  \emph{separable} field  extension $K/\k$,
there is a natural equivalence ${K}\o_\k\ck\cong \cC(\Spec K)$
and a natural bijection of the set of isomorphism
classes of objects of category ${K}\o_\k\cC$ with
the set $\X(\cC)(K)$.
\end{lem}

\begin{proof}
It will be convenient to consider a more general setting.
 Fix a  finite \emph{\'etale} morphism
  $S'\to S$. Given a  finite \'etale morphism $U\to S$,
let  $\cC'(U):= \cC_{U\x_SS'}\ (=\oo_{U\x_SS'}\o_{\oo_U}\cC(U))$,
resp.  $\cC''(U):=\cC_{U\x_SS'}$. 
We claim that the fibers over $S$ of the fibered categories $\cC'$ and
$\cC''$
are equivalent. 

To prove the claim, we introduce a 
category  $T$ of  `trivialized'  morphisms  $U\to S$.
That is, an object of $T$
is an \'etale  morphism  $U\to S$  equipped with an isomorphism
 $U\x \Xi\to U$ 
for some finite  set $\Xi$.
Let $\cC^{\Xi}$ be a fibered
category over $T$ such that its fiber over 
an object $U\in T$ as above  is $\prod_{\xi\in\Xi}\cC(U)$.
Then, the  Karoubi property of $\cC$
implies that the pullback of the  category
$\cC'$ to such an $U\in T$  is canonically equivalent to  $\cC^{\Xi}$.
Therefore, applying descent to  the sheaf of categories
 $U\mto\cC(U)$, we deduce
that $\cC'(U)$ is equivalent to the category of 
 sections of fiber  products of categories  $\cC^{\Xi}$ over the
 the diagram
$\dis
\xymatrix{U \ar@<0.3ex>[r]\ar@<-0.3ex>[r]&
 U \otimes_S U
\ar@<0.7ex>[r]\ar[r]\ar@<-0.7ex>[r]
&
U \otimes_S U \otimes_S U \ldots
}
$.
This yields our claim since all the above clearly applies to the  sheaf of categories
 $U\mto\cC''(U)$ as well. 

The statement of the lemma follows  as a special case
where the morphism  $S'\to S$ is taken to be  $\Spec K\to\Spec\k$.
\end{proof}

\begin{ex}\label{bun} An $I$-graded coherent sheaf
on a scheme $U$ is, by definition, 
a direct sum  $\V=\oplus_\ii\ \V_i$ where $V_i$ is a
coherent sheaf on $U$. Equivalently, one can view
an  $I$-graded coherent sheaf as a coherent sheaf on a disjoint union of $I$
copies of $U$. 
Let  $\VV^I(U)$, the category of
 $I$-graded algebraic
vector bundles  on $U$, be a full  subcategory of the
category of $I$-graded coherent sheaves 
whose objects are  locally free sheaves.
The category  $\VV^I(U)$ has  a natural  enhancement
 to a sheaf-category  defined by $\VV_U^I: S \mto \VV^I(U\times S)$.
The corresponding moduli stack $\X(\VV_U^I)$ is
a disjoint union $\X(\VV_U^I)=\sqcup_{\bv\geq 0}\ Bun_\bv(U)$
where $Bun_\bv(U),\ \bv=(v_i)_\ii$, is the stack that parametrizes
vector bundles  $\V=\oplus_\ii\ \V_i$, on $U$, such that $\rk \V_i=v_i,\
\forall \ii$.

In the special case where $I$ is a one element set we use simplified notation
$\VV(U)=\VV^I(U)$, resp. $\VV_U=\VV_U^I$, etc.
\end{ex}

\begin{defn}\label{fun-def} A functor of sheaf-categories is 
a functor $F: \cC \to \cC'$ such that $\pi_{\cC'} \circ F = \pi_\cC$
 {\em strictly}, and the induced functor $\cC_S\to \cC'_S$ is $\k[S]$-linear for all $S \in {\mathsf{AffSch}}$.
\end{defn}

We will use the notation $\VV^I=\VV(\pt)$, a sheaf enhancement of the categry of finite dimensional vector spaces. 
By definition, one has $\fX_\bv(\VV^I)\cong \pt/\GL_\bv$.
This stack  comes equipped with
a `tautological' universal $I$-graded vector bundle of $I$-graded rank $\bv$.

Let $\cC$ be a sheaf category and  $F:\cC\to\VV^I$  a sheaf functor .
This functor induces a morphism
$\X(\cC)\to \X(\VV^I)$, of moduli stacks.
For any dimension vector $\bv$, one has
a  substack $\fX_\bv(\cC)=\X(\cC)\times_{\X(\VV^I)}\X(\VV^I_\bv)$ parametrizing
all objects $x\in\cC$  such that $\op{\mathbf{rk}} F(x)=\bv$.
The substack  $\fX_\bv(\cC)$
is  both open and closed in  $\fX(\cC)$,
so  $\fX(\cC)=\sqcup_{\bv}\ \fX_\bv(\cC)$ is a disjoint union.
The universal vector bundle on $\X(\VV^I_\bv)$
gives, by pull-back,  a natural
$I$-graded rank $\bv$ vector bundle $\FF_\bv$ on $\X_\bv(\cC)$. 
Let  $\fX_\bv(\cC, F)\to\X_\bv(\cC)$ be the  `frame bundle'
of the vector bundle $\FF_\bv$,
cf. \S\ref{gstack}.
Thus, $\fX_\bv(\cC, F)$ is a $\GL_\bv$-torsor over $X_\bv(\cC)$
and  $\fX_\bv(\cC, F)/\GL_\bv\cong \X_\bv(\cC)$.
An object of  the
 category $\Spec\k\to\fX_\bv(\cC, F)$ is a
 pair $y=(x,{b})$ where $x$ is an object of $\cC_\k$ and
${b}$ is a frame, i.e. 
an isomorphism $F(x)\ccong \oplus_\ii \k^{v_i}$, of $I$-graded vector spaces.
We  will refer to $\fX_\bv(\cC, F)$ as the `framed stack' for
$\X_\bv(\cC)$.

\begin{ex}\label{par-ex} Let $S$ be a scheme over $\k$ and $D=\sqcup_\ii\ D_i\sset S$ a disjoint union
of  closed subschemes. Given a dimension vector $\bm\in \Z^I_{\geq 0}$
one has a  $\k$-linear category
$\VV(S,D,\bm)$ of parabolic bundles, cf. \S\ref{par-sec}.
An object of $\VV(S,D,\bm)$  is  a locally free sheaf $\V$ on $S$ equipped, for each $\ii$,
with a partial flag of subbundles $\V|_{D_i}=\V_i^{(0)}\supseteq \V_i^{(1)}\supseteq\ldots \supseteq\V_i^{(m_i)}=0$,
i.e. a chain of coherent subsheaves such that each quotient $\V_i^{(j)}/\V_i^{(j+1)}$ is a locally free sheaf on $D_i$.
Category $\VV(S,D,\bm)$  has an enhancement to a sheaf category.
Specifically, for any test scheme $U$ let $D\times U:=\sqcup_\ii\ D_i\times U$,
 a disjoint union
of  closed subschemes of $S\times U$.
Then, the sheaf category  is given by the presheaf
$U\mto \VV(U\times S, D\times S,\bm)$.
We will often abuse the notation and write $\VV(S,D,\bm)$ for the sheaf enhancement of
the category of parabolic bundles denoted by the same symbol.

Below, we are mostly be interested in the
special case where $S$ is a smooth curve and $D=\{c_i\in S(\k),\ \ii\}$ a collection of marked points.
In this case, we have the natural identification 
$\{c_i\}\times U\cong U$, so we may (and will)
view $V^{(j)}_i/V^{(j+1)}_i$  as
a  locally free sheaf on $U$. The assignment
\[\V\ \mto\ F(\V)=\bigoplus_\ii\ \bigoplus_{1\leq j\leq m_i}\
V^{(j)}_i/V^{(j+1)}_i.
\]
 yields a sheaf-functor
$F: \VV(C,D,\bm)\to\VV^{\bbm}(\pt)=\VV^{\bbm}$. 
For a type $\br\in \bbm$, the corresponding  framed stack $\X_\br(\VV(C,D,\bm), F)$ classifies pairs 
$(\V,b)$ where  $\V$  is  a 
parabolic
bundle on $S$  of type $\br$ and  $b=(b^{(j)}_i)_{(i,j)\in \bbm}$  is a
collection of bases, where 
 $b^{(j)}_i$ is a $\k$-basis  of the vector space
 $V^{(j)}_i/V^{(j+1)}_i$.
The forgetful map $\X_\br(\VV(S,D,\bm),\, F)\to \X_\br(\VV(S,D,\bm))$
is a $G_\br$-torsor.
\end{ex}

 \subsection{Endomorphisms and automorphisms}
\label{frame-sec}
Fix
a sheaf-category $\cC$ and write $\X=\X(\cC)$.
Let $S$ be an affine scheme, $\sch_S$ the category of affine schemes over $S$,
and $\X(S)$ the groupoid corresponding to $S$.
We write $\Aut(x)$ for the stabilizer group scheme of 
an object $x\in \X(S)$.

For any objects $x,y\in\X(S)$,  let $\Hom_{x,y}$ be a functor
$\sch_S\to \textsf{Sets}$ defined by the assignment
$(p:S'\to S)\ \mto \Hom_{\cC_{S'}}(p^*x,p^*y)$.

\begin{lem}\label{hom-sc} \vi  The functor $\Hom_{x,y}$ is represented by
an affine scheme $\underline\Hom(x,y)$ that
has the natural structure of  
a generalized vector bundle on $S$
in the sense of Definition \ref{vect-bun}. Furthermore,  there is a canonical $\k[S]$-module isomorphism
of $\Hom_{\cC_S}(x,y)$ and $\Ga(S, \underline\Hom(x,y))$, the  space of sections 
of this  generalized vector bundle.

\vii 
The stack $\X(\cC)$ has an affine diagonal. In particular, the projection
$I(\X(\cC))\to \X(\cC)$ is an affine morphism and $\X$ is a stack with affine stabilizers.
\end{lem}
\begin{proof} 
For any $f\in  \Hom_{\cC_{S'}}(p^*x,p^*y)$, the matrix 
 $\left|\begin{array}{cc}\Id_x & 0 \\ f & \Id_y
\end{array}\right|$ gives an automorphism of $p^*x\oplus p^*y$, so one has an
imbedding
of the set $\Hom_{\cC_{S'}}(p^*x,p^*y)$ into  the set $\Aut(p^*(x\oplus y))$.
The functor $(p:S'\to S)\ \mto \Aut(p^*x\oplus p^*y)$ is  represented by the group scheme
 $\Aut(x\oplus y)\in \sch_S$.
Hence, the presheaf of sets over  $\sch_S$ associated with  the functor $\Hom_{x,y}$
is a sub-presheaf of the presheaf  associated with  the functor  represented by $\Aut(x\oplus y)$.
Therefore, the functor  $\Hom_{x,y}$ is represented by a sub-groupscheme of $\Aut(x\oplus y)$.
It may be characterized as follows.
We have a natural homomorphism 
$\gamma: \GG\to\Aut(x\oplus y)$ given by $t\mto\textrm{diag}(\Id_x, t\cdot\Id_y)$.
Let  $\GG$ act on  $\Aut(x\oplus y)$ by group scheme automorphisms via conjugation.
Then, we have $\underline\Hom(x,y)=\{g\in \Aut(x\oplus y)\mid \underset{^{t\to0}}\lim\ \gamma(t)g\gamma(t)\inv\Id\}$. 
This can be reformulated  scheme-theoretically by identifying $\underline\Hom(x,y)$
with
the scheme of $\GG$-equivariant maps $u:\AA\to \Aut(x\oplus y)$ such that $u(0)=0$.
It follows from this last reformulation that the scheme $\underline\Hom(x,y)$ is affine.
 The rest of the proof of
part (i) are left to the reader.

To prove (ii) observe first that, for any $x,y,z\in \X(S)$, the composition
map $\Hom_{\ccc_S}(x,y)\times \Hom_{\ccc_S}(y,z)\to \Hom_{\ccc_S}(x,z)$ induces a morphism
$\ccirc: \underline\Hom(x,y)\times \underline\Hom(y,z)\to \underline\Hom(x,z)$, of affine schemes over $S$.

The pair $(x,y)$ gives a morphism $S\to \X\times\X$
and we also have 
a morphism $\{(\Id_x,\Id_y)\}_S:\ S\to \underline\Hom(x,y) \times_S \underline\Hom(y,x)$, of schemes over $S$.
Further, let 
$\underline\Hom(x,y)\times_S \underline\Hom(y,x)\to S$
be a morphism given by
$(f,g) \mto (g\ccirc f,f\ccirc g))$. 
With this notation, one has a natural  isomorphism 
$$
S\to \X\times\X=
 (\underline\Hom(x,y) \times_S \underline\Hom(y,x)) \times_{\underline\Hom(x,x) \times_S \underline\Hom(y,y)} \{(\Id_x,\Id_y)\}_S
$$
Part (i) implies that  the RHS of this formula is affine over $S$.
It follows that $S\times_{\X\times\X}\X$ is an affine schemes over $S$, proving  (ii).
\end{proof}

In the special case $x=y$, we put
$\End (x):=\Hom_{\ccc_S}(x,x)$.
By the above lemma, we have $\End (x)=\Ga(S, \underline\Hom(x,x))$.
Let $\underline\Aut(x)$ be a closed subscheme of
$\underline\Hom(x,x)\times_S \underline\Hom(x,x)$ defined by the equations $g'\ccirc g=g\ccirc g'=\Id$.
It is easy to show  that
 $\underline\Aut(x)$ is canonically isomorphic to $\Aut x$
as a  group scheme over $S$.

Now, we fix a sheaf functor $F: \ccc\to \VV^I$.
For $x\in \X(S)$, let
 $F_x: \End(x)\to \End F(x)$ be the $\k[S]$-algebra homomorphism
induced by  $F$.

\begin{lem}\label{end_lem}  
For any $x\in \X(S)$, the ideal $\Ker F_x$
 is  nilpotent, i.e. we have $(\Ker F_x)^N=0$ for $N\gg0$.
\end{lem}

\begin{proof} Put $B=\k[S]$. First, we consider the case where $B=\k$.
Let $J$ be the kernel of the map $F_x$
and let $A:=\k + J$. Thus, $A$ is  a  finite dimensional unital subalgebra of 
$\End(x)$ and $J$ is a codimension one two-sided ideal of $A$.
The lemma would follow provided we show that $J$ is contained in $\fN$,
the nilradical of $A$ (which is the same as the Jacobson radical).

Assume that $J\nsubseteq \fN$. Let $\bar J\cong J/(J\cap\fN)$ be
the image of $J$ in  $A/\fN$. Thus, $\bar J$ is a nonzero codimension one two-sided ideal
of  $A/\fN$. 
Any two-sided ideal of  $A/\fN$,  a semisimple finite dimensional algebra, is generated by  a central idempotent.

We deduce that the ideal $\bar J$ contains a nonzero idempotent
$\bar e$.
The ideal  $\fN$ being nilpotent,
the idempotent $\bar e$ can be lifted to an idempotent 
$e\in A$ such that $e\ (\op{mod} \fN) =\bar e\in\bar J$. 
Since $A=\k+J$ this implies that $e\in J$.
Thus, we have constructed a nonzero idempotent $e\in \End(x)$ 
that belongs to the kernel of the
homomorphism $F_x: \End(x)\to \End_\k F(x)$.

Associated with the  decomposition $1=e+(1-e)$, in $\End(x)$, there
is a direct sum decomposition
 $x=x'\oplus x''$, in $\scr C$. This gives a vector space
decomposition
$F(x)=F(x')\oplus F(x'')$.  It is clear, by functoriality,
that the map $F_x(e): F(x)\to F(x)$ is equal to the composition
$F(x)\onto F(x')\into F(x)$, of the first projection
and an imbedding of the first direct summand.
We deduce
\[e\neq 0\en \Rightarrow\en x'\neq 0\en \Rightarrow\en F(x')\neq 0\en
\Rightarrow\en F_x(e)\neq 0,\]
On the other hand,  by construction we have $e\in \Ker(F_x)$, hence
 $F_x(e)=0$. The contradiction implies
that $J\sset \fN$. This completes the proof in the special
case $B=\k$.

We now consider the general case, and put $S=\Spec B$.
Thus, $\End(x)$ is a $B$-algebra which is finite as a $B$-module.
For every closed point $s:\Spec\k\into S$, one has the composite
map
$x_s:\ \Spec\k\to X$ and the corresponding finite dimensional
$\k$-algebra $\End(x_s)$. The algebra $\End(x)$ being
finite over $B$,  there is an integer
$N$ such that $\dim_\k \End(x_s)<N$ for all closed points $s\in S$.
Observe further that given an arbitrary $\k$-algebra $R$ such that $\dim_\k R<N$
and  a two-sided ideal ${\mathfrak J}$ of $R$ such that
 ${\mathfrak J}^m=0$  for some $m\gg 0$, one automatically has
 that  ${\mathfrak J}^N=0$. Therefore, from the
special case
$B=\k$ of the lemma, we deduce that
 $(\Ker F_{x_s})^N=0$ holds for all $s$.

We claim that $(\Ker F_x)^N=0$.
Indeed, if $a_1\cdots a_N\neq 0$ for some elements
 $a_1,\ldots,a_N\in \Ker F_x$ then one can find
a closed point $s\in S$ such that $a_1|_s\cdots a_N|_s\neq 0$.
That would contradict the equation  $(\Ker F_{x_s})^N=0$, 
since  $a_j|_s\in \Ker F_{x_s},\ j=1,\ldots,N$.
\end{proof}

Fix a sheaf-functor $F:\cC\to \VV^I$.
We will use the notation $X=\X_\bv(\cC),\,
Y=\fX_\bv(\cC, F)$, resp.  $\g=\gl_\bv,\ G=\GL_\bv$ and
$\g(S)=\k[S]$, $G(S)$ for $S$-poits of $G$.
For $y=(x,b)\in Y(S)$ and  $f\in \End(x)$ there is a unique
element $p_{\g,y}(f)\in \g(S)$ such that $F_x(f)(b)=p_{\g,y}(f)(b)$.
Let $\g_y$ be  the image of the map $f\mto p_{\g,y}(f)$. Thus,
we have an  exact sequence
\beq{algalg}
\xymatrix{
0\ar[r]&\Ker F_x\ar[rr]
 &&\End(x)\ar[rr]^<>(0.5){p_{\g,y}}&&
\g_y.
}
\eeq 
Observe that the map $p_{\g,y}$ 
is a homomorphism of  associative subalgebras.

The group  $\Aut(y)$, resp. $\Aut(x)$,
is  the group of units of  the 
associative algebra $\End(y)$, resp. $\End(x)$.
Furthermore, the restriction of this
map to $\Aut(x)$ is nothing but the map 
$p_y$ in \eqref{algaut}.
Thus, the exact sequence in \eqref{algaut}
is obtained from  \eqref{algalg}, an  exact sequence
of associative algebras, by
 restriction to the groups of units.
On the other hand, clearly, one has  $\Lie\Aut(x)=\End(x)$,
where $\End(x)$ is viewed as a Lie algebra with
the commutator bracket.
Therefore,  \eqref{algalg} may be identified
with  the  exact sequence
of Lie algebras induced by \eqref{algaut}.
In particular,  the group $G_y$ is
 the group of units of  the 
associative algebra $\g_y$ and we have
$\Lie G_y=\g_y$. 

\begin{cor}\label{unip} Let $y\in Y(\k)$ and $x\in X(\k)$ be the image
  of $x$. Then,

\vi
The groups $\Aut(x)$ and $G_y$ are
  connected.

\vii The group $\Aut(y)$ is unipotent and
we have $\Aut(y)=\Ker\big[F_x:\Aut(x)\to\GL(F(x))\big]$.
\end{cor}
\begin{proof}
It is clear that the group of units of a finite dimensional
algebra is connected. It follows that the group $\Aut(x)$,
hence also $G_y$, is connected, proving (i).
The equation of part (ii) follows from the exact sequence
\eqref{algaut}. Hence, we deduce
$\Aut(y)=\Id+\Ker[F_x: \End(x)\to \End F(x)]$. Lemma \ref{end_lem} insures that 
this group  is unipotent, and (ii) follows.
\end{proof}

\begin{cor}\label{unicor}
If the functor $F$ is faithful then, for any $\bv$, the
stack $Y$ is an algebraic space.
\end{cor}
\begin{proof}
This  is an immediate consequence  of  Corollary
\ref{unip}(ii)
and \cite[Corollaire (8.1.1)(iii)]{LM}.
\end{proof}

\begin{cor}\label{rational} If $\k$ is a finite field
then each fiber of the natural map $[Y(\k)]\to [X(\k)]$
is a $G(\k)$-orbit.
\end{cor}
\begin{proof} 
The fiber of the map
$[Y(\k)]\to [(Y/G)(\k)]$ over $x\in X(\k)$ is  
isomorphic to the set of $\k$-rational points of a $G$-orbit
which is  defined over $\k$.
Our claim is a consequence of a combination of two  well known results.
The first result says that any such orbit
 contains a $\k$-rational point, say $y\in Y(\k)$.
Then,  $G_y$ is a connected subgroup of  $G=\GL_\bv$ defined over $\k$,
and we can identify the fiber with
$(G/G_y)(\k)$.
The second result says that, the group $G_y$ being connected,
the $G(\k)$-action on
$(G/G_y)(\k)$ is transitive.
\end{proof}

Note for any $y=(x,b)\in Y(\k)$ the group
$\Aut(x)$,  resp. $\Aut(y)$, contains the 
subgroup $\GG_{\text{scalars}}$ of scalar automorphisms
and one has $p_{\g,y}(\GG_{\text{scalars}})=\GG^\bv_\Delta\sset G_y$.

\begin{cor}\label{fitting} 
For any $y=(x,b)\in Y(\k)$ we have

\vi If $x$ is not absolutely  indecomposable then the group $G_y$ contains
a torus $H$ defined over $\k$ which is 
a $G(\bar\k)$-conjugate of a diagonal torus $\GG^\bv_\a\neq \GG^\bv_\Delta$.

\vii The following conditions are equivalent:
\[\text{$x$ is  absolutely indecomposable}
\quad\Leftrightarrow\quad
\text{$\Aut(x)/\GG_{\text{scalars}}\ $\ is unipotent}
\quad\Leftrightarrow\quad
\text{$G_y/\GG^\bv_\Delta\ $ is unipotent.}
\]
\end{cor}
\begin{proof} Let $J$ be the (nil)radical of the algebra
$\End(x)(\k)$ and $A=\End(x)(\k)/J$,  a finite 
dimensional semisimple $\k$ algebra. By the Wedderburn theorem, we have
$A$ is isomorphic to  a finite direct sum of matrix
algebras over some division algebras  of finite dimension
over $\k$. Let $B\sset A$ be a direct sum of the 
subalgebras of diagonal matrices of all these matrix algebras.
Thus, we have $B\cong \oplus_\al\ Q_\al$, where each 
$Q_\al$ is a division algebra over $\k$.
For every $\al$, we choose a maximal subfield $K_\al\sset Q_\al$.
Then, $K_\al$ is a finite extension of $\k$.
Therefore, the group $\prod_\al\ K_\al^\times$ defines, by Weil restriction,
an algebraic subgroup $H$ of $\Aut(x)$, which is  defined over $\k$
and is
 such that $H(\k)=\prod_\al\ K_\al^\times$.
Moreover,
 $H$  is 
a torus.

For each $\al$ we have $\bar\k\o_\k K_\al= \oplus_{\gamma\in\Gamma_\al}\ \bar\k 1_{\al,\gamma}$,
a direct sum of several copies of the field $\bar\k$ with the
correspoding unit
element being denoted by $1_{\al,\gamma}$. The elements $\{1_{\al,\gamma}\}$
form a collection  of orthogonal idempotents
of the algebra $\bar\k\o_\k A$ which can be lifted
to  a collection $\{e_{\al,\gamma}\}$ of orthogonal idempotents
of the algebra  $\End(x)(\bar\k)$. Let $\bar\k\o_\k x=
\oplus_{\al,\gamma}\ x_{\al,\gamma}$ be the corresponding
direct sum decomposition, where $x_{\al,\gamma}\in \cC_{\bar\k}$.
It is clear from the construction that
we have
$\End(x_{\al,\gamma})(\bar\k)=
e_{\al,\gamma}\End(x)(\bar\k)e_{\al,\gamma}=\bar\k\oplus
J_{\al,\gamma}$, where $J_{\al,\gamma}$ is a nilpotent ideal.
Thus, $\End(x_{\al,\gamma})(\bar\k)$ is a local $\bar\k$-algebra.
Hence, $x_{\al,\gamma}$ is an indecomposable object, by the Fitting
lemma. 

We have $F(x)=\oplus \ F(x_{\al,\gamma})$, a direct sum of
$\bar\k I$-modules. 
The homomorphism $p_y: \Aut(x)\to \GL$ induces an isomorphism
$H(\bar\k)=\prod \ \GG(\bar\k) \Id_{x_{\al,\gamma}} \iso\,
p_y(H)(\bar\k)=\prod \ \GG(\bar\k) \Id_{F(x_{\al,\gamma})}$
Let $b_{\al,\gamma}$ be a $\bar\k$-basis of $F(x_{\al,\gamma})$
which is compatible with the $I$-grading. 
Then, there is a unique element $g\in G$
that takes the given basis $b$ of $F(x)$ to  the basis $\sqcup_{\al,\gamma} \ b_{\al,\gamma}$.
It is clear that $\Ad g(p_y(H)(\bar\k))$ is a diagonal torus
that contains $\GG^\bv_\Delta$, proving (i).

To prove (ii) observe that
the object $x$ is absolutely indecomposable if and only if  $x=x_{\al,\gamma}$ for a single
$(\alpha,\gamma)$. The latter holds if and only if  $\dim H
=\sum_\al \ \dim_\k K_\al=1$, that is, if and only if  one has $p_y(H)=\GG^\bv_\Delta$.
\end{proof}

\subsection{Factorization of fixed point loci}
Let $G_\bv=\GL_\bv$, resp. $\g_\bv=\gl_\bv$ and $X_\bv=\fX_\bv(\cC),\ Y_\bv=\fX_\bv(\cC,F)$. 
We have the stack $I(Y_\bv, G_\bv)$, resp. $I(Y_\bv,\g_\bv)$,
whose objects are triples $(x,b,f)$ where $(x,b)\in Y$ and $f\in \Aut(x)$,
resp.  $f\in \End(x)$. We recall the isomorphism  $X_\bv=Y_\bv/G_\bv$ and
observe that the above definition of the stack $I\tpp(X_\bv)$
agrees with the definition $I\tpp(Y_\bv/G_\bv)=I(Y_\bv,\g_\bv)/G_\bv$, given
at the end of section \ref{ZZsec}.
Thus, there are natural
commutative diagrams of open imbeddings
\beq{GGAA}
\xymatrix{
I(Y_\bv,G_\bv)\ar[d]^<>(0.5){\pr_G}\ar@{^{(}->}[r] &
I(Y_\bv,\g_\bv)\ar[d]^<>(0.5){\pr_\g}
&& I(X_\bv)\ar[d]^<>(0.5){p_{G}}\ar@{^{(}->}[r] &
I\tpp(X_\bv)\ar[d]^<>(0.5){p_{\g}}\\
G_\bv\ar@{^{(}->}[r]^<>(0.5){\eps_G} &\g&&
G_\bv/\aad G_\bv\ar@{^{(}->}[r]^<>(0.5){\eps_{G/G}} &\g_\bv/G_\bv
}
\eeq

\begin{cor}\label{Zcart} Each of the diagrams \eqref{GGAA} is cartesian.
\end{cor}
\begin{proof} The map $\pr_\g$, resp. $\pr_G$, sends
$(x,b,f)$ to ${p_y}(f)$, where $y=(x,b)$. Thus,  at the
level of objects, the statement  amounts to the claim
that an endomorphism
$f\in\End(x)$ is invertible if and only if  the element ${p_y}(f)\in \g$
is invertible. The latter claim is a consequence of
Lemma \ref{end_lem}. Indeed, we have ${p_y}(f)\in\Im {p_y}
\cong \End(x)/\Ker{p_y}$ and,
the ideal $\Ker{p_y}$ being nilpotent, an element  of
the algebra  $\End(x)$  is invertible iff
so is its image in $\End(x)/\Ker{p_y}$.

The proof for morphisms is similar and is left for the reader.
\end{proof}

Let $\bv_1,\bv_2$ be a pair of dimension vectors.
Recall the notation of \S\ref{gfact} and
consider a natural commutative
diagram
\beq{d12}
\xymatrix{
& I\tpp(X_{\bv_1})\times I\tpp(X_{\bv_2})\ 
\ar[d]^<>(0.5){p_{\g_{\bv_1}}\times p_{\g_{\bv_2}}}
\ar[rr]^<>(0.5){{\mathfrak i}_{\bv_1,\bv_2}}&&\  I\tpp(X_{\bv_1+\bv_2})\
\ar[d]^<>(0.5){p_{\g_{\bv_1+\bv_2}}}\\
U\he:=\g\he_{\bv_1,\bv_2}/(G_{\bv_1}\times G_{\bv_2})\
\ar@{^{(}->}[r]^<>(0.5){\en \jmath_{\bv_1,\bv_2}}&\  \g_{\bv_1}/G_{\bv_1}\times
\g_{\bv_2}/G_{\bv_2}\ 
\ar[rr]^<>(0.5){\imath_{\bv_1,\bv_2}}&&\ 
\g_{\bv_1+\bv_2}/G_{\bv_1+\bv_2},
}
\eeq
where the map $\imath_{\bv_1,\bv_2}$, resp. ${\mathfrak i}_{\bv_1,\bv_2}$,
 sends a pair of objects to their direct sum.

\begin{prop}\label{Zfix} The morphism 
\[
U\he \bigtimes_{^{\g_{\bv_1}/G_{\bv_1}\times
  \g_{\bv_2}/G_{\bv_2}}}\
 \bigl(I\tpp(X_{\bv_1})\times
  I\tpp(X_{\bv_2})
\xrightarrow{\Id\times{\mathfrak i}_{\bv_1,\bv_2}}
U\he\bigtimes_{^{\g_{\bv_1+\bv_2}/G_{\bv_1+\bv_2}}}\ I\tpp(X_{\bv_1+\bv_2}),
\]
induced by 
\eqref{d12}, is an  isomorphism of stacks over $\g\he_{\bv_1,\bv_2}/(G_{\bv_1}\times G_{\bv_2})$.
\end{prop}
\begin{proof} We introduce
 simplified notation
$\g_j=\g_{\bv_j}, G_j=G_{\bv_j},\, j=1,2$,
resp. ${\g=\g_{\bv_1+\bv_2},
G=G_{\bv_1+\bv_2}}$.

It suffices to show that, for any affine scheme
$S$ and a morphism  $g: S\to \g\he_{\bv_1,\bv_2}$,
the corresponding functor 
\beq{fix_fun}
\{g\}\ \times_{(\g_1/G_1\times
  \g_2/G_2)(S)}\
\bigl(I\tpp(X_{\bv_j})(S)\times
I\tpp(X_{\bv_j})(S)\bigr)\too 
\{g\}\ \times_{(\g/G)(S)}\ I\tpp(X_{\bv_1+\bv_2})(S)
\eeq
is an equivalence of categories.
We first show that this functor is full. An object of the category on
the right is a triple $(x,f,h)$ where $x\in X(S),\ f\in \End(x)$,
and $h\in G(S)$ is  such that  $p_{\g}(f)=\Ad h(g)$, cf. \eqref{algalg}.
Without loss of generality we may (and will) assume that $h=1$,
so $p_{\g}(f)=g$.

The map $g$ is given by a pair of elements
$g_j\in \g_j\o \k[S],\ j=1,2$.
The vector space $\End(x)$, resp. $\g_j\o \k[S]$, has the natural structure
of an associative  $\k[S]$-algebra. Further, there is
a natural evaluation homomorphism
$\ev_f: \k[t]\o\k[S]=\k[\AA\times S]\to \End F(x)$,
 resp. 
$\ev_j: \k[t]\o\k[S]\to \g_j\o \k[S]$,
an algebra homomorphism
defined by the assignment $t\o 1\mto f$,
resp. $t\o 1\mto g_j$.
We put $J:=\Ker(\ev_f)$, resp. $J_j:=\Ker(\ev_j)$.
Let $A_f\sset \End(x)$, resp. $A_g\sset \g\o\k[S]$
and $A_j\sset \g_j\o\k[S]$, be a $\k[S]$-subalgebra
generated by $f$,
resp.  $g$ and $g_j$.
Thus, we get
the following homomorphisms:
$$
A_f= \k[\AA\times S]/J\ \stackrel{u}\onto\
A_g=\k[\AA\times S]/(J_1\cap J_2)\ \stackrel{v}\into\ 
A_1\bplus A_2 =\k[\AA\times S]/J_1\ \bplus\
\k[\AA\times S]/J_2,
$$
where the second map is induced by
the diagonal imbedding 
$\k[t]\o\k[S]\to (\k[t]\o\k[S])\ \oplus\ (\k[t]\o\k[S])$.

Further, let
$p_j(t,s)=\det(t\Id_{\k^{\bv_j}}-g_j(s))$
be  the characteristic polynomial of $g_j$.
Thus, we have   $p_j\in \k[\AA\times S]$
and, moreover, the Hamilton-Cayley theorem implies that $p_j(t,s)\in J_j$.
The polynomials $p_1$ and $p_2$ have
no common zeros since $g\in \g_{\bv_1,\bv_2}$.
Therefore, the ideals $J_1$ and $J_2$  have
no common zeros and, hence,
$J_1+J_2=\k[\AA\times S]$, by the Nullstellensatz.
This implies that  $v$, the second map in the diagram above, is
an isomorphism. We deduce that there
is  a canonical
direct sum decomposition $A_g\cong A_1\oplus A_2$.
Thus, the element $(1_{A_1},0)\in A_1\oplus A_2$
produces  a  nontrivial idempotent  of the algebra $A_g$.

Observe next that $\Ker(u)$ is a nilpotent ideal in $A_f$,
thanks to Lemma \ref{end_lem}. The algebra $A_f$ being
commutative, it follows that the
idempotent $(1_{A_1},0)$ can be lifted {\em uniquely} to an
idempotent $e=e_{x,f,g}\in A_f$. By construction, $e$ is an element of $\End(x)$
that commutes with $f$.
We put $x_1:=ex$ and $f_1=e\ccirc f$, resp. $x_2=(1-e)x$,
 and $f_2=(1-e)\ccirc f$. Thus, we have a direct sum 
$(x,f)=(x_1,f_1)\oplus (x_2,f_2)$
and, moreover, $p_G(f_j)=g_j$.
This proves that the functor in \eqref{fix_fun} is full.

Now, let $(x,f,g)\to (x',f',g')$ be a morphism in
the category on the right of \eqref{fix_fun}.
This means that we have
 a morphism $\vphi: x\to x'$
such that $\vphi\ccirc f=f'\ccirc \vphi$,
and also  $p_\g(f')=g'\in \Ad h(g)$
for some $h\in G(S)$.
Therefore, the map
$\wt\vphi:\ A_f\iso A_{f'},\ u\mto \vphi\ccirc u\ccirc \vphi\inv$,
is  an algebra isomorphism.
Furthermore, 
we have $\wt\vphi(e_{x,f,g})=e_{x',f',g'}$,
since  the elements $g_j$ and $\Ad h(g_j)$
have equal characteristic polynomials
and  the idempotent are determined by the
corresponding triples uniquely. It follows that
the  morphism $\vphi$ takes
the direct sum  decomposition $(x,f)=(x_1,f_1)\oplus (x_2,f_2)$
to the corresponding  decomposition $(x',f')=(x'_1,f'_1)\oplus (x'_2,f'_2)$,
resulting from $e_{x',f',g'}$. This implies that the 
functor \eqref{fix_fun} is fully faithful.
Hence, this functor is an equivalence, as required.
\end{proof}

 Proposition \ref{Zfix} and its proof have an immediate generalization to the case where 
the stack $I\tpp(\fX_\bv(\cC))$ is replaced
by the stack $\fX_\bv(\k[C]\o\cC)$, cf. \S\ref{Phi-sec}, where  $C\sset \AA$ is a fixed Zariski open nonempty subset.  In particular, 
for $C=\GG$ from Proposition \ref{Zfix}, using Corollary \ref{Zcart},
we obtain
\begin{cor} The natural morphism below is an isomorphism:
\[
{G}\he_{\bv_1,\bv_2}\ \bigtimes\nolimits_{({G}_{\bv_1}/\aad G_{\bv_1})\times 
  ({G}_{\bv_2}/\aad G_{\bv_2})}\
 \bigl(I(X_{\bv_1})\times
  I(X_{\bv_2})\bigr)
\too
{G}\he_{\bv_1,\bv_2}\ \bigtimes\nolimits_{G_{\bv_1+\bv_2}/\aad G_{\bv_1+\bv_2}}\
  I(X_{\bv_1+\bv_2}).\qedhere
\]
\end{cor}

\begin{rem}\label{general-curve} Proposition \ref{Zfix} can also be extended to the case of stacks
$\fX_\bv(\k[C]\o\cC)$, where  $C$ is a smooth affine curve which is not necessarily contained in $\AA$.
We will neither use nor prove such a generalization.
\end{rem}

\subsection{} \label{fsec} We now introduce the factirization sheaf $\rf$ associated to a triple
$(\cC,F,\phi)$, as has been outlined in \S\ref{h-intro} in a more general setting of an smooth scheme $C$.
Thus, we fix a potential, a morphism $\phi: \fX(\cC)\to \AA$, of stacks over $\k$.
For each $\bv$,
put $\phi_\bv=\phi|_{X_\bv}$.
According to  \eqref{llie}, one has a diagram
\[\AA\xleftarrow{\phi_\bv}\X_\bv(\cC)=\X_\bv(\cC,F)/\GL_\bv \xleftarrow{p_{\X_\bv(\cC,F)/\GL_\bv}} I\tpp(\X_\bv(\cC)\xrightarrow{p_{\gl_\bv}}\
\gl_\bv/\GL_\bv.\]

We define a sheaf $\rf_\bv$, on $\gl_\bv/\GL_\bv$, by
$\rf_\bv:= (p_{\gl_\bv})_!\,p^*_{\X_\bv(\cC,F)/\GL_\bv}\,\wp^{\phi_\bv}$, where $\wp^{\phi_\bv}=\phi_\bv^*\wp_\psi$ is a pull-back
of the Artin-Schreier local system.

\begin{thm}\label{Rfact}  If the morphism $\phi$ is
 additive in the sense of \S\ref{exp-sec}, then the collection $\rf=(\rf_\bv)_{\bv\in\Z^I_{\geq0}}$
 has the natural structure of   an  $\Aff$-equivariant factorization sheaf on
 $\gl$.
\end{thm}

\begin{proof}  Fix dimension vectors $\bv_1,\bv_2$. We use simplified
notation from the proof
of Proposition \ref{Zfix}. Thus, for $j=1,2$, we write $X_j=\fX_{\bv_j}(\cC)$, resp. $Y_{\bv_j}=\fX_\bv(\cC,F),\ \phi_j=\phi_{\bv_j},
\ G_j=\GL_\bv,\ \g_j
=\gl_{\bv_j},\ I_j\tpp=I_1\tpp(X_{\bv_j})$, and
  $I\tpp=I\tpp(X_{\bv_1+\bv_2}), \phi=\phi_{\bv_1+\bv_2}$.
Also,
$U\he:=\g\he_{\bv_1,\bv_2}/(G_1\times G_2),
\imath=\imath_{\bv_1,\bv_2},\jmath=\jmath_{\bv_1,\bv_2},
{\mathfrak i}={\mathfrak i}_{\bv_1,\bv_2}$. 
Then, using diagrams \eqref{llie} and \eqref{d12}, one obtains the following
 commutative diagram
$$
\xymatrix{
\AA\times\AA\ \ar[d]^<>(0.5){+}&
\ X_{\bv_1}\times X_{\bv_2}\  \ar[d]^<>(0.5){\oplus}\ar[l]_<>(0.5){\phi_1\times\phi_2}
 &&\ 
U\he\ \bigtimes_{\g_1/G_1\ \times\
  \g_2/G_2}\ (I_1\tpp\times I_2\tpp)\ 
\ar[d]_<>(0.5){\cong}^<>(0.5){\Id\times{\mathfrak i}}
\ar[ll]_<>(0.5){(p_{X_1}\times p_{X_2})\ccirc pr}
 \ar[rr]^<>(0.5){\Id\times(p_{\g_1}\times p_{\g_2})} &&\ \ U\he\ 
\ar@{=}[d]^<>(0.5){\Id\times\imath\ccirc\jmath}
\\
\AA\ & \ X_{\bv_1+\bv_2}\  \ar[l]_<>(0.5){\phi} &&\  U\he\
\bigtimes_{\g/G}\ I\tpp\ \ar[ll]_<>(0.5){p_X\ccirc pr}
\ar[rr]^<>(0.5){\Id\times p_\g} &&\  U\he\ \times_{\g/G} \ \g/G,
  }
$$
where $pr$ denotes the projection to the second factor.

$$
\xymatrix{
\AA\times\AA\ \ar[d]^<>(0.5){+}&
\ X_{\bv_1}\times X_{\bv_2}\  \ar[d]^<>(0.5){\oplus}\ar[l]_<>(0.5){\phi_1\times\phi_2}
 &&\ 
\ I_1\tpp\times I_2\tpp\ 
\ar[d]^<>(0.5){{\mathfrak i}}
\ar[ll]_<>(0.5){p_{X_1}\times p_{X_2}}
\ar[r]&(\g_1/G_1\ \times\ \g_2/G_2)\times_{\g/G}I\tpp
 \ar[rr]^<>(0.5){p_{\g_1}\times p_{\g_2}} &&\ \g_1/G_1\ \times\
  \g_2/G_2
\ar[d]^<>(0.5){\imath}
\\
\AA\ & \ X_{\bv_1+\bv_2}\  \ar[l]_<>(0.5){\phi} &&\ 
I\tpp\ \ar[ll]_<>(0.5){p_X}
\ar[rrr]^<>(0.5){p_\g} &&&\ 
\ \g/G,
  }
$$
By the additivity of the morphism $\phi$,
using the canonical isomorphism $(+)^*\wp_\psi\cong
\wp_\psi\boxtimes\wp_\psi$\,
where the map `$+$' is given by addition $(a,a')\mto a+a'$,
we obtain an isomorphism 
\beq{fact_phi}
(\oplus)^*\phi^*\,\wp_\psi\ =\ (\phi_1\times\phi_2)^*(+)^*\,\wp_\psi\ =\ 
(\phi_1\times\phi_2)^*\,(\wp_\psi\boxtimes\wp_\psi)
\ =\ 
\phi_1^*\wp_\psi\boxtimes\phi_2^*\wp_\psi.
\eeq
We compute
\begin{align*}
\imath^*\rf_{\bv_1+\bv_2}\ &=\ 
\imath^*\,(p_{\g})_!\,p^*_X\phi^*\,\wp_\psi\ =\ 
(\Id\times\imath)^*\,(\Id\times p_\g)_!\,
p_X^*\,\phi^*\wp_\psi\quad\text{(Base change)}\\
&=\ \big(\Id\times(p_{\g_1}\times p_{\g_2})\big)_!\,{\mathfrak i}^*\,
p_X^*\,\phi^*\wp_\psi\\
&=\ \big(\Id\times(p_{\g_1}\times p_{\g_2})\big)_!\,
 (p_{X_1}\times p_{X_2})^*\,(\oplus)^*\,\phi^*\wp_\psi
\quad\text{(by \eqref{fact_phi})}\\
&=\ \big(\Id\times(p_{\g_1}\times p_{\g_2})\big)_!\,(p_{X_1}\times p_{X_2})^*\,
(\phi_1^*\wp_\psi\boxtimes\phi_2^*\wp_\psi)\quad\text{(Base change)}\\
&=\ \jmath^*\,(p_{\g_1}\times p_{\g_2})_!\,(p_{X_1}\times p_{X_2})^*\,
(\phi_1^*\wp_\psi\boxtimes\phi_2^*\wp_\psi)\\
&=\ \jmath^*\,\big((p_{\g_1})_!\,p_{X_1}^*\,\phi_1^*\wp_\psi\ \boxtimes\
(p_{\g_2})_!\,p_{X_2}^*\,\phi_2^*\wp_\psi\big)
\ =\ \jmath^*(\rf_{\bv_1}\boxtimes \rf_{\bv_2}).
\end{align*}

\begin{align*}
\rf_{\bv_1}\boxtimes \rf_{\bv_2}
&=\ (p_{\g_1})_!\,p_{X_1}^*\,\phi_1^*\wp_\psi\ \boxtimes\
(p_{\g_2})_!\,p_{X_2}^*\,\phi_2^*\wp_\psi\ 
=\ (p_{\g_1}\times p_{\g_2})_!\,(p_{X_1}\times p_{X_2})^*\,
(\phi_1^*\wp_\psi\boxtimes\phi_2^*\wp_\psi)\\
&=\ (p_{\g_1}\times p_{\g_2})_!\,(p_{X_1}\times
p_{X_2})^*\,(\oplus)^*\,\phi^*\wp_\psi\
=\ (p_{\g_1}\times p_{\g_2})_!\,{\mathfrak i}^*\,p^*_X\phi^*\,\wp_\psi
\end{align*}
Therefore, the base change morphism $\imath^*(p_{\g})_!\to
(p_{\g_1}\times p_{\g_2})_!{\mathfrak i}^*$ yields a canonical
morphism
$$\imath^*\rf_{\bv_1+\bv_2}\ =\ 
\imath^*\,(p_{\g})_!\,p^*_X\phi^*\,\wp_\psi\xrightarrow{\varphi}
(p_{\g_1}\times p_{\g_2})_!\,{\mathfrak i}^*\,p^*_X\phi^*\,\wp_\psi
\ =\
\rf_{\bv_1}\boxtimes \rf_{\bv_2}
$$

To complete the proof observe that for any affine scheme $S$ and
an object $x\in \cC(S)$, there is
 an $\Aff$-action on
on $\End(x)$, resp. $\GG$-action on $\Aut(x)$, 
defined by  $"az+b":\ g\mto a\cdot g+b\cdot\Id_x$, resp. 
 $"a":\  g\mto a\cdot g$, similarly to \S\ref{gfact}
This gives the  stack $I\tpp(X_\bv)$, resp. $I(X_\bv)$, 
the structure of an $\Aff$-stack, resp. $\GG$-stack.
Furthermore, the morphism  $p_{\g}$ in  diagram \eqref{llie}
  is  clearly $\Aff$-equivariant,  resp.  the morphism $p_{G}$
in  diagram  \eqref{zz} is $\GG$-equivariant.
It follows $\rf_\bv$ is an $\Aff$-equivariant sheaf
on $\gl_\bv/G_\bv$.  It is straightforward to see that the factorization
sheaf structure on $\rf=(\rf_\bv)$ constructed in the proof of
Lemma \ref{Rfact} is compatible with the $\Aff$-equivariant structure.
This gives $\rf_{\rm Lie}$ the structure of an  $\Aff$-equivariant factorization sheaf.
\end{proof}

\section{Inertia stacks and Fourier transform}\label{if}

\subsection{Fourier transform for stacks}\label{Fstacks}

Let  $Y$ be a stack, viewed as a  $\GG$-stack with a trivial action.

\begin{defn}\label{vect-bun}   A morphism   $E\to Y$,  of  $\GG$-stacks, is called
\begin{itemize}
\item  a {\em vector bundle} on $Y$ if  for any
test scheme $S$ and a morphism $S\to Y$, the first projection
 $S\times_X Y\to S$ is a vector bundle on $S$ and the  induced $\GG$-action on  $S\times_X Y$ is by dilations.
\item  a {\em generalized vector bundle} on $Y$ if
$E$ is a representable stack over $Y$ of the form $\Spec_Y (\Sym M)$ where $M$ is
a coherent sheaf on $Y$.
\item  a {\em  vector space stack}
if
there exists an open covering $\{u_\al: Y_\al\to Y\}$ in
fppf-topology
and, for each $\al$, a morphism  $f_\al: E'_\al \to E''_\al$  of vector bundles on
$Y_\al$
and a $\GG$-equivariant isomorphism $u_\al^*(E)\cong \coker(f_\al)$.
In this case, we put $\rk E:= \rk E''_\al-\rk E'_\al$.
Morphisms of vector space stacks on $Y$ are defined to be the $\GG$-equivariant morphisms of stacks over $Y$. 
\end{itemize}
\end{defn}

Given a morphism   $f: E' \to E''$  of vector bundles on
a stack $Y$ one may view $E'$ as a flat group scheme over $Y$
and $E''$ as a $E'$-stack, where $E'$ acts on $E''$ by translation.
We write $\coker(f)$ for the corresponding quotient stack.
The map $f$ is equivariant under the dilation action of $\GG$
 on $E'$ and $E''$. Thus, $\coker(f)$ is a  vector space stack to be denoted
$E''/E'$. By definition, any vector space stack has an fppf-local presentation of the form $E''/E'$.


Let $f: Y_1\to Y_2$ be a morphism of stacks. Given a generalized vector bundle $E=\Spec_{Y_2}(\Sym M)$,
resp. a vector space stack $E=E''/E'$ on $Y_1$, one defines its pull-back by the formula
$f^*E=\Spec_{Y_1}(\Sym f^*M)$, resp. $f^*E=f^*E''/f^*E'$.
It is immediate to check that the latter formula is independent of the choice of a presentation
of $E$ in the form $E''/E'$. There is a canonical morphism $f_\sharp: f^*E\to E$, of stacks, such that 
$p_2\ccirc f_\sharp=f\ccirc p_1$, where $p_1: f^*E\to Y_1$ and $p_2: E\to Y_2$ denote the projections.

The dual of a vector space stack $E$ on $Y$ is , by definition,   a {\em generalized vector bundle}
$E^*$  on $Y$ defined as follows. Given  a test scheme $S$ let
$E^*(S,y)$ be category of morphisms
$E \times_Y S \to \AA_S = \AA \times S$ of  vector space stacks on $S$.
It is easy to see that the category $E^*(S)$ is in fact a groupoid
and we let $E^*$ be 
the stack defined by the assignment $S\mto E^*(S)$.
In a special case where $E=E''/E'$ this definition is equivalent to the formula
$E^*=E_2^* \times_{E_1^*} Y= \Spec_Y \Sym(\op{Coker}_Y(E_1 \to E_2))$,
where we have used the the morphism $Y\to E_1^*$ given by  the zero section.

\begin{ex} Let $Y$ be a smooth stack. Then, the tangent complex of $Y$ is a short complex  $d: T'_X\to T''_X$.
 In this case, we have $\TT Y\cong \coker(d)$, so the tangent stack is a vector space stack over $Y$. 
Further, we have $\TT^*Y=\Spec_Y(\coker(d))$. Thus, the cotangent stack of $Y$ is a
generalized vector bundle, the dual of the vector space stack $\TT Y$.
\end{ex}

\begin{rem}
Associated with a vector space stack $E$, there is also a dual vector space stack $E^*[1]$ defined by $\Hom_Y (S, E^*[1]) = \Hom(E \times_Y S, B\GG_a \times S)$ 
and a generalized vector bundle $\Omega E (=E[-1] ) := Y \times_E Y$. One has a
canonical isomorphism $\Omega(E^*[1])\cong E^*$.
\erem

We proceed to define the Fourier-Deligne transform in this setting, cf. \cite[\S7.3]{BG} for a special case.
Thus, we assume $\k$ to be a finite field and let $\wp_\psi$ be the Artin-Schreier sheaf associated with a fixed
additive character
$\psi: \k\to\CC^\times$. Let $Y$ be a stack over $\k$ and
$E$ a vector space stack on $Y$. We have a natural evaluation morphism $ev_E: E\times_Y E^* \to \AA$. 
For $\cf\in D\abs(E)$, define $\FD_E(\cf) =\pr_{2!}(\pr_1^*\cf \otimes ev_E^*(\wp_\psi)) \in D\abs(E^*)$,
where $\pr_i$ is the projection from $E \times_Y E^*$ to the $i$-th factor. 
It will be more convenient to use a renormalized Fourier-Deligne transform defined by 
$\FDN_E(\cf) =\FD_E(\cf)[ \ \mathrm{rk} E](\half \mathrm{rk} E)$.
The latter commutes with the Verdier duality.

\begin{prop} \label{FD} \vi The Fourier-Deligne transform $\FDN_E$ is an equivalence of categories;

\vii Let $h: E_1 \to {E_2}$ be a morphism of vector space stacks 
and $h^\vee: E_2^* \to E_1^*$  the dual morphism of generalized vector bundles.
Put $d=\rk E_2-\rk E_1$.
Then, there are isomorphisms of functors
 \beq{four-iso}\FDN_{E_2} \ccirc h_! = h^{\vee, *}\ccirc  \FDN_{E_1}[d](\half d),\qquad
\FDN_{E_1}\ccirc h^*  = h^\vee_! \ccirc \FDN_{E_2} [d](\half d).
\eeq

\viii Let $f: Y_1\to Y_2$ be a morphism of stacks, $E$ a vector space stack on $Y_2$, and $f^\sharp_E: f^*E\to E$,
resp. $f^\sharp_{E^*}: f^*E^*\to E^*$, the canonical morphisms.
Then,  there are isomorphisms of functors
\[ (f^\sharp_{E^*})_!\ccirc \FDN_{f^*E}= \FDN_E\ccirc (f_E^\sharp)_!,
\qquad
(f_{E^*}^\sharp)^*\ccirc \FDN_E=\FDN_{f^*E}\ccirc (f_E^\sharp)^*.
\]

\end{prop}

\begin{ex} 
Any vector space stack on $Y=\pt$ is isomorphic to a stack of the form
$E=\AA^m\times (\pt/\AA^n)$. For the dual generalized vector bundle $E^*$
one finds that $E^*=(\AA^m)^*\times (\pt/\AA^n)^*=(\AA^m)^*$. 
Let $pr: \AA^m\times \pt/\AA^n\to \AA^m$ be the first projection.
Then the functor $f^*: D\abs(\AA^m) \to D\abs(\AA^m\times \pt/\AA^n)$ is an equivalence
since the additive group $\AA^n$ is a unipotent group.
Thus, the equivalence of part (i) of the Proposition reduces, in this case,
to an analogous equivalence in the case of $E=\AA^m$.
\end{ex}

\begin{proof}[Sketch of proof of the Proposition] The first, resp. second, isomorphism in (ii) amounts to 
 functorial  isomorphisms $\FD_{E_2} (h_! A) = h^{\vee, *} \FD_{E_1} (A)$, resp. 
$\FD_{E_1}(h^* B) = h^\vee_! \FD_{E_2} (B)[2d](d)$ for any $A\in D\abs({E_1})$ and $B\in D\abs({E_2})$.
We only prove the first isomorphism.
To this end, consider the following commutative diagram

$$
\xymatrix{
\ & \
\ & \
{E_1} \ \ar[r]^<>(0.5){h}\ar@{}[dr]|{\Box}& \
{E_2} \ \\
\AA^1& \
{E_1} \times_Y E_1^*  \   \ar@{} [dr] |{\Box} \ar[l]^<>(0.5){ev_{E_1}}
              \ar[d]^<>(0.5){\pr_2^{E_1}} \ar[ur]^<>(0.5){\pr_1^{E_1}}&    \
{E_1} \times_Y E_2^* \  \ar[l]^<>(0.5){id \times h^{\vee}}  \ar[r]^<>(0.5){h \times id}
                \ar[d]^<>(0.5){\widetilde \pr_2} \ar[u]^<>(0.5){\widetilde \pr_1} &   \
{E_2} \times_Y E_2^* \ar[u]^<>(0.5){\pr_1^{E_2}} \ar[r]^<>(0.5){ev_{E_2}} \ar[dl]^<>(0.5){\pr_2^{E_2}}&  \
\AA^1 \ \\
\ & \
{E_1}^*& \ 
{E_2}^* \ar[l]^<>(0.5){h^{\vee}}& \
   }
$$

The squares in the diagram are Cartesian.
The required isomorphism is a consequence of the following chain of isomorphisms
\begin{align*}
h^{\vee,*}\FD_{E_1}(A)&=h^{\vee,*} \pr_{2,!}^{E_1}(\pr_1^{{E_1},*}(A) \otimes ev_{E_1}^*(\wp_\chi))\\
&=
\widetilde \pr_{2,!} (id \times h^{\vee})^*(\pr_1^{{E_1},*}(A) \otimes ev_{E_1}^*(\wp_\chi))\qquad\text{(base change)}\\
&=
\widetilde \pr_{2,!}(\widetilde \pr_1^*(A) \otimes (id \times h^{\vee})^*ev_{E_1}^*(\wp_\chi))\\
&=
\widetilde \pr_{2,!}(\widetilde \pr_1^*(A) \otimes (h \times id)^*ev_{E_2}^*(\wp_\chi))\qquad{(\text{since \ } ev_{E_1} \circ (id \times h^{\vee}) = ev_{E_2} \circ (h \times id))}\\
&=
\pr_{2,!}^{E_2} (h \times id)_! (\widetilde \pr_1^*(A) \otimes (h \times id)^*ev_{E_2}^*(\wp_\chi))\qquad\text{(projection formula)}\\
&=
\pr_{2,!}^{E_2} ((h \times id)_! \widetilde \pr_1^*(A) \otimes ev_{E_2}^*(\wp_\chi))\qquad\text{(base change)}\\
&=
\pr_{2,!}^{E_2} (\pr_1^{{E_2},*}h_!(A) \otimes ev_{E_2}^*(\wp_\chi)) = \FD_{E_2}(h_! A).
\end{align*}

The proof of the equivalence in (i) mimics the standard proof of a similar equivalence in the
case of vector bundles on a scheme, cf. \cite{KL}. 
Specifically, define a functor
 $\FDN_{E^*}: D\abs(E^*)\to D\abs(E)$ by the formula
$\FDN_{E^*}(\cf) =\pr_{1!}(\pr_2^*\cf \otimes ev_E^*(\wp_{-\psi}))[\mathrm{rk} E](\half \mathrm{rk} E)$.
Then, one establishes the Fourier inversion formula saying that there are  isomorphisms of functors
$\FDN_{E^*}\ccirc \FDN_E\cong \Id_{D\abs(E)}$ and $\FDN_E\ccirc \FDN_{E^*}\cong \Id_{D\abs(E^*}$,
respectively.

To prove the isomorphisms, 
let $i: E\times_Y E^*\into E\times E^*$, resp. $i^\vee:   E^*\times_Y E\into E^*\times E$, be the imbedding and put
$K_E:=i_!ev_E^*(\wp_{\psi})[\mathrm{rk} E](\half \mathrm{rk} E)$, resp. $K_{E^*}=i^\vee_!ev_{E^*}^*(\wp_{-\psi})[\mathrm{rk} E](\half \mathrm{rk} E)$.
The functors $\FD_E$, resp. $\FD_{E^*}$, can be written as a convolution functor
$\FDN_E(\cf)=\cf*K_E$, resp. $\FDN_{E^*}(\ce)=\ce*K_{E^*}$.
Let $\Delta_E: E\to E\times_Y E$, resp. $\Delta_{E^*}: E^*\times_Y E^*$,
denote the diagonal imbedding. Using adjunction, one constructs
natural morphisms $(\Delta_E)_!\C_E\to K_E*K_{E^*}$, resp. $(\Delta_{E^*})_!\C_{E^*}\to
K_{E^*}*K_E$.
The  Fourier inversion formula is equivalent to a statement that these morphisms
are isomorphisms. The latter statement is local with respect to $Y$. Therefore, we may (and will) assume 
that $E=E''/E'$. Thus, the category $D\abs(E)$ may be identified with the $E'$-equivariant derived 
category of $E''$. Further, the group $E'$ being unipotent, the forgetful functor
$D\abs(E)\to D\abs(E'')$ is an imbedding. The essential image of this imbedding
is a full subcategory of $D\abs(E'')$ whose objects are sheaves $\cf$ such that there exists
an isomorphism $a^*\cf\cong p^*\cf$, where $a$, resp.  $p$, is the action, resp. second projection, $E'\times_Y E'' \to E''$.
This way, we are reduced to proving the result in the case where $E$ is a vector bundle on $Y$.
The latter case follows from the corresponding result for schemes, which is known.

The proof of part (iii) is similar. We omit it.
\end{proof}

\subsection{Infinitesimal inertia and Fourier transform}
Let $G$ be an algebraic group and  $\mathfrak g$  its Lie algebra. 
Let $Y$ be a smooth $G$-stack, $i_Y: Y\to \TT Y$ the zero section,  $a_\g: \g\times Y \to \TT Y$ the infinitesimal action,
and $f: Y/G\to \pt/G$ the map induced by a constant map $Y\to \pt$.
 We have a commutative diagram
\beq{ddiag}
\xymatrix{
  Y/G\ar[d]^<>(0.5){\Id}\ar[r]^<>(0.5){i} \    &  \  \TT Y/G \   \ar[d] & \   (\g\times Y)/G  \  \ar[d]  \ar[l]_<>(0.5){a}   \ar[rr]^<>(0.5){f^\sharp_{\g/G}=pr_1} 
\ar@{}[drr]|{\Box} &&  \ \fg/G  \ar[d]  \\
  Y/G \ \ar@{=}[r]  &  \  Y/G \ \ar@{=}[r] &  \    Y/G \ \ar[rr] ^<>(0.5){f}&& \  \pt/G
    }
\eeq
Here, the first 3 vertical maps are vector space stacks, the map $i$,
resp.  $a$, induced by $i_Y$, resp.   $a_\g$, is a morphism
of vector space stacks on $Y/G$. The square on the right is cartesian,
so we have $pr_1=f^\sharp_{\g/G}$.

Let  $\mu: \TT^*Y \to \mathfrak g^*$ be    the moment map and
$\pi: \TT^*Y\to Y$ the projection.  
We have a commutative diagram
$$
\xymatrix{
  Y/G\ar[d]^<>(0.5){\Id}
&  \  T^*Y/G \  \ar[l]_<>(0.5){\pi'}\ar[d]  \ar[r]^<>(0.5){\nu}  &   (\fg^*\times Y)/G \ar@{}[drr]|{\Box} \ar[d] 
\ar[rr]^<>(0.5){f^\sharp_{\g^*/G}=pr_1}
&&  \ \fg^*/G\ar[d] \\
  Y/G \ \ar@{=}[r]  &  \  Y/G \ \ar@{=}[r] &  \    Y/G \ \ar[rr]^<>(0.5){f} && \pt/G}
$$
where the maps $\nu$ and $\pi'$ are induced by $\mu\times \pi$ and $\pi$, respectively.
The map $\nu$ is the  dual of the map $a'$ by the definition of the moment map.
Hence, the above diagram is a diagram of morphisms of generalized vector bundles 
which is dual to diagram \eqref{ddiag}.

Recall the notation of diagram \eqref{llie} and let 
$p: I\tpp(Y/G) \to (\g\times Y)/G$ be a composition of the map $\pr_\g\times \pr_Y: I\tpp(Y/G) \to \g\times Y$
and the quotient map $\g\times Y\to (\g\times Y)/G$.
Let $\mu':\ \TT^*Y/G \to Y/G$ be the map induced by $\mu$.


\begin{prop}\label{Zfour} For any morphism $\phi: Y/G\to\AA$, there is an isomorphism
$$
\FDN_{\mathfrak g/G \to {B} G}((p_{\fg})_! p_{Y/G}^*  \wp^\phi)=\mu^{\prime}_!(\pi')^*\wp^\phi\{2\dim Y-\dim G\}.
$$
\end{prop}

\begin{proof}
We have the following commutative diagram where the square is cartesian, cf.  \eqref{llie}:
$$
\xymatrix{
  \TT Y/G \     & \   (\g\times Y)/G \ar[l]_<>(0.5){a}   \ar[r]^<>(0.5){pr_1} &  \ \fg/G \\
   Y/G \   \ar[u]^<>(0.5){i}  &   I\tpp(Y/G)  \ar[l]_<>(0.5){p_{Y/G}} \ar[ur]_<>(0.5){p_\g} \ar[u]^{p} &  \ \ \\
   }
$$

By base change  we obtain

$$(p_{\fg/G})_! p_{Y/G}^* \wp^\phi= (pr_1)_! p_! p_{Y/G}^*  \wp^\phi= (pr_1)_! a_{\fg/G}^* i_! \wp^\phi=
(f^\sharp_{\g/G})_! a_{\fg/G}^* i_! \wp^\phi.$$


Thus, we compute
\begin{align*}
\FDN_{\g/G \to {B} G}((p_\g)_! p_{Y/G}^*   \wp^\phi) &=
\FDN_{\g/G \to {B} G} ((f^\sharp_{\g/G})_! a^* i_!  \wp^\phi)\qquad\text{(by Prop. \ref{FD}(iii))}\\
&=(f^\sharp_{\g^*/G})_!\FDN_{(\g^*\times Y)/G\to Y/G}(a^* i_!  \wp^\phi)\qquad\text{(by Prop. \ref{FD}(ii))}.
\end{align*}
The rank of the vector bundle  $TY/G\to Y/G$, resp. $(\g\times Y)/G\to Y/G$,
equals $\dim Y$, resp. $\dim\g$.
Hence, we have
\begin{multline*}
\FDN_{(\g\times Y)/G\to Y/G}(a^* i_!  \wp^\phi)=
\mu'_!\FDN_{(TY/G\to Y/G}(i_!  \wp^\phi)\{\dim Y\}\\
=\mu'_!(\pi')^*\wp^\phi\{\dim Y-\dim\g\}\{\dim Y\}=
\mu'_!(\pi')^*\wp^\phi\{2\dim Y-\dim\g\}.
\end{multline*}
The result follows using that  $f^\sharp_{\g^*/G}\ccirc\nu=\mu'$.
\end{proof}

\begin{rem}
There is a more general construction where ${B} G$ in $Y/G \to {B} G$ above is replaced by a more general stack. Suppose that we have a smooth map $q: Y \to W$ (`a family of smooth stacks'). Then there is a map $\TT^*(Y/W) \to q^*(\TT Y)^*[1]$ and the `dual' $q^*\Omega(\TT Y) \to \TT(Y/W)$ which specialize to $\mu$ and $a$ 
above for $Y/G \to {B} G$. However these are maps from generalized vector bundles to vector space stacks.
The functoriality with respect to such maps is not covered by 
Proposition \ref{FD}, though it could be stated separately.
\erem

\subsection{Proof of  theorem \ref{A-ind}}\label{fff} 

For each $\bv$ we have a  diagram, cf. \eqref{llie},
$$
\xymatrix{
\X_\bv(\cC) \ar@{=}[d]^<>(0.5){\Id}&& I\tpp(\X_\bv(\cC) )
\ar@{->>}[ll]_<>(0.5){p^+_\X}\ar[d]^<>(0.5){p_{\g}}\ar@{}[drr]|{\Box} 
&& I(\X_\bv(\cC) )\ar[d]^<>(0.5){p_{G}}\ar@{_{(}->}[ll]_<>(0.5){\wt\eps}\\
\X_\bv(\cC)&& \gl_\bv/\GL_\bv \ar@{->>}[ll]_<>(0.5){p_\X} && \GL_\bv/\aad\GL_\bv \ar@{_{(}->}[ll]_<>(0.5){\eps}
}
$$
where we have used the notation $p_G=p_{_{\GL_\bv}}$, resp. $p_\g=p_{\gl_\bv},\ p^+_\X= p^+_{_{\X_\bv(\cC,F)/\GL_\bv}}$,
and $p_\X=p_{_{\X_\bv(\cC,F)/\GL_\bv}}$.
Let $\phi:\X_\bv(\cC)\to \AA$ be an additive potential and $\rr_\bv:=\rf_\bv=(p_\g)_!(p^+_\X)^*\wp^{\phi}$.

We apply Proposition \ref{tr_prop}
in the special case where
$\cf=\C_{\GL_\bv/\aad \GL_\bv}\in D\abs(\GL_\bv/\aad \GL_\bv)$ and $\ce=\wp^{\phi}\in D\abs(X_\bv)$.
Clearly, we have
$\ltr\ce_x=\ltr(\wp^{\phi}|_x)=\psi(\phi(x))$ and 
$\ltr (\cf_{p_G(x,f)}=1$, for all
$[x]\in [X_\bv(\k)],\ f\in \Aut(x)(\k)$. 
Further, by  base change, from the above diagram we get
 $(p_{G})_!p^*_{\X}\ce=
(p_{G})_!\,\wt\eps^* (p^+_\X)^*\,\wp^{\phi}=\eps^*\rf_\bv$.
Thus,  formula \eqref{tr2} yields
\begin{align}
\ltr \RGam_c(\GL_\bv/\aad \GL_\bv,\ \eps^*\rr_\bv)\ &=\
\ltr \RGam_c(\GL_\bv/\aad \GL_\bv,\ 
(p_G)_!p_\X^*\ce)\nonumber\\
 &=\ \sum_{[x]\in [X_\bv(\k)]}\ \biggl(\,\,\sum_{^{f\in \Aut(x)(\k)}}\
 1\biggr)\cdot
 \mbox{\large{$\frac{1}{\#\Aut(x)}$}}\cdot
(\ltr\ce_x)\nonumber\\
&=\ \sum_{[x]\in [X_\bv(\k)]}\ (\ltr\ce_x)
\ =\ \sum_{[x]\in [X_\bv(\k)]}\ \psi(\phi(x)).
\label{tr3}
 \end{align}

Recall next that the $\GL_\bv$-action $\X_\bv(\cC,F)$ factors through a $\PGL_\bv$-action.
It follows that one has a natural isomorphism
$I\tpp(\X_\bv(\cC,F),\gl_\bv)=I\tpp(\X_\bv(\cC,F),\pgl_\bv)\,\times_{\pgl_\bv}\,\gl_\bv$.
Further, the group $\Aff$ acts naturally on $I\tpp(\X_\bv(\cC,F),\gl_\bv)$.
Write $I\tpp(\X_\bv(\cC,F),\gl_\bv)/\AA$ for the quotient
of $I\tpp(\X_\bv(\cC,F),\gl_\bv)$ by the additive group.
From the above isomorphism we deduce $I\tpp(\X_\bv(\cC,F),\gl_\bv)/\AA\cong I\tpp(\X_\bv(\cC,F),\pgl_\bv)$.
Let $p_\pgl: I\tpp(\X_\bv(\cC,F),\pgl_\bv)/\GL_\bv\to \pgl_\bv/\GL_\bv$, resp.
$\bar p_\X^+:\
I\tpp(\X_\bv(\cC,F),\pgl_\bv)/\GL_\bv\to \X_\bv(\cC)$, be the natural 
morphism. We deduce that the sheaf $\rr_\bv=(p_{\gl_\bv})_!(p_\X^+)^*\wp^{\phi}$, on $\gl_\bv/\GL_\bv$,
descends to a sheaf $\bar\rr_\bv$ on $\pgl_\bv/\GL_\bv$ and, moreover,
we have $\bar\rr_\bv=(p_\pgl)_!\bar p_\X^+\wp^{\phi}$.

Let $\pi': \TT^*\X_\bv(\cC,F)/\GL_\bv\to$ $ \X_\bv(\cC,F)/\GL_\bv=\X_\bv(\cC)$
be the projection and $\mu'_\bv: \TT^*\X_\bv(\cC,F)/\GL_\bv$ $\to \pgl_\bv^*/\GL_\bv$
the map induced by the moment map for the $\PGL_\bv$-action on
$T^*\X_\bv(\cC,F)$.
A slight modification of Proposition \ref{Zfour}  yields
$\FDN_{\pgl_\bv/\GL_\bv}((p_\pgl)_! \bar p_\X^*  \wp^{\phi})=\mu^{\prime}_!(\pi')^*\wp^\phi\{2\X_\bv(\cC,F)-\dim \pgl_\bv\}$.
Writing $d_\bv=\dim\X_\bv(\cC)$, we compute
\[2\dim \X_\bv(\cC,F)-\dim \pgl_\bv=2 d_\bv+2\dim \GL_\bv-\dim \pgl_\bv=
2 d_\bv+\bv\cdot\bv+1.
\]
By Theorem \ref{Rfact}, the collection $(\rr_\bv)$
is an $\Aff$-equivariant factorization sheaf on $\gl$.
 Thus, we obtain
\begin{align}
\Phi_{\pgl_\bv/\GL_\bv}(\rr_\bv)_\eta&=
\FDN_{\pgl_\bv/\GL_\bv}(\bar\rr_\bv)_\eta=\FDN_{\pgl_\bv/\GL_\bv}((p_\pgl)_!\bar p_\X^+\wp^{\phi})
\label{phi-form}\\
&=
(\mu^{\prime}_\bv)_!(\pi')^*\wp^\phi\{2\dim \X_\bv(\cC,F)-\dim \PGL_\bv\}
=
(\mu^{\prime}_\bv)_!(\pi')^*\wp^\phi\{2 d_\bv+\bv\cdot\bv+1\}.\nonumber
\end{align}

Finally, applying \eqref{3eqsG} and using \eqref{tr3}, we compute
\begin{align*}
\sum_\bv\ z^\bv &\cdot \Big(\sum_{[x]\in [X_\bv(\k)]}\ \psi(\phi(x))\Big)= \sum_\bv\ z^\bv\cdot 
[\RGam_c(\GL_\bv/\aad \GL_\bv,\,\eps^*_\bv\rr_\bv)]\\
&=\Sym\Big((\L^{-\frac{1}{2}}-\L^{\frac{1}{2}})\cdot\sum_{^{\bv>0}}\ (-1)^{|\bv|}\cdot z^\bv\cdot
\L^{\frac{|\bv|}{2}}\cdot
\big[\Phi_{\pgl_\bv/\GL_\bv}(\rr_\bv)^{\langle\sign\rangle}|_{\eta_\bv}\big]\Big)\\
&=\Sym\Big((\L^{-\frac{1}{2}}-\L^{\frac{1}{2}})\cdot\sum_{^{\bv>0}}\ (-1)^{|\bv|}\cdot z^\bv\cdot
\L^{\frac{|\bv|}{2}}\cdot\big[\mu^{\prime}_\bv)_!(\pi')^*\wp^\phi\big]^{\langle\op{triv}\rangle}|_\eta\{2 d_\bv+\bv\cdot\bv+1\}\Big)\\
&=\Sym\Big((\L^{-\frac{1}{2}}-\L^{\frac{1}{2}})\cdot\sum_{^{\bv>0}}\ (-1)^{|\bv|}\cdot z^\bv\cdot
\L^{\frac{|\bv|}{2}}\cdot(-1)^{2 d_\bv+\bv\cdot\bv+1}\cdot \L^{-d_\bv-\frac{\bv\cdot\bv}{2}-\frac{1}{2}}
\cdot\big[(\mu^{\prime}_\bv)_!(\pi')^*\wp^\phi\big]^{\langle\op{triv}\rangle}|_\eta\Big)\\
&=\Sym\Big((1-\L\inv)\cdot\sum_{^{\bv>0}}\  z^\bv \cdot \L^{-(d_\bv+\frac{\bv\cdot\bv-|\bv|}{2})}\cdot
\big[(\mu^{\prime}_\bv)_!(\pi')^*\wp^\phi\big]^{\langle\op{triv}\rangle}|_\eta\Big),
\end{align*}
where in the last equality we have used that $(-1)^{|\bv|}\cdot (-1)^{2 d_\bv+\bv\cdot\bv+1}=-1$
since $2d_\bv\bv+\cdot\bv+|\bv|$ is an even integer.

By definition, we have $\phi_\bv=(\pi')^*\wp^\phi$ and $\big[(\mu^{\prime}_\bv)_!(\pi')^*\wp^\phi\big]^{\langle\op{triv}\rangle}|_\eta=
H^\hdot(\mm_O, \ \wp^{\phi_\bv})$.
Thus, taking  the
trace of Frobenius in the RHS side of the last equation yields formula \eqref{thm11}.

To prove formula \eqref{thm1stack}, we consider the following cartesian diagram
\beq{in-f}
\xymatrix{
\X_\bv(\cC) \ar[d]^<>(0.5){p_0}\ar[rr]^<>(0.5){\wt i}\ar@{}[drr]|{\Box}  && I\tpp(\X_\bv(\cC))
\ar[d]^<>(0.5){p_{\g}}\\
0/\GL_\bv\ar@{^{(}->}[rr]^<>(0.5){i} && \gl_\bv/\GL_\bv 
}
\eeq
where we have used the notation $p_0=p_{\g}|_{\X_\bv(\cC)}$.
Applying base change for this diagram, we find
\begin{multline*}
\RGam_c(\X_\bv(\cC), \wp^\phi) =\RGam_c(\X_\bv(\cC),\, \wt i^*p_\X^*\wp^\phi)=
\RGam_c(\gl_\bv/\GL_\bv,\,(p_\g)_!
\RGam_c(0/\GL_\bv,\, (p_{0})_!\wt i^* p_\X^*\wp^\phi)\\
=
\RGam_c(\gl_\bv/\GL_\bv,\,i^*(p_\g)_!p_\X^*\wp^\phi)
\RGam_c(0/\GL_\bv,\, i^*\rr_\bv).
\end{multline*}
Thus, using  \eqref{phi-form} and \eqref{4eqsG} we compute
\begin{multline*}
\sum_{\bv}\   z^\bv\cdot \L^{-\frac{\bv\cdot\bv}{2}}\cdot [\RGam_c(X_\bv, \wp^\phi)]\
=\ 
\Sym\Big(\sum_{^{\bv>0}}\ z^\bv\cdot \L^{-\frac{|\bv|}{2}}\cdot[\Phi(\rr_\bv)|_{\eta}]^{\langle\op{triv}\rangle}\Big)\\
=\ 
\Sym\Big(\sum_{^{\bv>0}}\ z^\bv\cdot \L^{-\frac{|\bv|}{2}}\cdot(-1)^{2 d_\bv+\bv\cdot\bv+1}\cdot \L^{-d_\bv-\frac{\bv\cdot\bv}{2}-\frac{1}{2}}
\cdot\big[(\mu^{\prime}_\bv)_!(\pi')^*\wp^\phi\big]^{\langle\op{triv}\rangle}|_\eta\Big)\\
=\ 
\Sym\Big(-\L^{-\frac{1}{2}}\cdot\sum_{^{\bv>0}}\ 
(-1)^{|\bv|}\cdot z^\bv\cdot \L^{-(d_\bv+\frac{\bv\cdot\bv+|\bv|}{2})}\cdot\big[H^\hdot(\mm_O, \wp^{\phi_\bv})\big]^{\langle\op{triv}\rangle}\Big).
\end{multline*}

Equation \eqref{thm1stack} follows from the above by the Lefschetz trace formula
applied to the sheaf $\wp^\phi$ on the stack $\X_\bv(\cC)$.
This completes the proof of Theorem \ref{A-ind}.

\section{Parabolic  bundles on a curve}\label{higgs}
The goal of this section is to prove Theorem
\ref{ind-bundles}. 

\subsection{Approximation categories for vector bundles}
In this subsection we consider the case of vector bundles
without parabolic structure. 

Fix  a smooth geometrically connected curve $C$.
Recall that the slope of a vector bundle $\V$, on $C$, is
defined by the formula $\slope \V=\deg \V/\rk \V$.
One has the corresponding notion of  semistable, resp. stable, vector
bundle
in the sense of Mumford.
For any vector bundle $\V$ there is a unique
ascending filtration $0=\V_0\sset \V_1\sset \V_2\sset \ldots\sset \V_n=\V$,
by vector sub-bundles, called
the {\em Harder-Narasimhan filtration}, such that 
 each of the vector bundles
$\hn_j(\V):=\V_j/\V_{j-1},\ j=1,\ldots,n$,
is semistable and the following inequalities hold
\[
\slope \hn_1(\V)>\slope \hn_2(\V)>\ldots>
\slope \hn_n(\V).
\]

It is known that one has
\beq{hn}
\slope \hn_1(\V)\geq \slope \V\geq \slope \hn_n(\V).
\eeq
Also,  the Harder-Narasimhan filtration is functorial
in the sense that  any vector bundle morphism
$f: \V\to \V'$  maps $\V_i$ to $\V'_j$ whenever $\slope \V_i\leq
\slope \V'_j$, cf. e.g. \cite{An}.

Recall that $\VV(U)$ denotes the category of
vector bundles on a scheme $U$, cf. Example \ref{bun}.
For any real number $\gam>0$, we let $\VV_\gam(C)$ 
be  a  full subcategory of $\VV(C)$
whose objects are vector bundles $\V$ such that
one has $-\gam\leq\slope \hn_j(\V) \geq \gam$ for all $j$.
The functoriality of the Harder-Narasimhan filtration 
insures that $\VV_\gam$ is stable under finite
direct sums and direct summands,
so it is a $\k$-linear Karoubian category.
Moreover, the  category $\VV_\gam(C)$ is  quasi-abelian, \cite{An}.
Clearly, for any $\gam<\delta$ there is  a full imbedding $\VV_\gam(C)\into\VV_\delta(C)$,
so one has $\underset{^\gam}\varinjlim\ \VV_\gam(C)=\VV(C)$.

\begin{lem}\label{prestack} Let $S$ be  an affine connected scheme
and  $\V$ a  vector bundle on $C\times S$.
Then, the set formed by the (closed)  points $s\in S$
such that $\ \V|_{C\times\{s\}}\in\VV_\gam(C)\ $ 
is a Zariski open subset of $S$.
\end{lem}
\begin{proof} 
By a result of Shatz \cite{Sh}, the scheme $S$ has a canonical
 partition 
$S=\sqcup_P\ S_P$ parametrized by convex polygons $P$ in the plane
${\mathbb R}^2$, called Harder-Narasimhan polygons.
For each $P$,  the set $S_P$ is a locally-closed subscheme
of $S$, see  \cite[p. 183]{Sh},
all vector bundles $\V_s:=\V|_{C\times\{s\}},\ s\in S_P$, have the  Harder-Narasimhan
filtration 
of the same length  $n(P)$, the number of vertices of
$P$ and, moreover,
the coordinates of the $j$-th vertex are $(\rk \hn_j(\V_s),\
\deg  \hn_j(\V_s)$.

Proving the lemma amounts to showing that 
the set
\[S_\gam=\big\{s\in S\en\big|\en \exists j \en\text{such that}\en
\slope \hn_j(\V_s) \notin [-\gam,\gam]\big\}\]
 is a closed subset of
$S$. It is sufficient to show that, for any 
 Harder-Narasimhan polygon $P$, we have $\overline{S_\gam\cap S_P}\sset
 S_\gam$.

 Let $s\in S_\gam\cap S_P,\
s'\in \overline{S_\gam\cap S_P}$, and let
$Q$ be the polygon such that $\V_{s'}\in S_Q$ where one
may have $P=Q$, in general. The polygons
$P$ and $Q$ share the same  vertex $(0,0)$, that corresponds
to the $0$-th term of the  Harder-Narasimhan
filtration, and the same  vertex $(\rk \V_s,\deg \V_s)=(\rk \V_{s'},\deg
\V_{s'})$, that corresponds
to the last term  of the  Harder-Narasimhan
filtration. 
By the semi-continuity theorem \cite[Theorem 3]{Sh},
the polygon $Q$ is  on or above the polygon $P$.
Further, since  $s\in S_\gam$ we have that either
$\slope\hn_1(\V_s)> \gam$ or $\slope\hn_{n(P)}(\V_s)  < -\gam$.
Therefore, in the first case one has
$\slope\hn_1(\V_{s'})\geq\slope\hn_1(\V_s)> \gam$
and in the second case one has
$\slope\hn_{n(Q)}(\V_{s'})\leq\slope\hn_{n(P)}(\V_s)  < -\gam$.
Thus, in either case, we deduce that $s'\in S_\gam$,
and the lemma follows.
\end{proof}

Fix $\gam>0$.
Given an affine test scheme $S$, let $\VV_\gam(C\times S)$
be  a  full subcategory of $\VV(C\times S)$
whose objects are vector bundles $\V$ on $C\times S$  such that
$\V|_{C\times\{s\}}\in\VV_\gam(C)$ for any closed  point $s\in S$.
The assignment $S\mto \VV_\gam(C\times S)$ defines a sheaf subcategory,
$\VV_{C,\gam}$, of the sheaf category $\VV_C$, cf. Example \ref{bun}.
The corresponding moduli
stack  $\fX(\VV_\gam(C))$ is a substack of the stack
$\fX(\VV(C))$. Specifically,
we have $\fX(\VV_\gam(C))=\sqcup_r\ \buna$, where
$\buna$ is the stack parametrizing the objects of  category $\VV_\gam(C)$ of rank $r$. 
Lemma \ref{prestack} implies that $\buna$ is 
 an open substack of the stack $\bun_r(C)$.
Hence,  $\buna$  is a smooth Artin stack. 

Given $r$ and $\gam$,
let ${\mathcal P}_{r,\gam}$ be the set  of
 Harder-Narasimhan polygons $P$ with vertices contained
in 
the set  ${\mathcal P}_{r,\gam}=\{(r',d')\in \Z^2\mid 0\leq r'\leq r,\
-\gam\cdot r'\leq d'\leq \gam\cdot
r'\}$. The latter is finite. Hence, ${\mathcal P}_{r,\gam}$ is a finite set.
 The result of Shatz \cite[Proposition 11]{Sh}
yields
a decomposition
$\buna=\sqcup_{P\in {\mathcal P}_{r,\gam}}\  \bun_P$,
where  $\bun_P$
is a locally closed substack of $\bun_r(C)$ that parametrizes vector bundles with  Harder-Narasimhan polygon $P$. 
Furthermore, each stack $\bun_P$ has finite type.
We deduce that for any fixed $r$ and $\gam$
the  stack $\buna$ has finite type.
We obtain the following result

\begin{cor}\label{cor-a} For any $\gam>0$ and $r=0,1,\ldots$,
the stack $\fX_r(\VV_\gam(C))$ is an open 
substack of $\fX_r(\VV(C))$ of finite type
and we have $\fX_r(\VV(C))=\underset{^\gam}\varinjlim\ \fX_r(\VV_\gam(C))$.
\end{cor}

Our next result shows that the category $\VV_{\gam}(C)$
`captures' all indecomposable vector bundles provided
$\gam$ is large enough.

\begin{lem}\label{indec-slope} Let  $\V$ be an indecomposable vector
  bundle of rank $r\geq1$ and degree $d$. 
Put 
\[\gam_0=\frac{|d|}{r}+2(g-1)(r-1)\en\text{\em if}\en\textit{genus}(C)\geq
1,\quad
\text{\em resp.}\quad
\gam_0=\frac{|d|}{r}\en\text{\em if}\en\textit{genus}(C)=0.\]
 Then,
 $\V$ is an object of the category $\VV_\gam(C)$
 for any
$\gam> \gam_0$.
\end{lem}

\begin{proof} For $C={\mathbb P}^1$ any  indecomposable vector
  bundle has rank 1 and the result is clear.

Assume next that $\textit{genus}(C)\geq 1$ and
write  $0=\V_0\sset\ldots\sset \V_n=\V$ for the 
Harder-Narasimhan filtration.
By \cite[Proposition 8]{Sh}, for
an  indecomposable vector
bundle, we have $\slope \hn_1(\V)-
\slope \hn_n(\V)\leq (n-1)(2g-2)$.
Hence, using  \eqref{hn}, for any $j=1,\ldots,n$, we deduce
\[|\slope \hn_j(\V)-\slope \V|
\leq \slope\hn_1(\V)-
\slope\hn_n(\V)\leq (n-1)(2g-2)\leq 2(r-1)(g-1).
\]

The result follows.
\end{proof}

\subsection{Approximation categories for parabolic bundles}
\label{app} 
Below, we will freely use the notation of  section ~\ref{par-sec}. In particular, we have a collection $c_i\in C(\k),
\ i\in I$, of marked points and put $D=\sum_\ii\ c_i$.  
Following the ideas of \cite{Se},\cite{Fa},\cite{Bi}, \cite{KMM}
we are going to use the relation between parabolic bundles
and orbifold bundles on a ramified covering of the curve $C$.

In more detail, fix $r\geq 1$, the rank of the parabolic bundle.
Then, it follows from the proof of Theorem 1-1-1 in \cite{KMM}
that one can choose an integer
$N$, a simple effective divisor $D'$,  and a  ramified Galois  covering 
 $ p: \cc\to C$ 
such that the following holds:

\begin{enumerate}
\item  We have  $N>r$ and  $(N, \op{char}\k)=1$.

\item The geometric Galois group $\Gamma$, of the covering, is  a cyclic group of order $N$.


\item The divisor $\DD:=D+D'$  is the  ramification divisor of the covering.

\item The divisor $\DD$ 
is simple and it is linearly equivalent over $\k$ to $N\cdot H$, where $H$ is a very ample divisor. 
The covering is totally ramified at every point of the support of $\DD$.
\end{enumerate}

The covering is constructed as the scheme of $N$-th roots of the canonical section of $\mathcal O(\DD)$ in the total space of a line 
bundle $L$ corresponding to the divisor $H$ above, as in the proof of Theorem 1-1-1 in \cite{KMM}.




We write  $\DD=\sum_{i\in  {\mathsf I}}\ c_i$, where ${\mathsf I}\supseteq I$  is a labeling set for the points in the support of $\DD$.
For $i\in {\mathsf I}$, let $z_i$ be the unique point of $\cc$ over $c_i$.
Thus, $z_i$ is a $\Ga$-fixed point
and $\Ga$ acts on $\TT_{z_i}\cc$, the tangent space at $z_i$,
via a multiplicative character $\eps_i$.
Since $\Ga\cong {\mathbb Z}/(N)$, we have $(\eps_i)^N=1$. Furthermore, 
the set $\{1,\eps_i,(\eps_i)^2,\ldots,(\eps_i)^{N-1}\}$
is a complete set of  multiplicative characters of
the group $\Ga$ and  we have
$\eps_i=\eps_{i'}$ for all $i,i'\in \BI$.
Let $\VV^\Ga(\cc)$ be the $\k$-linear category
of $\Ga$-equivariant vector bundles on $\cc$.
For any  $\ce\in \VV^\Ga(\cc)$, the group $\Ga$
acts naturally on the fiber of $\ce$ at ${z_i}$ and there is a weight
decomposition $\ce|_{z_i}=\oplus_{0\leq k\leq N-1}\ \ce_i^{(k)}$,
where $\ce_i^{(k)}$ is the weight space
 of weight $(\eps_i)^k$.

Fix $\bm=(m_i)_{i \in \BI},\ 1\leq m_i \leq r$, and let
 $\V\in \VV(C,\DD,\bm)$ be a parabolic bundle.
Recall that  we write $V_i$ for the fiber of $\V$, resp. $V_i^\hdot$
for the partial flag in that fiber,  at the point $c_i,\ i\in\BI$.
We view  $\V$ as a locally free sheaf
and, for each $(i,j)$, let
$\V_i^{(j)}$ be a torsion free subsheaf 
formed by the sections $v$ of $\V$ such
that $v|_{c_i}\in V_i^{(j)}$. Thus, there is a short exact sequence
\beq{coh-flag}
0\to \V_i^{(j)}\to \V \to (V/V_i^{(j)})\o \k_{c_i}\to 0,
\eeq
where $\k_{c_i}=\oo_C/\oo_C(-c_i)$ is a skyscraper sheaf at $c_i$.
We have
 $\V_i^{(j)}/\V_i^{(j+1)}=(V_i^{(j)}/V_i^{(j+1)})\o \k_{c_i}$.

Let $\wt \DD=\sum_{i \in {\wt\BI}}\ z_i$ denote the simple
divisor in $\cc$ given by the collection of distinct ramification points
$z_i=p\inv(c_i)$.
We define a subsheaf  $\P(\V)$ of the sheaf $\oo_{\cc}(N\cdot \wt \DD)\o p^*\V$,
by the formula
\[ \P(\V):=\sum_{i \in {\wt\BI}}\ \sum_{j=0}^{m_i}\ \oo_{\cc}(j\cdot z_i)\o
p^*\V_i^{(j)}.
\]
The sheaf $\P(\V)$ comes equipped with a natural $\Ga$-equivariant
structure and it is torsion free, hence, locally free.
Thus, we have $ \P(\V)\in \VV^\Ga(\cc)$.

Let $\VGpp$ be  a full subcategory
of  $\VV^\Ga(\cc)$ whose objects
are the $\Ga$-equivariant vector bundles $\ce$, on $\cc$, such that,
for every $i \in {\wt\BI}$ and $m_i\leq k\leq N-1$, one has
$\ce_i^{(k)}=0$. Thus, we have
$\ce|_{z_i}=\oplus_{0\leq k\leq m_i-1}\ \ce_i^{(k)}$.
Observe that for any $\V\in \VV(C,\DD,\bm)$ and  $i \in {\wt\BI}$  it follows from the definition 
of 
the sheaf  $\P(\V)$  that for any $i \in {\wt\BI}$
we have canonical isomorphisms
\[\P(\V)|_{z_i}\ccong \bigoplus_j\ \oo(j\cd z_i)/\oo((j+1)\cd z_i)\,\bo\,
(\V_i^{(j)}/\V_i^{(j+1)})
\ccong \bigoplus_j\ (\TT_{z_i}\cc)^{\o j}\,\bo\, (V_i^{(j)}/V_i^{(j+1)})\o \k_{c_i}.
\]
We deduce that $\P(\V)$ is an object
of the category $\VGpp$. 

The idea of the following proposition goes back to
Seshadri \cite{Se}; the statement below is a slight strengthening
 of  \cite[\S5]{Fa}.

\begin{prop}\label{biswas} The functor $\ \dis\P:\ \VV(C,\DD,\bm)\to \VGpp,\
\V\mto \P(\V),\ $ is an equivalence.
\end{prop}
\begin{proof}
An inverse of $\P$ is constructed as follows.
Observe that for any $\Ga$-equivariant
coherent  sheaf $\ce$, on $\cc$, the sheaf
 $(p_*\ce)^\Ga$, of
$\Ga$-invariants,  is a direct summand
of $p_*\ce$, since the order of $\Ga$ is prime
to $\chark\k$. If $\ce$ is
locally
free, then  the sheaf
$p_*\ce$ is also locally
free and hence so is $(p_*\ce)^\Ga$.
Further,
$p_*[\ce(-j\cdot z_i)]$ is
a (not necessarily $\Ga$-stable) subsheaf of
$p_*\ce$. The sequence
 $p_*[\ce(-j\cdot z_i)]\,\cap\, (p_*\ce)^\Ga,\ j=0,1,\ldots$,
is  a flag of locally free subsheaves of  $(p_*\ce)^\Ga$
which induces a flag of subspaces
$(p_*\ce)^\Ga|_{c_i}=V^{(0)}_i\supseteq V^{(1)}_i\supseteq\ldots$,
in the fiber
of $(p_*\ce)^\Ga$ at $c_i$. Furthermore,
the definition of category $\VGpp$ insures that,
for any $\ce\in \VGpp$ and $i\in \BI$, one has 
$V^{(m_i)}_i=0$.
Moreover, it is not difficult to show that the
resulting functor $\VGpp\to \VV(C,\DD,\bm),\
\ce\mto (p_*\ce)^\Ga$ is an inverse of the
functor $\P$, cf. eg. \cite[section 2(c)]{Bi}.
\end{proof}

Recall that the destabilizing subsheaf
$\hn_1(\ce)$ of a coherent sheaf $\ce$ is defined
as the sum of all subsheaves of $\ce$ of the
maximal slope. It follows that if the sheaf $\ce$ is equipped
with an equivariant structure
then each term, $\hn_j(\ce)$, of the Harder-Narasimhan
filtration of $\ce$ is an equivariant subsheaf of $\ce$.
Therefore, for any real $\gam>0$ there is a well
defined $\k$-linear category
$\vgampp=\VGpp\cap\VV_\gam(\cc)$, a full subcategory
of $\VGpp$  whose objects are $\Ga$-equivariant 
vector bundles $\ce$ on $\cc$ such that  for all $k$
one has
$-\gam \leq \slope\hn_k \ce\leq\gam$.
Further, we use the functor $\P$ and let
$\VV_\gam(C,\DD,\bm)$ be a full subcategory of
$\VV(C,\DD,\bm)$ whose objects are the
parabolic bundles $\V$, on $C$, such that
$\P(\V)\in \vgampp$.

We claim that an analogue of 
Lemma \ref{indec-slope}
holds in the parabolic setting, specifically,
we have
\begin{lem}\label{par-ind} 
For any parabolic bundle $\V$
which is an
indecomposable  object of category $\VV(C,\DD,\bm)$,
 we have $\V\in \VV_\gam(C,\DD,\bm)$, for any $\gam>\gam_0$
where $\gam_0$ is the same constant as
in 
Lemma \ref{indec-slope}.
\end{lem}

\begin{proof} The functor $\P$ being an equivalence,
it suffices to prove the corresponding result for
the  category $\VGpp$.  Further, since
$\VGpp$ is a full Karoubian  subcategory
of the category $\VV^\Ga(\cc)$, 
 of all $\Ga$-equivariant vector bundles,
 it
is sufficient to establish an  analogue of 
Lemma \ref{indec-slope} in the case of category
 $\VV^\Ga(\cc)$.

It is clear from the proof that
Lemma \ref{indec-slope}  is, essentially, a consequence of
\cite[Proposition 8]{Sh}.  That proposition
is deduced in \cite{Sh} from   the vanishing of
certain $\Ext^1$-groups. Let $\Coh_\Ga(\cc)$, resp. $\Coh(\cc)$,
be the abelian category of  $\Ga$-equivariant, resp. non-equivariant,
coherent sheaves on $\cc$. Then, for any
$\Ga$-equivariant coherent sheaves $\ce,\ce'$ on $\cc$,
we have
\[\Ext^1_{\Coh_\Ga(\cc)}(\ce,\ce')\ =\
\big(\Ext^1_{\Coh(\cc)}(\ce,\ce')\big)^\Ga.\]
This isomorphism  follows from the exactness of
 the functor $(-)^\Ga$,
of $\Ga$-invariants, which holds since
  the order of the group $\Ga$ is
prime to $\chark\k$.
We conclude that the vanishing of  $\Ext^1_{\Coh(\cc)}(\ce,\ce')$
implies the vanishing of $\Ext^1_{\Coh_\Ga(\cc)}(\ce,\ce')$.
This insures that all  the arguments in the proof of  \cite[Proposition 8]{Sh}
go through in our $\Ga$-equivariant setting.
\end{proof}


 One has  a decomposition
$\X(\VV(C,\DD,\bm))=\sqcup_{\br}\ \ppar_{\br}$, where
 is $\ppar_{\br}$ is the stack parametrizing
parabolic bundles of type $\br$. For any  $\gam>0$, there is a similar decomposition
 $\X(\VV_\gam(C,\DD,\bm))
=\sqcup_{\br}\ \X_{\br}(\VV_\gam(C,\DD,\bm))$.

From now on, we let $\br$ be a type such that $|\br|=r$, where the $r$ is the integer
involved in the choice  of the covering $\wt C\to C$.
Then, we have a parabolic analogue of Corollary \ref{cor-a}:

\begin{lem}\label{par-a} For any $\br$ and $\gam$,
the stack $\X_{\br}(\VV_\gam(C,\DD,\bm))$ 
is an open  substack of $\ppar_{\br}$ of finite type
and we have
$\ppar_{\br}
=\underset{^\gam}\varinjlim\ \X^{^{}}_{\br}(\VV_\gam(C,\DD,\bm))$.
\end{lem}
\begin{proof} By standard results, it is known that
$\X_{\br,d}(\VGpp)$, the stack
parametrizing objects of category $\VGpp$ of type $\br$ and degree $d$,
 is an Artin stack of finite type.
We have a decomposition
$\X_{\br}(\VGpp)=\sqcup_d\ \X_{\br,d}(\VGpp)$.
We observe that the
 proof of Lemma \ref{prestack} as well as the argument leading to
Corollary \ref{cor-a}
extend to the $\Ga$-equivariant setting word for word. 
We deduce that $\X_{\br}(\vgampp)$ is an
 open substack of $\X_{\br}(\VGpp)$ of finite type
and $\X_\br(\VGpp)
=\underset{^\gam}\varinjlim\  \X_{\br}(\vgampp)$.
Transporting these observations via the equivalence $\P$,
we conclude that  $\X_{\br}(\VV_\gam(C,\DD,\bm))$  is an
 open substack of $\X_{\br}(\VV(C,\DD,\bm))$ of finite type
and $\X_{\br}(\VV(C,\DD,\bm))=\underset{^\gam}\varinjlim\  \X_{\br}(\VV_\gam(C,\DD,\bm))$.

It remains to show that the stack  $\X_{\br}(\VV(C,\DD,\bm))$ is smooth.
To see this, we use
a natural morphism of stacks $\ppar_{\br}\to
\bun_{r}$
that sends  parabolic bundles to the underlying vector bundles
and forgets the partial flag data.
It is clear that this is a smooth morphism and its 
fiber is isomorphic to a product of partial flag varieties.
The stack $\bun_r$ being smooth, it follows that so is $\ppar_{\br}$.
$\ppar_{\br}$.
\end{proof}

Recall that $\DD=\sum_{i\in {\mathsf I}}\ c_i$.
Given an $I$-tuple ${\mathbf n}=(n_i)$, where $n_i\geq 1$,
define an ${\mathsf I}$-tuple $\bm=(m_i)$ by letting $m_i=n_i$ for all $\ii$, resp. $m_i=1$ for all $i\in {\mathsf I}\sminus I$.
Let ${\mathsf{\Xi}}(\bm)=\{(i,j)\mid i\in\BI,\, j\in[1,m_i]\}$.
Note that parabolic bundles of  any type in ${\mathsf{\Xi}}(\bm)$ 
have a trivial flag data in the fibers over the point $c_i$ for each $i\in {\mathsf I}\sminus I$.
Therefore, adding trivial flags at these points  yields  a fully faithful functor $\rho:\
\VV(C,D,{\mathbf n})\to\VV(C,\DD,\bm)$.
For any $\gam>0$, let $\textsf{Vect}_\gam(C,D,\bm)$ be a full subcategory of  category $\VV(C,D,\bm)$
whose objects are parabolic bundles $\V$ such that $\rho(\V)$ is 
an object of the subcategory $\VV_\gam(C,\DD,\bm)$.

\begin{prop}\label{par-appr} Let $\br\in I^{\mathbf n}$ be a type such that $|\br|=r$. Then, we have

\vi  For any               $\gam>0$
the stack $\X_{\br}(\textsf{Vect}_\gam(C,D,{\mathbf n}))$ 
is a smooth open substack of $\ppar_{\br}$ of finite type. Furthermore,
we have
$\ppar_{\br}
=\underset{^\gam}\varinjlim\ \X^{^{}}_{\br}(\textsf{Vect}_\gam(C,D,{\mathbf n}))$.

\vii There exists $\gam_0>0$ such that  the natural map
$\op{AI}_\br(\textsf{Vect}_\gam(C,D,{\mathbf n}))\to \op{AI}_\br(\VV(C,D,{\mathbf n}))$ is a bijection
for all $\gam\geq \gam_0$.
\end{prop}
\begin{proof} Part (i) is a consequence of Lemmas \ref{par-ind}-\ref{par-a}. Define
 a type $\wt\br\in {\mathsf{\Xi}}(\bm)$ by $\wt r\ij=r\ij$ if $\ii$, resp. $\wt r\ij=0$ if  $i \in {\mathsf I}\sminus I,\,j>1$. 
It is immediate from definitions that the map $\fx_{\br}(\VV(C,D,{\mathbf n}))\to \fx_{\wt\br}(\VV(C,\DD,\bm))$,
induced by the functor $\rho$,
is an isomorphism of stacks and, moreover, this maps
restricts to an isomorphism $\fx_{\br}(\textsf{Vect}_\gam(C,D,{\mathbf n}))\iso \fx_{\wt\br}(\VV_\gam(C,\DD,\bm))$.
The required statements now follow from the corresponding statements  of Lemmas \ref{par-ind}-\ref{par-a}
via the functor $\rho$.
\end{proof}

\subsection{Reminder on the Hitchin base}
\label{hit-sec}
Let $C$ be a smooth projective curve,
and $D=\sum_\ii c_i$, as above. 
Let
 $\cl:=\Omega_C^1(D)$, an invertible sheaf on $C$.
Write $L$ for the total space of the corresponding line bundle,
$p_C: L \to C$ for the projection,
and put $L_i:=p_C\inv(c_i)=\cl|_{c_i}$,
a line in $L$.
The sheaf $p_C^*\cl$ has a  canonical section  $\lambda$
whose value at any element $\ell\in L$ equals $\ell$.

The residue map $\ri:\ \Gamma(C, \Omega_C^1(D))\to \AA$
at the point $c_i$
gives,
for any $j=1,\ldots,r$, a residue map
$\ri^{\o j}:\ \Gamma(C, \cl^{\o j})\to \AA$;
in particular, it induces a canonical isomorphism 
of the fiber $L_i\iso\AA$.

Fix an integer $r\geq 1$. The  {\em Hitchin base}, $\hit$,
is a vector space defined as follows
\[\hit=\bplus_{j=1}^r \Gamma(C,\ \Omega^{\o j}_C (j\cd D))=
\bplus_{j=1}^r \Gamma(C,\ \cl^{\o j}).
\]

Associated with an
element $\varkappa=(\varkappa^{(j)})_{j\in[1,r]}\in \hit$
there is a spectral curve $\Sigma_\varkappa$,
a closed subscheme of $L$ defined as the zero locus of the 
section
$\la^{\o r}+\sum_{j=1}^r \
p_C^*\varkappa^{(j)}\o \la^{\o (r-j)}\in\Gamma(L, \cl^{\o r})$.
The composition $\Sigma_\vk\into L\onto C$ is a finite morphism
of degree $r$. Thus, $\Sigma_\vk\cap L_i$,
the fiber of this morphism over $c_i$, is a scheme of length $r$.
Explicitly, it may be identified, via the canonical isomorphism
 $L_i\cong \AA$, with a
subscheme of $\AA$ defined by the
equation $z^r+\sum_{j=1}^r \
\ri^{\o j}(\varkappa^{(j)})\cdot z^{r-j}=0$.
Let $|\Sigma_\vk\cap L_i|\in \AA^r\dsl \si_r$ be the unorderd $r$-tuple of roots of this
polynomial counted with multiplicities, i.e., an effective divisor in
$\AA$ of degree $r$.
We  define a morphism
\[\reshit:\ \hit\too (\AA^r\dsl \si_r)^I=\mbox{$\prod_\ii$}\ \AA^r\dsl \si_r,\quad 
\varkappa\mto \mbox{$\prod_\ii$}\ 
|\Sigma_\vk\cap L_i|.
\]

Let $(\AA^r)^I_0$ be a codimension one hyperplane of $(\AA^r)^I$ formed
by the points with the vanishing total sum of all coordinates,
i.e. by the tuples $(z^{(j)}_i)_{\ii, j\in[1,r]}$ such that
$\sum_{i,j}\ z^{(j)}_i=0$. 
Let  $\hitirr\sset \hit$ be the locus of points $\varkappa\in\hit$
such that the spectral curve $\Si_\varkappa$ is reduced and
irreducible.
The following result is certainly known to the experts but we have
been unable to find a convenient reference in the literature.

\begin{lem}\label{reduced} \vi For any $\varkappa\in\hit$ one has
 $\en\reshit(\varkappa)\in (\AA^r)^I_0\dsl \si_r $.\vskip 2pt

\vii 
Assume that  $\varkappa\in\hit$ is such that the following two conditions hold:
\begin{align}
&\bullet\en \text{For each $\ii$, the divisor
$\ |\Si_\vk\cap L_i|\ $ is a sum of $r$ pairwise 
distinct points.} \label{b1} \\ 
&
\begin{array}{l}\bullet\en\text{For any  $0< r' <r$ and any collection   of subsets 
$\dis {Z_i}_{}\sset \Si_\vk\cap L_i,\ \ii$,}\\ 
\quad\text{such that $\# Z_i={r'}^{^{}}$ for all $i$, one has  $\ {\sum_\ii\
    \sum_{z\in Z_i}\ z \ \neq\ 0}^{^{}}$.}
\end{array} \label{b2} 
\end{align}

Then, 
the spectral curve $\Sigma_\vk$ is reduced and irreducible,
i.e., we have
$\vk\in\hitirr$.
\end{lem}

\begin{proof} For any $\ii$, the sum of roots
of the equation  $z^r+\sum_{j=1}^r \
\ri^{\o j}(\varkappa^{(j)})\cdot z^{r-j}=0$
equals $\ri(\vk^{(1)})$. Hence, proving part (i) of the lemma
amounts to showing that one has $\sum_\ii\ \ri(\vk^{(1)})=0$.
The latter equation holds thanks to the Residue Theorem
applied to $\vk^{(1)}\in\Gamma(C, \cl)$,  a rational 1-form on the curve $C$.

To prove (ii), we 
use the canonical section $\la$ of $p_C^*\cl$.
Specifically, let $C^\circ=C\sminus (\cup_i\ c_i)$ and
$L^\circ=p_C\inv(C^\circ)$.
By definition, one has a canonical isomorphism
$\cl|_{C^\circ}=\Omega_C^1(D)|_{C^\circ}\cong \Omega_{C^\circ}^1$.
Hence, $\la|_{L^\circ}$,
the restriction  of   $\lambda$ to the open set $L^\circ$, may be
identified with a 1-form on $L^\circ$.
This  1-form restricts further
to give   a 1-form $\lambda^\circ_\Si$ on $\Si^\circ_\vk:=\Si_\vk\cap L^\circ$,
a curve in $L^\circ$ (the form $\lambda^\circ_\Si$  may have poles at the points
of the finite set $\Si_\vk\sminus \Si_\vk^\circ=\sqcup_\ii\ \Si_\vk\cap L_i$).

Assume now that conditions
\eqref{b1}-\eqref{b2} hold. Then, one shows that 
$\lambda^\circ_\Si$ has a simple pole
at  any point $z\in \Si_\vk\cap L_i\sset \AA$
and, moreover,  for the corresponding residue 
we have $\textit{Residue}_z (\la^\circ_\Si) =z$.
Therefore, 
we deduce
\beq{sum-eta}
\sum_\ii\ \sum_{z\in \Si_\vk\cap L_i}\ z
=\sum_\ii\ \sum_{z\in \Si_\vk\cap L_i}\ 
\textit{Residue}_z (\la^\circ_\Si) =
\sum_{z\in \Si_\vk\sminus \Si^\circ}\textit{Residue}_z (\la^\circ_\Si) = 0,
\eeq
by the Residue Theorem.
More generally, a similar calculation implies that
for any irreducible component $\Si'$ of the
curve $\Si_\vk$ one  must have $\sum_\ii\ \sum_{z\in \Si'\cap L_i}\
z=0$.
Let $r'$ be the degree of the morphism $\Si'\to C$.
If $\Si'\neq \Si_\vk$ then for
every $\ii$ we have $\#(\Si'\cap L_i)=r'<r$, contradicting
condition \eqref{b2}.  It follows that the spectral curve can not
be 
reducible.

To complete the proof observe that condition \eqref{b1} clearly implies
that  the scheme $\Sigma_\vk$ is reduced
at any point of $\Sigma_\vk\cap L_i$.
Hence,  being  irreducible, $\Sigma_\vk$ is generically reduced.
Further, the projection $\Sigma_\vk\to C$ is a finite, hence a flat, morphism.
It follows that  $\Sigma_\vk$ is reduced.
\end{proof}

\subsection{Parabolic Higgs bundles}\label{higgs-sec}
Fix $\bm=(m_i)$.  Following  Example \ref{par-ex}, we view  $\VV(C,D,\bm)$ as a sheaf category
equipped with the sheaf  functor   $F:\ \VV(C,D,\bm)\to \VV^{\bbm}$.
Thus, given a type $\br\in \bbm$, we have  the framed stack 
$Y_\br=\fX_\br(\VV(C,D,\bm),\,F)$ of parabolic bundles of type $\br$.
The stack $\TT^*Y_\br$ classifies triples $(\V,b,u)$, where $(\V,u)$ is a parabolic Higgs 
bundle and $b$ is a framing on $\V$.  
The assignment $(\V,b,u)\mto (\gr\textit{res}_i(u))_\ii$ gives a moment
map $\mu_\br:\TT^* Y_\br\to \g^*_\br$, see \S\ref{par-sec}. We remark that $\g_\br^*$ is a codimension 1 hyperplane
in $\gl_\br^*\cong \gl_\br \sset(\gl_r)^I$ cut out by the equation 
$\sum_\ii\Tr g_i=0$. The fact that the image of $\mu_\br$ is contained in this hyperplane 
may be seen as a consequence of the Residue theorem $\sum_\ii\ \textit{res}_i(\Tr u)=0$,
where $\Tr u$ is a section of $\Omega^1_C(D)$. 
The  morphism   $f_{\op{bas}}:\ \TT^* Y_\br\to \higgs,\ (\V,b,u)\mto (\V,u)$ is a $G_\br$-torsor.
Hence, we have  $(\TT^* Y_\br)/G_\br=\higgs$ and the map $\mu_\br$ descends to a well defined morphism
$\reshiggs: \higgs\to \g_\br^*/G_\br$, of quotient stacks.

We briefly recall the construction of the {\em Hitchin map} in the case of 
 Higgs bundles without  parabolic structure.
Let  $\hig_r$ be the stack that parametrizes pairs 
 $(\V,u)$,  where  $\V$ is  a rank $r$ vector bundle
on $C$  and
 $u$
is a morphism $\V \to \V\o\Omega^1_C(D)$.
Given such a pair
 $(\V,u)$, one considers a vector
bundle $p_C^*\V$ on $L$ and  a morphism
$p_C^*u:\ p_C^*\V\to p_C^*\V\o\cl$, where we have used the notation of \S\ref{hit-sec}.
The corresponding `characteristic polynomial'
$\det(\lambda-p_C^*u)$ is a morphism $\wedge^r p_C^*\V\to
  \wedge^r p_C^*\V\o\cl^{\o r}$, that is,
an element
of $\Hom(\wedge^r p_C^*\V,\
\wedge^r p_C^*\V\o\cl^{\o r})=\Gamma(C, \cl^{\o r})$.
This element has an expansion of the form
$\det(\lambda-p_C^*u)=\sum_{j=0}^r \
\kap^{(j)}(u)\o \la^{\o (r-j)}$, where  $\kap^{(j)}(u)$ is a section of
$\cl^{\o
  j}$; in particular, we have $\kap^{(0)}(u)=1$. 
Then, the Hitchin map is a map
$\kap_r: \hig_r\to\hit$
defined by the assignment
$(\V,u)\mto \oplus_{j=1}^r\ \kap^{(j)}(u)$.

Each of the stacks $\higgs,\,Y_\br$, and  $\TT^* Y_\br$,
can be partitioned further by fixing a
degree $d$ of the underlying vector bundle, so
one has $\hig_r=\sqcup_{d\in \Z}\ \hig_{r,d}$, etc.
Put $\hig_{r,d}\cap \kap_r\inv(\hitirr):=\hig_{r,d}\times_\hit
\hitirr(D)$.
Below, we will use the following known result
that amounts, essentially, to a statement about compactified Jacobians
proved in  \cite{BNR}.
\begin{prop}\label{higgs-fin} For any degree $d$,
 the stack $\hig_{r,d}\cap \kap_r\inv(\hitirr)$
has  finite type.\qed
\end{prop}

The  standard Cartan subalgebra  of the Lie algebra $\gl_\br$ of block diagonal matrices
 is isomorphic to $\oplus_{(i,j)\in \bbm}\ \AA^{r\ij}=(\AA^r)^I$. 
Let $\ft_\br$ be the Cartan subalgebra of the Lie algebra $\g_\br$.
Thus, one has a natural identification  $\ft^*_\br\cong (\AA^r)^I_0$ and the Chevalley isomorphism
$\g_\br^*\dsl G_\br\cong \ft^*_\br\dsl \si_\br\cong(\AA^r)^I_0\dsl \si_\br$.

\begin{proof}[Proof of  Theorem \ref{ind-bundles}]
Forgeting the flag data on parabolic bundles yields a morphism 
$f_{\text{Higgs}}: \higgs\to\hig_r$.
Various maps considered above fit into 
the following diagram
\beq{hitch-diag}
\xymatrix{
\TT^*Y_{\br,d}\ \ar[d]^<>(0.5){\mu_\br}\ar@{}[drr]|{\Box} 
\ar[rr]^<>(0.5){f_{\text{bas}}}&& \ \higgsd\
\ar[r]^<>(0.5){f_{\text{Higgs}}}\ar[d]^<>(0.5){\reshiggs}&\
\hig_ {r,d} \
\ar[r]^<>(0.5){\kap}& \hit\ar[d]^<>(0.5){\reshit}\\
\g^*_\br   \  \ar[rr]&&  \g^*_\br/G_\br\ar@{=}[r]&
\ft_\br^*\dsl \si_\br\ar@{=}[r]&(\AA^ r \dsl \si_ r )^I_0
}
\eeq
It follows from definitions that this diagram commutes.
Therefore, we have $\reshiggs\inv(\g_\br\reg\dsl G_\br)= (\reshit\ccirc\kap_r\ccirc f_{\text{Higgs}})\inv(\ft\reg_\br\dsl\si_\br))$. The morphism 
$f_{\text{Higgs}}$ is a proper schematic morphism. Hence, from Lemma \ref{reduced} and Proposition \ref{higgs-fin} we deduce
that the stack $\TT^* Y_{\br,d}/G_\br=\reshiggs\inv(\g_\br\reg\dsl G_\br)$ has finite type.
In particular, the stack $\mm_O:=\mu_\br\inv(O)/G_\br$ is a stack of finite type for any coadjoint orbit
$O\sset\g_\br\reg$.

Further, for any $\gam>0$ let $Y_{\br,d,\gam}$ be the preimage of
the substack $\X_{\br,d}(\textsf{Vect}_\gam(C,D,\bm))\sset \ppar_{\br,d}$ in $Y_{\br,d}$.
By Proposition \ref{par-appr},  $Y_{\br,d,\gam}$ is an open substack of  $Y_{\br,d}$ of finite type.
The morphism $\TT^* Y_{\br,d}\to Y_{\br,d}$ being schematic,
we deduce that $(\TT^* Y_{\br,d,\gam})/G_\br$ is an open substack of  $\TT^* Y_{\br,d}/G_\br$
of finite type. Furthermore, Proposition  \ref{par-appr} says that
$\ppar_{\br,d}
=\underset{^\gam}\varinjlim\ \X^{^{}}_{\br,d}(\textsf{Vect}_\gam(C,D,\bm))$.
Put $\mm_{O,\gam}=\mm_O\times_{\higgsd}\,(\TT^* Y_{\br,d,\gam})/G_\br$, an open substack of $\mm_O$.
It follows that $(\TT^* Y_{\br,d})/G_\br=\underset{^\gam}\varinjlim\  (\TT^* Y_{\br,d,\gam})/G_\br$ and hence
$\mm_O=\underset{^\gam}\varinjlim\ \mm_{O,\gam}$.
Thus,  we obtain a diagram
\[\underset{^\gam}\varinjlim\ \mm_{O,\gam}=\mm_O
\ \into\ (\TT^* Y_{\br,d})/G_\br=\higgsd\  \to\ \ppar_{\br,d}\ =\
\underset{^\gam}\varinjlim\ \X^{^{}}_{\br,d}(\textsf{Vect}_\gam(C,D,\bm)).
\]

The stack $\mm_O$ being of finite type, it follows that there exists $\gamma$ such that
$\mm_{O,\gam}=\mm_O$ and, moreover, the image of $\mm_O$ in $\ppar_{\br,d}$ is contained
in $\X^{^{}}_{\br,d}(\textsf{Vect}_\gam(C,D,\bm))$.
Increasing $\gam$, if necessary, we may (and will) assume that
$\gam> \gam_0$, where $\gam_0$ is as in Lemma \ref{par-ind}.
Thus,  using  that  Lemma, we get
\begin{multline*}
\#\AI(\ppar_{\br,d}, \F_q)\ =\
\#\AI_{\br,d}(\textsf{Vect}_\gam(C,D,\bm)),\ \F_q)\\
=
q^{-\half\dim\mm_{O,\gam}}\cdot \ltr H^\hdot_c(\mm_{O,\gam},\qlb)^{\langle\sign\rangle}
=q^{-\half\dim\mm_O}\cdot \ltr H^\hdot_c(\mm_O,\qlb)^{\langle\sign\rangle},
\end{multline*}
where the second equality follows from Theorem \ref{A-ind} applied to category $\textsf{Vect}_\gam(C,D,\bm)$.
\end{proof}

\section{Deformation construction via coarse moduli spaces}
 \label{quiv-sec}
This section is devoted to the proof of Theorem \ref{la-thm}.
Although the theorem is stated under the
assumption that $\cC$ is either the category  of quiver representations
or the category of parabolic bundles,  a large part of the theory
may be developed in a much more general setting. 

\subsection{}\label{aaa}

We first consider quiver representations.
Let $Q$ be a finite quiver with vertex set $I$ and $\cC$ the category of finite dimensional
representations of $Q$ eqipped with the forgetful functor $F$.
Let  $\bv\in \Z^I$ be 
a dimension vector and $G=\PGL_\bv$, resp. $\g=\pgl_\bv$.
Then, $\X_\bv(\cC,F)=\Rep_\bv Q$ is the $G$-scheme of representations of $Q$
of  dimension $\bv$ and $\TT^*\X_\bv(\cC,F)=\Rep_\bv\bar Q$ is the $G$-scheme of representations of $\bar Q$,
the double of $Q$. 
Associated  with any $I$-tuple $\th=(\th_i)\in \BQ^I$ such that $\sum_\ii \th_i=0$
there is a stability condition on $\Rep_\bv\bar Q$.
Let $\cz^\th\sset{\Rep_\bv\bar Q}$ denote the $\th$-semistable locus
and $\mm^\th={\Rep_\bv\bar Q}\dsl_\th G$ the corresponding GIT quotient,
the coarse moduli space of  $\th$-semistable representations.
The canonical map $\varpi: \cz^\th\to\mm^\th$ is a categorical quotient by $G$.
The moment map $\mu: {\Rep_\bv\bar Q}\to\g^*$  restricts to a map
 $\mu^\th: \cz^\th\to\g^*$.

Next, we consider the setting of parabolic  bundles. We keep the notation of \S\ref{par-sec},
in particular, we fix a vector $\bm=(m_i)\in \Z^I$, where $m_i>0$ for all $i$.
Following  Mehta and Seshadri \cite{MeSe}, a stability condition for parabolic Higgs bundles 
is given by an $I$-tuple
$\th_i^{(1)}\leq \th_i^{(2)}\leq\ldots\leq\th_i^{(m_i)}$, of rational numbers, such that
$\th_i^{(m_i)}-\th_i^{(1)}<1$.  The slope of a parabolic Higgs bundle of type $\br\in \Z^\bbm_{\geq 0}$ and degree $d$ is defined
as $\frac{1}{r}(d+\sum_{i,j} r\ij\cdot\th\ij)$, where $r=|\br|$ is the rank of the underlying vector bundle.
Thus, one has the notion of $\th$-stable, resp. $\th$-semistable, parabolic Higgs bundles.
By definition, a framed  parabolic Higgs bundle is $\th$-stable, resp. $\th$-semistable, if so is
the corresponding  parabolic Higgs bundle with the framing being forgotten.
It is known  that for any type $\br$ and degree $d$ there exists a coarse moduli space $\mm^\th$, resp $\cz^\th$, of
 $\th$-semistable  parabolic Higgs bundles, resp.  framed  parabolic Higgs bundles.
Furthermore, these moduli spaces are normal quasi-projective varieties.
Let $G=\PGL_\br$ and  $\g=\Lie G$. The $G$-action on the stack $\TT^*\X_{\br,d}(\VV(C,D,\bm), F)$ induces
one on $\cz^\th$. Forgetting the framing yields a morphism $\varpi: \cz^\th\to \mm^\th$, which is a
categorical quotient by $G$.  The moment map induces a $G$-equivariant map $\mu^\th: \cz^\th\to\g^*$.

Below, we will consider the setting of quiver representations and of
parabolic Higgs bundles at the same time. In order to use uniform notation, in the Higgs bundle
case we will often write $\bv$ for the pair $(\br,d)$. Also, write $(-)^{\th\text{-stab}}$ for the 
$\th$-stable locus. 

The following result uses a special feature 
that there is a well behaved notion of a subobject of a given object
of the category ${\bar Q}\mmod$, resp. the category of parabolic Higgs bundles, in particular, the category in question is an exact category.

\begin{lem}\label{simple} Let $\xi$ be a closed point of $\TT^*\X_\bv(\cC,F)/G$ such that $\mu(\xi)\in \g\gen/G$.
Then, $\xi$ has no nonzero proper subobjects. In particular, 
$\xi$  is $\th$-stable, for any stability condition $\th$.
\end{lem}
\begin{proof} Let $\xi'$ be a subobject of $\xi$  and let $\xi''=\xi/\xi'$. Let $x$, resp. $x',x''$,
be the image of $\xi$, resp. $\xi',\xi''$, in $\X(\cC)$ and let $\bv'=\dim x',\ \bv''=\dim x''$.
Clearly, we have $\bv=\bv'+\bv''$, moreover, one can choose compatible framings of $x$ and $x',x''$, i.e. choose a basis of the vector space
$F(x)$  such that a subset of that basis is a basis of the vector space $F(x')\sset F(x)$ and the images
of the remaining base vectors form a basis of $F(x'')=F(x)/F(x')$. These choices provide a lift of $\xi$,
resp. $\xi',\xi''$, to an object $\wt\xi$, resp. $\wt\xi',\wt\xi''$, of the stack $\TT^*\X(\cC,F)$.
Let $g=\mu(\wt\xi)\in \g^*\cong\mathfrak{sl}$, resp. $g'=\mu(\wt\xi')\in \g^*_{\bv'}\cong\mathfrak{sl}_{\bv'},\ g''=\mu(\wt\xi'')\in \g^*_{\bv''}\cong\mathfrak{sl}_{\bv''}$.
By construction, the matrix $g$ has a block form
$g=\big(\begin{matrix} g'&*\\ 
0&g''
\end{matrix}\big)$. The assumption of the lemma that $g\in\g\reg$ forces one of the
two diagonal blocks to have zero size. We conclude that either $\bv'=\bv$ and $\xi'=\xi$, or
$\bv'=0$ and $\xi'=0$.
Finally, applying the Schur lemma, one deduces that  $\Aut \xi=\GG$.
\end{proof}

The following result is well known.
\begin{lem}\label{smooth-stab} The top, resp. bottom, horizontal arrow in the following natural commutative diagram is an
isomorphism of $G$-stacks, resp. stacks:
$$
\xymatrix{
\TT^*\X(\cC,F)^{\th\text{-stab}}\ \ar[d]\ar[rr]^<>(0.5){\cong} && \cz^{\th\text{-stab}} \ar[d]^<>(0.5){\varpi}\\
\big(\TT^*\X(\cC,F)/G\big)^{\th\text{-stab}}\ar[rr] ^<>(0.5){\cong} && \mm^{\th\text{-stab}}
}
$$
Furthermore, $\cz^{\th\text{-stab}}$
is a smooth symplectic variety and the map
$\mu^\th: \cz^{\th\text{-stab}}\to \g^*$ is a smooth morphism.
\end{lem}
\begin{proof}[Sketch of Proof] The case of quiver representations is clear, so we only consider
the case of  parabolic Higgs bundles. Let  $\xi$ be a stable closed point $\TT^*\X(\cC,F)$.
Then,  we have $\Hom(\xi,\xi)=\k$. In particular, we have $\Aut(x)=\GG$, hence $\TT^*\X(\cC,F)^{\th\text{-stab}}$ is an algebraic space.
An explicit construction of the coarse moduli space $\cz^\th$ shows that 
the scheme  $\cz^{\th\text{-stab}}$ represents the functor associated with the stack
 $\TT^*\X(\cC,F)^{\th\text{-stab}}$, so this algebraic space is in fact a scheme.

To prove that $\TT^*\X(\cC,F)^{\th\text{-stab}}$ is smooth, it suffices to show that  the function
$\xi\mto \dim\Ext^1 (\xi,\xi)$ is constant on $\TT^*\X(\cC,F)^{\th\text{-stab}}$.
To this end, one observes that the stack of parabolic bundles being smooth, the corresponding cotangent stack has
cohomological dimension $\leq 2$. It follows that  for all objects $\xi$ of $\cz^\th$
the  group $\Ext^j(\xi,\xi)=\op{R}^j\Hom(\xi,\xi)$ vanishes for
all $j\neq 0,1,2$. Hence, the Euler characteristic
$\dim\Ext^0(\xi,\xi)-\dim\Ext^1 (\xi,\xi)+\dim\Ext^2(\xi,\xi)$ is independent of $\xi$.
Furthermore, thanks to the  (derived) symplectic structure on a
cotangent stack, one has $\dim\Ext^2(\xi,\xi)=\dim\Ext^0(\xi,\xi)=1$. 
We deduce that the function
$\xi\mto \dim\Ext^1 (\xi,\xi)$ is constant on the stable locus, as required.
The remaining statements follow from standard results of symplectic geometry.
\end{proof}

\begin{cor}
\label{stab-cor} The variety $\mm^{\th\text{-stab}}$
is the fine moduli space of $\th$-stable objects and the
map $\varpi: \cz^{\th\text{-stab}}\to\mm^{\th\text{-stab}}$ is
a $G$-torsor.
\end{cor}

Let $\ft$ be the Cartan subalgebra, resp. $W$ the Weyl group, of $\g$.
Let $f^\th: \mm^\th\to \g^*\dsl G=\ft^*\dsl W$ be the map induced
by the moment map, and write $\mm\reg= (f^\th)\inv(\g\reg\dsl G)$,
resp. $\cz\reg= (\mu^\th)\inv(\g\reg)$.
Also, let $\mu\inv(\g\reg)$ denote the preimage of $\g\reg$ under
the  $\TT^*\X_\bv(\cC,F)$.

\begin{cor}
\label{quiv-gen} For any stability condition $\th$,
one has an isomorphism $\mu\inv(\g\reg)\cong \cz\reg$
of $G$-stacks, resp.  an isomorphism 
$\mu\inv(\g\reg)/G\cong \mm\reg$ of stacks. \qed
\end{cor}

\subsection{Constructing a resolution}\label{resolution} 
The proof of Theorem \ref{la-thm} is based on the construction of
 a smooth
variety $\wt\mm$ equipped with a morphism  $\pi:
\wt\mm\to{\mm\times_{\g^*\dsl G}\ft^*}$ such that
\begin{description}
\item[(P1)]  The map $p$ is a projective morphism;

\item[(P2)] The composite  $\wt f:\ \wt\mm\xrightarrow{\pi}{\mm\times_{\g^*\dsl G}\ft^*}\xrightarrow{pr_2}\ft^*$ is a
smooth morphism;
\item[(P3)] The restriction of $\pi$ yields an isomorphism
$\pi\inv(\mm\times_{\g^*\dsl G}\ft\reg)\,\iso\,\mm\times_{\g^*\dsl G}\ft\reg$.
\end{description}

In the quiver setting, the variety $\wt\mm$ was, essentially, constructed in \cite{HLV1}, \cite{HLV2} as follows.
One introduces an extended quiver $Q'$ obtained from 
$Q$ by attaching,  a `leg' of length $v_i-1$ at each vertex $\ii$.
Let $\tv$ be the dimension vector
that has the same coordinates as $\bv=(v_i)_{\ii}$ at the vertices of $Q$ and such
that the coordinates  at the additional vertices of the $i$-th leg are $(v_i-1,v_i-2,\ldots,1)$.
One has the group $\PGL_\tv$, the representation scheme
$\Rep_\tv\bar Q'$ and the moment map
$\mu_\tv: \Rep_\tv\bar Q'\to \pgl_\tv^*$. 
Let $\fz_\tv\sset \pgl_\tv^*$ be the fixed point set of
the coadjoint $\PGL_\tv$-action on $\pgl_\tv^*$ and put
$Z:=\mu_\tv\inv(\fz_\tv)$. For example, in the case where the quiver $Q$ has a single vertex $0$,
a point of $Z$ is a data $\xi=(\rho, \rho_j,\rho_j^*,\ j=1,\ldots,v-1)$:
\beq{rho-diag} 
\xymatrix{
 \ {\AA^v}\ 
 \ar@(ul,dl)_<>(0.5){\mu(\rho)}
\ar@<0.4ex>[r]^<>(0.5){\rho^*_1}&\ \AA^{v-1}\
\ar@<0.4ex>[l]^<>(0.5){\rho_1}\ar@<0.4ex>[r]^<>(0.5){\rho^*_2}&\
\AA^{v-2}\
\ar@<0.4ex>[l]^<>(0.5){\rho_2}\ar@<0.4ex>[r]^<>(0.5){\rho_3^*}\ &\ \ldots\  \ar@<0.4ex>[r]^<>(0.5){\rho^*_{v-1}}
\ar@<0.4ex>[l]^<>(0.5){\rho_3}&\ \AA^1 \ \ar@<0.4ex>[r]^<>(0.5){\rho_{v}^*}\ar@<0.4ex>[l]^<>(0.5){\rho_{v-1}}\ &\ 0 
\ \ar@<0.4ex>[l]^<>(0.5){\rho_{v}}
}
\eeq
 where $\rho$ is a representation 
of $\bar Q$ in $\AA^v$ and the other maps are subject to the `moment map equations':
\beq{g}
\mu_\bv(\rho)=z_0\Id_{\gl_v}+\rho_1\rho_1^*,\quad\text{and}\en
\rho_j^*\rho_j-\rho_{j+1}\rho^*_{j+1}=z^j\cdot\Id_{\gl_{v^j}}\quad
  j=1,\ldots,v-1,
\eeq 
for some $z=(z_0,z_1,\ldots,z_{v-1})\in\AA^v$ such that $\sum_j\ z^j=0$.
Let 
\[\mu_Z: Z\to \fz_\tv,\quad
\xi=(\rho, , \rho_j,\rho_j^*,\ j=1,\ldots,v)\mto (z_0,z_1,\ldots,z_{v-1})\]
be the restriction of the map $\mu_\tv$ to $Z$.
The assignment  $z=(z_0,\ldots,z_{v-1})\mto
\text{diag}(t_1,\ldots,t_v)$, where
$t_j=z_0+\ldots+ z_{j-1},\ j=1,\ldots,v$, gives an isomorphism
$\varkappa: \fz_\tv\iso\ft_v^*$.
Further, let $p: Z\to \Rep_\bv \bar Q$ be the forgetful map $(\rho, \, \rho_j,\rho_j^*,\ j=1,\ldots,v)\mto \rho$.

The case of an arbitrary quiver $Q$ with possibly more than one vertex is similar and
we keep the notation $p$ for the forgetful map.
In the general, it is easy to check that the
following diagram commutes
\beq{Zdiag}
\xymatrix{
\ Z\ar[d]^<>(0.5){\mu_Z}\ar[r]^<>(0.5){p}&\Rep_\bv\bar Q\ar[r]^<>(0.5){\mu_\bv}&
\ \g^*\ \ar@{->>}[r]&\g^*\dsl G\ \ar@{=}[d]\\
\ \fz_\tv\ \ar[rr]^<>(0.5){\varkappa}_<>(0.5){\cong}&&\ \ft^*\ \ar@{->>}[r]&
\ \ft^*\dsl W
}
\eeq
Hence, the assignment $\xi\mto \big(p(\xi),\,\varkappa\ccirc\mu_\tv(\xi)\big)$ gives
a well defined map $\wt f_Z: Z\to \Rep_\bv\bar Q\times_{\g^*\dsl G}~\ft^*$.

The implications in \eqref{sst-impl} below are well known; statements analogous to 
statements (i)-(ii) below may be found eg. in \cite{HLV2}.

\begin{lem}\label{sst} Let $\vth$ 
be a stability condition on $\Rep_\tv\bar Q'$ that has the same coordinates as the original stability
$\th$ at the vertices of $Q$ and  sufficienly general small negative rational numbers at the additional
vertices. Then, for $\xi\in Z$
one has
\beq{sst-impl}
p(\xi)\en \text{is $\th$-stable}\en  \Rightarrow \en
\xi\en  \text{is $\wt\th$-stable }\en  \Leftrightarrow \en \xi \en \text{is $\wt\th$-semistable }\en  \Rightarrow\en 
p(\xi)\en \text{is $\th$-semistable.}
\eeq

Furthermore, writing $\xi=(\rho, \rho_j,\rho_j^*)\in Z$, we have
\vskip 2pt

 \vi If
$\xi$ is
$\wt\th$-semistable then  each of the maps $\rho^*_j,\ j\geq 0$, is
surjective.

\vii If
$p(\xi)$ is $\th$-stable and each of the maps $\rho^*_j,\ j\geq 0$, is
surjective then $\xi$ is  $\wt\th$-stable. 
\end{lem}

Let $\GL\leg$ be a subgroup of  $\GL_\tv$ formed by the product of 
general linear groups corresponding to the vertices of the legs.
Thus, $\GL_\tv=\GL_\bv\times \GL\leg$ and the first projection
 induces a short exact sequence
\[1\to \GL\leg\to \PGL_\tv\to G\to1\]
where $G=\PGL_\bv$.
We fix $\vth$ as in  the above lemma, let $Z^{\wt\th}\sset Z$ be the open subset of $\wt\th$-semistable representations,
and  $\wt\cz$, resp.  $\wt\mm$, be the corresponding
quotient by $\GL\leg$, resp. $\PGL_\tv$.
By \eqref{sst-impl}, any  point of 
$Z^{\wt\th}$ is $\wt\th$-stable. 
It follows that  $\wt\cz$ and $\wt\mm$ are smooth and the
map $Z^{\wt\th}\to \wt\mm$ is a geometric quotient by $\GL_\tv$.
Furthermore, this map factors as a composition $Z^{\wt\th}\to \wt\cz\xrightarrow{\varpi}\wt\mm$,
where the first map is the geometric quotient by the group
$\GL\leg$  and the second map is a quotient map by $G$, which is also a geometric quotient.

The map  $\mu_\tv$ in diagram \eqref{Zdiag} induces a map $\wt\mm\to \fz_\tv$.
By part (i) of Proposition \ref{sst} the map $p$
restricts to a map $\wt\cz\to\cz^\th$ and also induces a map $\wt\mm\to\mm^\th$.
Further, the map $\mu_\tv$ in diagram \eqref{Zdiag} induces a map $\wt\mm\to \fz_\tv$.
Thanks to the commutativity of the diagram, these give a well defined map
$\wt\mm\to\mm^\th\times_{\g^*\dsl G}\ft^*$.

\begin{prop}\label{quivP} Properties {\em (P1)-(P3)} hold for the map $\wt\mm\to\mm^\th\times_{\g^*\dsl G}\ft^*$.
\end{prop}

For any $\bar\rho\in \Rep_\tv \bar \ql$,
we associate
 an $I$-graded partial
flag  $F^\hdot(\rho)$ in the  $I$-graded vector space $\AA^\bv$,
cf. \eqref{rho-diag},
 defined by $F^j=\Ker(\rho^*_j\ccirc\rho^*_{j-1}\ccirc\ldots\ccirc\rho^*_1)$,\
$j=1,2,\ldots$.
Note that if all the maps
 $\rho_j^*$ are surjective then  $F^\hdot(\bar\rho)$ is a complete
flag. Therefore, in the setting of Lemma \ref{sst}, the assignment 
$\xi=(\rho, \rho_j,\rho_j^*)\mto (\rho,\,F^\hdot(\rho))$ yields a map
$Z^\vth\to \Rep^\th\bar Q\times_{\g^*}\tg$, where $\tg$ is  the Grothendieck-Springer resolution of $\g^*\cong\pgl_\bv$.
This map factors through $\wt\cz$ and so does the composite map $Z^\vth\xrightarrow{\mu_\tv} \pgl_\tv^*=\pgl_\bv\to \g^*=\g$.
The proof of the following result is left for the reader.

\begin{lem}\label{wtcz} The above maps give a $G$-equivariant isomorphism
$\wt\cz^{\vth\stab}\cong \cz^{\th\stab}\times_{\g^*}\tg$. Hence,
one obtains a geometric quotient map
$\cz^{\th\stab}\times_{\g^*}\tg\to \wt\mm^{\vth\stab}$, by $G$.
\end{lem}

Next, we consider the setting of parabolic Higgs bundles.
Let $\wt\br=(\wt r^{(j)}_i)$ be the type that corresponds to taking complete flags at each of
the marked points $c_i,\ i\in I$, i.e. we have $\wt r^{(j)}_i=1$ for all $i,j$.
The group $\GL_{\wt\br}/G$ is the maximal  torus $T$, of $(\prod_\ii\ \GL_r)/\GG$,
so $\Lie(\GL_{\wt\br}/GG)=\ft$ and $W_{\wt\br}=\{1\}$.
Note that writing $\wt\br=(\wt r^{(j)}_i)$ we have $\mathsf{gcd}\{\wt r^{(j)}_i, (i,j)\in \Xi(\bm)\}=1$.
A stability condition $\vth$,
for parabolic Higgs bundles of type $\wt\br$,
may  be viewed as a point in  the interior of the
fundamental alcove in $\ft$.
For  a sufficiently general $\vth$,

Since $\mathsf{gcd}\{\wt r^{(j)}_i, (i,j)\in \Xi(\bm)\}=1$,
any $\vth$-semistable parabolic Higgs bundle is  $\wt\th$-stable, provided
the stability condition $\vth$ is sufficiently general. 
Thus, the corresponding coarse moduli space
$\mm^\vth=\mm^{\vth\stab}$ is smooth, the map $\cz^\vth\to\mm^\vth$ is a $T$-torsor,
and the moment $\mu^\vth: \cz^\vth\to\ft^*$ is a smooth morphism. It follows that the induced morphism
$f^\vth: \mm^\vth\to\ft^*$ is smooth as well.

Now, given an arbitrary type $\br\in\Xi(\bm)$, one has the natural morphism
$p: \hig_{\wt\br}\to \higgs$, of stacks, that sends complete
flags in the fibers to the corresponding partial flags of type $\br$.
The group $G=\GL_\br/\GG$ is a Levi subgroup of  $(\prod_\ii\ \GL_r)/\GG$, so
a  stability condition  on parabolic Higgs bundles of type $\br$ may be viewed as a point 
$\th$ on
a wall of the fundamental alcove.
It follows that for  a sufficiently general  element $\wt\th\in \ft$ in the interior of the
fundamental alcove
which is also  sufficiently close to $\th$ all implications in
\eqref{sst-impl} hold. We fix such a $\wt\th$.
Then, the map $p$ induces well-defined maps $p_\cz: \cz^\vth\to\cz^\th$, resp.
$p_\mm: \mm^{\wt\th}\to\mm^\th$, and one has a commutative diagram, cf. \eqref{hitch-diag}:
\beq{ppmm}
\xymatrix{
\mm^\vth\ar[d]^<>(0.5){f^\vth}\ar[rr]^<>(0.5){\varpi} && \mm^\th\ar[d]^<>(0.5){f^\th}\\
\ft^*\ar@{->>}[rr]&& \g^*\dsl G=\ft^*\dsl W.
}
\eeq

Let $\tg\to\g^*$  denote the Grothendieck-Springer resolution
viewed as a map of $G$-varieties. 
Associated with the $G$-torsor $\cz^{\th\text{-stab}}\to \mm^{\th\text{-stab}}$
and the $G$-viriety $\tg$, one has an associated bundle
$\cz^{\th\text{-stab}}\times_G \tg\to\mm^{\th\text{-stab}}$.
\begin{lem}\label{mmgroth}
There is a natural isomorphism $p_\mm\inv(\mm^{\th\text{-stab}})\cong \cz^{\th\text{-stab}}\times_G \tg$,
of schemes over $\mm^{\th\text{-stab}}$. Furthermore, the natural map
$\cz^{\th\text{-stab}}\times_{\g^*}\tg\to \cz^{\th\text{-stab}}\times_G \tg$
is a geometric quotient by $G$.
\end{lem}
\begin{proof}
The isomorphism in the bottom line of the diagram of Lemma
\ref{smooth-stab} implies that $\mm^{\th\text{-stab}}$ is a fine moduli space, in
particular, there is a universal vector bundle $\V$, on $\mm^{\th\text{-stab}}\times C$,
equipped with the parabolic structure of type $\br$ and with the
universal Higgs field. The quotient map
$\cz^{\th\text{-stab}}\to \mm^{\th\text{-stab}}$ is a $G$-torsor, the frame bundle associated with 
the vector bundle $F(\V)$ on $\mm^{\th\text{-stab}}$ given by formula \eqref{bun-functor}.
There is
a similar universal vector bundle $\wt\V$ on  $\wt\mm=\mm^{\wt\th\text{-stab}}\times C$.
Let $\wt\mm^{\th\text{-stab}}:=p_\mm\inv(\mm^{\th\text{-stab}})$.
Using universal properties of  moduli spaces, it is easy to show that
the restriction of $\wt\V$ to $\wt\mm^{\th\text{-stab}}$ is  canonically isomorphic to
the vector bundle $p_\mm^*\V$. Moreover,
the universal parabolic structure 
of type $\wt\br$ on this vector bundle and the universal Higgs field yield the isomorphism claimed in
the lemma. The second statement of the lemma is proved similarly.
\end{proof}

By commutativity of \eqref{ppmm}, the map $p_\mm\times f^\vth$ factors through a map
$\pi: \mm^{\wt\th}\to\mm^\th\times_{\g^*/G}\ft^*$.

\begin{lem}\label{higgsP} Properties {\em {(P1)-(P3)}} hold for  $\wt\mm=\mm^\vth$ and 
the map 
$\pi$  defined above.
\end{lem}
\begin{proof} 
Since every $\wt\th$-semistable point is $\wt\th$-stable, the variety $\mm^\vth$ is smooth.
The Hitchin map descends to a map $\kappa^\th: \mm^\th\to \hit$, resp. $\kappa^{\wt\th}: \mm^\vth\to\hit$.
It is known that  $\kappa^\th$ and $\kappa^{\wt\th}$ are projective morphisms.
Let $G/B$ be the flag variety.
It is known also that one has
\[\dim \mm^\vth=2\big(r^2(g-1)+1+\dim G/B\big)=2\dim \hit\]
and the fibers of the Hitchin map have dimension $\leq \dim \mm^\vth$.
It follows that the map $\kappa^{\wt\th}$ is dominant, hence surjective.
Further, it is clear that we have $\kappa^{\wt\th}=\kappa^\th\ccirc p_\mm$.
We deduce that $p_\mm$, hence $\pi$,
is a  projective morphism, proving (P1). 
Property (P2) is clear since the map $\wt f$ in (P2) equals $f^\vth$ and we know that $f^\vth$ is a smooth morphism.

Next, we use  Lemma  \ref{mmgroth} to identify the map $p_\mm\inv(\mm^{\th\text{-stab}})\to \mm^{\th\text{-stab}}\times_{\g^*/G}\ft^*$,
obtained from $\pi$ by restriction, with
the morphism 
\[\cz^{\th\text{-stab}}\times_G\wt\g\to (\cz^{\th\text{-stab}}\times_G\g^*)\times_{\g^*/G}\ft^*\]
induced by the natural map $\tg\to \g^*\times_{\g^*/G}\ft^*$.
Further, by Lemma \ref{simple} we have $\mm\reg\sset \mm^{\th\text{-stab}}$, resp.
$\cz\reg\sset \cz^{\th\text{-stab}}$.  Therefore, the restriction of $\pi$ to 
$\pi\inv(\mm\times_{\g^*/G}\ft\reg)$ may be identified with
the morphism $\cz\reg\times_G\wt\g\reg\to (\cz\reg\times_G\g\reg)\times_{\g^*/G}\ft\reg$.
The latter morphism is an isomorphism, proving (P3).
\end{proof}

\subsection{Proof of Theorem \ref{la-thm}(i),(iii),  and also parts (ii), (iv) in the coprime case}
The statement in (iii) follows  from Lemma \ref{smooth-stab} and Corollary \ref{quiv-gen}.
To prove other statements, 
let $\Nil$ be the nilpotent variety of $\g^*$.
We write $\wt\mm\reg=\pi\inv(\mm\times_{\g^*\dsl G}\ft\reg)$, resp.
$\mm\nil=(f^\th)\inv(0),\ \wt\mm\nil=\wt f\inv(\Nil)$, and
  $\pi\nil$, resp. $\pi\gen$, for the restriction of $p$ to
$\wt\mm\nil$, resp. $\wt\mm\reg$.
Thus, we obtain a commutative diagram:
\beq{grsp}
\xymatrix{
\wt\mm\nil \ar[d]^<>(0.5){\pi\nil}
\ar@{^{(}->}[rr]&&
\wt\mm\ar[d]^<>(0.5){\pi}\ar@/_4.5pc/[dd]|-<>(0.8){\wt
f}&&\ar@{_{(}->}[ll]\
\wt\mm\reg\ar[d]^<>(0.5){\pi\gen}_<>(0.5){\cong}\\
\mm\nil\ar[d]\ar@{^{(}->}[rr]^<>(0.3){\imath}&& \mm\times_{\g^*\dsl G}\ft^*\
\ar[d]^<>(0.5){pr_2}&&\ar@{_{(}->}[ll]_<>(0.5){\jmath}\
\mm\times_{\g^*\dsl G}\ft\gen\ \ar[d]^<>(0.5){pr_2}\\
\{0\}\ \ar@{^{(}->}[rr]&& \ft^* &&\ar@{_{(}->}[ll]\
\ft\gen
}
\eeq
Here, we have used
that the reduced scheme associated with 
$\mm\times_{\g^*\dsl G}\{0\}$ is isomorphic to
$\mm\nil$.

We consider  diagram
\eqref{pip} in the case where $\wt X=\wt\mm,\
X=\mm\times_{\g^*\dsl G}\ft^*$, and $S=\ft^*$, so the diagram
is:
 $\ \wt\mm\xrightarrow{\pi} \mm\times_{\g^*\dsl G}\ft^*\xrightarrow{pr_2}\ft^*$.
Properties (P1)-(P3) insure that   Proposition \ref{mylem} is applicable.
Using base change in diagram \eqref{grsp}, we deduce isomorphisms
\beq{ijiso}p^{\vth}_!\C_{\wt\mm^\vth}=\jmath_{!*}\C_{\mm^\th\times_{\g^*\dsl G}\ft\reg},\quad
(p^{\vth}\nil)_!\C_{\wt\mm^\vth\nil}=\imath^*\jmath_{!*}\C_{\mm^\th\times_{\g^*\dsl G}\ft\reg}.
\eeq
Part (i)  of the theorem now follows from this.

We have natural $\GG$-actions on $\mm$, resp. $\wt\mm$, etc. In the case of quiver representations,
the action is induced by the $\GG$-action on $\Rep_\bv\bar Q$, resp. $\Rep_\tv\bar Q'$, by dilations.
In the case of parabolic  Higgs bundles  the $\GG$-action
 is obtained by rescaling the Higgs field.
Further, the dilation action on $\ft_\bv^*$ induces a $\GG$-action on $\ft^*\dsl W$. Each of
the maps $\wt\mm\to\mm$ and $\mm\to \ft^*\dsl W$ is
$\GG$-equivariant. It follows,
thanks to \cite[Corollary 1.3.3]{HV2}
(cf. also \cite[Appendix B]{HLV2} and Theorem \ref{pure_res} of the
present paper),  that
the restriction map $H^\hdot(\wt\mm)\to
H^\hdot(\wt f\inv(z))$ is an isomorphism
for any $z\in\ft^*$ and that 
  the sheaf
$\wt f_!\C_{\wt\mm}$ is   geometrically constant,
cf. Theorem \ref{pure_res}). 
If  $z=0$ we have $\wt f\inv(0)=\wt\mm\nil$.
On the other hand, let  $z\in \ft\gen$ and  $\eta={\mathfrak
   w}_\ft(z)\in\ft\gen\dsl W=\g^*\dsl G$.
The preimage of $\eta$ under the quotient
map $\fw_\g: \g^*\to \g^*\dsl G$ is a regular semisimple coadjoint orbit
$O$ in
$\g^*$.
The map $\pi\gen$ yields an isomorphism
 $\wt f\inv(z)\iso  f_\mm\inv(\eta)=\mu_\cz\inv(O)\dsl G$.
Let $\IC:=\jmath_{!*}\C_{\mm^\th\times_{\g^*\dsl G}\ft\reg}$. We  deduce
\beq{f-iso}
H_c^\hdot(\mm\nil,\,\imath^*\IC)\ccong
H_c^\hdot(\wt\mm\nil,\,\C_{\mm\nil})\ccong H_c^\hdot(f_\mm\inv(\eta),\,\C)
\ccong H_c^\hdot(\mu_\cz\inv(O)\dsl G) \ccong H_c(\mm_O).
\eeq
It is clear from the construction that the composite isomorphism 
respects the natural monodromy actions. In particular, the monodromy
action on $H_c^\hdot\big(\mu_\cz\inv(O)\dsl G,\ \C\big)$
factors  through a $W$-action. Alternatively, 
this follows from the fact that  the sheaf
$\wt f_!\C_{\wt\mm}$ is   geometrically constant.

We now assume that we are in the coprime case, that is,
the dimension vector $\bv$, resp. the vector $(\br,d)$ in the parabolic bundle case,
is indivisible and, moreover, the stability condition $\th$ is sufficiently general.
In such a case, we have $\mm:=\mm^\th=\mm^{\th\stab}$, resp $\cz:=\cz^\th=\cz^{\th\stab}$.
Thus, the map $\mu_\cz:=\mu^\th: \cz\to\g^*$  is a smooth morphism, so the
map $pr_\tg: \ \cz\times_{\g^*}\wt\g\to \tg$, obtained by base change from $\mu_\cz$,
is a smooth morphism as well.
Hence  $\cz\times_{\g^*}\wt\g$ is a smooth variety and Lemma \ref{mmgroth} implies that the map $\wt\varpi: \cz\times_{\g^*}\wt\g\to \wt\mm$ 
is a universal geometric quotient by $G$.

To complete the proof, we use the following
commutative
 diagram 
\beq{bs2}
\xymatrix{
&\mu_\cz\inv(\Nil)\ \ar[dl]_<>(0.5){pr}\ar@{^{(}->}[r]\ar@{}[d]|{\Box}&
\ \cz\  \ar[dl]_<>(0.5){pr}\ar[dr]^<>(0.5){\mu_\cz}&
\ \mu_\cz\inv(\Nil)\ \ar@{}[d]|{\Box}
\ar@{_{(}->}[l]\ar[dr]^<>(0.5){\mu_\cz}&
\\
\mm\nil\ \ar[dr]\ar@{^{(}->}[r]&
\mm\ar[dr]^<>(0.5){f_\mm}&&\ \g^*\  
\ar[dl]^<>(0.5){\fw_\g}&\ \Nil\ 
\ar@{_{(}->}[l]\ar[dl]\\
&
\{0\}\ \ar@{^{(}->}[r]&\ \g^*\dsl G\ &
\{0\}\ \ar@{_{(}->}[l]&
}
\eeq

Let $pr_\cz: \cz\times_{\g^*}\tg\to\cz$, resp. $pr_\mm: \mm \times_{\g^*\dsl G}\ft^*\to\mm$
and $pr_{\g^*}: \g^*\times_{\g^*\dsl G}\ft^*\to\g^*$, be the 
first projection and $pr\reg_\cz$,  resp. $pr\reg_\mm$ and $pr\reg_{\g^*}: \g^*\times_{\g^*\dsl G}\ft\reg\to\g\reg$,
its restriction to $\cz\reg\times_{\g^*}\tg\to\cz$,
resp. $\mm\times_{\g^*\dsl G}\ft\reg$ and $\g^*\times_{\g^*\dsl G}\ft\reg$.
We also use similar notation $pr_{\cz,\op{nil}}$, resp. $pr_{\mm,\op{nil}},\, pr_{\g^*,\op{nil}}$
for the corresponding maps over $\Nil$.
For $\rho\in\Irr(W)$, let 
 $\IC_{\rho, \cz}$, resp.  $\IC_{\rho, \mm}$ and $\IC_{\rho, \g^*}$, denote the IC-extension of the $\rho$-isotypic
component of the local system $(pr\reg_\cz)_!\C_{\cz\reg\times_{\g^*}\tg}$,
resp. $(pr\reg_\mm)_!\C_{\mm\times_{\g^*\dsl G}\ft\reg}$ and $(pr\reg_{\g^*})_!\C_{\g^*\times_{\g^*\dsl G}\ft\reg}$.
We have $\wt f_!\C_{\wt\mm}=(pr_\mm)_!\pi_!\C_{\wt\mm}=(pr_\mm)_!\IC=\oplus_\rho\ \IC_{\rho, \mm}$.
On the other hand, by  smooth base change  for the cartesian squares in the above diagram, we get
\beq{bs3}
pr_{\cz,\op{nil}}^*\imath^*_{\mm\nil\to\mm}\IC_{\rho, \mm}=\imath_{\cz\nil\to\cz}^*\IC_{\rho, \cz}=\mu_{\cz,\op{nil}}^*\imath_{\Nil\to\g^*}^*\IC_{\rho, \g^*}.
\eeq
In the special case $\rho=\op{sign}$, the sheaf $\imath_{\Nil\to\g^*}^*\IC_{\rho, \g^*}$ is known to be the sky-scrapper sheaf $\C_0$,
at $\{0\}\sset\Nil$, by the theory of Springer representations.
It follows that  $\imath^*_{\mm\nil\to\mm}\IC_{\sign,\mm}=\C_{\mm_0}$
and, hence, $H^\hdot_c(\mm\nil)^{(\sign)}=H^\hdot_c(\mm_0)$, as required in Theorem
\ref{la-thm}(ii).\qed
\medskip

Using the definition of an IC-sheaf, from the first isomorphism  in \eqref{ijiso} we deduce the following corollary
that resembles a result of Reineke \cite{Rei}:
\begin{cor}
The map $p^{\vth}: \mm^{\vth}_\tv\to \wt\mm^\th$ is a small resolution.
\end{cor}

\subsection{The support of $\imath^*\IC(\cl_\sign)$}
To complete the proof of   Theorem \ref{la-thm} it remains to show that the sheaf $\imath^*\IC(\cl_\sign$
is supported on $\mm^\th_0$.
In the coprime case, this has been established in the previous subsection.
We reduce the general case to  the coprime case.
To this end,  
we  perform constructions of  previous sections in a slightly different setting 
where we `frame' by legs all elements of $I$ except one.

Specifically, assume first that  $\# I>1$,  fix $\ii$ and let $I\nat=I\sminus\{i\}$.
We  mimic the constructions of  previous sections with the set $I$ being replaced by $I\nat$.
Thus, in the parabolic bundle case
we only consider complete flags at the marked points $i\in I\nat$.
Similarly, in the quiver case, we only add legs at the vertices of $I\nat$. 
To simplify the notation, we  give a detailed  construction  in the setting of parabolic bundles, the quiver setting
being similar.

We write $\bm=(\bm\nat,m_i)$, where $\bm\nat$ is an $I\nat$-tuple.
Fix a type $\br=(r^{(j)}_{i'})\in \Xi(\bm)$ and put
 $G\nat=\prod_{(i',j)\in\Xi(\bm\nat)}\  GL_{r^{(j)}_{i'}}$,  resp. $W\nat:=\prod_{(i',j)\in \Xi(\bm\nat)}\ \si_{r^{(j)}_{i'}}$,
and $G_i=(\prod_j \GL_{r^{(j)}_{i}})/\GG$,  resp.  $ W_i=\si_{r^{(j)}_{i}}$. Let  
$\g\nat$ and $\ft\nat$, resp. $\g_i$ and $\ft_i$, be the Lie algebra and the Cartan subalgebra of $G\nat$, resp. $G_i$.
The natural imbedding $G\nat\into \GL_\br$ induces a short exact sequence
$
1\to G\nat\to G \to G_{v_i}\to 1$, where $G=\GL_\br/\GG$.
The induced extension of Lie algebras has 
a natural splitting $\g=\g\nat\oplus\g_i$.
To simplify the exposition, we will use this splitting to identify
$\ft=\ft\nat\oplus\ft_i,\ \g^*\dsl G=\g\nat^*\dsl G\nat\times \g_i^*\dsl G_i$, etc.

In addition to the type $\wt\br$ considered in \S\ref{resolution}, we now introduce the type $\wt\br\nat=((r\nat)_{i'}^{(j)})$  that corresponds to taking complete flags
at the marked points labeled by $I\nat$ only, i.e.  such that $(r\nat)_{i'}^{(j)}:=1$ for all $(i',j)\in \Xi(\bm\nat)$ and
$(r\nat)_i^{(j)}:=r_i^{(j)}$ for all $j=1,\ldots,m_i-1$.
Given a stability condition $\th$ for parabolic Higgs bundles of type $\br$, we choose a sufficiently
close  stability condition $\vth\nat$ for parabolic Higgs bundles of type $\wt\br\nat$ and also
a stability condition $\vth$ for parabolic Higgs bundles of type $\wt\br$ which is close to $\th$.
Associated with these stability conditions, there are coarse moduli
spaces $\cz^\th$, resp. $\cz^{\vth\nat},\cz^\vth$, and $\mm^\th$, resp.
$\mm^{\vth\nat},\mm^\vth$.

It is important to note, for what follows, that the vector $\wt\br\nat$ is indivisible, i.e.
we have $\op{gcd}\{(r\nat)_{i'}^{(j)},\ (i',j)\in \Xi(\bm)\}=1$. Therefore, 
choosing  $\vth\nat$ to be sufficiently general , we can ensure that any
$\vth\nat$-semistable object is $\vth\nat$-stable. Similarly,  any
$\vth$-semistable object is $\vth$-stable. 
Hence, the moment map
$\cz^{\vth\nat}\to \ft\nat^*\oplus\g_i^*$ is smooth.
Furthermore,  an analogue of Lemma \ref{wtcz} yields an isomorphism
$\cz^{\wt\th}=\cz^{\vth\nat}\times_{\g_i^*}\wt\g_i$ 
where $\pi_i: \wt\g_i\to\g_i^*$ is the Grothendieck -Springer resolution for the Lie algebra $\g_i$.
Also, we have a proper morphism
$p_i: \cz^{\vth\nat}\to \mm^{\vth\nat}\times_{\ft^*\dsl W_i}\ft_\bv^*$.

Write $\pi_{\Nil_i}: \wt\Nil_i\to \Nil_i$ for the Springer resolution of the nilpotent cone of $\g_i^*$.
Let $\cz\nil^{\vth\nat}=\cz^{\vth\nat}\times_{\g_i^*}\Nil_i$
and $pr_i: \cz\nil^{\vth\nat}\to\Nil_i$ be the restriction of the moment map $\mu^{\vth\nat}$.
One has  the following commutative diagram
{\footnotesize
$$\xymatrix{
\cz\nil^{\wt\th\nat}\times_{\Nil_i}\wt\Nil_i\ar[r]^<>(0.5){a}\ar[d]
& \cz\nil^{\vth\nat}\ar[d]^<>(0.5){p_{\Nil,i}}\ar[r]^<>(0.5){pr_i}&\Nil_i&&\\
\mm\nil^{\wt\th}\ar@{^{(}->}[d]^<>(0.5){q_1}\ar[r]^<>(0.5){b_i^1}&\mm\nil^{\vth\nat}
\ar[r]^<>(0.5){b_i^2}\ar@{^{(}->}[d]^<>(0.5){q_2}
&\mm^{\th}\nil
\ar@{=}[r]^<>(0.5){b_i^3}
\ar@{^{(}->}[d]^<>(0.5){q_3}& \mm^\th\nil\ar@{=}[r]^<>(0.5){b_i^4}\ar@{^{(}->}[d]^<>(0.5){q_4}& 
\mm^\th\nil\ar@{^{(}->}[d]^<>(0.5){q_5}
\\
\mm^{\wt\th}\ar[r]^<>(0.5){c_i^1}\ar[d]^<>(0.5){}&\mm^{\vth\nat}\times_{\ft^*/W_i}\ft^*
\ar[r]^<>(0.5){c_i^2}\ar[d]^<>(0.5){f_i}&\mm^\th\times_{\ft^*/W_i}\ft^*\ar@{=}[r]^<>(0.5){c_i^3}\ar[d]^<>(0.5){}&\mm^\th\times_{\ft^*/W}\ft^*
\ar[r]^<>(0.5){c_i^5}\ar[d]^<>(0.5){}&\mm^\th\ar[d]^<>(0.5){f\bvt}\\
\ft^*\ar@{=}[r]&\ft^*\ar@{=}[r]& \ft^*
\ar@{=}[r]& \ft^*\ar[r]&\ft^*/W
}
$$}

Now let $\op{sign}_i$ be the sign representation of the group $W_i$,
write $\IC:=\jmath_{!*}\C_{\mm^\th\times_{\g^*\dsl G}\ft\reg}$, as in 
\eqref{f-iso}, and let $\IC^{\op{sign}_i}$ be the corresponding $\op{sign}_i$-isotypic component.
According to the theory of Springer representations, the $\op{sign}_i$-isotypic component
of the sheaf $(\pi_{\Nil_i})_!\C$ is supported at $0\in \Nil_i$.
Using base change in the above diagram and  an argument similar to the proof of isomorphisms in \eqref{bs3}
one deduces that $\supp(\IC^{\op{sign}_i})\sset (b_i^4\ccirc b_i^3\ccirc b^2_i\ccirc p_{\Nil,i})(pr_i\inv(0))$.
We can use this inclusion for all elements $\ii$ at the same time provided we choose (as we may)
$\wt\th$ in such a way that it is a sufficiently general deformation of $\vth\nat$ for every choice of distinguished element $\ii$.
This way, we obtain that
$\supp(\IC^{\op{sign}_i})\sset \cap_\ii\ (b_i^4\ccirc b_i^3\ccirc b^2_i\ccirc p_{\Nil,i})(pr_i\inv(0))$.
The intersection on the right is  equal to $\mm^\th_0$.

This completes the proof of Theorem \ref{la-thm}(ii)  provided the set $I$ contains at least 2 elements.
It remains to consider the case where $\# I=1$, i.e. the case
where there is only one marked point. 
Then, one can add a second marked point with trivial parabolic type and
apply the result in the case of two marked points. On the other hand, it is clear that
adding a marked point with  trivial parabolic type  doesn't affect the moduli space of parabolic Higgs bundles.
This completes the proof.

\section{Appendix A: Purity}\label{pure-sec}
\subsection{Local acyclicity}\label{acyc}
Throughout this section we fix a smooth geometrically irreducible scheme $S$
and a morphism $p: X\to S$.

We will use the following version of  the standard definition of  local acyclicity.
\begin{defn} 
A  sheaf  $\cF$ on $X$ is called {\em locally acyclic} with respect to $p$ if for any   cartesian
diagram
\beq{lac} 
\xymatrix{
C\times_S X\ar[rr]^<>(0.5){\wt g}\ar[d]^<>(0.5){\wt p} && X \ar[d]^<>(0.5){p}\\
C\ar[rr]^<>(0.5){g} && S
}
\eeq
where $C$ is a smooth curve,
the sheaf $\wt g^*\cf$  has zero vanishing cycles with respect to the map $\wt g$ at any closed point of the curve $C$.
\end{defn}

Local acyclicity behaves well under proper push-forwards in the following sense. 
Let  $\pi$  be  a proper morphism, so we have a diagram 
\beq{pip}
\wt X\xrightarrow{\pi} X\xrightarrow{p}S.
\eeq

\begin{lem}\label{la-proper}
Let $\cF$ be a sheaf on $\wt X$ which is locally acyclic with respect to the composite
map $f=p\ccirc \pi$.
Then, the sheaf $\pi_!\cf'$  is locally acyclic with respect to $p$. 
\end{lem}
\begin{proof}
This follows from a similar property of the vanishing cycles functor.
\end{proof}

Next, 
for each $n\in\Z$, let ${\scr H}^n$ denote the $n$-th perverse cohomology 
sheaf of the complex  $\pi_!\C_{\wt X}$ and, given  a Zariski open subset  $U\sset S$
let $\jmath_U: p\inv(U)\into X$ denote the corresponding open imbedding.
For any closed point $s\in S$ write $\wt X_s:=f\inv(s)$,
resp. $X_s:=p\inv(s)$.

The main result of this subsection, which is equivalent to Proposition \ref{virt},  is as follows

\begin{prop}\label{mylem} Assume that  $\pi$ is a projective morphism,
the composite $f=p\ccirc \pi$, in \eqref{pip}, is a smooth morphism,
and  $U\sset S$ is a Zariski open and dense subset such that  $\pi:\ f\inv(U)\to p\inv(U)$ is a smooth morphism.
Then, we have

\vi ${\scr H}^n=\IC({\scr H}^n|_{p\inv(U)})$ for all $n$, equivalently, the map $\pi$ is {\em virtually small}
in the sense of \cite{Rei}.

\vii Assume in addition that   $\pi$ is generically  finite. Then $\pi$ is small. Furthermore, for any $s\in S$
the map  $\wt X_s\to X_s$, induced by $\pi$, is semi-small.
\end{prop}

\begin{proof} By smooth base change the morphism $\wt f:  C\times_S \wt X\to C$
is  smooth for any smooth curve $g: C\to S$. Furthermore, $\wt X$ and $C\times_S \wt X$ are smooth schemes.
It follows that the sheaf $\C_{\wt X}$ is locally acyclic with respect to $f$.
Hence, the sheaf $\pi_!\C_{\wt X}$ is  locally acyclic with respect to $p$, by Lemma \ref{la-proper}.
Decomposition theorems implies that $\pi_!\C_{\wt X}=\oplus_n\ {\scr H}^n$ and, moreover,
for each $n$ one has ${\scr H}^n=\oplus_\al\ \IC(X_{n,\al}, \cl_{n,\al})$ for some smooth connected locally closed 
and pairwise distinct subvarieties
$X_\al\sset X$ and some local systems $\cl_{n,\al}$ on $X_{n,\al}$.
We deduce that each of the sheaves $\IC(X_{n,\al}, \cl_{n,\al})$ must be 
 locally acyclic with respect to $p$.

To prove (i) we must show that the direct sum $\oplus_\al\ \IC(X_{n,\al}, \cl_{n,\al})$
consists of a single summand $\IC(X_{n,\al_0}, \cl_{n,\al_0})$ where $X_{n,\al_0}$ is a Zariski open and dense subset
of $X$. 
Assume that $X_{n,\al}$ is not dense in $X$. 
The sheaf ${\scr H}^n|_{p\inv(U)}$ being a local system, it follows
that $\bar X_{n,\al}$, the closure of $X_{n,\al}$, is contained in $p\inv(S\sminus U)$. 
We choose
$x\in X_{n,\al}$ such that
$s=p(x)$ is a sufficiently general smooth point of $p(X_{n,\al})$,
a constructible subset of $S$. Further, we choose (as we may) a smooth connected curve $C$ in $S$ such that
$C\cap (S\sminus U)=\{s\}$. 
Let $\bphi_{\wt p, s}$ denote the functor of vanishing cycles for the map $\wt p: C\times_S X\to C$, in diagram \ref{lac}, at 
the point $s\in C\sset S$. In our case, the map
$\wt g: C\times_S X\to X$, in diagram \ref{lac}, is a locally closed imbedding and $\bar X_{n,\al}\sset \{s\}\times_S X=p\inv(s)$.
It follows that $\bphi_{\wt p, s}(\wt g^*\IC(X_{n,\al}, \cl_{n,\al}))=\wt g^*\IC(X_{n,\al}, \cl_{n,\al})$.
The restriction of this sheaf to the point $\wt g\inv(x)$ is isomorphic to $\cl_{n,\al}|_x$, hence nonzero.
This contradicts the local acyclicity of $\IC(X_{n,\al}, \cl_{n,\al})$.
Thus, $\IC(X_{n,\al}, \cl_{n,\al})=0$ and part (i) is proved.

Assume now that the map $\pi$ is generically finite. It follow that the morphism
$\pi:\ (p\ccirc \pi)\inv(U)\to p\inv(U)$ is finite. Hence, $\dim \wt X=\dim X$ and ${\scr H}^n|_{p\inv(U)}=0$ unless
$n=-\dim X$. We conclude that $\pi_!\C_{\wt X}=\IC({\scr H}^{-\dim X}|_{p\inv(U)})$, that is, the map $\pi$ is small.

The last statement in (ii)  is clear for all $s\in U$.
Thus, let $s\in S\sminus U$ and let $\pi_s: \wt X_s\to X_s$ be the restriction of $\pi$.
Proving that $\pi_s$ is semismall amounts to showing that
$(\pi_s)_!\C_{\wt X_s}[\dim  \wt X_s]$ is a perverse sheaf.
 To this end, choose a smooth curve $C$ in $S$ such that
$s\in C$ and $C\sminus \{s\}\sset U$. The map  $\wt g: C\times_S X\to X$ is a locally closed imbedding 
such that $(C\sminus \{s\})\times_S X\sset p\inv(U)$, cf. diagram \ref{lac}.
It follows that the restriction of the sheaf $\wt g^*\pi_!\C_{\wt X}$ to $(C\sminus \{s\})\times_S X=
\wt p\inv(C\sminus \{s\})$ is a local system, say $\cl_{C\sminus \{s\}}$.
 Let $\bpsi_{F, s}$ denote the functor of nearby cycles for a map $F: Y\to C$
at the point $s\in C$. This functor takes perverse sheaves to perverse sheaves.
We deduce
that $\bpsi_{\wt p, s}\cl_{C\sminus \{s\}}[\dim X_s]$ is a perverse sheaf on $X_s$.
On the other hand, let $\wt \pi: C\times_S \wt X\to C\times_S X$ be the map induced by $\pi$.
The map $\wt p\ccirc \wt \pi: \ C\times_S \wt X\to C$ being smooth,
we have $\bpsi_{\wt p\ccirc \wt \pi, s}\C_{C\times_S\wt X}=\C_{\wt X_s}$.
Hence, the proper base change theorem for nearby cycles
yields
 \[\bpsi_{\wt p, s}\cl_{C\sminus \{s\}}=(\pi_s)_!\bpsi_{\wt p\ccirc \wt \pi, s}\C_{C\times_S\wt X}=(\pi_s)_!\C_{\wt X_s}.\]
Thus, $(\pi_s)_!\C_{\wt X_s}[\dim  \wt X_s]$ is a perverse sheaf, as required.
\end{proof}

The proof of part (ii) also yields the following result.
Let $i_s: X_s\into X$ denote the closed imbedding
\begin{cor}\label{mycor} In the setting of Proposition \ref{mylem}(ii), the sheaf
$i_s^*\IC({\scr H}^{-\dim X}|_{p\inv(U)})$
is a perverse sheaf which is isomorphic to $(\pi_s)_!\C_{\wt X_s}[\dim X_s]$.
\end{cor}

\subsection{}\label{positive-sec}
 The following definition was introduced by Hausel and Rodriuez-Villegas \cite{HV2}.

\begin{defn}\label{semipositive} A smooth quasi-projective variety $X$ is called {\em semi-projective} if $X$ is equipped
with a $\GG$-action such that $X^\GG$, the fixed point set, is a projective
vatiety and, moreover,  the $\GG$-action contracts $X$ to $X^\GG$, i.e., for any $x\in X$ the map $\GG\to X,\ x\mto zx$
extends to a regular map $\AA\to X$.
\end{defn}

In the above setting, $X^\GG$ is necessarily smooth and we have $\underset{^{z\to0}}\lim\ z\cdot x\in X^\GG$.

The goal of this section is to give an algebraic proof of  \cite[Corollary 1.3.3]{HV2}, which is stated below.
The original proof, essentially contained in \cite{HV}, works 
over the complex numbers and it
involves a $C^\infty$-argument  and a compatification of $X$ constructed earlier by C. Simpson.
The statement for  \'etale cohomology is then deduced from the corresponding result for the ordinary cohomology
by certain comparison arguments. In the special case of quiver varieties, the theorem was proved earlier
in \cite{CBvdB} where the proof  also used transcendental methods (the hyper-K\"ahler trick) .

\begin{thm}\label{pure_res} 
Let $\h$ be a finite dimensional vector space
equpped with a linear $\GG$-action with positive weights, so the $\GG$-action is a contraction to
$0\in\h$, the origin.  Let $\cx$ be  a semi-projective variety and 
 $f: \cx\to \h$  a smooth $\GG$-equivariant morphism.
 Then, for any  $h\in\h$ and $k\geq 0$, we have

\vi The  restriction map $H^\hdot(\cx)\to H^\hdot(X_h)$, where $\cx_h:={f}\inv(h)$,
is an isomorphism and the  cohomology group $H^k(\cx_h)$ is pure of weight $k/2$.

\vii  The sheaf
$\dis R^k{f}_*\CC_{ \cx}$ is a constant sheaf on $\h$.
\end{thm}

\begin{proof} Following \cite{HV}, for a $\GG$-variety $Y$
one defines $\text{Core}(Y)$, the {\em core} of $Y$, to be the
set of all points $y\in Y$ such that $\underset{^{z\to\infty}}\lim\ z\cdot y$ exists.
It is clear that $\cx_0^\GG=\cx^\GG$ and
$\text{Core}(\cx_0)=\text{Core}(\cx)$. Hence,
$X_0$ is a semi-projective variety and
one has restriction maps $H^\hdot(\cx)\to H^\hdot(\cx_0)\to H^\hdot(\text{Core}(\cx_0))$.
It is not difficult to show, see e.g.   \cite[Theorem 1.3.1]{HV}, that the second map, as well as the composite of the two maps, is an isomorphism.
It follows that the first map is an isomorphism.
Hence, the cohomology groups of $\cx$, resp. $\cx_0$, are isomorphic to those of $\cx^\GG$,
in particular, they are pure, \cite[Corollary 1.3.2]{HV}.

Next, fix  $h\neq 0$ and put $\cx':=\AA\times_\h \cx$, where the fiber product is taken with respect to the
action map $\AA\to\h,\ z\mto z(h)$. We let $\GG$ act on $\cx'$ via dilations on the first factor.
Thus, we have a commutative diagram of  $\GG$-equivariant maps
$$
\xymatrix{
pr_1\inv(0)     \ar@{^{(}->}[rr]^<>(0.5){i}
\ar[d]&&\cx'\ar[d]^<>(0.5){pr_1}\ar[rr]^<>(0.5){pr_2}&& \cx\ar[d]^<>(0.5){f}\\
\{0\}\ar@{^{(}->}[rr]&&\AA\ar[rr]^<>(0.5){z\ \mto\ z(h)}&& \h
}
$$
Note that $pr_1\inv(0) \cong \cx_0$ and
the $\GG$-action on $\cx'$ is a contraction.
We deduce similarly to the above that each of  the pull-back morphisms 
$H^\hdot(\cx)\xrightarrow{pr_2^*}H^\hdot(\cx')\xrightarrow{i^*}H^\hdot(\cx_0)$
is an isomorphism. In particular, the cohomology of $\cx'$ is pure.
Note further that  $pr_1\inv(1)\cong \cx_h$ and
the $\GG$-action  provides
a $\GG$-equivariant isomorphism
$\GG \times \cx'_h\  \iso\ \pr_1\inv(\GG)$.
Thus, we have a  diagram
\beq{kun}
\xymatrix{
\cx_0\ 
\ar@{^{(}->}[r]^<>(0.5){{i}}&
\  \cx'\ & \
\cx'\sminus \cx_0\ \ar@{_{(}->}[l]_<>(0.5){j}
\ar@{=}[r]& \ \GG \times \cx_h,
}
\eeq
where $i$  and $j$ are  a closed and an  open imbedding,
respectively.

The cohomology groups $H^0(\GG)=\C(0)$
and  $H^1(\GG)=\C(2)$ are
1-dimensional vector spaces
which have
weights $0$ and $2$, respectively.
Hence, one has the following
K\"unneth  decomposition
\[H^\hdot(\cx'\sminus \cx_0)\cong [H^0(\GG)\o H^\hdot(\cx_h)]\ \ 
\bplus\ \ [H^1(\GG)\o H^{\hdot-1}(\cx_h)]\ =\
H^\hdot(\cx_h)\
\bplus\ H^{\hdot-1}(\cx_h)(2).\]
Using the K\"unneth  decomposition, the long exact sequence 
associated with  diagram \eqref{kun},
takes the 
following form,
\beq{pure_k}
\ldots\to 
\xymatrix{
H^{k-1}(\cx_0)(1)\ar[r]^<>(0.5){{i}_!} &
  H^k(\cx')\ar[r]^<>(0.5){j^*}&
\ H^k(\cx_h)\
\bplus\ H^{k-1}(\cx_h)(2) \ar[r]^<>(0.5){[1]}&
 H^{k+1}(\cx_0)
}\to\ldots
\eeq

Let $\gr_\ww H^\hdot(-)$ denote
an associated graded term of weight $\ww\in\Z$
in the weight filtration on the cohomology. 
Applying the functor $\gr_\ww(-)$, which is an exact functor,
to  \eqref{pure_k}, one obtains an exact sequence
of spaces of  weight $\ww$.
Thanks to the purity result proved earlier, we know
that $\gr_\ww H^n(\cx_0)=\gr_\ww H^n(\cx')=0$,
whenever $\ww\neq n$. Hence,
 for any $\ww\geq k+2$,  the fragment of the resulting
exact sequence of spaces of  weight $\ww$
that corresponds to  \eqref{pure_k} reads
\beq{pure2}
\ldots \to \xymatrix{0
\ar[r]^<>(0.5){{i}_!} &0
\ar[r]^<>(0.5){j^*}&
 \gr_\ww
H^k(\cx_h)\ \bplus\  
[\gr_{\ww-2}H^{k-1}(\cx_h)](2)\ar[r]^<>(0.5){[1]}&0}\to\ldots
\eeq

From \eqref{pure2}, we see that
$\gr_{\ww-2}H^{k-1}(\cx_h)=0$. It follows that the group
$\gr_\ww H^k(\cx_h)$ vanishes
for any pair of integers $\ww,k$, such that $\ww >k$.
On the other hand,
since $\cx_h$ is smooth,
we also have $\gr_\ww H^k(\cx_h)=0$ for all $\ww<k$.
Thus, we conclude that, for each $k$, the group
$H^{k}(\cx_h)$  is pure of
weight $k$. 

The purity implies that  the
long exact sequence  \eqref{pure_k}   breaks up into a direct sum 
$\oplus_{\ww\in\Z}\ E^{(\ww)}$, of long exact sequences  $E^{(\ww)},\ \ww\in\Z$, 
such that, for any $\ww$, all  terms in   $E^{(\ww)}$ are
pure of weight ~$\ww$. Furthermore, one finds  that
each of  these
long exact sequences  actually splits further into length two
exact sequences. Specifically, for $k\in\Z$, the long exact sequence
$E^{(k)}$
reduces, effectively, to the following
pair of isomorphisms:
\beq{ji}
\jmath^*:\ 
 H^k(\cx')\ \iso\ H^k(\cx_h)
\quad\text{and}\quad 
{i}_!:\ H^{k-1}(\cx_0)(2)\ 
\iso \ H^{k+1}(\cx').
\eeq
Here, the isomorphism $\jmath^*$ is induced by the imbedding
$\jmath:\ \cx_h=\{1\}\times \cx_h\into \cx'$.

To complete the proof of part (i) of the theorem, we factor
the restriction map $H^\hdot(\cx)\to H^\hdot(\cx_h)$ as a composition
$\dis H^\hdot(\cx)\xrightarrow{pr_2^*} H^\hdot(\cx')\xrightarrow{\jmath^*}H^\hdot(\cx_h)$.
We have shown that each of these maps is an isomorphism, proving (i).

To prove (ii), write 
$i_h: \{h\}\into \h$,
resp. $\imath_h: \cx_h\into\cx$,
for
the corresponding closed imbeddings and 
 $p:\cx\times \h\onto \h$ for
 the second projection.
Also, define a map $\eps:\ \cx\into
\cx\times \h$ by the assignment
$x\mto (x,{f}(x))$. Thus, $\eps$ is a closed embedding
via the graph of  ${f}$, so one has a factorization ${f}=p\ccirc\eps$.

Now, we begin the proof by noting
 that each cohomology sheaf $R^k p_*\C_{\cx\times \h}$ is 
a constant sheaf,
by the K\"unneth formula.
Next, we observe that
 there is a canonical morphism
\[u:\  p_*\C_{\cx\times \h}\too  p_*(\eps_*
\eps^*\C_{\cx\times \h})= {f}_*\C_{\cx}.\]
Thus, we would be done provided
we can prove  that the morphism $u$ is, in fact, an isomorphism.
We will prove this by
showing that,
for every $h\in \h$,
the  morphism $i_h^!(u):\, i_h^!( p_*\C_{\cx\times \h})\to
i_h^!({f}_*\C_{\cx})$,
induced by $u$,
is an isomorphism. This is known to be sufficient
to conclude that $u$ is an isomorphism,

The argument below involves 
the following  diagram, where ${f}_h:={f}|_{\cx_h}$ and $p_h$ stands for a constant map,
\beq{sqp}
\xymatrix{
\cx_{h_{}}\ \ar@{_{(}->}[d]^<>(0.5){\imath_h}
\ar[rrr]_<>(0.5){\eps|_{\cx_h}}\ar@/^1.3pc/[rrrrrr]^<>(0.35){{f}_h}
&&&\ {\cx\times \{h\}_{}}\ \ar@{_{(}->}[d]^<>(0.5){\Id\times i_h}
\ar[rrr]_<>(0.5){p_h}&&&\  {\{h\}_{}}\ \ar@{_{(}->}[d]^<>(0.5){i_h}\\
\cx\ \ar[rrr]^<>(0.5){\eps}\ar@/_1.3pc/[rrrrrr]_<>(0.35){{f}} &&&\
\cx\times \h\ \ar[rrr]^<>(0.5){p} &&&\ \h.
}
\eeq

It is clear that  all commutative squares
in the diagram are cartesian. Also, 
${f}$ is a smooth morphism, so $\cx_h$
is a smooth subvariety  of codimension
$n:=\dim\h$ in $\cx$. Therefore, applying proper base change
to the above diagram one gets the following canonical
isomorphisms:
\begin{align*}
i_h^!( p_*\C_{\cx\times \h})\ &=
( p_h)_*(\Id\times i_h)^!\C_{\cx\times \h}=
( p_h)_*(\C_{\cx}[2n])=
H^{\hdot+2n}(\cx);\\
i_h^!( {f}_*\C_{\cx})\ & =\ 
( {f}_h)_*(\imath_h^!\C_{\cx})\ =\ 
( {f}_h)_*(\C_{\cx_h}[2n])\ =\ 
H^{\hdot+2n}(\cx_h).
\end{align*}

Thus, the morphism $i_h^!(u)$
goes, via base change, to a morphism
$H^{\hdot+n}(\cx)\to H^{\hdot+n}(\cx_h)$.
One can check that the latter morphism 
is the  restriction morphism $\imath^!_h$
induced by the imbedding $\imath_h: \cx_h\into\cx$.
Furthermore, the first isomorphism in \eqref{ji} implies
that $\imath^!_h$ is an isomorphism,
for any $h\in\h$.
It follows that the morphism $u$ is an isomorphism,
 completing the proof of
the theorem.
\end{proof}
\section{Appendix B: An application to the Calogero-Moser variety of type $A$}\label{CM-sec}
The goal of this Appendix is to illustrate the relation between indecomposable representations and the geometry of
the moment map 
in a special case of the Calogero-Moser quiver.

\subsection{}
Fix a quiver $Q$ with vertex set $I$ and a dimension vector $\bv\in\Z^I$.
Let $\bar{Q}$ be the double of the quiver $Q$ and $\mathcal M_O$ be as in the formulation of Theorem \ref{A-ind}. 
Further, let $(\eta_i)_\ii\in\k^I$ and write  $\eta=(\eta_i\cdot \Tr_{v_i})_\ii\in\g^*_\bv$.
One has  a diagram
\beq{pqr}
\xymatrix{
\mu_\bv^{-1}(\eta)\ar@{^{(}->}[r]\ar@{->>}[d]^<>(0.5){q}& \Rep_\bv\bar{Q}\ \ar@{->>}[r] & \Rep_\bv Q\ar@{->>}[d]^<>(0.5){\fw}\\
\mu_\bv^{-1}(\eta)\dsl G_\bv\ar[rr]^<>(0.5){\pi}&& \Rep_\bv\dsl G_\bv
}
\eeq
Here the first map in the top row  is the natural inclusion and
the second map  is a restriction of representions of $\bar Q$ to $Q$, viewed as a subquiver of $\bar Q$.
Let $p$ denote the composite of these two maps.
Further, for any $x\in \Rep_\bv\dsl G_\bv$ let $\op{Ind(x)}$ be the set of  $G_\bv$-orbits in  $\fw\inv(x)$
formed by the  indecomposable representations.

It was shown by  Crawley-Boevey that in the case where   $\bv$ is indivisible and the $I$-tuple $(\eta_i)$ is sufficiently general
the image of $p$ is equal to
the  set of indecomposable 
representations.
Moreover, 
$\eta$ is a regular value of the moment map $\mu: \Rep_\bv\bar Q\to\g^*_\bv$.
In particular, $\mu_\bv^{-1}(\lambda)$ is a smooth scheme, the $G_\bv$-action on this scheme is free, and
the map  $\mu_\bv^{-1}(\lambda)\to \mu_\bv^{-1}(\lambda)\dsl G_\bv$ is a geometric quotient.



The following result is essentially \cite[Proposition 2.2.1]{CBvdB}.

\begin{prop} \label{nochar} Let $\bv$ be an indivisible dimension vector and assume $\eta$ is sufficiently general. Then we have 

\vi For any $G_\bv$-orbit $y\sset {\rm Rep}_\bv Q$, of indecomposable representations,
we have $q(p\inv(y))\cong \AA^{\frac{1}{2}\dim\mm_O}$, where $\mm_O$ is as in Theorem \ref{A-ind}.

\vii For any $\F_q$-rational point $x\in{\rm Rep}_\bv Q \dsl  G_\bv$,  the number,  $\#\op{Ind(x)}({\mathbb F}_q)$, of isomorphism  classes of
absolutely idecomposable representations defined over $\F_q$, is given by the formula


$$\#\op{Ind(x)}({\mathbb F}_q)= q^{\frac{1}{2}\dim \cM_O} \cdot \ltr H^\hdot_c(\pi^{-1}(x)) .$$

\viii  If the the set $\op{Ind(x)}$ is finite then we have a partition 
$\pi^{-1}(x)=\sqcup_{y\in \op{Ind(x)}}\ q(p\inv(y))$ into a disjoint union of finitely many 
locally closed subvarieties, each isomorphic to $\AA^{\frac{1}{2}\dim \cM_O} $.
\end{prop}

\begin{proof}

\vi According to Lemma 3.1 in \cite{CB1}, for $y \in {\rm Rep}_\bv Q$, which is in the image of $p \circ i$, the fiber $(p \circ i)^{-1}(y)$ is 
a quotient of a vector space $V$ isomorphic to  $\Ext_{\k Q}^1(y,y)^*$ by an affine-linear  free action on $V$ of a unipotent group $U$,
where $U=G_y$. The latter group is unipotent since $y$ is absolutely indecomposable. representation $y$.
Any such  quotient $V/U$ is known to be isomorphic  an affine space of dimension
$\dim V-\dim U$. Thus, we have $(p \circ i)^{-1}(y)\cong \AA^m$ where $m=\dim \Ext_{\k Q}^1(y,y)-\dim Aut(y)=
\dim \Rep_\bv Q-\dim G_\bv=\half \dim\mm_O$. This proves (i). In particular, for  any $\F_q$-rational point $x\in{\rm Rep}_\bv Q \dsl  G_\bv$,  we have
$$\ltr H^\hdot_c(\pi^{-1}(x)) =\#\pi^{-1}(x)({\mathbb F}_q) =q^{\frac{1}{2}\dim \cM_O} \cdot \# r\inv(x)(\F_q).$$
Part (ii) follows.
In the case where the number of somorphism classes of absolutely indecomposable representations is finite
we deduce from (i) that the scheme $\pi^{-1}(x)$ is a disjoint union of finitely many pieces each of which is isomorphic to 
$\AA^{\frac{1}{2}\dim\mm_O}$.
\end{proof}

\subsection{}\label{CM}
Let $Q$ be the Calogero-Moser quiver $Q$, i.e. a quiver with  two vertices, a loop at one of the vertices, and an edge joining the vertices, which we fix to point toward the vertex with the loop. We fix the the dimension vector $(n,1)$, where $n$ labels the vertex with the loop. 
A representation with this dimension vector is a pair $(u,v)$, where $u \in {\mathfrak g \mathfrak l}(V), v \in V$, and $V$ is an $n$-dimensional vector space. 
It is known that the map that assigns to $(u,v)\in \gl_n\times V$ the unordered $n$-tuple of eigenvalues of
the operator $u$ yields an isomorphism 
\[\Rep_{n,1}(Q)\dsl \GL_n=(\gl_n\times V)\dsl \GL(V)\iso \gl(V)\dsl \GL(V)\cong\AA^n/\si_n.\]

Let $\Nil=\Nil({\mathfrak g \mathfrak l}(V))$.
Conjugacy classes in $\Nil$ are parametrized by partitions of $n$ according to Jordan normal form.
Given a partition $\nu=(\nu_1\geq \nu_2\geq\ldots)$,
write $|\nu|=\sum\ \nu_i$, resp. $\ell(\nu)$ for the number of parts, i.e. 
$\ell(\nu)=\#\{i\mid \nu_i\neq 0\}$.
Let $\Sigma$ be the set  of pairs, $(\lambda, \mu)$, of partitions  such that 
\begin{enumerate}
\item
$|\la|+|\mu|=n$;

\item
The parts of each of the partitions $\lambda$ and $\mu$ strictly decrease;

\item One has that either  $\ell(\la)=\ell(\mu)$ or  $\ell(\la)=\ell(\mu)+1$.
\end{enumerate}

The  definition of  the set $\Sigma$ and the lemma  below were also obtained in a recent paper by Bellamy and Boos, see \cite[Lemma 5.3]{BeBo},
where the authors use the notation $P_F(n)$ for our $\Sigma$.

\begin{lem} \label{pairorbits} For any field $\k$, there is a natural bijection $(\la,\mu)\mto O_{\la,\mu}$
between
the set $\Sigma$ and the set of  $\k$-rational $\GL(V)$-orbits in $\Nil\times V$ which consist of absolutely 
indecomposable representations such that for $(u,v)\in O_{\la,\mu}$
the   Jordan normal form of $u$ is given by the partition
with parts $\lambda_i + \mu_i=\nu_i$, i.e. the size of the $i$-th block of the nilpotent $u$
equals $\lambda_i + \mu_i=\nu_i$.
\end{lem}

\begin{proof}

First we recall the description of the $GL(V)$-orbits in the space of pairs $(u,v) \in {\mathcal N} \times V$, where $V$ is a vector space and $\mathcal N \subset {\mathfrak g \mathfrak l}(V)$ is the nilpotent cone. This classification is the result of Theorem 1 in \cite{Tr} or Proposition 2.3 in \cite{AH}, and is described as follows (we follow \cite{Tr}). The orbits in $\mathcal N \times V$ are in bijection with pairs of partitions $\lambda, \mu$ such that $|\lambda|+|\mu|=\sum \lambda_i + \sum \mu_i = n$. This bijection has the property that the type of the nilpotent $u$ is equal to $\lambda+\mu=(\lambda_1 + \mu_1, \lambda_2 + \mu_2, ...)$, and is constructed in the following way. Given a pair of partitions $(\lambda, \mu)$ as above, let $\nu=\lambda+\mu$, and let $u$ be a nilpotent of type $\nu$. Let $D_\nu$ be the set of boxes of the Young diagram $\nu$, so that $D_\nu=\{ (i,j):1 \leq j \leq \nu_i \}$. We choose a basis of $V$ with basis vectors $e_{i,j}$, where $(i,j) \in D_\nu$ such that $ue_{i,j}=e_{i,j-1}$ for $2 \leq j \leq \nu_i$ and $ue_{i,1}=0$. Let $v=\sum e_{i,\lambda_i}$, where we put $e_{i,0}=0$.

To single out the orbits consisting of indecomposable representations, we start by recalling another bijection in Proposition-Construction 1 in \cite{Tr} between the orbits in $\mathcal N \times V$ and pairs of partitions $\nu, \theta$, where $\nu$ is the type of the nilpotent $u$ and $\theta$ is the type of the nilpotent $u$ acting on the quotient vector space $V/\langle v, uv, u^2 v, ... \rangle$. Note that $\nu$ and $\theta$ have the property that their columns with the same number differ in length at most by one. 

We observe that a representation is reducible if and only if $\nu$ and $\theta$ have a row of the same length. Indeed, a representation is reducible if and only if $u$ has a Jordan block in its Jordan normal form such that the component of $v$ is zero in a basis of generalized eigenvectors for $u$, which translates into the condition that $\nu$ and $\theta$ have a row of the same length. Because of the equalities $\lambda_i - \lambda_{i+1} = \nu_i - \theta_i$, $\mu_i - \mu_{i+1} = \nu_i - \theta_{i+1}$ we obtain that $\lambda$ and $\mu$ have to have distinct parts. 

Finally we remark that because of the explicit description of indecomposables, which by the above are indexed by the same data over any field, we have that absolutely indecomposable representations coincide with indecomposable ones.
\end{proof}

An element of $\Rep_{n,1}\bar Q$ is a quadruple $(B_1,B_2,v,w) \in \gl(V) \times \gl(V) \times V \times V^*$.
The moment map $\mu: \Rep_{n,1}\bar Q\to \gl_n$ is given by the formula
$\mu(B_1,B_2,v,w) =[B_1,B_2]+vw$.
It is known that for any  $c\neq 0$ the element $c\cdot \Id \in\gl(V)$ is sufficiently general.
$\Id\in\gl(V)$ is a regular value of the moment map. Thus, $\Id\in\gl(V)$ is a regular value of the moment map
and $\mm=\mu\inv(\Id)\dsl GL_n$ is a smooth
symplectic variety of dimension $2n$,
called the {\em Calogero-Moser space}. 
We use the notation of diagram \eqref{pqr} and put
$$
\mm_0:=\mu\inv(\Id)\cap \pi\inv(0)=\{(B_1,B_2,v,w) \in \Nil \times \gl(V) \times V \times V^*\mid
[B_1,B_2]=I+vw=\Id\}\dsl \GL(V).
$$

The above variety has been studied by G. Wilson \cite{Wi}.
He constructed a parametrization of connected components of the
variety $\mm_0$ by partitions of $n$ and
showed that each connected component  of $\mm_0$ is isomorphic
to $\AA^n$.
On the other hand,  the number of $\GL(V)$-orbits in
$\fw\inv(0)=\Nil\times V$ being finite, 
Proposition \ref{nochar}(iii) is applicable and it yields an alternative proof of the
latter isomorphism.
Furthermore, Proposition \ref{nochar}(iii) provides a natural parametrization
of connected components of  $\mm_0$ by the $\GL(V)$-orbits  of indecomposables in
$\Nil\times V$. Applying  Lemma \ref{pairorbits} we obtain
a  parametrization
of connected components of  $\mm_0$ by  the elements of the set $\Sigma$.

The latter parametrization is related to the one given by Wilson via  a 
bijection between the set $\Sigma$ and the set of partitions of $n$, viewed as Young diagrams with
$n$ boxes. The bijection  amounts, essentially, to describing Young diagrams 
by its  `Frobenius form' ,  Specifically, let $(\la,\mu)\in\Sigma$.
First, shift the rows of the Young diagram $\mu$ so that the $i$-th row is shifted by $i$ units to the right (where the first row is the longest, 
the second row is the second longest, and so on). Next, shift the columns of the Young diagram $\lambda$ so that the $i$-th column is hifted by $i-1$ units up (where the first column is the longest, the second column is the second longest, and so on). Now, the bijection in question  is defined by  assigning  to  $(\la,\mu)$
the Young diagram with $n$ boxes obtained by gluing 
the resulting shifts of $\lambda$ and $\mu$ along the line that becomes 
the main diagonal of the new  diagram.
Lemma \ref{pairorbits} ensures that the diagram thus defined has the shape of a
Young diagram.

\begin{rem} Frobenius form was also mentioned by Wilson, see \cite[p.28]{Wi}.
Our Lemma  \ref{pairorbits} gives a geometric interpretation, in terms of indecomposable 
quiver representations, of somewhat mysterious inequalities in Wilson's work \cite[Lemma 6.9]{Wi}.
\end{rem}

\begin{rem} 
One can generalize\footnote{This has also been done in \cite{BeBo}.} the case of the Calogero-Moser quiver to a cyclic quiver with one additional vertex connected by an edge to one of the vertices of the cycle, where the component of the dimension vector is equal to $n$ at each of the vertices of the cycle and $1$ at the additional vertex. Below we describe the bijection which is the content of Proposition \ref{nochar} for the cyclic quiver (i.e. the analog for the cyclic quiver of the bijection of  Lemma \ref{pairorbits} for the Calogero-Moser quiver). 

The torus-fixed points in a quiver variety for the cyclic quiver are described, for example, in \cite{Ne}, and the description is as follows. Let $(v_1,...,v_n)$ be the dimension vector at the vertices of the cycle so that $\dim V_k = v_k$ for the standard vector spaces $V_k$ associated to the quiver variety (note that we will take $k$ modulo $n$ below), and let the framing be given by $1$ at the first vertex and $0$ at all the other vertices of the cycle. Then the torus-fixed points in the corresponding quiver variety are labelled by Young diagrams in which each box is colored in one of the $n$ colors corresponding to the content of the box (i.e. the difference of the $x$- and the $y$-coordinates of the center of the box, when the Young diagram is placed in the first quadrant in the standard way) in such a way that in total there are exactly $v_i$ boxes of the $i$-th color in the Young diagram.  

Now we describe the other side of the bijection of   Lemma \ref{pairorbits}. The absolutely indecomposable representations of the undoubled quiver in this case coincide with the indecomposable ones, as in the above case of the Calogero-Moser quiver. The indecomposable (nilpotent when one goes around the circle enough times) representations of our undoubled quiver (with an oriented cycle and the remaining arrow pointing inside) can be described as follows. 

We call a "snake" a nilpotent representation of the cyclic quiver which consists of a string of vectors $u_1,...,u_N$ where $u_k$ lies in $V_{k \ {\rm mod} \ n}$ and $u_k$ is mapped to $u_{k+1}$ by the arrow $V_{k \ {\rm mod} \ n} \to V_{(k+1) \ {\rm mod} \ n}$) with a marked point in each snake at a place with residue $1$ modulo $n$ (i.e. the place mod $n$ corresponding to where the nonzero component of the framing vector lies). Then an indecomposable representation of our undoubled quiver is labelled by a collection of "snakes," so that when all snakes are unrolled and placed on the coordinate line so that the marked place of each snake is at the origin, any two snakes have the following property: one of them lies strictly within the other inside of the coordinate line (i.e. the longer is the head of a snake, the longer is its tail). 

The bijection of Lemma \ref{pairorbits} of the above data describing an indecomposable representation of the undoubled quiver with the data describing a torus-fixed point in the quiver variety for the cyclic quiver is obtained by taking each snake, bending it at the marked point to form the right angle, and placing snakes which have been bent in this way on top of each other to form the Young diagram. 
\end{rem}

\end{document}